\documentclass[12pt,letterpaper]{report}

\newcommand{\thesistitle}{Two inquiries about finite groups and well-behaved quotients}
\newcommand{\thesisauthor}{Ben Blum-Smith}
\newcommand{\thesisadvisora}{Yuri Tschinkel}
\newcommand{\thesisadvisorb}{Fedor Bogomolov}
\newcommand{\graddate}{May 2017}

\usepackage{setspace}

\usepackage{lpic}
\usepackage{amssymb,amsmath}
\usepackage{pifont}
\usepackage{relsize}
\usepackage{color}
\usepackage[usenames,dvipsnames,svgnames,table]{xcolor}
\usepackage{float}
\usepackage{rotating}
\usepackage{verbatim}
\usepackage{letltxmacro}
\usepackage[normalem]{ulem}
\usepackage{bbm}
\usepackage{mathtools}
\usepackage{amscd}
\usepackage{mathrsfs}
\allowdisplaybreaks[4]
\usepackage{lineno}
\usepackage{tikz}
\usepackage{hyperref}
\usepackage{graphicx} % use this line instead of the above to suppress graphics in draft copies
\usepackage{indentfirst}

\usepackage{amsthm} % facilities for theorem-like environments
\usepackage{amsfonts}
\usepackage{geometry}

\usepackage{glossaries}

\makeglossaries

\newglossaryentry{naturals}
{
	name = {\ensuremath{\mathbb{N}}},
	description = {Nonnegative integers. Thus an $\mathbb{N}$-graded ring has a degree zero piece},
	sort = N
}

\newglossaryentry{[n]}
{
	name = {\ensuremath{[n]}},
	description = {The set $\{1,\dots,n\}$},
	sort = n
}

\newglossaryentry{spec}
{
	name = {\ensuremath{\operatorname{Spec}}},
	description = {The prime spectrum of a ring, i.e. the set of prime ideals, topologized by the Zariski topology},
	sort = Spec
}

\newglossaryentry{coprod}
{
	name = {\ensuremath{\coprod}},
	description = {The coproduct in the category of sets, i.e., the disjoint union},
	sort = coprod,
	symbol = {\ensuremath{\coprod}}
}

\newglossaryentry{height}
{
	name = {\ensuremath{\operatorname{ht}}},
	description = {The height of a prime ideal, i.e. the maximum length of a chain of prime ideals descending from the given one},
	sort = ht
}

\newglossaryentry{partitions}
{
	name = {\ensuremath{\mathscr{P}}},
	description = {The monoid of partitions with at most $n$ parts, describing the shapes of monomials in $R = A[x_1,\dots,x_n]$},
	sort = P
}

\newglossaryentry{conjparts}
{
	name = {\ensuremath{\overline{\mathscr{P}}}},
	description = {The monoid of partitions with parts of size at most $n$, describing the fine grades of elements in the Stanley-Reisner ring $S = A[B_n\setminus\emptyset]$},
	sort = P'
}

\newglossaryentry{Bn}
{
	name = {\ensuremath{B_n}},
	description = {The boolean algebra of subsets of $[n]$, seen as a poset},
	sort = Bn
}

\newglossaryentry{garsia}
{
	name = {\ensuremath{\mathscr{G}}},
	description = {The Garsia map, defined in \ref{def:garsiamap}},
	sort = garsia
}

\newglossaryentry{Grr}
{
	name = {\ensuremath{G_{rr}}},
	description = {The subgroup of a transformation group $G$ that is generated by its reflections and rotations. In the main case, $G$ is a permutation group, and this is the subgroup generated by transpositions, double transpositions, and 3-cycles},
	sort = Grr
}

\renewcommand{\emptyset}{\ensuremath{\varnothing}}

\newcommand{\C}{\ensuremath{\mathbb{C}}}
\newcommand{\R}{\ensuremath{\mathbb{R}}}
\newcommand{\Q}{\ensuremath{\mathbb{Q}}}
\newcommand{\Z}{\ensuremath{\mathbb{Z}}}
\newcommand{\N}{\gls{naturals}}

\newcommand{\F}{\ensuremath{\mathbb{F}}}
\newcommand{\A}{\ensuremath{\mathbb{A}}}
\newcommand{\Proj}{\ensuremath{\mathbb{P}}}

\newcommand{\primep}{\ensuremath{\mathfrak{p}}}
\newcommand{\primeq}{\ensuremath{\mathfrak{q}}}
\newcommand{\maxm}{\ensuremath{\mathfrak{m}}}
\newcommand{\maxn}{\ensuremath{\mathfrak{n}}}

\newcommand{\depth}{\ensuremath{\operatorname{depth}}}
\newcommand{\height}{\gls{height}}
\newcommand{\Bn}{\gls{Bn}}
\newcommand{\scrP}{\gls{partitions}}
\newcommand{\scrPconj}{\gls{conjparts}}
\newcommand{\garsia}{\gls{garsia}}
\newcommand{\Grr}{\gls{Grr}}

\newcommand{\Spec}{\gls{spec}}
\newcommand{\MaxSpec}{\ensuremath{\operatorname{MaxSpec}}}

\newcommand{\Aut}{\operatorname{Aut}}

\newcommand{\Tr}{\operatorname{Tr}}
\newcommand{\Ind}{\operatorname{Ind}}
\newcommand{\Res}{\operatorname{Res}}

\newcommand{\gr}{\operatorname{gr}}
\newcommand{\lk}{\operatorname{lk}}
\newcommand{\rk}{\operatorname{rk}}
\theoremstyle{plain}
\newtheorem{thm}{Theorem}[section]
\newtheorem{prop}[thm]{Proposition}
\newtheorem{lemma}[thm]{Lemma}
\newtheorem{cor}[thm]{Corollary}
\newtheorem{conj}[thm]{Conjecture}
\newtheorem*{introthm}{Theorem}

\theoremstyle{definition}
\newtheorem{definition}[thm]{Definition}
\newtheorem{remark}[thm]{Remark}
\newtheorem{question}[thm]{Question}
\newtheorem{example}[thm]{Example}
\newtheorem{notation}[thm]{Notation}
\newtheorem{theme}{Theme}

\begin{document}

%%%%%% Title page %%%%%%%%%%%
%% Sets page numbering to "roman style" i, ii, iii, iv, etc:
\pagenumbering{roman}

% No numbering in the title page:
\thispagestyle{empty}

\begin{center}
  {\large\textbf{\thesistitle}}
  \vspace{.7in}
	
  by
  \vspace{.7in}

  \thesisauthor
  \vfill

\begin{doublespace}
  A dissertation submitted in partial fulfillment\\
  of the requirements for the degree of\\
  Doctor of Philosophy\\
  Department of Mathematics\\
  Courant Institute of Mathematical Sciences\\
  New York University\\
  \graddate
\end{doublespace}
\end{center}
\vfill

\noindent\makebox[\textwidth]{\hfill\makebox[2.5in]{\hrulefill}}\\
\makebox[\textwidth]{\hfill\makebox[2.5in]{\hfill\thesisadvisora\hfill}}\\
\\
\noindent\makebox[\textwidth]{\hfill\makebox[2.5in]{\hrulefill}}\\
\makebox[\textwidth]{\hfill\makebox[2.5in]{\hfill\thesisadvisorb\hfill}}
\newpage
%%%%%%%%%%%%% Blank page %%%%%%%%%%%%%%%%%%

 %\doublespacing % requires package setspace, invoked above
 
\thispagestyle{empty}
\vspace*{0in}
\begin{center}
{\copyright{} \thesisauthor}\\
All Rights Reserved, 2017
\end{center}

\newpage

%%%%%%%%%%%%%% Dedication %%%%%%%%%%%%%%%%%
\section*{Dedication}\addcontentsline{toc}{section}{Dedication}
\vspace*{\fill}
\begin{center}
{\it dedicated to Diane and Ben, who never doubted where this was headed}
\end{center}
\vfill
\newpage
%%%%%%%%%%%%%% Acknowledgements %%%%%%%%%%%%
\section*{Acknowledgements}\addcontentsline{toc}{section}{Acknowledgements}
% !TEX root = ./thesis.tex

It gives me tremendous pleasure to acknowledge the many people who helped this thesis come into being.

My advisors Yuri Tschinkel and Fedor Bogomolov were each critical. Yuri provided reliable professional advice since the beginning of my studies and throughout the research process, as well as a certain helpful impatience with my tendency to try to understand everything before beginning. This empowered me to get on the court before I felt I was ready, which I then discovered is how you get ready. 

Fedor was incredibly generous with me throughout my time at NYU. He proposed several problems for me, each of which I learned a great deal from, and one of which has become chapter \ref{ch:absing}. He also spent {\em a lot} of time with me and with other students during my years here. I am still amazed at how he made a habit of dropping by my office after class to follow up on questions I had asked during lecture. This interest in his students goes hand in hand with an attitude toward mathematics full of humility and wonder, which is a delight to be around.

I owe a very special debt of gratitude to Sophie Marques. Sophie was the first mathematician to collaborate with me seriously on a project aimed at advancing the frontier. Although that project has not ended up in this thesis, it provided me with my first real research experience: working on something that fits into an active research program, and having a research need inform my choices about what to learn. Sophie was always generous with her knowledge while also seeking and valuing my input. In addition to all of this, she has given me extremely helpful, detailed feedback on an earlier draft of this thesis.

The research in chapter \ref{ch:invar} was bookended by encounters with two recent PhD theses, and correspondence with their authors was critical to its success. Owen Biesel was the source of the project itself. His work on $G$-closures of rings was the starting point: I knew I wanted to follow it up somehow. In the ensuing exchange, he posed the question to which chapter \ref{ch:invar} is addressed. At the other end of the process, after forming a topological conjecture about sphere quotients and then spending a summer trying to prove it, a challenge for which I was woefully underprepared, I learned of the work of Christian Lange, which supplied everything I needed. Like Owen, Christian was a generous correspondent, patiently answering my questions about his work, and inquiring with curiosity about my application of it. As Owen helped me find the question, Christian helped me find the answer.

I am also delighted for the opportunity to thank the many mathematicians who shared their expertise with me over the course of the research detailed here. Nothing could have made me feel more welcomed into the worldwide mathematical community. Special thanks go to Victor Reiner, Gregor Kemper, Allan Steel, John Voight, Cory Colbert, Mohamed Omar, Robert Young, Josephine Yu, David Eisenbud, Naoki Terai, Sylvain Cappell, and Harold Edwards.

My fellow students at NYU have grounded the process of becoming mathematicians in a sense of in-it-togetherness, and I have learned a great deal from them about all parts of mathematics. I would like to mention by name Edgar Costa, Jin Qian, Yash Jhaveri, Carlos Amendola Ceron, Or Hershkovits, Alex Blumenthal, Aukosh Jagganath, Manas Rachh, Mihai Nica, Ian Tabasco, Joey MacDonald, Rachel Hodos, Ethan O'Brien, Lukas Koehler, Jordan Thomas, and Guillaume Dubach. A special acknowledgement goes to my officemate Federico Buonerba, whose area is just close enough to mine that we often had questions to ask each other, and whose point of view is just different enough that we never failed to learn something from the answer.

Finally it gives me great joy to acknowledge my friends and family. I grew up with four deeply thoughtful and intellectually curious adults, with interests spanning a tremendous swath of the intellectual landscape: my parents Judith Smith and Lawrence Blum, and Noel Jette and Alan Zaslavsky. All of my scholarly pursuits have been built on the foundation of curiosity and thoughtfulness I learned from them.

My friends and family have shown me that they are in my corner in countless ways -- from hosting me when I was in town for a math conference, and asking me insistently to explain what I was working on and being willing to spend 45 minutes working to parse the answer, to being understanding when I needed to disappear socially in order to write this, to coming over at 11pm bearing banana bread in the final weeks.

In particular I want to acknowledge Ben Spatz, my brother in all but blood, whose friendship has always been a catalyst for my will to pursue the most beautiful truth.

Lastly, I must express my profound gratitude to my partner Diane Henry, who not only took over the joint administration of our household almost completely during this last stretch while she simultaneously pursued her own goals, but whose wisdom has been a guiding light for me since the beginning.
\newpage
%%%% Abstract %%%%%%%%%%%%%%%%%%
\section*{Abstract}\addcontentsline{toc}{section}{Abstract}
%limited to 350 words

This thesis addresses questions in representation and invariant theory of finite groups. The first concerns singularities of quotient spaces under actions of finite groups. We introduce a class of finite groups such that 
the quotients have at worst abelian quotient singularities. We prove that supersolvable groups belong to this class  and show that nonabelian finite simple groups do not belong to it.
The second question concerns the Cohen-Macaulayness of the invariant ring $\Z[x_1,\dots,x_n]^G$, where $G$ is a permutation group. We prove that this ring is Cohen-Macaulay if $G$ is generated by transpositions, double transpositions, and 3-cycles, and conjecture that the converse is true as well.
\newpage
%%%% Table of Contents %%%%%%%%%%%%
\tableofcontents
%%%%% List of Figures %%%%%%%%%%%%%
\cleardoublepage
\addcontentsline{toc}{section}{List of Figures}
\listoffigures
\newpage
%%%%% List of Tables %%%%%%%%%%%%%
\listoftables\addcontentsline{toc}{section}{List of Tables}
\newpage
%%%%% List of Symbols %%%%%%%%%%%%
\addcontentsline{toc}{section}{List of Symbols}
\printglossary[title={List of Symbols}]
\newpage
%%%%% Body of thesis starts %%%%%%%%%%%%
\pagenumbering{arabic} % switches page numbering to arabic: 1, 2, 3, etc.
%%%%% Introduction %%%%%%%%%%%%
\section*{Introduction}\addcontentsline{toc}{section}{Introduction}
% !TEX root = ./thesis.tex

This thesis concerns questions in representation theory and invariant theory of finite groups. 

In chapter \ref{ch:absing}, we study quotients of products of projective spaces by actions of finite groups. 
The question is to determine when the quotient has only {\em abelian singularities}, which are amenable to an explicit desingularization process. Our methods are group- and representation-theoretic. We work over the complex numbers \C.

In chapter \ref{ch:invar}, the principal object is a polynomial ring over the integers $\Z$ or a finite field $\F_p$, and our question is when the invariant ring of a group $G$ is free as a module over a polynomial subring. Geometrically, this is the question of when the quotient of an affine space by $G$ has a finite flat morphism to affine space. The methods are invariant-theoretic and combinatorial. This question is intimately related to a third, purely topological one: if the object being acted on is $\R^n$, viewed as a piecewise linear (PL) manifold, when is the quotient also a PL manifold? This question has been recently resolved by Christian Lange, and his results will allow us draw conclusions
about the ring of invariants. 

Explicating the close connection between the second and third questions is a major goal of chapter \ref{ch:invar}. The first question is not as closely related; indeed, chapters \ref{ch:absing} and \ref{ch:invar} are logically independent. But the following themes unite our work.

Throughout, $G$ is a finite group. 

\begin{theme}
\textbf{The local structure of the quotient of a smooth object by $G$ is determined by the stabilizers.}
\end{theme}

In chapter \ref{ch:absing}, we let $G$ act on a smooth projective variety, and we want to control the structure of singularities in the quotient. In chapter \ref{ch:invar}, we let $G$ act on a simplicial ball in $\R^n$, and we want to control the homology of links in the quotient. In both cases, we rely on the fact that we can control the local structure by looking at stabilizers. 

Here is an algebraic version of this statement:

%[\cite{vistoli}]

\begin{introthm}[algebraic]\label{lem:quotientstabalg}
If $G$ acts on a variety $X$, and $x\in X$ is a point, and $G_x$ its stabilizer, then the canonical map $X/G_x\rightarrow X/G$ is \'{e}tale over a neighborhood of the image of $x$ in $X/G$.\qed
\end{introthm}

\begin{introthm}[smooth corollary]
If $X$ is smooth, the local structure of any singular point $y$ of $X/G$ is given by the local structure of the image of $x$ in $X/G_x$ for some preimage $x\in X$ of $y$.\qed
\end{introthm}

And here is a topological version:

\begin{introthm}[topological]\label{lem:quotientstabtop}
If $G$ acts on a simplicial complex $\Delta$ simplicially and in such a way that fixed point sets are subcomplexes, and $\alpha\in \Delta$ is a face with stabilizer $G_\alpha$, then the link of the image of $\alpha$ in $\Delta / G$ is the quotient of the link of $\alpha$ in $\Delta$ by $G_\alpha$.\qed
\end{introthm}

The definitions of {\em simplicial complexes}, {\em simplicial actions}, and {\em links}, will be given in chapter \ref{ch:invar}.

\begin{introthm}[smooth corollary]
If $\Delta$ is a PL triangulation of a PL manifold, the link of the image of $\alpha$ in the quotient is the quotient of a simplicial sphere by $G_\alpha$.\qed
\end{introthm}

\begin{theme}
\textbf{Groups generated by elements with low-codimension fixed point sets have well-behaved quotients.}
\end{theme}

This thesis involves three different theorems of this kind. The first is a central result in invariant theory which is a lemma for our work in chapter \ref{ch:absing} and an inspiration for our main result in chapter \ref{ch:invar}.

Let $V$ be a vector space over a field $k$. A \textbf{pseudoreflection} is a finite-order linear transformation that fixes a hyperplane pointwise. The ring $k[V]$ is the ring of polynomial functions on $V$. Assume that $G$ acts on $V$ and therefore on $k[V]$.

\begin{introthm}[Chevalley-Shephard-Todd]
Suppose that the characteristic of $k$ does not divide the order of $G$. Then the ring $k[V]^G$ of $G$-invariants is a polynomial subalgebra of $k[V]$ if and only if $G$ is generated by pseudoreflections.
\end{introthm}

Geometrically, the statement that $k[V]^G$ is a polynomial algebra is the statement that the quotient of $V$, viewed as affine $n$-space over $k$, by $G$, is itself isomorphic to affine $n$-space. More generally this implies, in view of theme 1, that if a finite group $G$ acts on a smooth variety $X$ over a field $k$ of characteristic not dividing $|G|$, then the quotient will be smooth if the point stabilizers $G_x$ are generated by pseudoreflections.

The second such theorem is a recent result of Christian Lange, building on work of Marina Mikhailova, resolving the third question mentioned above. For the purposes of this statement (and throughout this thesis), we use the word \textbf{rotation} to mean a linear transformation that fixes a codimension 2 subspace pointwise.

\begin{introthm}[Lange]
If $G$ acts linearly on $\R^n$, viewed as a PL manifold, then $\R^n/G$ is a PL manifold (with or without boundary) if and only if $G$ is generated by reflections and rotations. When this does happen, $\R^n/G$ is homeomorphic to $\R^{n-1}\times\R^{\geq 0}$, respectively $\R^n$, if $G$ does, respectively does not, contain a reflection.
\end{introthm}

One may view this a topological analogue to the Chevalley-Shephard-Todd theorem. In fact, the Chevalley-Shephard-Todd theorem is one of the many tools that Lange uses in the proof.

The third theorem illustrating the present theme is the most significant result of this thesis.

\begin{thm}[Main result]\label{thm:mainintro}
Let $G$ be a permutation group, acting on $\Z[x_1,\dots,x_n]$ by permuting the variables. If $G$ is generated by transpositions, double transpositions, and 3-cycles, then the invariant ring $\Z[x_1,\dots,x_n]^G$ is a Cohen-Macaulay ring, and is therefore free as a module over the subring of symmetric polynomials.
\end{thm}

This theorem is the principal objective of chapter \ref{ch:invar}, and is proven in section \ref{sec:mainresult}. Geometrically, the conclusion states that the quotient of affine space over $\Z$ by $G$ has a finite flat morphism to affine space.

\section{Overview}

In chapter \ref{ch:absing}, we introduce a class of groups called SEP groups, which are guaranteed to have an action on a smooth projective variety with abelian quotient singularities. We investigate which groups are SEP. The main findings are:
\begin{itemize}
\item  nilpotent and supersolvable groups are always SEP,
\item  nonabelian finite simple groups are never SEP, and 
\item metabelian groups are often but not always SEP.
\end{itemize}

We then begin to consider groups which are not SEP but still have an action on a projective space with at worst abelian quotient singularities. 
Our main results are that $A_5$ and $PSL(2,7)$ have such an action.

In chapter \ref{ch:invar}, we investigate the ring $\Z[x_1,\dots,x_n]^G$. The basic structural question is 
whether or not it is module-free over a polynomial subring. This turns out to be equivalent to {\em Cohen-Macaulayness}. To prove this equivalence we use the theory of Cohen-Macaulay rings.

Then we apply combinatorial ideas developed by Garsia and Stanton to connect the structure of this ring to the topological question mentioned above. The connection is via a fundamental construction in combinatorial commutative algebra called a {\em Stanley-Reisner ring}. We develop the theory of Stanley-Reisner rings, and their application by Garsia and Stanton to invariant theory, in detail.

Finally we apply the theorem of Lange, and a new argument about permutation group actions on a simplicial complex, to show that the Cohen-Macaulayness of a certain Stanley-Reisner ring is equivalent to $G$ being generated by transpositions, double transpositions, and 3-cycles. This allows us to deduce theorem \ref{thm:mainintro} via the work of Garsia and Stanton.

We also discuss methods for constructing explicit bases for $\Z[x_1,\dots,x_n]^G$ using the geometry of the associated cell complex.

In the final section we collect several open questions. 
The most pressing of these is whether or not the converse to theorem \ref{thm:mainintro} holds. We conjecture that it does.
%%%%% chap1 %%%%%%%%%%%%
\newpage
% !TEX root = ./thesis.tex

\chapter{Abelian singularities}\label{ch:absing}

In this chapter we introduce a class of groups characterized by a representation-theoretic property we call {\em SEP}. This property guarantees that the group has an action on a smooth projective variety with mild quotient singularities. It has intrinsic group-theoretic interest in addition to geometric consequences. 

In the first section, we define and investigate this class. In the second, we describe the results of a preliminary search for non-SEP groups that still have the desired action.

\section{SEP groups}

A fundamental fact in linear algebra is that any pair of diagonalizable commuting matrices shares a full basis of eigenvectors. If a pair of matrices fails to commute, then they may still share some eigenvectors, although not a full basis. In this case, they act by restriction on the subspace spanned by the common eigenvectors, and their actions on this subspace do commute. Thus one may see the sharing of eigenvectors as a kind of {\em partial commuting}. If two diagonalizable matrices do not share any eigenvectors, they {\em noncommute purely}.

In the representation theory of a finite group on an algebraically closed field of characteristic zero, group elements always act as diagonalizable transformations. In this context, if two elements of an abstract group do commute, then in every representation they will be forced to share a full basis of eigenvectors. But if they do not commute abstractly, it may still be the case that in every concrete representation of this group on a vector space, they are forced to share some eigenvectors, i.e. commute partially. This prompts us to ask: given a finite group $G$, and two elements $x,y\in G$ that do not commute, is it possible to find a representation of $G$ in which this fact is expressed in an unadulterated way, i.e. their abstract failure to commute is realized in a pair of transformations that do not share an eigenvector?

This question motivates the following definition:

\begin{definition}[Fedor Bogomolov]
A group $G$ is called  \textbf{SEP}\footnote{Bogomolov does not remember why he chose the name SEP. Our best guess is that it stands for ``shared eigenvector property."}
if for every pair of noncommuting elements $x,y\in G$ there exists a representation $\rho$ of $G$ 
such that $\rho(x),\rho(y)$ do not share a common eigenvector.
\end{definition}

\begin{notation}
Throughout, $G$ is a finite group. All vector spaces are over $\C$. If we have a representation $\rho:G\rightarrow GL(V)$ of a group $G$ on a vector space $V$, we will use the word {\em representation} to refer freely to either $\rho$ or $V$.
\end{notation}

We may also be interested to know if we can find a single representation that has this property for all of $G$'s noncommuting pairs. Therefore we make a second definition:

\begin{definition}\label{def:SEP}
If there exists a representation $\rho$ of $G$ such that for all noncommuting $x,y\in G$ the elements 
$\rho(x),\rho(y)$ do not share any common eigenvectors, then we say $G$ is \textbf{SSEP} (for ``strong SEP").
\end{definition}

We also name the condition on $\rho,x,y$ in these definitions:

\begin{definition}
Given $x,y\in G$, and a representation $\rho$ such that $\rho(x),\rho(y)$ do not share eigenvectors, we say $\rho$ is \textbf{SEP for $x,y$}.
\end{definition}

In this language, a group $G$ is SEP if for each noncommuting pair one can find a SEP representation, and it is SSEP if one can find a single representation that is SEP for all pairs.

This section is an investigation into SEP and SSEP groups. We find that all nilpotent, and more generally supersolvable, groups are SEP (section \ref{sec:supersolvable}); no nonabelian simple group is SEP (section \ref{sec:nonabsimple}); and some metabelian groups are SEP while other are not (section \ref{sec:metabelian}). We also prove that a group is SSEP if its nonabelian subgroups are sufficiently large (proposition \ref{bigirrep}), and that a metabelian group is SEP if its commutator subgroup has a certain structure (theorem \ref{thm:metabcrit}).

\subsection{An example}\label{sec:example}

Before stating general results, we present a concrete example:

\begin{prop}[Bogomolov]\label{prop:A5}
The alternating group $A_5$ is not SEP.
\end{prop}

\begin{table}
\begin{center}
\[
\begin{array}{ c | c c c c c }
 & 1 & (123) & (12)(34)  & (12345) & (13524)\\
 \hline \text{Triv} & 1 & 1 & 1 & 1 & 1\\
 \text{3a} & 3 & 0 & -1 & \phi & \hat \phi \\
 \text{3b} & 3 & 0 & -1 & \hat \phi & \phi \\
 \text{4} & 4 & 1 & 0 & -1 & - 1\\
 \text{5} & 5 & -1 & 1 & 0 & 0
\end{array}
\]
\end{center}
\caption{Character table of $A_5$. The symbols $\phi,\hat\phi$ represent $(1\pm\sqrt{5})/2$.}\label{tbl:A5char}

\begin{center}
\[
\begin{array}{ c | c c c }
 & 1 & (123) & (12) \\
\hline \text{Triv} & 1 & 1 & 1\\
\text{Sign} & 1 & 1 & -1\\
\text{Sta} & 2 & -1 & 0
\end{array}
\]
\end{center}
\caption{Character table of $S_3$.}\label{tbl:S3char}
\end{table}

\begin{proof}
The group $A_5$ contains a conjugacy class of subgroups isomorphic to $S_3$, embedded as $\langle(123),(12)(45)\rangle$, for example. If we take $x,y$ to be a pair of generators for $S_3\subset A_5$, we claim that $x$ and $y$ share a common eigenspace in any representation of $A_5$.

This is equivalent to the statement that for any representation $V$ of $A_5$, its restriction 
\[
\Res_{S_3}^{A_5}V
\]
to $\langle x,y\rangle = S_3$ will necessarily contain a one-dimensional subrepresentation of $S_3$. It is enough to check this statement for irreducible representations of $A_5$, since every representation splits into irreducible ones, so if every irreducible $V$ contains a one-dimensional subrepresentation of $S_3$, then every representation does.

The character table of $A_5$ is given in table \ref{tbl:A5char}, and that of $S_3$ in table \ref{tbl:S3char}. There is only one irreducible character of $S_3$ of degree greater than 1: the character $\operatorname{Sta}$ of the  standard representation.\footnote{The \textbf{standard} representation of $S_n$ is the nontrivial irreducible subrepresentation of the \textbf{defining} (or \textbf{canonical}) representation, which is $S_n$'s action on $\C^n$ via permutations of a basis.} Thus if $\Res_{S_3}^{A_5}V$ does not contain any one-dimensional representation of $S_3$, its character must be a multiple of $\operatorname{Sta}$. Since $\operatorname{Sta}$ is zero on the class of involutions and negative on the class of order 3 elements, this means that the character of $V$ itself must be zero on involutions of $A_5$ and negative on order 3 elements of $A_5$.

The group $A_5$ has only one irreducible character that is zero on the class of involutions, the degree 4 character, and it is positive on the class of order 3 elements. Thus no irreducible representation of $A_5$ can restrict on $S_3$ to a multiple of $\operatorname{Sta}$. So all of $A_5$'s irreducible representations' restrictions to $S_3$ must contain a one-dimensional representation. Thus a $x,y$ have a common eigenspace in any representation of $A_5$.
\end{proof}

On the other hand, there is a non-split $\Z/2\Z$-central extension of $A_5$, known as the \textbf{binary icosahedral group} and written $\tilde A_5$, and it is even SSEP.

\begin{prop}[Bogomolov]
If $G$ has a faithful two-dimensional representation, it is SSEP.
\end{prop}

\begin{proof}
Let $V$ be the faithful two-dimensional representation. Let $x,y\in G$. Then the restriction
\[
\Res^G_{\langle x,y\rangle} V
\]
to the subgroup they generate is a two-dimensional representation of this subgroup. Suppose $x,y$'s actions on $V$ have a common eigenvector. Then $\Res^G_{\langle x,y\rangle} V$ contains a one-dimensional representation of $\langle x,y\rangle$. But since it is only two-dimensional, this means it splits completely into one-dimensional representations of $\langle x,y\rangle$. In other words, $x$ and $y$'s actions on $V$ share a full basis of eigenvectors. Then the actions of $x$ and $y$ on $V$ commute. Since $V$ is faithful, this means $x$ and $y$ commute in $G$.

Thus if $x,y$ do not commute, their actions on $V$ do not have a common eigenvector. Thus $V$ realizes $G$ as SSEP.
\end{proof}

\begin{prop}[Bogomolov]\label{prop:binicos}
The binary icosahedral group $\tilde A_5$ is SSEP.
\end{prop}

\begin{proof}
By the last proposition, one just needs to see that $\tilde A_5$ has a faithful two-dimensional representation. One sees this using the orthogonal representation $\varphi: SU_2 \rightarrow SO_3$ of the special unitary group $SU_2$, which is a double cover (\cite{artin}, Section 8.3). One embeds $A_5$ in $SO_3$ as the rotations of an icosahedron and then realizes $\tilde A_5$ as the $\varphi$-preimage. Then $\tilde A_5$ acts faithfully on $\C^2$ through the canonical action of $SU_2$. 
\end{proof}

\subsection{Notation and preliminaries}\label{sec:preliminaries}

\begin{notation}
Throughout, for commutators and conjugates we adopt the right-action notation $x^y = y^{-1}xy$ and $[x,y]=x^{-1}y^{-1}xy = x^{-1}x^y$.
\end{notation}

We use the word {\em character} in two well-established but distinct senses: 

\begin{definition}\label{def:characterofrep}
When speaking of a representation $\rho:G\rightarrow GL(V)$ of arbitrary degree of a {\em nonabelian} group $G$, \textbf{the character of the representation} is the class function $\chi:G\rightarrow\C$ giving traces of the actions of elements of $G$ on the representation space. In other words, 
\begin{align*}
\chi:G &\rightarrow \C \\
g &\mapsto \Tr \rho(g).
\end{align*}
We will also refer to $\chi$ in this situation as \textbf{a character of the group $G$}. If $\rho$ is an irreducible representation of $G$, then $\chi$ is called an \textbf{irreducible character} of $G$.
\end{definition}

\begin{definition}\label{def:characterofabgroup}
When speaking of an {\em abelian} group, a \textbf{character of the group} is an element of its Pontryagin dual, i.e. a one-dimensional representation of the group.
\end{definition}

\begin{remark}
A one-dimensional representation 
\[
\chi: A \rightarrow GL(1,\C)
\]
of an abelian group $A$, i.e. a character in the sense of definition \ref{def:characterofabgroup}, can also be seen as a character in the sense of definition \ref{def:characterofrep}, by identifying $GL(1,\C)$ with $\C^\times\subset \C$, so that $\chi$ is a class function. But it has the added feature of being multiplicative: for $x,y\in A$, we have
\[
\chi(xy) = \chi(x)\chi(y).
\]
This does not hold for characters in the sense of definition \ref{def:characterofrep}.

The word {\em character} is sometimes also used to refer to one-dimensional representations of groups that are not abelian. We will avoid this usage entirely and merely call them ``one-dimensional representations."
\end{remark}

\begin{notation}\label{not:G/A}
If $G$ is a group and $A\subseteq G$ is a (not necessarily normal) subgroup, then we use the symbol $G/A$ to mean the left coset space  of $A$ in $G$. Then the statement
\[
[s]\in G/A
\]
should be interpreted to mean that $[s]$ is a coset and $s$ is some representative in $G$ of this coset. We will then sometimes write
\[
g^s
\]
to mean the conjugate of $g$ by {\em any representative} of the coset $[s]$. We will only use this notation when the setting renders the choice of coset representative inconsequential. The main example is below in \ref{rmk:extendedchipreservesmult}.

Note that when $A$ is normal, this notation is consistent with the meaning of $G/A$ as the quotient group.
\end{notation}

\begin{notation}\label{not:extendedchi}
If $\chi$ is a character (in either sense) of a group $A$ that is embedded in a larger group $G$, we adopt the convention that $\chi$ can be extended to a function on $G$, also called $\chi$, by assigning it the value $0$ outside $A$. More precisely, define a new function $\overline{\chi}$ by
\[
\overline{\chi}(g) =
\begin{cases} 
\chi(g), & g\in A\\
0,& g\notin A,
\end{cases}
\]
and then set $\chi = \overline{\chi}$.
\end{notation}

\begin{remark}\label{rmk:extendedchipreservesmult}
Notation \ref{not:extendedchi} allows us to write the formula 
\[
\Ind_A^G\chi(g) = \sum_{[s]\in G/A} \chi(g^s),
\]
giving the character of an induced representation. Per \ref{not:G/A}, this formula does not depend on the choice of coset representative $s\in[s]$: if $g^s\notin A$, then 
\[
g^{sa} = \left(g^s\right)^a\notin A
\]
either, so $\chi(g^s) = \chi(g^{sa})=0$, while if $g^s\in A$, then 
\[
\chi(g^{sa}) = \chi \left(\left(g^s\right)^a\right) = \chi(g^s)
\]
because $\chi$ is a class function on $A$.

In the case that $A$ is abelian, so that $\chi$ is multiplicative on $A$, the extended meaning of $\chi$ given by \ref{not:extendedchi} preserves the multiplicativity relation $\chi(gh)=\chi(g)\chi(h)$ as long as at least one of $g,h$ is in $A$. For if one of $g,h$ is in $A$ while the other is not, then $gh$ is not in $A$, so that $\chi(gh) = 0 = \chi(g)\chi(h)$.
\end{remark}

\begin{remark}\label{rmk:pairstosubgroups}
If $G$ fails to be SEP, then it means that there is a noncommuting pair $x,y\in G$ such that in every representation of $G$, $x$ and $y$ share a common eigenvector. As in the proof of \ref{prop:A5}, this is equivalent to the statement that any representation of $G$, when restricted to the nonabelian subgroup $H$ generated by $x$ and $y$, will contain some one-dimensional representation of $H$. Notice that this is a fact about $H$ that does not depend on the choice $x,y$ of its generators. In other words, in this situation, any other $x',y'$ that also generate $H$ will also obstruct SEPness, i.e. they will share a common eigenvector in every representation of $G$. 

Conversely, if $G$ is SEP, then for every pair $x,y$ of noncommuting elements, there is a representation $V$ in which they do not share a common eigenvector. This means that the restriction of $V$ to $H=\langle x,y\rangle$ must not contain any one-dimensional representations of $H$. Again, this is a statement about $V$ and $H$ that does not depend on the choice of generators $x,y$ for $H$.

These considerations motivate the following definition:
\end{remark}

\begin{definition}\label{def:sepforH}
Let $H\subset G$ be a nonabelian subgroup. Given a representation $\rho$ (respectively $V$) of $G$, we say $\rho$ (respectively $V$) is \textbf{SEP for} $H$ if $\rho$'s (respectively $V$'s) restriction to $H$ does not contain any one-dimensional representations of $H$. We also say $G$ is \textbf{SEP for $H$} if $G$ has a representation $V$ that is SEP for $H$.
\end{definition}

\begin{lemma}\label{lem:SEP2generated}
The group $G$ is SEP if and only if it is SEP for each of its 2-generated nonabelian subgroups, and it is SSEP if and only if it has a representation $V$ that is SEP for all of its 2-generated nonabelian subgroups.
\end{lemma}

\begin{proof}
This amounts to unwinding definitions \ref{def:SEP} and \ref{def:sepforH} in view of remark \ref{rmk:pairstosubgroups}. SEPness of $G$ is, by definition, the statement that for any pair $x,y$ of noncommuting elements, there exists a representation $V$ in which $x$ and $y$ do not share an eigenvector. By remark \ref{rmk:pairstosubgroups}, this is true if and only if for every nonabelian $H$ that can be generated by $2$ elements, there is a representation $V$ whose restriction to $H$ contains no one-dimensional representations of $H$. This is the statement that $G$ is SEP for each of its $2$-generated nonabelian subgroups.

Likewise, SSEPness of $G$ is, by definition, the statement that there exists a representation $V$ of $G$ in which no noncommuting pair $x,y\in G$ shares a common eigenvector. By \ref{rmk:pairstosubgroups}, this is true if and only if $V$'s restriction to $H$ contains no one-dimensional representations of $H$, for every $H=\langle x,y\rangle$ with $x,y$ noncommuting, i.e. $V$ is SEP for $H$. Thus SSEPness is equivalent to the statement that $G$ has a representation $V$ that is SEP for every 2-generated nonabelian subgroup $H$.
\end{proof}

Lemma \ref{lem:SEP2generated} reduces SEPness to a statement quantified over subgroups rather than pairs of elements. We can reduce the quantification further:

\begin{definition}
A \textbf{minimal nonabelian group} is a nonabelian group all of whose proper subgroups are abelian.
\end{definition}

\begin{example}
The alternating group $A_4$ is minimal nonabelian. Its proper subgroups are all of orders $3$ and $4$.
\end{example}

\begin{remark}\label{rmk:minnonabexists}
Any (finite) nonabelian group contains a minimal nonabelian subgroup: it is a minimal element among the nonabelian subgroups in the subgroup lattice.
\end{remark}

\begin{lemma}\label{lem:minimal2generated}
All minimal nonabelian groups are $2$-generated, by any pair of noncommuting elements.
\end{lemma}

\begin{proof}
If $G$ is a minimal nonabelian group and $x,y$ are any pair of noncommuting elements, then $\langle x,y\rangle$ is a nonabelian subgroup. It cannot be proper since $G$ is minimal nonabelian; thus $G= \langle x,y\rangle$.
\end{proof}

\begin{prop}
A finite group $G$ is SEP if and only if it is SEP for each of its minimal nonabelian subgroups.
\end{prop}

\begin{proof}
If $G$ is SEP, then by lemma \ref{lem:SEP2generated} it is SEP for all its 2-generated subgroups. By lemma \ref{lem:minimal2generated}, the minimal nonabelian subgroups are among these, so it is SEP for each of them.

In the other direction, if $G$ fails to be SEP, then it fails for some 2-generated nonabelian subgroup $H$ by lemma \ref{lem:SEP2generated}; thus for every representation $V$ of $G$, its restriction to $H$ has a one-dimensional subrepresentation. This is also a one-dimensional subrepresentation for any subgroup of $H$; in particular, for a minimal nonabelian subgroup $H'\subset H$, which exists by remark \ref{rmk:minnonabexists}. Thus failure to be SEP can always be detected on some minimal nonabelian subgroup. 
\end{proof}

One advantage of this reduction is that it expedites ascertaining the SEP property by computer search since in many cases there are good algorithms to enumerate conjugacy classes of subgroups. Another is that the structure of minimal nonabelian groups is well understood. A classical result of Miller and Moreno (\cite{millermoreno}) gives a classification:

\begin{prop}
If $H$ is a (finite) minimal nonabelian group, then either $H$ is a $p$-group of nilpotency class 2 with a cyclic commutator subgroup of order $p$ and the $p$-rank of $Z(H)$ is at most 3, or else $H$ is a semidirect product $\F_p^a\rtimes C_{q^b}$ where the generator of $C_{q^b}$ acts on $\F_p^a$ by an irreducible automorphism of prime order $q$; in this latter case the commutator subgroup is $\F_p^a$.\qed
\end{prop}

The representation theory of such groups is also well understood (\cite{mastnaka}).

Among the minimal nonabelian groups are the dihedral groups of order $2p$, $p$ an odd prime, and $8$. Dihedral groups play an important role in a number of arguments in this chapter because they are particularly adept at obstructing SEPness, so we take a moment to recall the definition and highlight the property that will be useful to us.

\begin{definition}\label{def:dihedral}
The \textbf{dihedral group} $D_n$ of order $2n$ is the nontrivial semidirect product 
\[
C_n\rtimes C_2 = \langle r,f\mid r^n=f^2= 1, r^f = r^{-1}\rangle.
\]
Note that we must have $n\geq 3$ or else the action of $f$ on $r$ is trivial. The \textbf{rotation subgroup} is the subgroup $C_n = \langle r\rangle$, and the \textbf{reflections} are the elements of the rotation subgroup's nontrivial coset $\langle r\rangle f$. 
\end{definition}

\begin{remark}
These names come from the description of $D_n$ as the group of Euclidean symmetries of a regular $n$-gon ({\em dihedron}) in $\R^2$. One may take $r$ to be the $2\pi/n$-rotation of the $n$-gon about its center, and $f$ to be any reflection symmetry.
\end{remark}

Recall that a representation that splits into multiple copies of the same irreducible representation is said to be \textbf{isotypical}.

We have the following:

\begin{lemma}[Representation theory of $D_n$]\label{dihedralreps}
All the irreducible representations of $D_n$ are 1- or 2-dimensional. All the 2-dimensional irreducible characters are zero on all reflections and sum to zero on the rotation subgroup.
\end{lemma}

\begin{proof}
As the rotation subgroup $C_n$ is maximal, the isotypical-or-induced lemma (\ref{isotypicalorinduced} in the appendix) implies that all irreducible representations of $D_n$ are either induced from irreducible representations of this subgroup or else are isotypical when restricted to it. As it is also abelian, in the former case the representations are induced from a one-dimensional representation of $C_n$, while in the latter case their restriction to $C_n$ is scalar. The former representations evidently have dimension 2, while the latter have dimension 1, as follows:

If the restriction of an irreducible representation $\rho:D_n\rightarrow GL(V)$ to $C_n$ is scalar, then $C_n$'s image in $\rho(D_n)$ is central. In particular, the image of a generator $r$ is central, which implies the image of the commutator $[r,f]$ of $D_n$'s generators is trivial, and thus that $\rho(D_n)$ is abelian. Since $\rho$ is irreducible this implies it is one-dimensional.

Thus all the two-dimensional irreducible representations are induced from one-dimensional representations $L$ of $C_n$. Such a representation has character zero outside $C_n$ because it is normal (see lemma \ref{inducingfromnormal} of the appendix). It follows that this character has sum zero on $C_n = \langle r\rangle$, because it is orthogonal to the trivial representation and zero outside of $C_n$.
\end{proof}

We use this lemma in several proofs in sections \ref{sec:nonabsimple} and \ref{sec:metabelian}.

\subsection{Supersolvable groups}\label{sec:supersolvable}

A finite solvable group admits a normal series with abelian quotients (the derived series), and a subnormal series with cyclic quotients (any composition series). If we strengthen this requirement to a normal series with cyclic quotients, we can guarantee that the group is SEP. Recall that groups with such a normal series are called \textbf{supersolvable}.

We require some lemmas:

\begin{lemma}\label{lem:cyclic1dfaithful}
Cyclic groups are precisely those finite groups possessing a one-dimensional faithful representation.
\end{lemma}

\begin{proof}
For a cyclic group of order $n$, one obtains a one-dimensional faithful representation by mapping a generator to an $n$th root of unity $\zeta_n$.

In the other direction, the image of any one-dimensional representation of a finite group $G$ is cyclic, since all finite subgroups of $\C^\times$ are cyclic. If the representation is faithful, this means $G$ is cyclic.
\end{proof}

\begin{lemma}\label{lem:resofind}
If $N$ is a normal subgroup of a group $G$, and $\rho: N\rightarrow GL(V)$ is a representation of $N$, then we have the formula
\begin{equation}\label{eq:resofind}
\Res^G_N\Ind_N^G \rho \cong \bigoplus_{[s]\in G/N} \rho^s,
\end{equation}
where each $\rho^s:N\rightarrow GL(V)$ is defined by $\rho^s(x) = \rho(x^s)$.
\end{lemma}

\begin{proof}
Proposition 22 in \cite{serre} is a formula giving the restriction to an arbitrary subgroup $K$ of the representation induced from an arbitrary subgroup $H$. This is that formula in the special case that $H$ is normal and equal to $K$.
\end{proof}

\begin{remark}
The formula \eqref{eq:resofind} illustrates the point discussed in \ref{not:G/A}. The map $\rho^s: N\rightarrow GL(V)$ may depend on the choice of representative $s$ for a coset $[s]\in G/N$. However, if $s' = sn$ with $n\in N$, we have
\[
\rho^{s'}(x) = \rho(x^{sn}) = \rho(n^{-1}x^sn) = \rho(n^{-1})\rho(x^s)\rho(n) = \rho(n)^{-1}\rho^s(x)\rho(n),
\]
so that $\rho^{s'}$ and $\rho^s$ are isomorphic representations of $N$, with the isomorphism induced by the automorphism $\rho(n)\in GL(V)$. Thus the right side of \eqref{eq:resofind} is unambiguous up to isomorphism of representations, which is what the context requires.
\end{remark}

\begin{lemma}\label{lem:commutator1}
If $V$ is a vector space and $X,Y\in GL(V)$ share an eigenvector $v$, then their commutator $[X,Y]\in GL(V)$ also shares this eigenvector, and it has eigenvalue $1$, i.e.
\[
[X,Y]v = v.
\]
\end{lemma}

\begin{proof}
Suppose the eigenvalues of $X,Y$ corresponding to the eigenvector $v$ are $\alpha,\beta$. We have
\[
[X,Y]v = X^{-1}Y^{-1}XYv = \alpha^{-1}\beta^{-1}\alpha\beta v = v.\qedhere
\]
\end{proof}

\begin{thm}\label{thm:supersolvable}
Let $G$ be a finite group admitting a normal series with cyclic quotients, i.e. a finite supersolvable group. Then $G$ is SEP.
\end{thm}

\begin{proof}[Proof of Theorem \ref{thm:supersolvable}]
Let $x,y$ be any two noncommuting elements. Then $[x,y]$ is a nontrivial element of $G$. Let
\[
G=G_0\triangleright G_1\triangleright\dots\triangleright G_n = \{1\}
\]
be the assumed normal series with cyclic quotients. Since $[x,y]$ is nontrivial, it is in $G_i\setminus G_{i+1}$ for some $i=0,\dots,n-1$. By assumption, $G_i/G_{i+1}$ is cyclic, so there exists a faithful one-dimensional representation 
\[
\rho : G_i/G_{i+1}\rightarrow \C^\times
\]
by lemma \ref{lem:cyclic1dfaithful}. Precomposing with the canonical homomorphism 
\[
\pi: G_i\rightarrow G_{i}/G_{i+1},
\]
we obtain a one-dimensional representation $\phi =\rho \pi$ of $G_i$ that is nontrivial outside of $G_{i+1}$.

We will now show that the induced representation
\[
\Phi = \Ind_{G_i}^G\phi
\]
is SEP for $x,y$. This will be done by showing that $\Phi([x,y])$ does not have $1$ as an eigenvalue. It will then follow that $\Phi(x),\Phi(y)$ do not share an eigenvector, for if they did, their commutator $\Phi([x,y])$ would have $1$ as an eigenvalue, by lemma \ref{lem:commutator1}. As $x,y$ are an arbitrary noncommuting pair, this will complete the proof that $G$ is SEP.

Since $G_i$ is normal in $G$, lemma \ref{lem:resofind} tells us that
\[
\Phi|_{G_i} = \bigoplus_{[s]\in G/G_i} \phi^s.
\]
Since $\phi$ is one-dimensional, $\phi^s$ is as well, for each $s$, so that this formula splits $\Phi$ into one-dimensional representations on $G_i$. It follows that for any given $g\in G_i$, the eigenvalues of $\Phi(g)$ are just the values of the $\phi^s(g)\in \C^\times$.

We apply this with $g=[x,y]\in G_i$. By assumption, $[x,y]$ lies outside of $G_{i+1}$. Since $G_{i+1}$ is normal in $G$, $[x,y]^s$ also lies outside of $G_{i+1}$ for each $s$. Since $\phi$ is nontrivial outside of $G_{i+1}$ by construction, this means that $\phi^s([x,y]) = \phi([x,y]^s)$ is not equal to $1$ for any $s$. Thus no eigenvalue of $\Phi([x,y])$ is $1$. This completes the proof.
\end{proof}

\begin{cor}[Bogomolov]
Finite nilpotent groups are SEP.
\end{cor}

\begin{proof}They are supersolvable.\end{proof}

\begin{remark}
Theorem \ref{thm:supersolvable} and its proof are essentially due to Bogomolov (personal communication), who formulated it for nilpotent groups. The author's only contribution was to note that the proof works with little change for supersolvable groups.
\end{remark}

Recall that a group $G$ is called \textbf{cyclic-by-abelian} if it has a cyclic normal subgroup $C$ such that the quotient $G/C$ is abelian.

\begin{cor}\label{cyclicbyabelian}
If a finite group $G$ is cyclic-by-abelian, then it is SSEP.
\end{cor}

\begin{proof}
Let $C\triangleleft G$ be cyclic with $G/C$ abelian. Then $[G,G]\subset C$, thus every nontrivial commutator is in $C\setminus\{1\}$. There exists a character $\phi$ of $C$ that is nontrivial on $C\setminus\{1\}$, by lemma \ref{lem:cyclic1dfaithful}. Then $\Phi = \Ind_{C}^G \phi$ is a representation in which, by the exact same argument as in the proof of theorem \ref{thm:supersolvable}, no nontrivial commutator has 1 as an eigenvalue, and therefore in which no pair of noncommuting elements shares an eigenvector. Thus $\Phi$ is SEP for any noncommuting $x,y$, so it manifests $G$ as SSEP.
\end{proof}

There is no hope of a similar result about groups which are merely solvable:

\begin{prop}\label{prop:S4}
The symmetric group $S_4$ is not SEP.
\end{prop}

This can be proven by direct reference to $S_4$'s character table, as was done for $A_5$ in section \ref{sec:example}, but we prefer a more conceptual proof:

\begin{proof}
Consider the subgroup $D_4 = \langle (1234),(13)\rangle \subset S_4$. We will show that no representation of $S_4$ is SEP for this subgroup.

$D_4$ acts faithfully on the plane as the symmetry group of a square. This is its only irreducible representation of degree greater than one. (Proof: it has five conjugacy classes, thus five irreducible representations. It has three subgroups of index 2, implying three nontrivial homomorphisms to $\{\pm 1\}$. These, and the trivial representation, account for four of the five.) The character $\chi$ of this representation is given by
\[
\begin{array}{ c | c c c c c }
 & 1 & (13)(24) & (1234) & (13) & (12)(34) \\
\hline \chi & 2 & -2 & 0 & 0 & 0
\end{array}
\]
Notice that $\chi$ separates the central element $(13)(24)$ from the class of the reflection $(12)(34)$. On the other hand, in $S_4$ these elements are conjugate. Therefore no class function on $S_4$, in particular no character of $S_4$, can separate them. It follows that no character of $S_4$ restricts to a multiple of $\chi$; thus the restriction to $D_4$ of any representation of $S_4$ must contain some one-dimensional representation of $D_4$, so no representation of $S_4$ is SEP for $D_4$.
\end{proof}

The only feature of $S_4$ used in this proof is that it contains $D_4$ in such a way that the central involution is conjugate to one of the other involutions. Therefore the argument generalizes:

\begin{cor}[$D_4$ obstruction]\label{D4obstruct}
If a finite group $G$ contains $D_4$ in such a way that the nontrivial central element in $D_4$ is conjugate to one of the other involutions, then $G$ is not SEP.\qed
\end{cor}

\subsection{Nonabelian simple groups}\label{sec:nonabsimple}

While section \ref{sec:supersolvable} shows that there are plenty of SEP groups, there are also plenty of groups which are not SEP. Recall that a family of objects $\mathscr{F}$ is said to be \textbf{upward-closed} if whenever an object $A$ is in $\mathscr{F}$ and embeds in an object $B$, then $B$ is in $\mathscr{F}$ too.

\begin{lemma}\label{lem:upwardclosed}
The family of non-SEP groups is upward-closed.
\end{lemma}

\begin{proof} 
If a noncommuting pair in a group $G$ shares an eigenspace in every representation of $G$, it also does so in every representation $\rho$ of any group containing $G$, since $\rho$ is also a representation of $G$ by restriction. Thus if $G$ is not SEP for $x,y\in G$, no overgroup of $G$ can be SEP for $x,y$ either.
\end{proof}

Thus by propositions \ref{prop:A5} and \ref{prop:S4}, no group containing $S_4$ or $A_5$ is SEP.  But more broadly:

\begin{thm}\label{nonabsimple}
No nonabelian simple group is SEP.
\end{thm}

This theorem is the main goal of the section. The structure of the proof is as follows. By a 1997 result of Barry and Ward, every nonabelian simple group contains a {\em minimal simple group}. Such groups were classified in 1968 by Thompson, and they are all of the form $PSL(2,q)$, $Sz(2^p)$, or $PSL(3,3)$. We will show that none of these groups is SEP by giving explicit subgroups for which they are not SEP. The result for all nonabelian simple groups will then follow by lemma \ref{lem:upwardclosed}. Here are the precise details:

\begin{definition}
A \textbf{minimal simple group} is a nonabelian finite simple group all of whose proper subgroups are solvable.
\end{definition}

\begin{lemma}[\cite{barryward}, Theorem 1]\label{lem:barryward}
Every nonabelian finite simple group contains a minimal simple group.
\end{lemma}

\begin{lemma}[\cite{thompson}, Corollary 1]\label{lem:thompson}
Every minimal simple group is among the following:
\begin{enumerate}
\item $PSL(2,q)$ for $q$ a prime power $\geq 4$.
\item The Suzuki group $Sz(2^p)$ for $p$ an odd prime.
\item $PSL(3,3)$.
\end{enumerate}
\end{lemma}

\begin{remark}
The group $PSL(2,q)$ is only minimal simple for certain $q$. The version of the statement in \cite{thompson} is sharper. But \ref{lem:thompson} is all we will need.
\end{remark}

\begin{lemma}[Bogomolov]\label{psl2q}
For a prime power $q\geq 4$, $PSL(2,q)$ is not SEP.
\end{lemma}

In fact, we can already know this for $q=\pm 1\mod 8$ since in this case $PSL(2,q)$ contains $S_4$, and for $q=\pm 1\mod 10$ since in this case it contains $A_5$. One could hope to proceed to the remaining cases. Bogomolov's proof involves a different, but also somewhat delicate, case analysis. We give a more uniform proof, although some case analysis is inevitable because the representation theory of $PSL(2,q)$ depends on $q$ mod $4$.

\begin{proof}
We will show that the SEP property is obstructed by a dihedral group of order $q-1$ or $q+1$, if $q$ is odd, or $2(q - 1)$, if $q$ is even. It is well-known that $PSL(2,q)$ contains dihedral subgroups of these orders (\cite{dickson}, \S 246, or \cite{king}, Theorem 2.1(d)-(i)). Let $D$ be a dihedral subgroup of $PSL(2,q)$, of order to be specified shortly.

By lemma \ref{dihedralreps}, all irreducible characters of $D$ that are not one-dimensional share the following properties: (1) they are identically zero on all the reflections, and (2) they sum to zero on the rotation subgroup. As these properties are both linear, the character of any representation of $D$ that does not contain a one-dimensional subrepresentation also possesses them.

It follows that an irreducible representation of $PSL(2,q)$ cannot be SEP for $D$ unless its character has these same two properties. We will show that the order of $D$ can always be chosen so that no irreducible character of $PSL(2,q)$ meets this standard.

The relevant part of the character table of $PSL(2,q)$ is given in table \ref{tbl:psl2q}. The notation below is explained in the caption.

\begin{table}
\[
\text{Case $q=1$ mod $4$.}
\]
\[
\begin{array}{c | c c c }
& 1 & a^\ell & b^m \\
\hline \text{Triv} & 1 & 1 & 1\\
\psi & q & 1 & -1\\
\chi_i & q+1 & \rho^{i\ell}+\rho^{-i\ell} & 0\\
\theta_j& q-1 & 0 & -\left(\sigma^{jm} + \sigma^{-jm}\right)\\
\xi_1 & (q+1)/2 & (-1)^\ell & 0\\
\xi_2 & (q+1)/2 & (-1)^\ell & 0
\end{array}
\]
\vspace{2mm}
\[
\text{Case $q=3$ mod $4$.}
\]
\[
\begin{array}{c | c c c }
 & 1 & a^\ell & b^m\\
\hline \text{Triv} & 1 & 1 & 1\\
\psi & q & 1 & -1\\
\chi_i & q+1 & \rho^{i\ell}+\rho^{-i\ell} & 0\\
\theta_j& q-1 & 0 & -\left(\sigma^{jm} + \sigma^{-jm}\right)\\
\eta_1 & (q+1)/2 & 0 & (-1)^{m+1}\\
\eta_2 & (q+1)/2 & 0 & (-1)^{m+1}
\end{array}
\]
\vspace{2mm}
\[
\text{Case $q$ even.}
\]
\[
\begin{array}{ c | c c c c }
 & 1 & c & a^\ell & b^m\\
\hline \text{Triv} & 1 & 1 & 1 & 1\\
\psi & q & 0 & 1 & -1\\
\chi_i & q+1 & 1 & \rho^{i\ell} + \rho^{-i\ell} & 0\\
\theta_j & q-1 & -1 & 0 & - \left(\sigma^{jm} + \sigma^{-jm}\right)
\end{array}
\]
\caption{Character table of $PSL(2,q)$. The symbols $\rho,\sigma$ are primitive $(q-1)$th and $(q+1)$th roots of unity respectively. For odd $q$, respectively even $q$, $a$ is the class of elements of order $(q-1)/2$, respectively $q-1$, and $b$ is the class of elements of order $(q+1)/2$, respectively $q+1$. For odd $q$, $i$ and $j$ are even integers and we have omitted the classes of elements of order dividing $q$. For even $q$, $i,j$ are integers and $c$ is the class of involutions. Source: \cite{dornhoff}, \S 38.}\label{tbl:psl2q}
\end{table}

There are three cases to consider: $q=1$ mod $4$, $q=3$ mod $4$, and $q$ even.

In the case $q=1\mod 4$, the class of involutions is $a^{(q-1)/4}$ in the table. Take $D$ to be of order $q+1$, so its rotation subgroup, of order $(q+1)/2$, consists of the identity and elements in the classes $b^m$. The only irreducible characters of $PSL(2,q)$ that are zero on $a^{(q-1)/4}$ are those of the cuspidal representations $\theta_j$. Their absolute value is $q+1$ on the identity and is bounded by $2$ on the elements $b^m$. Thus the sum of any of these characters across the rotation subgroup has absolute value bounded below by
\[
q+1 - 2\left(\frac{q+1}{2} - 1\right) =2>0.
\]
Therefore no irreducible character of $PSL(2,q)$ is simultaneously zero on $D$'s reflections and sums to zero on $D$'s rotation subgroup.

In the case $q=3\mod 4$, the class of involutions is $b^{(q+1)/4}$. Take $D$ to be of order $q-1$, so its rotation subgroup consists of the identity and elements $a^\ell$. In this case it is only the principal series characters $\chi_i$ that are zero on the class of involutions, and they are $q-1$ on the identity and bounded by $2$ in absolute value on the nontrivial rotations. Thus these characters' sums across $D$'s rotation subgroup has absolute value again bounded below by
\[
q-1 - 2\left(\frac{q-1}{2} - 1\right) = 2>0
\]
and so cannot be SEP for $D$.

In the final case of even $q$, the class of involutions is $c$. Take $D$ to be of order $2(q-1)$, so the nontrivial rotations are $a^\ell$. The only irreducible character that is zero on $c$ is the Steinberg character $\psi$, of degree $q$, which is positive on the $a^\ell$'s, so the sum across the rotation subgroup of $D$ is positive.
\end{proof}

\begin{remark}
This proof generalizes the proof for $A_5$ given earlier. In that proof (\ref{prop:A5}), the SEP-obstructing subgroup was $S_3$, which is isomorphic to the dihedral group of order $6$. Now $A_5\cong PSL(2,4)\cong PSL(2,5)$. For $q=4$, the even case, this proof selects the dihedral group of order $2(q-1) = 2\cdot 3 = 6$ to obstruct SEPness, while for $q=5$, the $1$ mod $4$ case, it selects the dihedral group of order $q+1 = 5+1 = 6$.
\end{remark}

\begin{lemma}\label{psl33}
$PSL(3,3)$ is not SEP.
\end{lemma}

\begin{proof}
The argument is identical to that given for $S_4$ (proposition \ref{prop:S4}). $PSL(3,3)=SL(3,3)$ contains a subgroup $D_4$ generated by
\[
r = \begin{pmatrix} &-1& \\1& & \\ & &1\end{pmatrix},\;f = \begin{pmatrix}1& & \\ &-1& \\ & &-1\end{pmatrix}
\]
The central involution in this copy of $D_4$ is
\[
\begin{pmatrix}-1& & \\ &-1& \\ & &1\end{pmatrix}
\]
But this is conjugate in $PSL(3,3)$ to $f$, so apply corollary \ref{D4obstruct}.
\end{proof}

\begin{lemma}\label{suzuki}
The Suzuki group $Sz(2^p)$ ($p$ an odd prime) is not SEP.
\end{lemma}

\begin{proof}
The proof is the same as that for $PSL(2,q)$ when $q$ is even.

Like $PSL(2,q)$, the Suzuki group $Sz(q), \;q = 2^p$ has a single conjugacy class of involutions (\cite{suzuki}, Proposition 7). It also contains a dihedral group $D$ of order $2(q-1)$: when $Sz(q)$ is realized as a permutation group as in Suzuki's original presentation, this is the normalizer of the stabilizer of two points (\cite{suzuki}, Proposition 3).

The character table of $Sz(q)$ is given in table \ref{tbl:Szq}, which is taken from \cite{suzuki}, Theorem 13.

\begin{table}
\[
\begin{array}{ c | c c c c c c }
\text{Class name:} & 1 & \sigma & \rho,\rho^{-1} & \pi_0 & \pi_1 & \pi_2\\
\text{Order divides:} & 1 & 2 & 4 & q-1 & q + r + 1 & q - r + 1\\
\hline X & q^2 & 0 & 0 & 1 & -1 & -1 \\
X_i & q^2 + 1 & 1 & 1 & \varepsilon_0^i(\pi_0) & 0 & 0\\
Y_j & (q - r + 1)(q-1) & r-1 & -1 & 0 & -\varepsilon_1^j(\pi_1) & 0\\
Z_k & (q + r + 1)(q-1) & -r-1 & -1 & 0 & 0 & -\varepsilon_2^k(\pi_2)\\
W_l & r(q-1)/2 & -r/2 & \pm ri/2 & 0 & 1 & -1
\end{array}
\]
\caption{Character table of $Sz(q)$. Here, $r = \sqrt{2q}$. The classes called $\pi_0,\pi_1,\pi_2$ consist of elements belonging to certain cyclic subgroups $A_0,A_1,A_2$, and the $\varepsilon$'s are certain characters of these subgroups.}\label{tbl:Szq}
\end{table}

Per lemma \ref{dihedralreps}, as in the proof for $PSL(2,q)$, in order for an irreducible representation of $Sz(q)$ to be SEP for $D$ its character would have to be zero on the class of involutions and sum to zero on the elements of $D$'s rotation subgroup. There is only one irreducible character of $Sz(q)$ that is zero on the class of involutions ($X$ in the table), and it is positive on all the elements of $D$'s rotation subgroup (which, besides the trivial class, are in the classes Suzuki calls $\pi_0$, as their orders divide $q-1$). So no irreducible representation of $Sz(q)$ is SEP for $D$; therefore $Sz(q)$ is not SEP.
\end{proof}

\begin{proof}[Proof of theorem \ref{nonabsimple}]
By \ref{lem:barryward} and \ref{lem:thompson}, every nonabelian finite simple group contains $PSL(2,q)$ for $q\geq 4$, $Sz(2^p)$ for $p$ odd prime, or $PSL(3,3)$, and none of these is SEP by \ref{psl2q}, \ref{psl33}, and \ref{suzuki}, so \ref{lem:upwardclosed} then implies that no nonabelian finite simple group is SEP.
\end{proof}

\subsection{A family of SSEP groups}\label{sec:SSEPfam}

The argument of theorem \ref{thm:supersolvable} shows supersolvable groups are SEP by finding, for any noncommuting pair, a representation in which its commutator does not have 1 as an eigenvalue. A group can be SEP without this. For example, $A_4$ is a group in which, in {\em every} representation, {\em every} commutator has $1$ as an eigenvalue. Nonetheless, the standard 3-dimensional representation of $A_4$ actually realizes it as SSEP: $A_4$ is a minimal nonabelian group, thus any pair of noncommuting elements generates the whole group (\ref{lem:minimal2generated}), and therefore cannot have a common eigenspace in this representation because it is irreducible.

In a similar way one sees immediately that any minimal nonabelian group is SSEP. This is actually a special case of a more general phenomenon that forces a group to be not just SEP but SSEP:

\begin{prop}\label{bigirrep}
Let $G$ be a finite group with an irreducible representation $V$ of degree $d$ that exceeds the index of any of its nonabelian subgroups. Then $V$ realizes $G$ as SSEP.
\end{prop}

\begin{proof}
Let $x,y$ be a pair of noncommuting elements of $G$, and let $H=\langle x,y\rangle$. By assumption, $[G:H]<d$. Now let $L$ be any one-dimensional representation of $H$. Then, by Frobenius reciprocity, the number of times that $L$ occurs in the restriction of $V$ to $H$ is equal to the number of times $V$ occurs in the induced representation $\Ind_H^G L$. Since the dimension of this representation is 
\[
[G:H]<d = \dim V,
\]
this number is zero. So no one-dimensional representation $L$ occurs in the restriction of $V$ to $H$, i.e. $x,y$ do not have a common eigenspace in $V$.
\end{proof}

\begin{remark}
As we have seen, failure to be SEP is always caused by specific obstructing subgroups. For example, the obstruction for $S_4$ is $D_4$ (\ref{prop:S4}). Proposition \ref{bigirrep} shows that a subgroup obstructing SEPness cannot be ``too big." Indeed, $D_4\subset S_4$ is ``as big as possible," since $S_4$ has an irreducible representation (in fact, two) of degree $3$, equal to the index.
\end{remark}

The following construction shows that proposition \ref{bigirrep} has some content beyond the minimal nonabelian groups (such as $A_4$):

\begin{prop}\label{bigirrepexample}
Let $p$ be a prime congruent to $1$ mod $4$ and let $d=(p+1)/2$. There exists an automorphism $A$ of $C_p^2$ of order $d$. Let $G$ be the semidirect product 
\[
C_p^2\rtimes C_d
\]
where $C_d = \langle A\rangle$. Then $G$ has an irreducible representation of degree $d$, and all of its nonabelian subgroups have index $<d$.
\end{prop}

\begin{proof}
We will show that $G$ has an irreducible representation of degree $d$ and that every nonabelian subgroup of $G$ contains $C_p^2$ properly (and thus has index $<d$).

First, $A$ exists. Interpret $C_p^2$ as the additive group of $\F_{p^2}$. Let $\alpha\in\F_{p^2}^\times$ be an element of order $d$, which exists because $d$ divides 
\[
p^2-1 = |\F_{p^2}^\times|.
\]
Then let $A$ be the action of $\alpha$ on $\F_{p^2}$ by multiplication.

Second, $A$'s action on $C_p^2$ is irreducible. We see this as follows. $(d,p-1)=1$ since $2d-(p-1) = 2$ and $d$ is odd. Since $A^d - I = 0$, the eigenvalues of $A$ in the algebraic closure $\overline{\F}_p$ are $d$th roots of unity. Since $d$ is relatively prime with $p-1$, they cannot lie in $\F_p$ without both being $1$; but if they were both $1$ then $A$ would have order dividing $p$, which is also relatively prime with $d$. Thus $A$'s eigenvalues do not lie in $\F_p$, and it follows that $A$ has no eigenvectors in $\F_p^2$. Since $\F_p^2$ is $2$-dimensional, this implies that the action of $A$ is irreducible.

Third, note that in the last paragraph the only property of $A$ that was used in the argument was that its order $d$ is relatively prime to $p-1$ and $p$. Since the same is also true of all nontrivial factors of $d$, which are the orders of $A$'s powers, it follows that $A^2,\dots,A^{d-1}$ also all act irreducibly on $C_p^2$.

As a corollary, $C_d = \langle A\rangle$ acts {\em freely} on $C_p^2\setminus\{1\}$, since a nontrivial point stabilizer would imply an eigenvector (with eigenvalue $1$) of some $A^k$ ($1\leq k\leq d-1$).

Now let $L$ be any nontrivial one-dimensional representation of $N = C_p^2$, and let 
\[
\chi: N\rightarrow GL(L)\cong\C^\times
\]
be the associated homomorphism. Then the desired degree $d$ representation $V$ is 
\[
\Ind_N^G L = \bigoplus_{[s]\in G/N} sL.
\]
The degree of this representation is $[G:N] = d$, and we claim that it is irreducible. Since $N$ is normal, it acts separately on each translate $sL$, and then by Mackey's criterion (\cite{serre}, proposition 23), irreducibility of $V$ follows from irreducibility of $L$ and distinctness of each $sL$ as a representation of $N$. $L$ is irreducible since it is one-dimensional, and we see that each $sL$ is distinct as follows:

The $sL$'s are given (as representations) by the homomorphisms $\chi\circ c_s$ where $c_s$ is the automorphism of $N$ induced by conjugation by $s$ in $G$, where $s$ is a representative of a coset $[s]\in G/N = \langle A\rangle$. I.e. the $c_s$'s are exactly the actions of the group $\langle A\rangle$ on $N$. Thus the representations $sL$ correspond with the orbit of $\chi$ in the dual group $\widehat{N}$ under the action of $C_d = \langle A\rangle$ induced by its action on $N$. They are all distinct because, as mentioned above, $\langle A\rangle$'s action on $N\setminus\{1\}$ is free, and if a finite group acts freely on a finite abelian group (minus its identity) then the induced action on the dual group (minus its identity) is also free. (See the appendix, lemma \ref{freeaction}, for an elementary proof.)

This establishes that $V$ is irreducible, and it is clear that it is degree $d = |\langle A\rangle|$.

Now we show that any nonabelian subgroup of $G$ contains $N$ properly (and therefore has index $<d$). Let $H$ be such a subgroup. The canonical homomorphism $G\rightarrow G/N = C_d$ restricts to a homomorphism $H\rightarrow C_d$, and since $H$ is nonabelian, the kernel of this homomorphism, $H\cap N$, must be nontrivial, so $H$ contains a nonidentity element of $N$. Thinking of $N$ as the vector space $\F_p^2$, this means $H\cap N$ is a nontrivial subspace of $N$.

By the same token, $H$'s image in $C_d$ must be nontrivial since the kernel $H\cap N$ is abelian. Let $A^k$, $0<k\leq d-1$ be any nonidentity element in the image of $H$ in $C_d=\langle A\rangle$, and let $h$ be a preimage in $H$. Then conjugation by $h$ acts as $A^k$ on $N=C_p^2$, and $H\cap N$ is invariant under this action. But since we saw above that the action of $A^k$ on $N=C_p^2$ is irreducible, and last paragraph that $H\cap N$ is nontrivial, we conclude that $H\cap N = N$.

Since $h\notin N$, this shows that $H$ contains $N$ properly.
\end{proof}

Thus $G$ satisfies the hypothesis of proposition \ref{bigirrep} and is therefore SSEP.

\subsection{Metabelian groups}\label{sec:metabelian}

Recall that a group is called \textbf{metabelian} if it has an abelian normal subgroup with an abelian quotient, in other words if it is solvable of height two. In this section we investigate the SEP property for metabelian groups.

The results of sections \ref{sec:supersolvable}, \ref{sec:nonabsimple}, and \ref{sec:SSEPfam} show that SEPness is loosely correlated with abelianness -- the ``extremely nonabelian" simple groups are never SEP, while ``almost abelianness" of various kinds (nilpotence and supersolvability, a propos of section \ref{sec:supersolvable}, and having all nonabelian subgroups ``large," a propos of section \ref{sec:SSEPfam}) guarantee SEPness. (Of course abelian groups themselves are SEP, vacuously.) Based on this intuition, Bogomolov and the author expected that metabelian groups might be always SEP; but this turns out not to be the case.

Recall that the \textbf{affine group}, or \textbf{affine linear group}, $AGL(n,q)$, is the group of transformations of $\F_q^n$ generated by the linear transformations $GL(n,q)$ and the group $T_{n,q}\cong \F_q^n$ of translations 
\begin{align*}
t_x:\F_q^n&\rightarrow \F_q^n\\
v &\mapsto x+v.
\end{align*}
Since conjugation by a linear transformation sends translations to translations, $T_{n,q}$ is a normal subgroup of this group, and since $GL(n,q)$ stabilizes the origin, while $T_{n,q}\cong\F_q^n$ acts freely on the points of $\F_q^n$, their intersection is trivial. Thus
\[
AGL(n,q) = \F_q^n \rtimes GL(n,q)
\]
is a Frobenius group, with Frobenius kernel $\F_q^n$ and Frobenius complement $GL(n,q)$. In this context one sees $\F_q^n$ as the affine space $\A_{\F_q}^n$, hence the name, because the presence of the translations means one cannot distinguish the origin among the points of $\F_q^n$ from the abstract group action on these points. A choice of origin is equivalent to the choice of a section $GL(n,q)\rightarrow ALG(n,q)$.

For $n=1$, $AGL(1,q) = \F_q\rtimes \F_q^\times$ is metabelian. However:

\begin{thm}\label{affinenotSEP}
If $q=p^k$ is an odd prime power with $k>1$, the affine group 
\[
AGL(1,q) = \F_q\rtimes \F_q^\times
\]
is not SEP.
\end{thm}

We defer the proof to the end of the section.

In spite of this negative result, a large class of metabelian groups is SEP. If $G$ is metabelian then the commutator subgroup $[G,G]$ is abelian. There is a criterion on the structure of $[G,G]$ that lets us conclude SEPness without knowing anything else about $G$:

\begin{thm}\label{thm:metabcrit}
Let $G$ be a metabelian group and let
\[
[G,G] \cong C_{p_1^{e_1}}\times\dots\times C_{p_k^{e_k}}
\]
be the expression of its commutator as a direct product of cyclic factors of prime power order. If the factors are pairwise nonisomorphic, then $G$ is SEP.
\end{thm}

\begin{remark}
This is not a necessary condition. Many metabelian groups not satisfying the hypothesis of theorem \ref{thm:metabcrit} are still SEP, for example the family of SSEP groups described in section \ref{sec:SSEPfam}. But it shows that SEP metabelian groups are easy to come by.
\end{remark}

The organization of the section is motivated by the proof of \ref{thm:metabcrit}. The fundamental tool is the following technical lemma, which is of independent utility in investigating the SEP property:

\begin{lemma}[Commutator criterion]\label{commutatorcriterion}
Let $G$ be an arbitrary finite group, $H$ a nonabelian subgroup of $G$, and $V$ a representation of $G$ with character $\chi$. Then $V$ is SEP for $H$ if and only if $\chi$ sums to zero on each coset of $[H,H]$ in $H$.
\end{lemma}

\begin{proof}
The representation $V$ is SEP for $H$ if and only if $V|_H$ contains no one-dimensional representations of $H$. By the orthogonality relations, this is the case if and only if $\chi|_H$ is orthogonal to all of $H$'s one-dimensional representations, with respect to the $H$-invariant inner product
\[
\langle \chi_1,\chi_2\rangle_H = \frac{1}{|H|}\sum_{h\in H} \chi_1(h)\overline{\chi}_2(h).
\]
Now $H$'s one-dimensional representations are precisely the pullbacks to $H$ of all of the characters of the abelian group $H/[H,H]$. These characters span the full space of functions on $H/[H,H]$ (as for any abelian group), so their pullbacks to $H$ span the space of all functions on $H$ constant on each coset of $[H,H]$. The orthogonal complement of this space is clearly the space of class functions that sum to zero on each coset of $[H,H]$, and $V$ is SEP for $H$ if and only if $\chi|_H$ lies in this orthogonal complement.
\end{proof}

Now we begin to assemble the proof of theorem \ref{thm:metabcrit}.

If $G$ is metabelian, then its commutator subgroup $[G,G]$ is abelian, and therefore acts trivially on itself by conjugation. It follows that the conjugation action of $G$ on $[G,G]$ makes the latter a $G/[G,G]$-module.

\begin{notation} 
In what follows we fix the notation that $G$ is a finite metabelian group, $A=[G,G]$, and $Q=G/[G,G]$, so $A$, $Q$ are abelian and $A$ is a $Q$-module.
\end{notation}

\begin{prop}\label{subgroupcharacterH}
Let $H\subset G$ be a nonabelian subgroup, so that $K=[H,H]$ is a nontrivial subgroup of $A$. Suppose that $A$ has a character $\chi$ whose kernel does not contain any image of $K$ under the $Q$-action. Then the representation $\Ind_A^G\chi$ is SEP for $H$.
\end{prop}

For example, if $A$ is cyclic (so $G$ is cyclic-by-abelian), it has a character $\chi$ that is a faithful representation of $A$ and is thus nontrivial on all nontrivial subgroups. Therefore $\Ind_A^G\chi$ is SEP for all nonabelian subgroups; thus $G$ is SSEP. This reproduces corollary \ref{cyclicbyabelian}, although without the information (obtained in the proof in section \ref{sec:supersolvable}) that no commutator has an eigenvalue $1$ in this representation.

\begin{proof}[Proof of proposition \ref{subgroupcharacterH}]
We want to show $\Ind_A^G\chi$ is SEP for $H$, and by the commutator criterion (lemma \ref{commutatorcriterion}), this is equivalent to showing that 
\[
\sum_{k\in K}\Ind_A^G\chi(kh) = 0
\]
for all $h\in H$. Actually we even have $\sum_{k\in K}\Ind_A^G \chi(kg)=0$ for every $g\in G$. We see this as follows:
\begin{align*}
\sum_{k\in K}\Ind_A^G \chi(kg) &= \sum_{k\in K}\sum_{[s]\in G/A} \chi((kg)^s)\\
&=\sum_{k\in K} \sum_{[s]\in G/A}\chi(k^sg^s)\\
&=\sum_{k\in K} \sum_{[s]\in G/A}\chi(k^s)\chi(g^s).
\end{align*}
The last equality is because $\chi$, being a character of $A$, is multiplicative, since $k^s\in A$ as $k\in K\subset A$ and $A$ is normal (see \ref{rmk:extendedchipreservesmult}). Reversing the summations and then reindexing the inner sum, we have
\[
\sum_{[s]\in G/A}\left(\sum_{k\in K} \chi(k^s)\right)\chi(g^s) = \sum_{[s]\in G/A}\left(\sum_{k\in K^s} \chi(k)\right)\chi(g^s).
\]
But the inner sum is zero, because by assumption $K^s$ is not contained in $\ker \chi$, therefore $\chi$ restricts to a nontrivial character of $K^s$, and the sum of a nontrivial character over a group is always zero.
\end{proof}

This proposition immediately implies that the following condition on the module structure guarantees SEPness:

\begin{prop}[Subgroup character condition]\label{subgroupcharactercondition}
Suppose $A$ (as $Q$-module) has the property that for any nontrivial subgroup $K\subset A$, $A$ has a character whose kernel does not contain any image of $K$ under the $Q$-action. Then $G$ is SEP.
\end{prop}

\begin{proof}
If $H$ is any nonabelian subgroup of $G$, then proposition \ref{subgroupcharacterH} shows how to construct a representation of $G$ that is SEP for $H$.
\end{proof}

\begin{remark}
In fact, we found theorem \ref{affinenotSEP} by looking for a group where the condition of this proposition fails. 
\end{remark}

The next lemma links \ref{subgroupcharactercondition} with the hypothesis of theorem \ref{thm:metabcrit}.

\begin{lemma}\label{ladischlemma}
The following conditions on a finite abelian group $A$ are equivalent:

\begin{enumerate}
\item $A$ has a nontrivial subgroup $K$ such that every character of $A$ is trivial on an image of $K$ under some automorphism of $A$.
\item In a decomposition of $A$ into cyclic factors of prime power order, two of the factors are isomorphic.
\end{enumerate}
\end{lemma}

This result and its proof are due to Frieder Ladisch (personal communication).

\begin{proof}
We write $A$ additively.

Condition 2 is fulfilled by $A$ if and only if it is fulfilled by at least one of $A$'s Sylow subgroups. We will show the same for condition 1. If a nontrivial subgroup $K$ of a Sylow subgroup $A_p$ fulfills condition 1 for $A_p$, it also does so for $A$ because automorphisms of $A_p$ extend to $A$ and characters of $A$ restrict to $A_p$. Conversely, if a nontrivial subgroup $K$ of $A$ fulfills condition 1 for $A$, then there is a prime $p$ (any prime dividing $|K|$ in fact) such that $K_p = A_p\cap K$ fulfills condition 1 for $A_p$, since characters of $A_p$ extend to $A$ and automorphisms of $A$ act on $A_p$.

Thus without loss of generality we can suppose $A$ is a $p$-group. If it fulfills condition 2, it has the form $A = F \oplus B$ where $F\cong C_{p^k}\times C_{p^k}$. Then let $K$ be the cyclic subgroup generated by any nonzero element in $F$. Note $\Aut F\subset\Aut A$. Every character of $A$ restricts to a character on $F$ and we assert any character of $F$ is trivial on some $\Aut F$-image of $K$. Indeed, $\Aut F = GL(2,\Z/p^k\Z)$ acts transitively on the elements of $F$ of any given order, and therefore on the order-$|K|$ cyclic subgroups of $F$. Meanwhile every character of $F$ is trivial on some order-$|K|$ cyclic subgroup since it is trivial on some maximal (order $p^k$) cyclic subgroup, as otherwise its image would not be cyclic. This shows 2$\Rightarrow$1. 

In the other direction, suppose $K$ fulfills condition 1 for $A$ and consider the subgroup $A_0$ of $A$ of elements of order dividing $p$. This is an $\F_p$-vector space of dimension the $p$-rank of $A$. It has a filtration 
\[
A_0\supset A_1\supset \dots A_k = 0,
\]
where $A_i = A_0\cap p^iA$ is the subgroup of $A_0$ consisting of $p^i$-divisible elements, and $p^k$ is the exponent of $A$. Both $A_0$ and this filtration of it are invariant under automorphisms of $A$. 

Now as $K$ is nontrivial it contains elements of order $p$, so it must meet $A_0$ nontrivially. Thus there is a maximal $i<k$ such that $A_i$ meets $K$ nontrivially; fix this $i$, so that $K\cap A_{i+1} = 0$. Furthermore, as $A_i$ and $A_{i+1}$ are both automorphism invariant, we must have that every image $K'$ of $K$ under $\Aut A$ also meets $A_i$ nontrivially and $A_{i+1}$ trivially.

We assert $A_{i+1}$ is codimension $>1$ in $A_i$. If it were codimension $1$, it would be the kernel of some character on $A_i$, which could be extended to a character $\chi$ of $A$. This character would be nontrivial on every $\Aut A$-image $K'$ of $K$, since they all meet $A_i$ nontrivially outside of $A_{i+1}$. This contradicts the assumption that $K$ fulfills condition 1 for $A$, so we conclude $A_{i+1}$ is codimension $>1$ in $A_i$.

We claim this in turn implies that at least two of $A$'s cyclic factors are isomorphic. Indeed, the dimension of $A_i$ is the number of cyclic factors of $A$ of order at least $p^{i+1}$.\footnote{In fact, if $\lambda = (\lambda_1,\dots,\lambda_r)$ with $\lambda_1\geq \dots\geq \lambda_r$ is the partition describing the type of $A$, so that $A \cong \prod_{\lambda_j} C_{p^{\lambda_j}}$, then the tuple $(\dim A_0,\dots,\dim A_{k-1})$ is the conjugate partition $\lambda'$.} That $A_{i+1}$ is codimension at least two in $A_i$ thus implies that the number of cyclic factors of order at least $p^{i+1}$ is at least two greater than the number of cyclic factors of order at least $p^{i+2}$. This implies that there are at least two cyclic factors of order exactly $p^{i+1}$.

This establishes 1$\Rightarrow$2.
\end{proof}

\begin{proof}[Proof of theorem \ref{thm:metabcrit}]
In this situation, by lemma \ref{ladischlemma}, for every subgroup $K$ of $A=[G,G]$, there is a character $\chi$ of $A$ whose kernel does not contain any image of $K$ under $\Aut A$, so $A$ fulfills the hypothesis of proposition \ref{subgroupcharactercondition} for any possible $Q$-action.
\end{proof}

It remains to prove our claim about $AGL(1,q)$. (Recall that $q=p^k$ is an odd, composite prime power.) We fix notation:

\begin{notation}\label{not:AGL}
Let $G = AGL(1,q)$. Then $G = A\rtimes Q$ where $A$ is isomorphic to the additive and $Q$ to the multiplicative group of $\F_q$. The commutator subgroup $[G,G]$ is equal to $A$, so these labels are consistent with those used throughout the section. We identify $G$ as a group of permutations of the  elements of $\F_q$, with $A$ being translations by $x\in \F_q$ and $Q$ being multiplications by $a\in \F_q^\times$.

Let $x\in \F_q$; then denote by $t_x$ the element 
\begin{align*}
t_x: \F_q &\rightarrow \F_q\\
z & \mapsto z+x
\end{align*}
of $A$, i.e. the translation of the affine line $\A_{\F_q}^1$ by $x$. Likewise, if $a\in \F_q^\times$, let $m_a$ be the element 
\begin{align*}
m_a: \F_q &\rightarrow\F_q\\
z&\mapsto az
\end{align*}
of $Q$, i.e. the linear map of $\A^1$ given by multiplication by $a$.
\end{notation}

\begin{remark}\label{rmk:freeandtrans}
If $j$ is an integer, we have 
\[
t_x^j = t_{jx}
\]
and 
\[
m_a^j = m_{a^j},
\]
thus all $t_x$'s obey the relation $t_x^p=1$ and all $m_a$'s obey $m_a^{q-1}=1$. If to match the right-action notation of the rest of this chapter we agree that $G$ acts on $\A^1$ from the right, then we also have for all $y\in \F_q$ that
\[
yt_x^{m_a} = y(m_a)^{-1}t_xm_a = (a^{-1}y)t_xm_a = (a^{-1}y+x)m_a = y+ax = yt_{ax},
\]
so that
\[
t_x^{m_a} = t_{ax}
\]
for any $a,x$.

As the equation $ax=y$ in $\F_q$ always has a unique solution for $a$ given nonzero $x,y$, there is a unique $m_a$ such that $t_x^{m_a} = t_y$. In other words, the action of $Q$ on the nonidentity elements of $A$ is free and transitive.
\end{remark}

The key to theorem \ref{affinenotSEP} is this lemma:

\begin{lemma}[Representation theory of $AGL(1,q)$]\label{lem:repthyofAGL}
Let $V$ be an irreducible representation of $G=AGL(1,q)$. Then $V$ is
\begin{itemize}
\item one-dimensional and pulled back to $G$ from a character of $Q$, or else 
\item the unique representation $W$ induced from any nontrivial character on $A$. 
\end{itemize}
The character of $W$ restricted to $A$ is the sum of all nontrivial characters of $A$.
\end{lemma}

\begin{proof}
This is Theorem 6.1 of \cite{piatetskishapiro}. For a different proof using the isotypical-or-induced lemma, see \ref{rmk:ourproof} in the appendix.
\end{proof}

We are ready to prove that $AGL(1,q)$ is not SEP if $q$ is composite:

\begin{proof}[Proof of theorem \ref{affinenotSEP}]
Since the $W$ of \ref{lem:repthyofAGL} is the only irreducible representation of $G$ of degree greater than $1$, all $G$'s hope of being SEP lies with $W$.

Consider the character $\chi_W$ of $W$. Since $A$ is normal, $\chi_W$ is zero outside of $A$ (lemma \ref{inducingfromnormal} of the appendix). On $A$, as it is the sum of all nontrivial characters of $A$, it is one less than the sum of {\em all} characters, which is the character of the regular representation. Thus 
\[
\chi_W(1) = |A|-1 = q-1,
\]
and 
\[
\chi_W(t_x) = 0-1 = -1
\]
for all $x\neq 0$ in $\F_q$.

Let $D$ be the subgroup generated by $m_{-1}$ and any $t_x$ with $x\neq 0$. Because $q$ is odd, $m_{-1}$ is the negation map on $\F_q$. Therefore $D$ is a dihedral group of order $2p$, where $p$ is the characteristic of $\F_q$. We now show $W$ is not SEP for $D$.

The rotation subgroup of $D=\langle t_x,m_{-1}\rangle$ is $\langle t_x\rangle$, of order $p$. Thus the sum of $\chi_W$ over the rotation subgroup of $D$ is $1(q-1)+(p-1)(-1) = q-p$. Since $q=p^k$ with $k>1$, this is nonzero, and we can conclude from lemma \ref{dihedralreps} that $W$ is not SEP for $D$. This concludes the argument.
\end{proof}

\section{Abelian singularities}

When one takes the quotient of a smooth complex algebraic variety $X$ by the action of a finite group $G$, the resulting variety typically has singular points. By Hironaka's theorem, the singularities can be resolved by a sequence of blowups. However, in general, it is a hard problem to make the desingularization process completely constructive.

On the other hand, if the singularities are {\em abelian}, meaning that they are locally isomorphic to the quotient of $\C^n$ by a finite abelian group, then the desingularization can be accomplished in an explicit way (see \cite{cox}, Chapters 10 and 11). Thus abelian singularities are mild from the point of view of resolution of singularities.

\begin{definition}
If $V$ is an algebraic variety over $\C$, $v\in V$ is an \textbf{abelian quotient singularity} (or simply \textbf{abelian singularity}) if the completion 
\[
\widehat{\mathcal{O}_{V,v}}
\]
of the local ring at $v$ is isomorphic to the completion 
\[
\widehat{\mathcal{O}_{Y,0}},
\]
where $Y$ is the quotient of $\C^n$ by a finite abelian group $A$, acting linearly.
\end{definition}

The singular points of the quotient $X/G$ are automatically abelian if for any point $x\in X$, the point stabilizer $G_x$ is abelian, per the discussion in the introduction. In this case, the image of $x$ in the quotient $X/G$ is locally isomorphic to the quotient of $\C^n$ by $G_x$.

Let $G$ be a finite group with a representation $V$ that realizes it as SSEP. Then $G$ acts on the projective space $\Proj(V)$. The SSEP property guarantees that the quotient will be only mildly singular:

\begin{lemma}\label{lem:SSEPabelianstab}
Any singular points of the quotient variety $\Proj(V)/G$ are abelian.
\end{lemma}

\begin{proof}
This follows from knowing that if $x\in\Proj(V)$ then $G_x$ is abelian. But indeed, if $x\in\Proj(V)$ is stabilized by any two $g,h\in G$, then a representative of $x$ in $V$ is a shared eigenvector for the actions of $g$ and $h$ on $V$. Since $V$ is SEP for every pair of noncommuting elements of $G$, it must be that $g,h$ commute. Thus $G_x$ is abelian.
\end{proof}

SEP groups themselves have a similar property.

\begin{lemma}
If $G$ is a SEP group, and
\[
X = \prod\Proj(V_i),
\]
where the product is taken over the irreducible representations of $G$, then any singular points of $X/G$ are abelian.
\end{lemma}

\begin{proof}
Again, we show that for any $x\in X$, $G_x$ is abelian. If $x\in X$ is stabilized by both $g$ and $h$, then its projection to each factor $\Proj(V_i)$ is represented in $V_i$ by a shared eigenvector for the actions of $g$ and $h$ on $V_i$. Thus $g,h$ share an eigenvector in every irreducible representation of $G$, and therefore in every representation. Since $G$ is SEP, this is impossible unless they commute.
\end{proof}

Thus SEP groups have an action on a product of projective spaces with a quotient that has well-behaved singularities. This was, in fact, Bogomolov's original motivation for giving the definition.

However, SEPness is a much more stringent criterion than necessary to guarantee this. This section is a preliminary probe into other ways that a group can have the desired action.

A first observation is this:

\begin{prop}\label{prop:centralext}
If a group $G$ has a SSEP (respectively SEP) central extension, then it has an action on a projective space (respectively a product of projective spaces) with abelian point stabilizers.
\end{prop}

\begin{proof}
If $V$ is an irreducible representation of a group $\tilde G$, then the center $Z(\tilde G)$ acts trivially on $\Proj(V)$ since it acts by scalars on $V$. Therefore the quotient 
\[
G = \tilde G / Z(\tilde G)
\]
acts on $\Proj(V)$. If $\tilde G$ is realized as SSEP by $V$, then the above shows that the point stabilizers $\tilde G_x$ (for $x\in \Proj(V)$) are abelian. But the point stabilizers $G_x$ in $G$ are exactly the images of the $\tilde G_x$'s in $G$ under the canonical map
\[
\tilde G \rightarrow \tilde G / Z(\tilde G) = G
\]
(see lemma \ref{lem:quotientstabilizer} in the appendix to chapter \ref{ch:invar}). It follows that the $G_x$'s, as homomorphic images of abelian groups, are also abelian.
\end{proof}

\begin{example}
Although $A_5$ is not SEP, its $\Z/2\Z$-central extension $\tilde A_5$, the binary icosahedral group, is SSEP, with the action on $\C^2$ described above in \ref{prop:binicos} realizing it as such. The induced action of $\tilde A_5$ on $\Proj_\C^1$ has abelian point stabilizers by lemma \ref{lem:SSEPabelianstab}, and factors through $A_5$. Thus $A_5$ acts on $\Proj_\C^1$ with abelian point stabilizers.
\end{example}

One can get much further with the Chevalley-Shephard-Todd theorem (\cite{neuselsmith}, Theorem 7.1.4), discussed in the introduction. This theorem implies that the portions of the point stabilizers $G_x$ that are generated by pseudoreflections do not lead to singularities. Thus one does not actually need the $G_x$'s to be abelian. The objective of this section is to show that both of the smallest nonabelian simple groups have actions on $\Proj^2$ with good quotients.

\begin{prop}\label{A5}
Let $V$ be a $3$-dimensional faithful irreducible representation of $A_5$. Then $\Proj(V)/A_5$ is smooth.
\end{prop}

\begin{prop}\label{PSL27}
Let $W$ be a $3$-dimensional faithful representation of $PSL(2,7)$. Then $\Proj(W)/PSL(2,7)$ has only abelian singularities.
\end{prop}

To prove these, we need to understand the way that the point stabilizers for a projective representation act on neighborhoods of the points they stabilize:

\begin{lemma}\label{lem:tangentsp}
Let $V$ be a representation of a group $G$ and suppose that $V = W\oplus L$ is a decomposition into sub-representations with $L$ one-dimensional. $L$ corresponds to a fixed point $x\in \Proj(V)$ for the action of $G$ on $\Proj(V)$. We have that $W' = T_x\Proj(V)$ is a representation of $G$, and $W' \cong W \otimes L^{-1}$ as representations. Furthermore, the image of $x$ in $\Proj(V)/G$ has a Zariski neighborhood isomorphic to $W'/G$.
\end{lemma}

\begin{proof}
One chooses coordinates $(v_1,v_2,\dots,v_n)$ for $V$ such that $v_1$ is a coordinate for $L$ and $v_2,\dots,v_n$ are coordinates for $W$. Then $(v_1:v_2:\dots:v_n)$ are projective coordinates for $\Proj(V)$, $x=(1:0:\dots:0)$ is contained in the affine patch 
\[
(1:*:\dots:*),
\]
and we may identify this affine patch with $T_x\Proj(V)\cong \C^{n-1}$ via 
\[
(1:v_2:\dots:v_n)\mapsto (v_2,\dots,v_n).
\]
Then just by writing down the action of $G$ on $(v_2,\dots,v_n)$ via this identification, we see it is precisely the action on $W$ scaled back by the character on $L$ to preserve that the first coordinate is $1$, in other words it is the representation $W\otimes L^{-1}$. It is clear that the image in the quotient $\Proj(V)/G$ of the affine patch we have described is precisely isomorphic to $W\otimes L^{-1}/G$.
\end{proof}

Now we can prove the propositions. We make use of standard information about the subgroup lattices and characters of $A_5$ and $PSL(2,7)$.

%ideally, add tables, figures, and references

\begin{proof}[Proof of proposition \ref{A5}]
Any singularity in $\Proj(V)/A_5$ is the image of a fixed point of a nontrivial point stabilizer. Thus we consider the point stabilizers with regard to $A_5$'s action on $\Proj(V)$. These are the subgroups such that restriction of $V$ to them has a one-dimensional subrepresentation, and that are maximal with respect to this property. 

$A_5$'s maximal subgroups, up to conjugacy, are $A_4$, $D_5$, and $D_3$. The restriction of $V$ to either dihedral subgroup splits into $W\oplus L$ with $W$ a faithful $2$-dimensional representation and $L$ the sign representation.\footnote{By \textbf{sign representation} we refer to the map $D_n\rightarrow \{\pm 1\}$ that is $+1$ on the rotation subgroup and $-1$ on the reflections. One sees the decomposition $W\oplus L$ straightforwardly from the character table, but it can also be seen by realizing $V$ as the representation of $A_5$ as rotational symmetries of an icosahedron: then $D_3$ is the stabilizer of an axis through the centers of a pair of opposite faces, and $D_5$ the stabilizer of an axis through a pair of opposite vertices.} Then $W'=W\otimes L^{-1}$ is also the faithful $2$-dimensional representation of a dihedral group, which realizes it as a reflection group, so the Chevalley-Shephard-Todd theorem gives us that the quotient $W'/D_3$, respectively $W'/D_5$, is affine space and therefore smooth. Lemma \ref{lem:tangentsp} then implies that these quotients are locally isomorphic to the images in $\Proj(V)/A_5$ of the stable points of the $D_3$ and $D_5$ subgroups, respectively. Thus, these points do not yield singular points in the quotient. 

Meanwhile, the restriction of $V$ to $A_4$ is irreducible. Since every subgroup of $A_4$ is abelian, restriction to any of them splits $V$ into 3 one-dimensional representations, thus all subgroups of $A_4$ stabilize three points. However, the Sylow 3-subgroups of $A_4$ are also contained in $D_3$ and so they are not the full stabilizers. Thus the only new point stabilizer we obtain is $A_4$'s other maximal subgroup, the Sylow 2-subgroup isomorphic to the Klein $4$-group; call it $K$. Restriction to this subgroup splits $V$ into $L_1\oplus L_2\oplus L_3$, the three nontrivial characters of $K$.\footnote{Again, this is easy to see from the character table, but one obtains it more suggestively by recognizing $K$ as the stabilizer of the axis through the centers of a pair of opposite edges in the realization of $V$ in terms of an icosahedron. If one positions the icosahedron appropriately in $\R^3$, $K$ is the stabilizer of the three coordinate axes, and it is clear it acts nontrivially on all three.} Choosing any $L_i$ to regard as the stable point, we have $W' = (L_{i-1}\oplus L_{i+1})\otimes L_i^{-1} = L_{i+1}\oplus L_{i-1}$ is a faithful two-dimensional representation of $K$, which is necessarily a reflection group, so again Chevalley-Shephard-Todd shows that $W'/K$ is smooth, and the lemmas tell us that this is locally isomorphic to the image in $\Proj(V)/A_5$ of the stable point of $K$. Since we have considered all point stabilizers for $A_5$'s action on $\Proj(V)$, we can conclude that the quotient is smooth.
\end{proof}

\begin{proof}[Proof of proposition \ref{PSL27}]
The proof is identical in structure to the above: we identify point stabilizers for $PSL(2,7)$'s action on $\Proj(W)$, and analyze the local structure of the corresponding candidates for singular points in the quotient.

The character of $W$, up to an automorphism of $PSL(2,7)$, is as follows:
\[
\begin{array}{ c | c c c c c c }
\text{Representative:} & I & \begin{pmatrix} &-1\\1& \end{pmatrix} & \begin{pmatrix}2& \\ &4\end{pmatrix} & \begin{pmatrix}2 & -2\\2&2\end{pmatrix} & \begin{pmatrix}1&1\\ &1\end{pmatrix} & \begin{pmatrix}1&-1\\ &1\end{pmatrix} \\
\text{Order:} & 1 & 2 & 3 & 4 & 7 & 7\\
 \hline \chi_W & 3 & -1 & 0 & 1 & \alpha & \bar\alpha
\end{array}
\]
Here, $\alpha, \bar\alpha$ are the two roots of $\lambda^2+\lambda+2$. The maximal subgroups of $PSL(2,7)$ are two classes of $S_4$ (exchanged by $PSL(2,7)$'s outer automorphism) and a class of the Frobenius group $7:3$ of order $21$; call it $F_{21}$. The restriction of $W$ to any of the maximal subgroups is irreducible. The restriction to either class of $S_4$'s is the representation of $S_4$ as rotations of a cube in $\R^3$; the stabilizers of one-dimensional subspaces in this representation are the $D_3$ that stabilizes an axis through a pair of opposite vertices and the $D_4$ that stabilizes an axis through the centers of a pair of opposite faces. 

These stabilizers lead to smooth points in the quotient for exactly the same reason as in the proof above for $A_5$: restriction to each dihedral group splits $W$ into the faithful two-dimensional representation and the sign representation; tensoring with the sign representation does not change the faithful two-dimensional representation, which realizes the dihedral group as a reflection group; so the Chevalley-Shephard-Todd theorem says the quotient of the neighborhood of the stable point by this stabilizer is smooth. Thus no stabilizers contained in the classes of $S_4$ contribute any singular points to the quotient. 

It remains to consider $F_{21}$. Its (maximal) Sylow 3-subgroups are contained in the classes of $S_4$, so they are already considered. The Sylow 7-subgroup $C_7$ is maximal as well, and of course $W$ splits into three characters of this abelian group upon restriction; thus $C_7$ is the stabilizer of these points. They do end up being singular in the quotient; but as $C_7$ is abelian, the structure of the singularities is only abelian.
\end{proof}

\section{Appendix: algebraic lemmas}

Let $N$ be a finite abelian group (written additively), and let $A$ be a group of automorphisms of $N$, acting on the left. $A$ has a natural left action on $N$'s character group $\widehat N$ (written multiplicatively) by $a\chi = \chi\circ a^{-1}$ (for $a\in A, \chi\in \widehat N$).

\begin{lemma}[Free action lemma]\label{freeaction}
If $A$'s action on $N\setminus \{0\}$ is free, then $A$'s action on $\widehat N\setminus \{1\}$ is also free.
\end{lemma}

(We used this lemma in the proof of proposition \ref{bigirrepexample}. We believe it is a standard fact but have not encountered a reference.)

\begin{proof}
The assumption that $A$'s action on $N\setminus\{0\}$ is free also implies the same is true of the restriction of the action to any subgroup of $A$, and in particular to the cyclic subgroup generated by any element $a\in A$.

The statement that $A$'s action on $\widehat{N}\setminus\{1\}$ is free is equivalent to the statement that for any nontrivial character $\chi\in \widehat N$, $a\chi = \chi$ implies $a=1$. Now any nontrivial character $\chi:N\rightarrow\C^\times$ is a homomorphism to a cyclic group of order $r\geq 2$. The fibers of this homomorphism all have cardinality $m=|N|/r$. One of them contains $0$ and thus it contains $m-1$ nonidentity elements, and there is also at least one other fiber, with $m$ nonidentity elements.

Suppose $a$ satisfies $a\chi = \chi\circ a^{-1} = \chi$. Then $\chi\circ a^{-1}$ has the same fibers as $\chi$, and it follows that $a^{-1}$, and thus $a$, acts separately on each fiber. Thus each fiber is a union of orbits for the action of $\langle a\rangle$ on $N$.

As observed above, the action of $\langle a\rangle$ on $N\setminus\{0\}$ is free, which means that all orbits for $\langle a \rangle$'s action on $N\setminus\{0\}$ have the same length, namely the order of $a$ (call it $d$). Thus $m$ (the cardinality of the fibers not containing the identity, of which as noted above there is at least one) is a multiple of $d$. But $m-1$ is also a multiple of $d$, since this is the cardinality of the part of $N\setminus\{0\}$ sitting in the fiber containing $0$.

Thus $d\mid (m,m-1)=1$, i.e. $a$ is order $1$, i.e. $a=1$. This proves the action of G on $\widehat N\setminus\{1\}$ is free.
\end{proof}

\begin{lemma}[Inducing from a normal subgroup]\label{inducingfromnormal}
If $N\triangleleft G$ is a normal subgroup and $\chi$ is the character of any representation $V$ of $N$, then $\Ind_N^G\chi$ is zero outside of $N$.
\end{lemma}

\begin{proof}[First proof]
We have
\[
\Ind_N^G\chi(g) = \sum_{[s]\in G/N} \chi(g^s),
\]
where $\chi$ is defined to be zero outside of $N$, as in \ref{not:extendedchi}. Since $N$ is normal, $g\notin N$ implies $g^s\notin N$, so in this case $\Ind_N^G\chi(g)$ is a sum of zeros.
\end{proof}

\begin{proof}[Second proof]
$\Ind_N^G\chi(g)$ is the trace of a matrix describing the action of $g$ on 
\[
\bigoplus_{[s]\in G/N} sV.
\]
In order to have nonzero trace, $g$ has to stabilize some $sV$. Now for any $s$, 
\[
gsV = sV \Leftrightarrow g^sV = V,
\]
i.e. $g^s\in N$. But $g^s\in N\Leftrightarrow g\in N$ because $N$ is normal. Thus if $g$ lies outside of $N$, it does not stabilize any $sV$.
\end{proof}

The following statement appears, in a slightly different form, as Proposition 24 in \cite{serre}. Because we make heavy use of it, we state it here for convenience, and offer the proof from \cite{serre}. Recall that a representation is said to be \textbf{isotypical} if it is a direct sum of copies of a single irreducible representation.

\begin{lemma}[Isotypical-or-induced lemma]\label{isotypicalorinduced}
If $N$ is a normal subgroup of a group $G$ and $V$ is an irreducible representation of $G$, then $V = \Ind_M^G W$ where $M$ is a subgroup containing $N$ and $W$ is an irreducible representation of $M$ whose restriction to $N$ is isotypical.
\end{lemma}

It is named after its immediate corollary, that either $V|_N$ is isotypical to begin with (the case $M=G$), or else $M<G$ and $V$ is induced from an irreducible representation of the proper subgroup $M$. It is also possible to derive the statement from the corollary by induction on the size of $M$. This corollary is actually what one finds as proposition 24 in \cite{serre}, but the proof found there actually proves the statement itself with no added work:

\begin{proof}
Let $V=\bigoplus V_\alpha$ be the canonical decomposition of $V$ into isotypical representations of $N$. Because $N$ is normal, $G$ acts on the set of $N$-submodules of $V$, within it the set of isotypical ones, and within that the set of maximal isotypical ones, which are the $V_\alpha$'s. Thus $G$ permutes the $V_\alpha$'s, and because $V$ is irreducible it permutes them transitively. Let $W$ be any $V_\alpha$ and let $M$ be its stabilizer in $G$. Then certainly $M\supset N$, and we have $V=\Ind_M^G W$ and $W|_N$ is isotypical.
\end{proof}

\begin{lemma}[Commutator of split metabelian groups]\label{commofsplitmetab}
Let $G = A\rtimes Q$ be a split metabelian group (i.e. a semidirect product of abelian groups) and view $A,Q$ as subgroups in the natural way. Then $[G,G]=[A,Q]$.
\end{lemma}

\begin{proof}
It is obvious that $[A,Q]\subset [G,G]$. The opposite inclusion follows from a calculation showing that an arbitrary commutator of $G$ is in $[A,Q]$. Let $a,b\in A$ and $x,y\in Q$, so that two arbitrary elements of $G$ are $ax$ and $by$. We freely use the fact that the pairs $a,b$ and $x,y$ commute (and so stabilize each other under conjugation).
\begin{align*}
[ax,by] &= (ax)^{-1}(ax)^{by}\\
&= (ax)^{-1}a^{by}x^{by}\\
&= x^{-1}a^{-1}a^yx^{by}\\
&= x^{-1}[a,y]x^{by}\\
&=[a,y]^x x^{-1}x^{by}\\
&=[a,y]^x(x^{-1})^y x^{by}\\
&=[a,y]^x(x^{-1}x^b)^y\\
&=[a,y]^x[x,b]^y\\
&=[a^x,y^x][x^y,b^y]\\
&=[a^x,y][x,b^y]\\
&=[a^x,y][b^y,x]^{-1}.
\end{align*}
As $A$ is normal, this last is in $[A,Q]$.
\end{proof}

The assumption that both $A$ and $Q$ are abelian is necessary for the result. For example taking $G=S_4$ with $A=V$ and $Q=S_3$, we find $[G,G]=A_4$, but $[A,Q]=A=V$.

As an aside, as this calculation only used $[a,b]=[x,y]=1$ and $a^x,b^y\in A$, it generalizes without change to show that in an arbitrary group, if $A$ is any abelian subgroup and $Q$ any abelian subgroup that normalizes $A$ (so that $AQ$ is also a subgroup), $[AQ,AQ] = [A,Q]$.

\begin{remark}\label{rmk:ourproof}
We promised a proof of lemma \ref{lem:repthyofAGL} based on the isotypical-or-induced lemma. Recall that this is the statement that $G=AGL(1,q)$ only has one irreducible representation $W$ of degree greater than one, and it is induced from any nontrivial character of $A$. See \ref{not:AGL} for notation.
\end{remark}

\begin{proof}[Proof of lemma \ref{lem:repthyofAGL}]
In \ref{not:AGL}, $A$ and $Q$ are both identified as specific subgroups of $G$. But we also have a canonical surjective homomorphism $G\rightarrow Q$, since $G=A\rtimes Q$. Thus we can speak of an element of $Q$ as an element of $G$ but we can also ask for its preimage in $G$.

Since $A\triangleleft G$ is normal, the isotypical-or-induced lemma (\ref{isotypicalorinduced}) tells us that $V$ is induced from an irreducible representation $L$ of some subgroup $M$ containing $A$ such that the restriction of $L$ to $A$ is isotypical. It will turn out that either $M=A$, $L$ is a nontrivial character of $A$, and $V$ is $W$, or else $M=G$ and $V=L$, and $L$ is the pullback to $G$ of one of the $q-1$ characters of $G/A\cong Q$. The plan will be to show that in all cases $L$ is one-dimensional, and then to show that if $M$ contains $A$ properly then this also implies $M=G$.

We first claim that $L$ is necessarily one-dimensional. We see this as follows:

If $M=A$, then it is abelian, and any irreducible representation of it is one-dimensional. 

On the other hand suppose $M$ contains $A$ strictly, in which case its image $Q_M$ in $Q$ contains $m_a$ for some specific $a \neq 1\in\F_q^\times$. Now $Q_M\subset M$, because $A\subset M$ is the entire kernel of $G\rightarrow Q$, so $M$ contains all the preimages of the elements of $Q_M$, and in particular the elements themselves. Thus $M$ contains $m_a$.

Clearly $A$ contains $[M,M]$, as $M/A$ is isomorphic to the subgroup $Q_M$ of $Q$ and therefore abelian. But also, $[M,M]$ contains 
\[
[t_x,m_a] = t_x^{-1}t_x^{m_a} = t_{-x}t_{ax} = t_{(a-1)x}
\]
for any $x\in A$. Since $a-1\neq 0$, the map $x\mapsto (a-1)x$ is surjective onto $\F_q$, and therefore the map
\begin{align*}
A &\rightarrow [M,M]\subset A\\
t_x &\mapsto [t_x,m_a] = t_{(a-1)x}
\end{align*}
is surjective onto $A$. Thus $A = [M,M]$.

But on closer examination of $[t_x,m_a]$, one sees that it is $\in [A,Q_M]\subset [M,M]$. Thus the previous calculation even shows that $A\subset [A,Q_M]$, and therefore that $[A,Q_M] = [M,M] = A$. (For an alternative proof that $[A,Q_M] = [M,M]$, appeal to lemma \ref{commofsplitmetab}.)

Therefore,
\[
[M,M]=[A,Q_M]\subset [A,M] = [[M,M],M] \subset [M,M],
\]
so we conclude $A=[[M,M],M]$ as well. 

Now $L$'s restriction to $A$ is isotypical; as $A$ is abelian that means it is scalar, which is to say that the image of $A$ in $GL(L)$ is central. This means that the actions of elements of $A$ commute with any element of $M$ on $L$. This implies that $M\rightarrow GL(L)$ factors through 
\[
M/[A,M] = M/[[M,M],M].
\]
But 
\[
M/[[M,M],M] = M/A = M/[M,M]
\]
is abelian, so because $L$ is irreducible, it is one-dimensional after all.

Therefore it restricts to a character $\chi$ on $A$. Let $\chi_L$ be the character of the one-dimensional representation $L$, so that $\chi_L|_A = \chi$, and let $\chi_V$ be the character of the induced representation $V$, so that for any $g\in G$ we have
\[
\chi_V(g) = \Ind_M^G\chi_L(g) = \sum_{[s]\in G/M} \chi_L(g^s).
\]

Our plan is to show that if  $\chi$ is trivial, then $V$ is one-dimensional (and equal to $L$; and $M=G$); and then that if $M$ contains $A$ properly this case must hold. Therefore either $V$ is one-dimensional, or else $V$ is induced from a {\em nontrivial} character of $M=A$. Then we will check that all nontrivial characters of $A$ induce the same representation $W$ and that it is irreducible. Here are the arguments:

Since $A$ is normal, $t_x^s \in A$ for any $s\in G$. If $\chi=\chi_L|_A$ is trivial, this implies $\chi(t_x^s) = 1$ regardless of $s$. Thus
\begin{align*}
\chi_V(t_x) &= \Ind_M^G\chi(t_x) \\
&= \sum_{[s]\in G/M} \chi(t_x^s) \\
&= \sum_{[s]\in G/M} 1 \\
&= [G:M] = \Ind_M^G\chi(1).
\end{align*}
This equality implies that this induced representation $V$ is trivial on $A$, and therefore $G\rightarrow GL(V)$ factors through $G/A\cong Q$. But because $Q$ is abelian and $V$ is presumed irreducible, this implies that 
\[
1=\dim V = \Ind_M^G\chi(1),
\]
so we conclude $[G:M]=1$. Thus if $\chi$ is trivial, $M=G$ and $V=L$.

Above, we found that if $M$ contains $A$ properly, then $[M,M]=A$. Since $L$ is one-dimensional, it is trivial on $[M,M]=A$. Then $\chi = \chi_L|_A$ is trivial, so we are in the case $M=G$ and $V=L$ just described, and $V$ is one-dimensional.

To summarize, if $\chi$ is trivial then $V$ is one-dimensional, and if $M$ contains $A$ properly then this holds. Therefore, if $V$ is not one-dimensional, it must be that $M=A$ {\em and} $\chi$ is not trivial.

Thus the only way for an irreducible representation $V$ of $G$ to have degree $>1$ is for it to be induced from a nontrivial character of $A$. 

It was observed previously (remark \ref{rmk:freeandtrans}) that the conjugation action of $Q=G/A$ is free and transitive on the nonidentity elements of $A$. It follows that the induced action on the nontrivial elements of the character group $\widehat{A}$ is also free and transitive. (Use lemma \ref{freeaction} of the appendix to conclude that the action is free, and then count the elements of $A$ and $\widehat{A}$ to conclude transitivity.) Therefore, if $\chi$ is now any nontrivial character of $A$, then
\[
\Ind_A^G\chi(g) = \sum_{[s]\in G/A} \chi(g^s)
\]
restricted to $g\in A$ yields precisely the sum of all $q-1$ nontrivial characters of $A$.

Thus inducing from any nontrivial character $\chi$ of $A$ yields the same representation. This is $W$.

Since $A$ is normal and the action on the nontrivial characters is free, $t_x\mapsto\chi(t_x^s)$ is a distinct character of $A$ for every $s$. Thus $W$ is irreducible by Mackey's criterion (\cite{serre}, Proposition 23).
\end{proof}
%%%%% chap2 %%%%%%%%%%%%
\newpage
% !TEX root = ./thesis.tex

\chapter{Polynomial invariants of permutation groups over $\Z$ and $\F_p$}\label{ch:invar}

\section{Background and motivation}\label{sec:background}

In this chapter we consider a typical problem in invariant theory: given a group acting on a ring, describe the subring fixed under the action. This subject dates back to the second half of the 19th century (\cite[p.~1]{derksenkemper}). The original impetus came from geometry: interest in coordinate-independent descriptions of important geometric quantities naturally led to the question of which quantities do not change when you change coordinates. It became apparent, over time, that the changes of coordinates constituted a group (\cite[pp.~24--5]{eisenbud}); meanwhile, the totality of quantities being considered constituted a ring, usually a ring of polynomial functions over a field (generally $\R$ or $\C$), and the invariant quantities were a subring.

The 19th century era culminated with groundbreaking papers \cite{hilbert1}, \cite{hilbert2}, by David Hilbert, which showed that for the groups under consideration, including finite groups, the invariant subring of a polynomial ring is always finitely generated as an algebra over the ground field. This was the first instance of what is now a long tradition of theorems that give general conditions under which an invariant ring is guaranteed to be well-behaved in some important respect. Other examples, in the case of finite groups, are:

\begin{itemize}
\item {\em Noether's bound} (1916), which states that if the characteristic of the ground field does not divide the group order $|G|$, then the invariant ring is generated in degrees $\leq |G|$.\footnote{Emmy Noether's original formulation of her result required the assumption that the characteristic of the ground field, if not zero, {\em exceeds} the group order. However, it was long suspected that the bound held under the weaker assumption we have mentioned here. This discrepancy between what had been proven and what was expected became known as the ``Noether gap." It was finally closed around the turn of this century by John Fogarty (\cite{fogarty}) and Peter Fleischmann (\cite{fleischmann}), working independently.}
\item The {\em Chevalley-Shephard-Todd theorem}, discussed in the introduction, which under the same assumption gives a precise characterization of the groups whose invariant ring is ``as nice as possible", i.e. polynomial. 
\item Hochster and Eagon's 1971 result (\ref{thm:hochstereagon} below) that under the same assumption, the invariant ring always has the desirable property known as {\em Cohen-Macaulayness}, which we will explicate at length shortly.
\end{itemize}

One is struck by the shared assumption in these three theorems: that the characteristic of the ground field is prime to the group order. This is the so-called {\em nonmodular case}, and these results are representative of the fact that the theory as a whole is well-behaved in this case. In the {\em modular case}, the invariant ring can be much more pathological; all three of these theorems, and many others, can fail. Thus, interest has arisen in recent decades in sorting out when in the modular case such pathologies arise. 

We will be focused on the question of Cohen-Macaulayness. We give some sample results, to set the stage. Let $k$ be a field of characteristic $p$. Let $G$ be a finite group with order divisible by $p$, acting linearly on a $k$-vector space $V$ of dimension $n$. Let $k[V]$ be the coordinate ring of $V$ seen as a $k$-variety, in other words, the polynomial algebra $k[x_1,\dots,x_n]$ generated by coordinate functions $x_1,\dots,x_n$ on $V$, with the induced action of $G$. Let $k[V]^G$ be the ring of polynomials invariant under the action.

\begin{itemize}
\item In 1980, Ellingsrud and Skjelbred (\cite{ellingsrud}) showed that if $G$ is cyclic of order $p^k$, then $k[V]^G$ is not Cohen-Macaulay unless $G$ fixes a subspace of $V$ of codimension $\leq 2$. 
\item In 1996, Larry Smith (\cite{smith96}) showed that if $\dim_k V = 3$, then $k[V]^G$ {\em is} Cohen-Macaulay. (This was priorly known to hold for $\dim_kV\leq 2$.)
\item In 1999, Campbell et al (\cite{campbelletal}) showed that if $G$ is a $p$-group, and $V=W^{\oplus 3}$, where $W$ is any nontrivial representation of $G$, then $k[V]^G$ is not Cohen-Macaulay. 
\item Also in 1999, Gregor Kemper (\cite{kemper99}) showed that if $G$ is a $p$-group and $k[V]^G$ is Cohen-Macaulay, then $G$ is necessarily generated by elements $g$ whose fixed-point sets in $V$ have codimension $\leq 2$, generalizing Ellingsrud and Skjelbred's result beyond cyclic groups.
\end{itemize}

See \cite{kemper12} for a 2012 summary of the state of the art. 

Our interest will be in permutation groups $G\subset S_n$, acting on $V$ by permuting a basis. These have the feature that $k$ is not an essential part of the definition of the action, allowing it to be varied. Thus we can ask: 

\begin{question}\label{q:CMpermutationgroups}
For which $G\subset S_n$ is $k[V]^G$ Cohen-Macaulay regardless of $k$?
\end{question}

Here, little appears to have been priorly known. Kemper in \cite{kemper01} gave a criterion to determine Cohen-Macaulayness when the characteristic of $k$ divides $|G|$ exactly once, but for most permutation groups this leaves out several primes. A few results are established:

\begin{itemize}
\item If $G$ is a Young subgroup, then $k[V]^G$ is a polynomial algebra over $k$, so it is Cohen-Macaulay regardless of $k$. %(We give definitions and details below.)

\item Kemper (\cite{kemper99}) showed that if $G\subset S_n$ is {\em regular} (i.e. its action on $\gls{[n]}$ is free and transitive), then $k[V]^G$ is Cohen-Macaulay over every $k$ if it is isomorphic to $C_2$, $C_3$, or $C_2\times C_2$, but not otherwise. (In fact, in other cases, it is not Cohen-Macaulay for any $k$ with $\operatorname{char}k$ dividing $|G|$.)

\item Victor Reiner (\cite{reiner92}) and Patricia Hersh (\cite{hersh}, \cite{hersh2}) have shown that $A_n$, the diagonally embedded $S_n\hookrightarrow S_n\times S_n\subset S_{2n}$, and the wreath product $S_2\wr S_n\subset S_{2n}$, have invariant rings that are Cohen-Macaulay regardless of the field.
\end{itemize}

Our primary contribution in this chapter is to give a sufficient criterion for a permutation group to have a Cohen-Macaulay invariant ring over all fields. It unites and extends the results of Reiner and Hersh, and echoes Kemper's result mentioned above. We state our theorem in the next section and prove it over the course of 
this chapter.

We conjecture that our criterion is also necessary and present supporting evidence in the final section.

There is an independent motivation for question \ref{q:CMpermutationgroups}. A striking recent development in Galois theory is Manjul Bhargava and Matthew Satriano's work \cite{bhargavasatriano} extending the notion of Galois closure from finite field extensions to (locally module-free of finite rank) ring extensions. Bhargava's student Owen Biesel further developed this idea in \cite{biesel}, generalizing Bhargava and Satriano's {\em $S_n$-closure} to a more general {\em $G$-closure} operation for an arbitrary permutation group $G$.

In analyzing the $G$-closures of monogenic ring extensions (see \cite{biesel}, Chapter 5), a certain ``universal" role is played by the invariant ring of $G$ over the {\em integers}. Thus it is natural to ask questions about the structure of this ring. 

Let $R = \Z[x_1,\dots,x_n]$, and consider $R^G$. In the cases analyzed in his thesis ($A_n,\;\forall n$; $D_4\subset S_4$), Biesel found, and made use of the fact, that $R^G$ is free as a module over $R^{S_n}$. But when Riccardo Ferrario brought Biesel's approach systematically to each of the subgroups of $S_4$, he discovered (\cite{ferrario}, \S 2.2) that for $G=C_4\subset S_4$, $R^G$ is not free.

\begin{question}[Biesel] \label{q:bieselsquestion} For what permutation groups $G$ is the invariant ring $R^G$ free as an $R^{S_n}$-module? \end{question}

Below, it will be shown that this question is equivalent to question \ref{q:CMpermutationgroups}.

\section{The result, and the plan of attack}

Let $G\subset S_n$ be a permutation group, acting on the polynomial ring 
$$
R = \Z[x_1,\dots,x_n]
$$ 
by permuting the variables. Let $R^G\subset R$ be the subring of $G$-invariants. Our major goal is to prove:

\begin{thm}[Cohen-Macaulayness of integer invariants]\label{thm:mainresult}
Let $\Grr$ be the subgroup of $G$ generated by its transpositions, double transpositions, and three-cycles. If $G = \Grr$, then $R^G$ is Cohen-Macaulay.
\end{thm}

It will also be shown that the conclusion is equivalent to $R^G$ being free as an $R^{S_n}$-module.

This result is an analog to the Chevalley-Shephard-Todd theorem. Like the latter, it asserts that the invariant ring is ``nice" when the group is generated by elements fixing ``big" subspaces.\footnote{Note that, when viewed as linear transformations of $\R^n$, transpositions, double transpositions and 3-cycles are precisely the permutations that pointwise-fix subspaces of codimension at most 2.} To fold it into the theme of this thesis, we can also view it as a statement about ``well-behaved quotients." In the language of schemes, the polynomial ring's spectrum is affine $n$-space over $\Z$, carrying an action of a permutation group $G$, and the spectrum of the invariant ring is the quotient of this scheme by the action. The theorem asserts that the quotient is a Cohen-Macaulay scheme if $G=G_{rr}$. This will turn out to mean it has a finite flat morphism to affine space.

As noted in the last section, we also conjecture the converse:

\begin{conj}\label{conj:notCM}
If $R^G$ is Cohen-Macaulay, then $G=\Grr$.
\end{conj}

We discuss our evidence for this conjecture in section \ref{sec:notCM}.

Figure \ref{fig:schematic} is a schematic diagram of the argument. $G$, $R$, $R^G$, and $\Grr$ were mentioned above; the rest of the notation ($S$, $S^G$, $A[\Delta / G]$, etc.) will be defined over the course of the chapter.

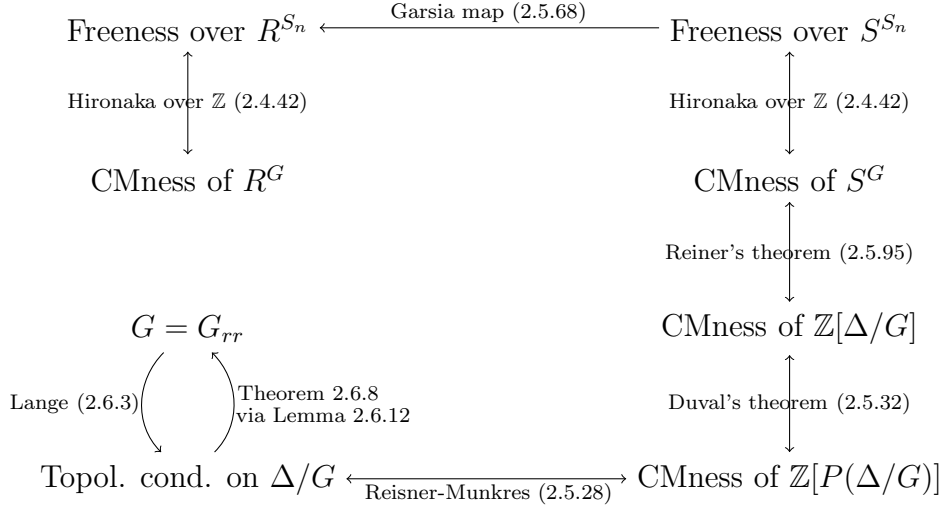
\begin{figure}
\begin{center}
\begin{tikzpicture}

\node (CMnessR) at (0,0) {CMness of $R^G$};
\node (FreenessR) at (0,2) {Freeness over $R^{S_n}$};
\draw [<->] (CMnessR) -- (FreenessR);
\node at (0,1) {\scriptsize Hironaka over $\Z$ (\ref{thm:HironakaoverZ})};

\node (CMnessS) at (8,0) {CMness of $S^G$};
\node (FreenessS) at (8,2) {Freeness over $S^{S_n}$};
\draw [<->] (CMnessS) -- (FreenessS);
\node at (8,1) {\scriptsize Hironaka over $\Z$ (\ref{thm:HironakaoverZ})};

\draw [->] (FreenessS) -- (FreenessR);
\node at (4,2.2) {\scriptsize Garsia map (\ref{thm:garsiabasis})};

\node (CMnessADel) at (8,-2) {CMness of $\Z[\Delta / G]$};
\draw [<->] (CMnessADel) -- (CMnessS);
\node at (8,-1) {\scriptsize Reiner's theorem (\ref{thm:reinerSRquotient})};

\node (CMnessPADel) at (8,-4) {CMness of $\Z[P(\Delta / G)]$};
\draw [<->] (CMnessADel) -- (CMnessPADel);
\node at (8,-3) {\scriptsize Duval's theorem (\ref{thm:duval})};

\node (HomCond) at (0,-4) {Topol. cond. on $\Delta / G$};
\draw [<->] (HomCond) -- (CMnessPADel);
\node at (4,-4.2) {\scriptsize Reisner-Munkres (\ref{thm:reisnermunkres})};

\node (Grr) at (0,-2) {$G=\Grr$};
\draw [->] (Grr) to [out=225,in=135] (HomCond);
\draw [->] (HomCond) to [out=45,in=315] (Grr);
\node at (-1.5,-3) {\scriptsize Lange (\ref{thm:lange})};
\node at (1.6,-2.85) {\scriptsize Theorem \ref{thm:quotientCMforpermgroups}};
\node at (1.8,-3.15) {\scriptsize via Lemma \ref{lem:primestabilizer}};

\end{tikzpicture}
\end{center}
\caption{Schematic diagram of proof of theorem \ref{thm:mainresult}.}\label{fig:schematic}
\end{figure}

The structure of the chapter is as follows:

In section \ref{sec:FTSPandbasics}, we state the fundamental theorem on symmetric polynomials (FTSP), and collect some other basic facts about invariant rings that are used in the sequel.

In section \ref{sec:CMrings}, we use the theory of Cohen-Macaulay rings to 
show that the Cohen-Macaulay property of $R^G$ is equivalent to the existence of a free basis for it as a module over the subring of symmetric polynomials, and we prove the equivalence of questions \ref{q:bieselsquestion} and \ref{q:CMpermutationgroups}. 

%As a first application of the theory, we prove that Young subgroups have Cohen-Macaulay invariant rings. 

%We will use this application as a running example, proving it twice more over the course of the chapter to illustrate the theory as it develops.\com{do I keep this Young thing?}%attend

In section \ref{sec:stanleyreisner}, we use work of Adriano Garsia and Dennis Stanton \cite{garsia}, \cite{garsiastanton} to connect the problem of the Cohen-Macaulayness of the polynomial invariant ring $R^G$ to the topology of a certain regular CW complex. This takes several steps:

In subsection \ref{sec:SRrings}, we review some fundamentals regarding simplicial complexes and define the {\em Stanley-Reisner ring} of a simplicial complex and a poset. We also define boolean complexes and their Stanley-Reisner rings.

In subsection \ref{sec:reisnermunkres}, we report on foundational work by Gerald Reisner, James Munkres, and Richard Stanley that relates the Cohen-Macaulayness of the Stanley-Reisner ring to the topology of the associated complex.

In subsection \ref{sec:garsiamap}, we use the ideas of Garsia and Stanton to connect the Cohen-Macaulayness of a polynomial invariant ring to that of an invariant ring $S^G$ inside a Stanley-Reisner ring $S$. We develop these ideas in a self-contained way. We define the notion of {\em stacking up}. We see this as the key to understanding the power of the {\em Garsia map}, a $\Z$-linear isomorphism defined by Garsia from the Stanley-Reisner ring to the polynomial ring. The Garsia map is ``approximately a homomorphism," in a sense that is made precise in the section, and can be used to transfer module bases. We then give a version of the classical proof of the FTSP reformulated in the language of the Garsia map. As a bonus, the theory delivers the classical homogeneous system of parameters (h.s.o.p.) for $R^G$, and a corresponding h.s.o.p. for $S^G$.

In subsection \ref{sec:reinerthm}, we introduce the notion of {\em balanced} boolean complexes and use a theorem of Victor Reiner to identify the invariant subrings $S^G$ with full Stanley-Reisner rings of balanced boolean complexes $\Delta / G$. This allows us to finally establish a connection between our original problem concerning polynomial invariants and a topological problem. This subsection concludes section \ref{sec:stanleyreisner}.

In section \ref{sec:quotientsofspheres}, we solve the topological problem, using recent work in orbifold theory by Christian Lange. We are able to determine for which groups $G$ our boolean complexes have the topological property implied by the Cohen-Macaulay property of the Stanley-Reisner ring.

In section \ref{sec:mainresult}, we assemble the proof of \ref{thm:mainresult}, per figure \ref{fig:schematic} above.

However, we aim not only to prove Cohen-Macaulayness, but to give explicit bases for the invariant ring with $\Z$-coefficients as a module over the symmetric subring. For this purpose, in section \ref{sec:shellings} we again make use of the work of Garsia and Stanton, which shows how to explicitly construct a basis from a {\em shelling} of the boolean complex, when such exists. We also extend their work, defining the notion of a {\em cell basis} of a balanced boolean complex, from which one can immediately extract a basis for the invariant ring over $\Q$, whether such exists over $\Z$, and conjecture that Cohen-Macaulayness of a balanced boolean complex always implies that a cell basis exists.

Finally, in section \ref{sec:notCM}, we discuss our conjecture \ref{conj:notCM}. We also collect other open questions prompted by our inquiry.

\section{The FTSP, and some invariant theory basics}\label{sec:FTSPandbasics}

The fundamental theorem on symmetric polynomials (FTSP) may be thought of as the primordial theorem of invariant theory, since it describes the invariant ring of a group acting on a ring, in terms of explicit generators, and yet it predates the existence of the field of invariant theory by arguably over a century. Indeed, it predates the concepts of groups and rings.\footnote{There is some ambiguity regarding how exactly to date the theorem. Two respected historical monographs on Galois theory, \cite{edwards} and \cite{tignol}, tell somewhat different stories based on differing standards about what to count. Harold Edwards, in \cite{edwards}, credits Newton with knowledge of the essence of the theorem, and both authors describe the theorem as entering the collective consciousness of mathematicians over the course of the 18th century. Edwards sets a standard for what constitutes a ``precise statement and proof" that likely was not met before the 1816 proof by Gauss, based on the lexicographic monomial order, mentioned in the next paragraph. Jean-Pierre Tignol, in \cite{tignol}, credits Edward Waring's 1770 {\em Meditationes Algebraicae} with the first printed proof (see \cite{waring}, Problems I -- III in Chapter 1) and even with the lexicographic order idea. In our opinion, Waring's work does contain a proof over $\Q$, modulo a techicality important to Edwards, but not over $\Z$, and we think Tignol is being much too generous to read the lexicographic argument into it. In addition to his own proof, Waring does (Problem III, section 3) describe in vague terms a method for writing down a representation of a symmetric polynomial that is reminiscent of Gauss' lexicographic algorithm, but he does not define lexicographic order or argue that the algorithm terminates.}

This theorem, and its classical 1816 proof by Gauss, are wellsprings of inspiration for our inquiry, so it serves us to recall them here. 

\begin{definition}\label{def:elemsympolys}
Let $x_1,\dots,x_n$ be indeterminates in a polynomial ring $R$. The \textbf{elementary symmetric polynomials} in the $x_i$ are the $n$ polynomials
\begin{align*}
\sigma_1 &= \sum_{j\in \gls{[n]}} x_j = x_1 + x_2 + \dots + x_n\\
\sigma_2 &= \sum_{S\in \binom{\gls{[n]}}{2}} \prod_{j\in S} x_j = x_1x_2 + x_1x_3 + \dots + x_{n-1}x_n\\
&\:\vdots \\
\sigma_i &= \sum_{S \in \binom{\gls{[n]}}{i}} \prod_{j\in S} x_j\\
&\:\vdots \\
\sigma_n &= \prod_{j\in \gls{[n]}} x_j = x_1x_2\dots x_n
\end{align*}
where $\binom{\gls{[n]}}{i}$ is the set of $i$-subsets of $\gls{[n]}$.
\end{definition}

\begin{thm}[FTSP]\label{thm:ftsp}
Let $A$ be an arbitrary commutative unital ring, and let $R=A[x_1,\dots,x_n]$. Let $S_n$ act on $R$ by permuting the variables. Then the subring $A[\sigma_1,\dots,\sigma_n]$ is a polynomial algebra, and it coincides precisely with $R^{S_n}$.\footnote{In terms that would have been somewhat more familiar to the seventeenth-through-nineteenth century authors who discovered it, this theorem is asserting that every symmetric polynomial is representable in a unique way as a polynomial in the $\sigma_i$'s. Most of these authors thought of the $x_i$'s as being roots of a polynomial of which the $\sigma_i$'s are coefficients; the technicality important to Edwards, mentioned above, is that they did this without ever proving that polynomials always have roots.}
\end{thm}

Gauss' beautiful proof, based on the lexicographic order on monomials, can be found in many sources, e.g. \cite{jacobson}, Theorem 2.20, and \cite{sturmfels}, Theorem 1.1.1, and the original \cite{gauss}, paragraphs 3--5. We have discussed this proof at length elsewhere (\cite{FTSP}). Below (\ref{thm:ftsp2}), we give a proof based on the tools we develop in section \ref{sec:stanleyreisner}. At its heart it is nothing but the Gauss proof. In our view, the tools of section \ref{sec:stanleyreisner} provide insight into these classical ideas.

We turn to a few elementary results in invariant theory of which we will make use in the sequel.

\begin{notation}
Let $A$ be a commutative, unital, noetherian ring, and let $R$ be a finitely generated, graded $A$-algebra. In subsequent sections, except for section \ref{sec:CMrings}, we will always take $R=A[x_1,\dots,x_n]$, with the standard grading (i.e. the grading given by assigning degree $1$ to each indeterminate), unless otherwise noted, and $A=\Z$ or a field. Let $G$ be a finite group acting on $R$ by graded $A$-algebra automorphisms. If $f\in R$, we denote by $g(f)$ the image of $f$ under the action of $g$.
\end{notation}

\begin{prop}\label{prop:Risintegral}
$R$ is finite as a module over $R^G$. 
\end{prop}

\begin{proof}
Any $f\in R$ satisfies the monic polynomial $F(X) = \prod_{g\in G} (X-g(f))$. This polynomial's coefficients are elementary symmetric polynomials in the images $g(f)$ of $f$; thus they are $G$-invariant, i.e. $F(X)\in R^G[X]$; this proves integrality. Since $R$ is finitely generated as an $A$-algebra and thus as an $R^G$-algebra, being integral over $R^G$ implies being module-finite over it (this is corollary 5.2 in \cite{atiyahmacdonald}).
\end{proof}

\begin{prop}\label{prop:fgasalgebra}
$R^G$ is finitely generated as an $A$-algebra.
\end{prop}

\begin{proof}
We have $A\subset R^G\subset R$. Because $A$ is noetherian, $R$ is finitely generated over $A$, and per \ref{prop:Risintegral}, $R$ is integral over $R^G$, the Artin-Tate lemma (\cite{atiyahmacdonald}, Proposition 7.8) immediately implies $R^G$ is finitely generated as an $A$-algebra.
\end{proof}

\begin{prop}\label{prop:reynoldsoperator}
If $|G|$ is invertible in $A$, there exists an $R^G$-linear projection 
$$
\Omega: R\rightarrow R^G.
$$
\end{prop}

\begin{proof}
The map $\Omega$ is given by averaging over $G$:
\[
\Omega(\cdot) = \frac{1}{|G|}\sum_{g\in G} g(\cdot)
\]
It is clear from the formula that it restricts to the identity on $R^G$ and that its image lies in $R^G$, thus it is a projection. The linearity as $R^G$-module follows from a direct calculation. Let $f_i\in R^G$ (and thus $\Omega(f_i)= f_i$), let $r_i\in R$, and keep in mind that each $g\in G$ is a ring automorphism:
\begin{align*}
\Omega\left(\sum_i f_i r_i\right) &= \frac{1}{|G|}\sum_{g\in G}g\left(\sum_i f_ir_i\right) \\
&= \frac{1}{|G|}\sum_{g\in G} \sum_i g(f_i)g(r_i) \\
&= \frac{1}{|G|}\sum_{g\in G} \sum_i f_ig(r_i) \\
&= \sum_i f_i \frac{1}{|G|}\sum_{g\in G} g(r_i) \\
&= \sum_i f_i \Omega(r_i).\qedhere
\end{align*}
\end{proof}

\begin{definition}
The projection $\Omega(\cdot)$ is called the \textbf{Reynolds operator}. It is the main tool for proving structural results in the nonmodular case. It does not exist in the modular case, which leads to substantial complications. 
\end{definition}

\begin{notation}
Now specialize to the case that $R=A[x_1,\dots,x_n]$ and $G$ is a permutation group, acting by permutations of the indeterminates $x_i$. Then we have $R^{S_n}\subset R^G \subset R$.
\end{notation}

This situation has several particularly nice features:

\begin{prop}\label{prop:integraloverSn}
As a  module, $R^G$ is finite over $R^{S_n}$, and thus integral over $R^{S_n}$.
\end{prop}

\begin{proof}
By applying \ref{prop:Risintegral} to $S_n$'s action on $R$, we find that $R$ is a finite $R^{S_n}$-module, and $R^G$ is a submodule.
\end{proof}

Another feature is that taking invariants commutes with base change:

\begin{prop}\label{prop:basechange}
If $B$ is any $A$-algebra, then $G$ acts on the tensor product $B\otimes_A R$ through its action on $R$, and we have
\[
(B\otimes_A R)^G = B \otimes_A R^G
\]
\end{prop}

\begin{proof}
This is a result of the fact that $G$ acts on the set of monomials in $R$. A consequence is that $R^G$ has an $A$-basis consisting of sums of monomials across orbits (which we call \textbf{orbit monomials}, following Reiner \cite{reiner95}). This becomes a $B$-basis of $B\otimes_A R^G$; but the same argument starting with $B\otimes_A R = B[x_1,\dots,x_n]$ shows that $(B\otimes_A R)^G$ has the same basis.
\end{proof}

For a quick introduction to the invariant theory of finite groups, we refer the reader to \cite{stanley4}. For a fuller treatment, with attention to the modular case, see \cite{neuselsmith}, \cite{smith95}, or \cite{derksenkemper}, the latter of which also treats some infinite groups. The recent \cite{campbellwehlau} is devoted to the modular case.

\section{Cohen-Macaulay rings}\label{sec:CMrings}

In this section we introduce the basic concepts of the theory of Cohen-Macaulay rings 
and explain why they are of interest to modular invariant theorists. We prove the equivalence of questions \ref{q:bieselsquestion} and \ref{q:CMpermutationgroups}. We rely heavily on the treatments in \cite{brunsherzog} and \cite{eisenbud}.

\begin{notation}
In this section, $R$ will be an arbitrary noetherian ring except where otherwise noted. It will revert to its meaning as a polynomial algebra in the next section.
\end{notation}

\begin{definition}\label{def:regularseq}
Let $R$ be a ring and let $x_1,\dots,x_n$ be elements of $R$. If $x_1$ is a nonunit nonzerodivisor of $R$, and $x_i$ is a nonunit nonzerodivisor of $R/(x_1,\dots,x_{i-1})$ for $i = 2,\dots, n$, then $x_1,\dots,x_n$ is a \textbf{regular sequence} in $R$.
\end{definition}

The condition in this definition is more compactly stated if we adopt the natural convention that an empty sequence of elements generates the zero ideal; then we can state the condition as that $x_i$ is a nonunit nonzerodivisor on $R/(x_1,\dots,x_{i-1})$ for all $i$.

Regular sequences behave ``kind of like indeterminates," in the following sense:

\begin{prop}[\cite{brunsherzog}, Theorem 1.1.8]
Let $x_1,\dots,x_k$ be a regular sequence in $R$, and let $I = (x_1,\dots,x_k)$ be the ideal they generate. Let 
\[
\gr_I(R) = R/I \oplus I/I^2 \oplus \dots
\]
be the associated graded ring. Then $\gr_I(R)$ is isomorphic to the polynomial ring $R/I[X_1,\dots,X_k]$, where $X_1,\dots,X_k$ are indeterminates.\qed
\end{prop}

\begin{lemma}[\cite{brunsherzog}, Corollary 1.1.3]
If $x_1,\dots,x_n$ is a regular sequence in $R$, and $\primep$ is a prime ideal of $R$ containing $x_1,\dots,x_n$, then the images of $x_1,\dots,x_n$ in $R_\primep$ also form a regular sequence in the latter ring. \qed
\end{lemma}

\begin{definition}\label{def:depth}
Let $(R,\maxm)$ be a noetherian local ring. The \textbf{depth} of $R$ is the length of the longest regular sequence in $\maxm$.
\end{definition}

\begin{prop}[\cite{brunsherzog}, Proposition 1.2.12]\label{prop:depthleqdim}
If $(R,\maxm)$ is a noetherian local ring, its depth is bounded above by its Krull dimension: 
\[
\pushQED{\qed}
\operatorname{depth} R \leq \dim R.\qedhere
\popQED
\]
\end{prop}

\begin{definition}\label{def:CMlocal}
The noetherian local ring $(R,\maxm)$ is said to be \textbf{Cohen-Macaulay} if $\operatorname{depth} R = \dim R$.
\end{definition}

Cohen-Macaulayness of arbitrary noetherian rings is defined locally:

\begin{definition}[Cohen-Macaulayness]\label{def:CM}
A noetherian ring $R$ is \textbf{Cohen-Macaulay} if $(R_\primep, \primep_\primep)$ is a Cohen-Macaulay local ring for each prime $\primep\in\Spec R$.
\end{definition}

\begin{remark}
It is sufficient to check the condition in definition \ref{def:CM} on maximal ideals $\maxm\in\MaxSpec R$. Note that the definition we give here presumes that the ring is noetherian.
\end{remark}

\begin{cor}
A Dedekind domain is Cohen-Macaulay. In particular, $\Z$ is Cohen-Macaulay.
\end{cor}

\begin{proof}
The localization of a Dedekind domain at a maximal ideal is a discrete valuation ring (DVR), which is a one-dimensional local ring with maximal ideal generated by (and therefore containing) a nonunit nonzerodivisor, which thus constitutes a regular sequence of length one.
\end{proof}

Cohen-Macaulayness is preserved under localization, completion, adjoining indeterminates, and taking quotients by regular sequences; see \cite{brunsherzog}, Chapter 2.

%\begin{prop}[\cite{brunsherzog}, Theorem 2.1.3(a)]\label{prop:regularquotient}
%If $R$ is a Cohen-Macaulay ring and $x_1,\dots,x_n$ is a regular sequence in $R$, then $R/(x_1,\dots,x_n)$ is a Cohen-Macaulay ring.
%\end{prop}

The theory takes on a particularly elegant form when the ring $R$ is a finitely generated $\N$-graded algebra over a field $k$.

\begin{definition}\label{def:connected}
In the context of a graded $A$-algebra, \textbf{connected} is used here to mean that the degree zero piece equals the coefficient ring $A$.
\end{definition}

\begin{remark}\label{rmk:connectedisconnected}
This should not be confused with the common use of the same word to mean that the ring lacks idempotents other than $0$ and $1$, or equivalently that its $\Spec$ is connected as a topological space. To avoid confusion, when referring to a ring, we will only use the word {\em connected} in the sense of definition \ref{def:connected}, and when we need to assert that the prime spectrum is connected as a topological space, we will use phrases like {\em $R$ has a connected spectrum.} The author believes that the usage here (which is taken from \cite{smith95}, \cite{smith97}) is inspired by the cohomology ring of a topological space, which in degree zero is just the coefficient ring if and only if the space is connected. Nonetheless, the present usage actually makes the connectedness of the $\Spec$ equivalent to that of the coefficient ring $A$; see \ref{lem:connectedisconnected} below.
\end{remark}

\begin{notation}\label{notation:graded}
Until stated otherwise, let $R=\bigoplus_{\ell \geq 0} R_\ell$ be a finitely generated connected $\N$-graded algebra over a field $k$.
\end{notation}

\begin{definition}\label{def:hsop}
Let $\theta_1,\dots,\theta_n\in R$ be homogeneous elements of $R$ that are algebraically independent over $k$ and such that $R$ is finite as a module over the subring $k[\theta_1,\dots,\theta_n]$. Then $\theta_1,\dots,\theta_n$ is a \textbf{homogeneous system of parameters}, or $\textbf{h.s.o.p.}$ for short.
\end{definition}

\begin{prop}[\cite{smith97}, Theorem 1.5]
The Krull dimension of $R$ is equal to the length of any h.s.o.p.\qed
\end{prop}

While Cohen-Macaulayness is defined locally, for a finitely generated connected $\N$-graded $k$-algebra there is a simple global test:

\begin{prop}\label{prop:CMforgraded}
In the present setting, $R$ is a Cohen-Macaulay ring if and only if some h.s.o.p. is a regular sequence, if and only if any h.s.o.p. is a regular sequence.
\end{prop}

\begin{proof}
If some h.s.o.p. is a regular sequence, it remains regular after localizing at the unique maximal graded ideal $\maxm$. Then $R_\maxm$ is Cohen-Macaulay since it has a regular sequence of length equal to its Krull dimension. By \cite{brunsherzog} exercise 2.1.27(c), this implies $R$ is Cohen-Macaulay. 

In the other direction, the implication
\[
\text{Cohen-Macaulay $\Rightarrow$ every h.s.o.p. is a regular sequence}
\]
 is Corollary 6.7.7 in \cite{smith95}.
\end{proof}

\begin{cor}\label{cor:polyringCM}
A polynomial ring over a field is Cohen-Macaulay.
\end{cor}

\begin{proof}
The indeterminates form a regular sequence.
\end{proof}

Since the elements of a h.s.o.p. $\theta_1,\dots,\theta_n$ are algebraically independent, the subring $k[\Theta] = k[\theta_1,\dots,\theta_n]\subset R$ is a polynomial algebra. The ring $R$ is then a finite module over this polynomial algebra. One may measure the homological complexity of $R$ by the length of a minimal free resolution of it as a $k[\Theta]$-module. By this measure, Cohen-Macaulay rings are homogically simple:

\begin{prop}[\cite{stanley4}, Lemma 3.3]
An h.s.o.p. $\theta_1,\dots,\theta_n$ is a regular sequence if and only if $R$ is free as a module over $k[\Theta] = k[\theta_1,\dots,\theta_n]$.\qed
\end{prop}

\begin{cor}[Hironaka's criterion]\label{cor:CMfreemodule}
The ring $R$ is Cohen-Macaulay if and only if is free as a $k[\theta_1,\dots,\theta_n]$-module, where $\theta_1,\dots,\theta_n$ is any h.s.o.p.\qed
\end{cor}

\begin{remark}\label{rmk:depthofgraded}
One may have detected in the above the strong analogy between the local and graded-algebra-over-a-field cases, with the role of the unique maximal ideal in the former played by the unique graded maximal ideal in the latter. Indeed, in the present (latter) setting, the maximum length of a regular sequence contained in the positively graded ideal is again called the \textbf{depth}. Thus a finitely generated graded $k$-algebra, like a local ring, is Cohen-Macaulay if its depth equals its dimension.
\end{remark}

The relation of Cohen-Macaulayness to the invariant theory of finite groups begins with the fact mentioned in the chapter introduction that in the nonmodular case, invariant rings are always Cohen-Macaulay:

\begin{thm}[Hochster-Eagon theorem]\label{thm:hochstereagon}
If $k$ is a field, and $G$ is a finite group with order not divisible by the characteristic of $k$, having a graded action on the polynomial ring $R = k[x_1,\dots,x_n]$, then the invariant ring $R^G$ is Cohen-Macaulay.
\end{thm}

Hochster and Eagon's 1971 paper actually proves the more general statement that if $R$ is any unital noetherian Cohen-Macaulay ring (not necessarily graded or finitely generated over a field), $G$ acts by automorphisms (not necessarily graded automorphisms), and $|G|$ is a unit in $R$, then $R^G$ is Cohen-Macaulay (\cite{hochstereagon}, proposition 13). However the version we give here is the classical situation of invariant theory, and it is subject to a short proof using tools so far developed:

\begin{proof}
We use proposition \ref{prop:CMforgraded}. Let $\theta_1,\dots,\theta_n$ be a h.s.o.p. for $R^G$. The ring extension $R^G\subset R$ is finite by proposition \ref{prop:Risintegral}. Thus the composed ring extension $k[\theta_1,\dots,\theta_n]\subset R^G\subset R$ is finite, and $\theta_1,\dots,\theta_n$ is a h.s.o.p. for $R$.

$R$ is Cohen-Macaulay by \ref{cor:polyringCM}, so by \ref{prop:CMforgraded}, $\theta_1,\dots,\theta_n$ must be a regular sequence in it. Unwinding definition \ref{def:regularseq}, the statement that $\theta_i$ is a nonzerodivisor in $R/(\theta_1,\dots,\theta_{i-1})R$ translates to the statement that
\[
\theta_i f \in R\theta_1 + \dots + R\theta_{i-1} \: \Rightarrow \: f \in R\theta_1 + \dots + R\theta_{i-1}
\]
for each $i=1,\dots,n$ (where the ideal $R\theta_1+\dots+R\theta_{i-1}$ is a void sum when $i=1$ and thus is taken as $(0)$).

We will show $\theta_1,\dots,\theta_n$ is a regular sequence in $R^G$. For any $i=1,\dots,n$, suppose $f\in R^G$ and $\theta_i f\in R^G\theta_1 + \dots + R^G\theta_{i-1}$. Then certainly $\theta_i f\in R\theta_1+\dots+R\theta_i$, so the above applies, and $f\in R\theta_1 + \dots + R\theta_{i-1}$. Let
\begin{equation}\label{eq:finRtheta}
f = \sum_{j=1}^{i-1} r_j\theta_j,\: r_j\in R
\end{equation} 
be the expression for $f$ whose existence is implied by this. Now apply the Reynolds operator $\Omega$ whose existence was established in \ref{prop:reynoldsoperator} to both sides of \eqref{eq:finRtheta}. Recalling that $\Omega$ is $R^G$-linear, and that $f, \theta_1,\dots,\theta_{i-1}\in R^G$, we obtain
\[
f = \Omega(f) = \sum_{j=1}^{i-1}\Omega(r_j)\theta_j.
\]
Each $\Omega(r_j)$ lies in $R^G$, thus this equation expresses $f$ as an element of 
\[
R^G\theta_1+\dots+R^G\theta_{i-1}.
\]
Therefore $\theta_i$ is a nonzerodivisor in $R^G/(\theta_1,\dots,\theta_{i-1})R^G$. Furthermore it is a nonunit, since it is positive degree (and $\theta_1,\dots,\theta_{i-1}$ generates a homogeneous ideal, so that $R^G/(\theta_1,\dots,\theta_{i-1})$ inherits $R^G$'s grading). Thus, $\theta_1,\dots,\theta_n$ is a regular sequence in $R^G$, which is therefore Cohen-Macaulay by \ref{prop:CMforgraded}.
\end{proof}

\begin{remark}
The existence of $\Omega$ is the only use of the assumption $\operatorname{char}k\nmid |G|$ in this proof.
\end{remark}

The fact that the theorem fails, but not uniformly, in the modular case, was mentioned in the chapter introduction. 

\begin{example}
Let $k=\F_2$ and let $G=C_4 = \langle (1234)\rangle \subset S_4$, acting by permutations on the indeterminates of $R = k[x_1,\dots,x_4]$. Then $R^G$ is not Cohen-Macaulay. A h.s.o.p. is formed by $\sigma_1,\dots,\sigma_4$, as we will see below in \ref{prop:hsops}, but there is no module basis for $R^G$ over $k[\sigma_1,\dots,\sigma_4]$. The polynomials
\begin{align*}
g_0 &= 1\\
g_2 &= x_1x_3 + x_2x_4\\
g_3 &= x_1^2x_2 + x_2^2x_3 + x_3^2x_4 + x_4^2x_1\\
g_{4a} &= x_1^2x_2x_3 + x_2^2x_3x_4 + x_3^2x_4x_1 + x_4^2x_1x_2\\
g_{4b} &= x_1x_2^2x_3 + x_2x_3^2x_4 + x_3x_4^2x_1 + x_4x_1^2x_2\\
g_5 &= x_1^2x_2^2x_3 + x_2^2x_3^2x_4 + x_3^2x_4^2x_2 + x_4^2x_1^2x_2\\
\end{align*}
do form a module basis for $\Q[x_1,\dots,x_4]^G$ over $\Q[\sigma_1,\dots,\sigma_4]$, but they fail either to be linearly independent or to span $R^G$ over $k[\sigma_1,\dots,\sigma_4]$ because of relations (over $\Q$) such as
\[
2\left(x_1^3x_2^2x_3+x_2^3x_3^2x_4+x_3^3x_4^2x_1 + x_4^3x_1^2x_2 \right)= \sigma_3g_3 + \sigma_2g_{4b} + \sigma_1g_5.
\]
Non-Cohen-Macaulayness of $R^G$ means every candidate basis will have a similar problem.

On the other hand, if $G = D_4 = \langle C_4, (13)\rangle$, then $R^G$ is Cohen-Macaulay.
\end{example}

Since our primary aim regards the ring $\Z[x_1,\dots,x_n]^G$ for a permutation group $G$, which is a graded algebra over $\Z$ rather than a field, we need to tie this ring to the setting in which we have been working:

\begin{prop}\label{prop:equivofZandallfields}
Let $R$ be a finitely generated graded connected $\Z$-algebra that is free as a $\Z$-module. Then the following are equivalent:
\begin{enumerate}
\item $R$ is Cohen-Macaulay.
\item $R\otimes\Q$ is Cohen-Macaulay, and $R\otimes \F_p$ is Cohen-Macaulay for every prime $p$.
\item $R\otimes_\Z k$ is Cohen-Macaulay for every field $k$.
\item $R\otimes_\Z A$ is Cohen-Macaulay for every Cohen-Macaulay ring $A$.
\end{enumerate}
\end{prop}

\begin{proof}
This is a generalization of exercise 5.1.25 in \cite{brunsherzog}. Here are the details:

1$\Rightarrow$2: for $\Q$, localizations of Cohen-Macaulay rings are Cohen-Macaulay since Cohen-Macaulayness is a local property. For $\F_p$, the claim follows because Cohen-Macaulayness is preserved by quotients of regular sequences, since $p$ is a nonzerodivisor of $R$ because it is free as a $\Z$-module.

2$\Rightarrow$3: Every field $k$ contains one of the fields $\Q,\F_p$, and for finitely generated algebras over a field, Cohen-Macaulayness is retained under arbitrary field extensions (\cite{brunsherzog}, Theorem 2.1.10).

3$\Rightarrow$4: As $R$ is free as a $\Z$-module, $R\otimes A$ is free, and therefore faithfully flat, as an $A$-module. Then \cite{brunsherzog} exercise 2.1.23 applies, which states that the Cohen-Macaulayness of a faithfully flat extension can be deduced from that of the base and of every fiber. The base $A$ is Cohen-Macaulay by assumption, and the fibers $R\otimes A \otimes \kappa(\primep) = R\otimes\kappa(\primep)$, $\primep\in \Spec A$, are Cohen-Macaulay by 3.

4$\Rightarrow$1: Take $A=\Z$.
\end{proof}

\begin{remark}
An examination of this proof shows that the assumption that $R$ is graded (and connected) can be dropped. But we will only use this result on graded rings.
\end{remark}

We also have a $\Z$-analogue to the Hironaka criterion \ref{cor:CMfreemodule}. In the following subsection, we develop the needed machinery to state and prove this result.

\subsection{Hironaka decomposition over $\Z$}

%consider adding examples to illustrate all the dimension-theoretic concepts in this section.

The theorem we aim to prove in this subsection states essentially that, just as for a $k$-algebra, Cohen-Macaulayness can be detected for a graded $\Z$-algebra by freeness as a module over the polynomial subalgebra generated by an h.s.o.p. The precise statement is given below in \ref{thm:HironakaoverZ}. Although it is a known result (David Eisenbud, personal communication), we are not aware of a careful proof in the literature. We will derive it as a consequence of a more general theorem (\ref{thm:generalHironaka}) relating the Cohen-Macaulayness of a noetherian ring with connected spectrum to its projectivity as a module over an equidimensional regular subring.

Applying the general result in our situation requires knowing that our $\Z$-algebra is equidimensional. Thus we need to develop some dimension theory. We now recall the relevant notions and lemmas.

%\begin{definition}\label{def:dimensionofideal}
%The \textbf{dimension of an ideal $I$} in a ring $R$ is the Krull dimension of the ring $R/I$, i.e. the maximum length of an ascending chain of prime ideals containing $I$.
%\end{definition}

%\begin{definition}\label{def:height}
%The \textbf{height} of a prime ideal $\primep$ in a ring $R$ is the Krull dimension of the ring $R_\primep$, i.e. the maximum length of a chain of primes descending from $\primep$. If $I$ is an arbitrary ideal, its height is the infimum of the heights of the primes containing it.
%\end{definition}

\begin{definition}
A ring $R$ is \textbf{equidimensional} if all of its maximal ideals have the same height and all of its minimal prime ideals have the same dimension. It is \textbf{catenary} if all saturated chains of prime ideals connecting any specific two have the same length. It is \textbf{universally catenary} if all finitely generated algebras over it are catenary.
\end{definition}

\begin{lemma}[\cite{brunsherzog}, Theorem 2.1.12]\label{lem:catenary}
Cohen-Macaulay rings are universally catenary.\qed
\end{lemma}

%\begin{lemma}
%The minimal number of generators of the maximal ideal $\maxm$ of a local noetherian ring $(R,\maxm)$ is equal to the dimension of $\maxm/\maxm^2$ as a $k$-vector space, where $k$ is the residue field $R/\maxm$, and we have
%\[
%\dim_k \maxm/\maxm^2 \geq \dim_{Krull} R.
%\]
%\end{lemma}

%\begin{proof}
%Let $y_1,\dots,y_\ell$ be a set of minimal generators for $\maxm$ over $R$. Since they generate, their images in $\maxm/\maxm^2 = \maxm\otimes_R R/\maxm$ generate it over $R/\maxm = k$. Thus 
%\[
%\ell \geq \dim_k \maxm/\maxm^2.
%\]
%On the other hand, let $y'_1,\dots,y'_m$ be elements of $\maxm$ whose images in $\maxm/\maxm^2$ are a $k$-basis of it (so $m=\dim_k\maxm/\maxm^2$). Then $(y'_1,\dots,y'_m) + \maxm^2 = \maxm$, so by the Nakayama lemma, $y'_1,\dots,y'_m$ generate $\maxm$. By the minimality of the $y_i$'s, 
%\[
%\ell \leq m = \dim_k \maxm/\maxm^2,
%\]
%so we have equality. Now the Krull height theorem (\cite{eisenbud}, Theorem 10.2) implies that the height of $\maxm$, which by definition is $\dim_{Krull} R$, cannot exceed $\ell=m = \dim_k\maxm/\maxm^2$.
%\end{proof}

%\begin{definition}\label{def:regularring}
%A noetherian ring is \textbf{regular} if all of its localizations at prime ideals are regular local rings.
%\end{definition}

%\begin{lemma}[\cite{brunsherzog}, Proposition 2.2.9]
%The localization of a regular local ring at any of its primes is a regular local ring. Thus it is only necessary to check the condition in \ref{def:regularring} at maximal ideals.\qed
%\end{lemma}

\begin{prop}[\cite{eisenbud}, Corollary 18.11]
In a Cohen-Macaulay local ring, all minimal primes have the same dimension.\qed
\end{prop}

\begin{cor}\label{prop:locallyequidimensional}
A Cohen-Macaulay ring $R$ is \textbf{locally equidimensional}, i.e. for any fixed prime ideal $\primeq\triangleleft R$, the length of a saturated chain from a minimal prime $\primep$ to $\primeq$ does not depend on the choice of $\mathfrak{p}$.\qed
\end{cor}

One might hope that local equidimensionality in this sense would mean that the ring is equidimensional on the connected components of its spectrum, but this is not so in general. For example, the localization of a polynomial ring over a field at the complement of the union of two primes of different heights is Cohen-Macaulay but not equidimensional. However, in a circumstance of use to us, a Cohen-Macaulay ring with connected spectrum can be guaranteed to be equidimensional -- see \ref{prop:connectedimpliesequidimensional} below.

To prepare this result, we recall a basic fact about integral ring extensions:

%\begin{lemma}[\cite{atiyahmacdonald}, Theorem 5.10, ``lying-over"]\label{lem:lyingover}
%There exists a prime $\primep$ of $R$ with $\primeq = \primep\cap S$.\qed
%\end{lemma}

%\begin{definition}
%Such a prime $\primep$ is said to \textbf{lie over} $\primeq$.
%\end{definition}

%\begin{lemma}[\cite{atiyahmacdonald}, Corollary 5.9, ``incomparability"]\label{lem:incomparability}
%There cannot be a proper containment between two primes of $R$ lying over $\primeq$.
%\end{lemma}

%\begin{lemma}[\cite{atiyahmacdonald}, Proposition 5.6(i) and Corollary 5.8]\label{lem:lyingoverprops}
%If $\primep$ lies over $\primeq$, then $R/\primep$ is integral over $S/\primeq$, and $\primep$ is maximal if and only if $\primeq$ is maximal.
%\end{lemma}

%\begin{lemma}[\cite{atiyahmacdonald}, Theorem 5.11, ``going-up"]\label{lem:goingup}
%If $\primeq_1\subset\primeq_2\subset\dots\subset\primeq_s$ is a chain of primes in $S$, and $\primep_1\subset\primep_2\subset\dots\subset\primep_r$, with $r<s$, is a chain of primes in $R$ satisfying $\primeq_i = \primep_i \cap S$ for each $i\in[r]$, then the chain $\primep_1\subset\dots\subset\primep_r$ can be extended to a chain $\primep_1\subset\dots\subset\primep_s$ with $\primeq_i = \primep_i\cap S$ for each $i\in [s]$.
%\end{lemma}

\begin{lemma}\label{lem:integralpreservesdim}
If $R$ is an integral ring extension of $S$, then the Krull dimensions of $R$ and $S$ are equal.
\end{lemma}

\begin{proof}
Let $\dim R = \ell$ and $\dim S = d$.

Let $\primeq_0\subset\dots\subset\primeq_d$ be a maximal chain of primes in $S$. By lying-over (\cite{atiyahmacdonald}, Theorem 5.10) and going-up (\cite{atiyahmacdonald}, Theorem 5.11), there exists a chain of primes $\primep_0\subset\dots\subset\primep_d$ in $R$ with $\primep_i\cap S = \primeq_i$ for each $i$. Thus $\ell \geq d$.

Let $\primep_0\subset\dots\subset\primep_\ell$ be maximal chain of primes in $R$. Then $\primep_0\cap S \subset\dots\subset\primep_\ell\cap S$ is a chain of primes in $S$, and the primes are all distinct, by incomparability (\cite{atiyahmacdonald}, Corollary 5.9). Thus $\ell \leq d$.
\end{proof}

We make the following definition in order to state lemma \ref{lem:graphandspec}, which will be used to link equidimensionality and the connectedness of the spectrum.

\begin{definition}
If $R$ is a noetherian ring, define the \textbf{containment graph} $\Gamma_R$ to be the bipartite graph with vertex sets $V_{min} = \text{minimal primes of $R$}$ and $V_{max} = \text{maximal ideals of $R$}$, with an edge for every pair $\primep\in V_{min}, \maxm\in V_{max}$ satisfying $\primep\subset\maxm$.
\end{definition}

\begin{lemma}\label{lem:graphandspec}
The spectrum of a noetherian ring $R$ is connected if and only if its containment graph $\Gamma_R$ is connected.
\end{lemma}

\begin{proof}
$\Rightarrow$ Suppose $\Spec R$ is disconnected, and let $\Spec R = W_0\cup W_1$ be a partition into disjoint nonempty closed sets. Neither $W_0$ nor $W_1$ can contain all the minimal primes without containing all of $\Spec R$, thus $W_0\cap V_{min}$ and $W_1\cap V_{min}$ are also disjoint and nonempty. But $W_0$ contains every prime of $R$ containing any element of $W_0\cap V_{min}$, and in particular every maximal ideal containing any such element. This means there are no elements of $W_1\cap V_{max}$ that have an edge in $\Gamma_R$ to any element of $W_0\cap V_{min}$. Likewise, no elements of $W_0\cap V_{max}$ have an edge to any element of $W_1\cap V_{min}$. Thus $W_0$ and $W_1$ also disconnect $\Gamma_R$. This half of the argument does not rely on the noetherian hypothesis.

$\Leftarrow$ Suppose $\Gamma_R$ is disconnected and let $W_0$ and $W_1$ be a partition of $V_{min}\cup V_{max}$ into nonempty sets of vertices with no edges between them. Note that because $R$ is noetherian, $V_{min}$ is a finite set.

It must be that $W_0$ and $W_1$ both meet $V_{min}$ nontrivially. Since $W_0$ is nonempty, then either it meets $V_{min}$ in the first place, or else it meets $V_{max}$, say at $\maxm\in W_0\cap V_{max}$, and then we must have $\primep\in W_0\cap V_{min}$ for any minimal prime $\primep$ contained in $\maxm$. Similar logic applies to $W_1$.

Let
\[
I_0 = \bigcap_{\primep\in W_0\cap V_{min}} \primep,\; I_1 = \bigcap_{\primep\in W_1\cap V_{min}} \primep
\]
Then $V(I_0) = \bigcup_{\primep\in W_0\cap V_{min}} V(\primep)$ contains every prime ideal containing any minimal prime in $W_0\cap V_{min}$, and similarly for $V(I_1)$ and $W_1\cap V_{min}$. If $V(I_0)$ and $V(I_1)$ shared any element $\primeq$ of $\Spec R$, then they would also share any maximal ideal containing $\primeq$, and then this maximal would be an element of $V_{max}$ with edges to both $W_0\cap V_{min}$ and $W_1\cap V_{min}$, contradicting the assumption that $W_0,W_1$ disconnect $\Gamma_R$. Therefore $V(I_0)$ and $V(I_1)$ are disjoint. Since their union is a closed set containing every minimal prime of $\Spec R$, it exhausts $\Spec R$, so $V(I_0)$ and $V(I_1)$ disconnect $\Spec R$.
\end{proof}

\begin{prop}\label{prop:connectedimpliesequidimensional}
Let $R$ be a Cohen-Macaulay ring with a connected spectrum, that is integral over an equidimensional, catenary, integrally closed domain $S$. Then $R$ is equidimensional.
\end{prop}

\begin{proof}
Suppose the dimension of $S$ is $d$. Since it is a domain, the zero ideal is its unique minimal prime. Because it is equidimensional, every maximal ideal $\maxn$ has a saturated chain of length $d$ down to zero. Since $S$ is catenary, for any prime $\primeq$ of $S$, the length of any saturated chain from $\primeq$ down to $0$ is $\height(\primeq)$, i.e. does not depend on the choice of chain. Applying this with $\primeq = \maxn$ maximal, we conclude that every maximal chain in $\Spec S$ has length $d$. 

Note that $R$ is integral over $S$ because it is module-finite, thus incomparability, lying-over, and going-up apply.

Let $W_0$ be the set of prime ideals of $R$ lying over the zero ideal of $S$. Every $\primep\in W$ is minimal, by incomparability (\cite{atiyahmacdonald}, Corollary 5.9). We claim that (a) every prime in $W_0$ has dimension $d$, (b) every maximal ideal containing a minimal prime in $W_0$ has height $d$, and (c) actually every minimal prime of $R$ is in $W_0$. Claims (a)-(c) combine to establish the equidimensionality of $R$.

Note that for any $\primep\in W_0$, $R/\primep$ contains $S$ since $\primep\cap S = 0$. Furthermore, it is integral over $S$. (Each element of $R$ satisfies a monic polynomial over $S$, so it certainly satisfies the same polynomial mod $\primep$).

For (a), $R/\primep$'s dimension must be $d$ since integral extensions preserve dimension (\ref{lem:integralpreservesdim}), and this is by definition the dimension of $\primep$.

For (b), let $\maxm$ be a maximal ideal of $R$ containing $\primep\in W_0$. Its height is not greater than $d$, since this is the dimension of $R$, by \ref{lem:integralpreservesdim}. We can thus establish (b) by exhibiting a chain of primes in $R$ of length $d$ from $\primep$ to $\maxm$. Since the ideals of $R$ containing $\primep$ are in inclusion-preserving bijective correspondence with the ideals of $R/\primep$, for this purpose we can replace $R$ with $R/\primep$, and $\maxm$ with its image in this ring. Thus we may temporarily assume $R$ is an integral domain and $\primep = 0$.

Since $S$ is integrally closed, the going-down theorem (\cite{atiyahmacdonald}, Theorem 5.16) applies to the extension $S\subset R$. The ideal $\maxn = \maxm\cap S$ is maximal in $S$ by the going-up theorem (\cite{atiyahmacdonald}, Theorem 5.11), thus height $d$, so we may take a chain of length $d$ descending from $\maxn$, and apply the going-down theorem to obtain a chain of length $d$ descending from $\maxm$ in $R$.

For (c), we return to the original setting, dropping the temporary assumption that $R$ is a domain. Let $W_1$ be the set of minimal primes of $R$ whose intersections with $S$ are not zero. We will show that $W_1$ is empty.

For a contradiction, suppose that $\primep \in W_1$. Note that $\height(\primep \cap S) > 0$ since $\primep\cap S$ is nonzero and $S$ is a domain. Let $\maxm$ be a maximal ideal of $R$ containing $\primep$. If a saturated chain of primes from $\primep$ to $\maxm$ has length $\ell$, one obtains by intersecting with $S$ a chain of length $\ell$ ascending from $\primep\cap S$. It follows from $\height(\primep \cap S) > 0$ that $\ell < d$.

If $\maxm$ also contains a member of $W_0$, then by (b) it has height $d$. But $R$ is a Cohen-Macaulay ring, and \ref{prop:locallyequidimensional} asserts that a maximal ideal cannot possess saturated chains of two different lengths to different minimal primes. Thus there is no maximal ideal of $R$ containing elements of both $W_0$ and $W_1$. 

This implies that there is no path in the containment graph $\Gamma_R$ from $W_0$ to $W_1$. The set $W_0$ is nonempty, by lying-over (\cite{atiyahmacdonald}, Theorem 5.10). If $W_1$ is also nonempty, this means $\Gamma_R$ is disconnected. This contradicts the assumption that $\Spec R$ is connected, by \ref{lem:graphandspec}. Thus $W_1$ is empty. This establishes (c), so $R$ is equidimensional.
\end{proof}

The relevance of equidimensionality is that it allows us to apply a far-reaching generalization of the Hironaka criterion, given as theorem \ref{thm:generalHironaka} below, which works for integral extensions of regular domains. This is essentially a consequence of a local version, lemma \ref{lem:localHironaka}. We also need the fact that regular domains are Cohen-Macaulay and integrally closed. 

%Before giving the statement and proof, we collect the needed facts from the theory of integrally closed and regular rings. %projective modules and

%\begin{lemma}[\cite{kaplansky}]\label{lem:projoverlocalisfree}
%A  projective module $M$ over a local ring $(R,\maxm)$ is free.\qed
%\end{lemma}

%\begin{proof}
%Let $m$ be the minimal number of generators of $M$, and let $\varphi:R^m \rightarrow M$ be a surjection obtained by mapping a basis to a minimal set of generators. Because $M$ is projective, the surjection splits, and we have
%\[
%R^m \cong M \oplus N
%\]
%where $N = \ker \varphi$. Tensoring with $k = R/\maxm$, we get
%\[
%k^m \cong \frac{M}{\maxm M} \oplus \frac{N}{\maxm N}
%\]
%Nakayama's lemma implies that a $k$-vector space basis of $M/\maxm M$ lifts to a set of generators for $M$, so $m = \dim_k M/\maxm M$, and we conclude $\dim_k N/\maxm N = 0$, thus $\maxm N = N$. A second application of Nakayama's lemma now implies $N = 0$. Thus $M \cong R^m$.
%\end{proof}

%\begin{remark}
%Kaplansky showed (\cite{kaplansky}) that the finite generation hypothesis in \ref{lem:projoverlocalisfree} can be removed.
%\end{remark}

%\begin{lemma}[\cite{eisenbud}, Theorem 19.2]\label{lem:projislocallyfree}
%A finitely generated module $M$ over a noetherian ring $R$ is projective if and only if its localization $M_\maxm$ at every maximal ideal $\maxm$ of $R$ is free.\qed
%\end{lemma}

\begin{definition}
Recall that a noetherian local ring $(R,\maxm)$ is a \textbf{regular local ring} if the minimal number of generators of its maximal ideal $\maxm$ is equal to its Krull dimension. A noetherian ring is \textbf{regular} if all its localizations at maximal ideals are regular local rings. 
\end{definition}

\begin{lemma}[\cite{brunsherzog}, Proposition 2.2.5]\label{lem:regularisCM}
If $(R,\maxm)$ is a regular local ring, any minimal set of generators for $\maxm$ forms a regular sequence. Thus $R$ is Cohen-Macaulay.\qed
\end{lemma}

\begin{lemma}[\cite{atiyahmacdonald}, Proposition 5.13]\label{lem:normalityislocal}
An integral domain is integrally closed if and only if its localization at every maximal ideal is integrally closed.\qed
\end{lemma}

\begin{lemma}\label{lem:regularisnormal}
A regular domain is integrally closed.
\end{lemma}

\begin{proof}
Apply \ref{lem:normalityislocal}. Each localization at a maximal is a regular local ring. Regular local rings are unique factorization domains (\cite{eisenbud}, Theorem 19.19), which are always integrally closed.
\end{proof}

\begin{lemma}[Local Hironaka criterion, \cite{eisenbud}, Corollary 18.17]\label{lem:localHironaka}
If $R$ is a ring whose maximal ideals are all the same height, and it contains a regular local ring $(S,\maxn)$ and is module-finite over it, then $R$ is Cohen-Macaulay if and only if it is a free $S$-module.\qed
\end{lemma}

\begin{thm}[General Hironaka criterion]\label{thm:generalHironaka}
Let $R$ be a ring with connected spectrum that is finite as a module over an equidimensional regular domain $S$. Then $R$ is Cohen-Macaulay if and only if it is projective as an $S$-module.
\end{thm}

\begin{proof}
$\Rightarrow$ Let $\dim_{Krull} S = d$. Suppose $R$ is finitely generated and projective as an $S$-module. Let $\maxm$ be a maximal ideal of $R$. The intersection $\maxn = \maxm\cap S$ is a maximal ideal of $S$ by going-up. Then $R\otimes_S S_\maxn$ is finitely generated projective as an $S_\maxn$-module, and therefore finite free, since $S_\maxn$ is local and projective modules over local rings are free (\cite{kaplansky}). The dimension of $S_\maxn$ is $d$, and $S_\maxn$ is regular, so $\maxn S_\maxn$ contains a regular $S_\maxn$-sequence of length $d$ by \ref{lem:regularisCM}, which is also regular on $R\otimes_S S_\maxn$ since it is a free $S_\maxn$-module. This same sequence is furthermore regular on $R_\maxm$ since this is a localization of $R\otimes_S S_\maxn = R_{S\setminus\maxn}$. Therefore 
\[
\depth R_\maxm \geq d.
\]
But
\[
\dim R_\maxm \leq \dim R = \dim S = d,
\]
with the first equality because integral inclusions preserve dimension (\ref{lem:integralpreservesdim}). Meanwhile, 
\[
\depth R_\maxm \leq \dim R_\maxm,
\]
by \ref{prop:depthleqdim}. Thus 
\[
\depth R_\maxm\leq \dim R_\maxm \leq d \leq \depth R_\maxm,
\]
so we can conclude equality. Therefore $R_\maxm$ is Cohen-Macaulay. Since $\maxm$ was an arbitrary maximal of $R$, $R$ is Cohen-Macaulay.

$\Leftarrow$ Suppose that $R$ is a Cohen-Macaulay ring with connected spectrum that is finitely generated as an $S$-module. Since regular domains are integrally closed (\ref{lem:regularisnormal}), and Cohen-Macaulay (\ref{lem:regularisCM}) and therefore catenary (\ref{lem:catenary}), $R$ and $S$ satisfy the hypotheses of proposition \ref{prop:connectedimpliesequidimensional}, and we conclude that $R$ is equidimensional.

Let $\maxn$ be a maximal ideal of $S$. $R\otimes_S S_\maxn$ is a localization of $R$ and thus a Cohen-Macaulay ring. Furthermore it is finite over the regular local ring $S_\maxn$ ($=S\otimes_S S_\maxn$), since $R$ is finite over $S$. Lastly, it is equidimensional, since its prime spectrum consists of exactly those primes of $R$ that do not meet the complement of $\maxn$ in $S$, i.e. those primes of $R$ whose intersection with $S$ is contained in $\maxn$. These are the primes of $R$ which lie over $\maxn$ in $S$, and the primes they contain. The former are all maximal by incomparability. Since $R$ is equidimensional, they are all the same height.

Therefore \ref{lem:localHironaka} applies to the rings $R\otimes_S S_\maxn$ and $S_\maxn$. Since $R\otimes_S S_\maxn$ is Cohen-Macaulay, \ref{lem:localHironaka} implies it is a free module.

Since this holds for every maximal $\maxn$ of $S$, and finitely generated projective modules are exactly those that are locally free (\cite{eisenbud}, Theorem 19.2),  $R$ is projective over $S$. %\ref{lem:projislocallyfree} implies that
\end{proof}

\begin{remark}
The ``if" direction of theorem \ref{thm:generalHironaka} remains true if we drop the requirements that the spectrum of $R$ is connected and that $S$ is a domain, and weaken the requirement on $S$ from regularity to mere Cohen-Macaulayness, and the same proof works.
\end{remark}

With just one more piece of preparation, we are ready to state and prove the version of the Hironaka criterion of which we will make use in the sequel.

%\begin{thm}[Quillen-Suslin Theorem, \cite{lam}, Corollary V.4.12]\label{thm:quillensuslin}
%If $A$ is a principal ideal domain and $S = A[x_1,\dots,x_n]$, then any finitely generated projective $S$-module is free.
%\end{thm}

\begin{lemma}\label{lem:connectedisconnected}
If $R$ is an $\N$-graded $A$-algebra connected in the sense of definition \ref{def:connected}, then $R$ has a connected spectrum if and only if the same is true of $A$.
\end{lemma}

\begin{proof}
It is well-known that the connectedness of the $\Spec$ of a ring is equivalent to the existence of nontrivial (i.e. $\neq 0$ or $1$) idempotents (e.g. \cite{eisenbud}, exercise 2.25). Suppose $x\in R$ is idempotent and let $x = x_0+x_1+\dots+x_n$ be its decomposition into homogeneous components. Then $x^2 = x$ implies $x_k = \sum_{i+j=k} x_ix_j$ for all $0\leq k \leq n$. Because $A$ has connected $\Spec$ and is the entire degree zero component of $R$ by assumption, the first of these relations, $x_0 = x_0^2$, implies $x_0 = 0$ or $1$. In either case, the next relation, $x_1 = 2x_0x_1$, implies $x_1=0$. Then $x_j=0$ for $j>0$ follows by induction: the induction assumption reduces the relation $x_k = \sum_{i+j=k} x_ix_j$ to $x_k = 2x_0x_k$, which again implies $x_k=0$ for either choice of $x_0$. Thus $x=0$ or $1$.

In the converse direction, clearly if $A$ contains a nontrivial idempotent then so does $R$.
\end{proof}

%\begin{lemma}\label{lem:dimensionfalls}
%Let $\varphi:A\rightarrow B$ be a homomorphism of noetherian rings, and suppose that $\ker \varphi$ contains a nonzerodivisor of $A$. Then $\dim_{Krull} B < \dim_{Krull} A$.
%\end{lemma}

%\begin{proof}
%The set of zerodivisors of a noetherian ring is the union of the primes associated to zero, which includes all the minimal primes. Thus $\ker \varphi$ is not contained in any of the minimal primes of $A$. Thus a chain of primes in $A$ containing $\varphi$ cannot include any minimal prime of $A$ and so cannot be of maximal length in $A$. The correspondence theorem identifying $\Spec B$ with $V(\ker \varphi)\subset\Spec A$ finishes the proof.
%\end{proof}

And finally:

\begin{thm}[Hironaka criterion over $\Z$]\label{thm:HironakaoverZ}
Let $R=\bigoplus_{\ell\geq 0}R_\ell$ be a finitely generated graded connected $\mathbb{Z}$-algebra. Let $n = \dim_{Krull} R - 1$. Suppose $\theta_1,\dots,\theta_n$ are homogeneous, positive-degree elements of $R$ such that $R$ is finite as a module over $\mathbb{Z}[\theta_1,\dots,\theta_n]$. Then $R$ is Cohen-Macaulay if and only if it is free as a $\mathbb{Z}[\theta_1,\dots,\theta_n]$-module.
\end{thm}

\begin{proof}
Let $S = \Z[\theta_1,\dots,\theta_n]$. Since $S\subset R$ is finite, $\dim_{Krull} S = \dim_{Krull} R = n+1$. %, by \ref{lem:finitepreservesdim}.
 It follows immediately that $\theta_1,\dots,\theta_n$ are algebraically independent over $n$: the homomorphism from the polynomial ring $\Z[X_1,\dots,X_n]$ to $S$ mapping $X_i\mapsto \theta_i$ cannot have a nontrivial kernel without causing $\dim S < \dim \Z[X_1,\dots,X_n] = n+1$, %by \ref{lem:dimensionfalls},
 since $\Z[X_1,\dots,X_n]$ is a domain. Thus $S$ is a polynomial ring over the principal ideal domain $\Z$. In particular, it is an equidimensional regular domain.

Meanwhile, $R$ has connected spectrum, by \ref{lem:connectedisconnected}, since $\Z$ has connected spectrum. Therefore, theorem \ref{thm:generalHironaka} applies, and $R$ is Cohen-Macaulay if and only if it is a projective $S$-module. But since $\Z$ is a principal ideal domain, a finitely generated $S$-module is projective if and only if it is free, by the Quillen-Suslin theorem (\cite{lam}, Corollary V.4.12). %\ref{thm:quillensuslin}.
\end{proof}

\begin{remark}
One can avoid the heavy machinery of the Quillen-Suslin theorem in this proof, since for graded modules over a graded, connected algebra over a p.i.d., one can prove that projectiveness coincides with freeness using only the graded Nakayama lemma (\ref{lem:gradedNakayama} in the algebraic lemmas appendix) and elementary arguments.
\end{remark}

The analogy between \ref{thm:HironakaoverZ} and \ref{cor:CMfreemodule} motivates us to name the $\theta_1,\dots,\theta_n$ of \ref{thm:HironakaoverZ} by analogy with the field case:

\begin{definition}\label{def:hsopoverZ}
If $R$ is a finitely generated graded connected $\Z$-algebra and $\theta_1,\dots,\theta_n$ are positive-degree homogeneous elements with $n=\dim R - 1$ and $R$ module-finite over $\Z[\theta_1,\dots,\theta_n]$, as in the hypothesis of proposition \ref{thm:HironakaoverZ}, then we will say $\theta_1,\dots,\theta_n$ is a \textbf{homogeneous system of parameters (h.s.o.p.) for $R$}.
\end{definition}

We are now in a position to prove the equivalence of questions \ref{q:bieselsquestion} and \ref{q:CMpermutationgroups}.

\begin{definition}\label{def:orbitmonomial}
Let $A$ be a ring and $R$ an $A$-algebra with an $A$-basis of monomials (in some specified set of generators). Let $G$ be a finite group that acts by automorphisms of $A$ that fix this basis setwise, so that $G$ has an action on the monomial basis. Then an \textbf{orbit monomial} is the sum of the monomials in a single orbit of this action. If $m$ is a monomial, we denote the corresponding orbit monomial by $Gm$.
\end{definition}

In this circumstance, $R^G$ has an $A$-basis of orbit monomials.

\begin{thm}\label{thm:equivofquestions}
If $G\subset S_n$ is a permutation group, acting on the polynomial ring $R=\Z[x_1,\dots,x_n]$ by permuting the indeterminates, then the following are equivalent:
\begin{enumerate}
\item $R^G$ is a Cohen-Macaulay ring.
\item $(R\otimes\F_p)^G$ is a Cohen-Macaulay ring for all primes $p$.
\item $R^G$ is free as a module over the subring $R^{S_n} = \Z[\sigma_1,\dots,\sigma_n]$.
\end{enumerate}
\end{thm}

\begin{proof}
Note that for any field $k$, $R^G\otimes k$ and $(R\otimes k)^G$ are both the $k$-vector spaces spanned by all orbit monomials of $G$, thus $(R\otimes k)^G = R^G \otimes k$. The equivalence of 1 and 2 is then given by the equivalence of 1 and 2 in \ref{prop:equivofZandallfields} once we take into account that $R^G\otimes\Q$ is automatically Cohen-Macaulay by theorem \ref{thm:hochstereagon}.

We get 1$\Leftrightarrow$3 as follows:

$R$ is a finite module over $R^G$, by proposition \ref{prop:Risintegral}, so they have the same dimension. Thus $\dim R^G = n+1$. The elements $\sigma_1,\dots,\sigma_n$ are homogeneous, positive degree, and generate $R^{S_n}$ over $\Z$ by the FTSP (\ref{thm:ftsp}); and $R^G$ is finite as a module over $R^{S_n}$ by proposition \ref{prop:integraloverSn}. Thus $\sigma_1,\dots,\sigma_n$ is an h.s.o.p. for $R^G$, i.e. $R^G$ and $\sigma_1,\dots,\sigma_n$ meet the conditions of \ref{thm:HironakaoverZ}, so the Cohen-Macaulayness of $R^G$ is equivalent to its being a free module over $\Z[\sigma_1,\dots,\sigma_n]$.
\end{proof}

\section{The Stanley-Reisner ring and the invariant ring}\label{sec:stanleyreisner}

This section reports on fundamental work in combinatorial commutative algebra done in the 1970's and 80's by Melvin Hochster, Gerald Reisner, Richard Stanley, Adriano Garsia, Dennis Stanton, and others, which connects the algebraic properties of the polynomial invariant ring of a permutation group to those of a related ring, the {\em Stanley-Reisner ring} of a certain cell complex, and shows that these properties are also reflected in the topology of this complex. This allows us to connect our question to a purely topological one about quotients of spheres and balls, and thus to apply recent work in orbifold theory by Christian Lange, which we do in the next section.

\subsection{Stanley-Reisner rings}\label{sec:SRrings}

A {\em Stanley-Reisner ring}, or {\em face ring}, is a ring that encodes combinatorial information about a poset or regular cell complex, allowing algebraic techniques to be brought to bear on combinatorial and topological questions, and vice versa. The original definition was given for a simplicial complex. We review some basic terminology:

\begin{definition}\label{def:simplicial}
Let $V$ be a finite set of cardinality $n$. A (finite abstract) \textbf{simplicial complex} $\Delta$ is a family of subsets of $V$ that is downward-closed, i.e. $\alpha\in \Delta$ and $\beta \subset \alpha$ implies $\beta\in \Delta$. Each $\alpha \in \Delta$ is called a \textbf{face} of $\Delta$, and the \textbf{dimension} of a face $\alpha$ is $\#\alpha-1$, where $\#\alpha$ is the cardinality of $\alpha$. A face maximal with respect to inclusion is a \textbf{facet}. A singleton $\{v\}\in\Delta,\;v\in V$ is called a \textbf{vertex}. The \textbf{dimension of $\Delta$} is the maximum dimension of any of its faces. If all facets are the same dimension, the simplicial complex is \textbf{pure}.
\end{definition}

The geometric langauge in definition \ref{def:simplicial} is justified by the following definition.

\begin{definition}\label{def:totalspace}
If $\Delta$ is a finite abstract simplicial complex on the vertex set $V$, let $W$ be a real vector space of dimension $\#V-1$, and identify $V$ with any set of $\#V$ distinct points of $W$ in general position (i.e. not contained in a hyperplane). Then the \textbf{total space} or \textbf{geometric realization} $|\Delta|$ of $\Delta$ is the union of the convex hulls of the points in each face $\alpha$ of $\Delta$. The convex hull of any individual $\alpha$ is a \textbf{simplex}. See figure \ref{fig:examplecomplex}.
\end{definition}

\begin{remark}
Note that $|\cdot|$ is {\em not} being used to denote cardinality in this context.
\end{remark}

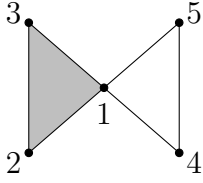
\begin{figure}
\begin{center}
\begin{tikzpicture}

\draw [fill = lightgray] (0,0) -- (-1, -0.866) -- (-1,0.866) -- (0,0);
\draw (0,0) -- (1, -0.866) -- (1,0.866) -- (0,0);
\draw [fill] (0,0) circle [radius = 0.05];
\draw [fill] (-1,-0.866) circle [radius = 0.05];
\draw [fill] (-1,0.866) circle [radius = 0.05];
\draw [fill] (1,0.866) circle [radius = 0.05];
\draw [fill] (1,-0.866) circle [radius = 0.05];

\node at (0,-0.35) {1};
\node at (-1.2,-1) {2};
\node at (-1.2,1) {3};
\node at (1.2,-1) {4};
\node at (1.2,1) {5};

\end{tikzpicture}
\end{center}
\caption{The geometric realization of the abstract simplicial complex $K$ on the vertex set $V=[5]$, with facets $123$, $34$, $35$, and $45$, projected into a plane. It is not pure, as one facet is dimension 2 while others are dimension 1.}\label{fig:examplecomplex}
\end{figure}

\begin{remark}\label{rmk:totalspaceasCWcomplex}
One chooses the space $W$ and the points in it with which to identify $V$ in definition \ref{def:totalspace}. However, it is clear that there is a homeomorphism between the total spaces arrived at through any two sets of choices, which is even linear when restricted to (the geometric realization of) any individual face. Thus $|\Delta|$ is well-defined as a topological space, and also as a piecewise-linear (PL) space, modulo the technicalities involved in defining the latter. Furthermore, because the convex hull of $k$ points in general position in a real vector space is homeomorphic to a closed $(k-1)$-ball, and because if $\alpha\subset\beta \subset V$ then the convex hull of $\alpha$ is a subset of the boundary of that of $\beta$, the geometric realization $|\Delta|$ naturally carries the structure of a regular CW complex, with the faces as the cells.

For basic definitions regarding PL spaces, see \cite{lange}, \S 2.2.1. For definitions and technical details on CW complexes, see \cite{hatcher}, Chapter 0 and the Appendix.
\end{remark}

\begin{definition}[Stanley-Reisner ring of a simplicial complex]\label{def:SRringofsimplicial}
Let $\Delta$ be a finite abstract simplicial complex on the vertex set $V$, and let $A$ be a ring, typically a field or $\Z$. Let $R_\Delta$ be the polynomial ring $A[\{x_v\}_{v\in V}]$ with indeterminates indexed by the vertex set of $\Delta$. A monomial $\prod_{v\in V} x_v^{e_v}$ of $R_\Delta$ is \textbf{supported on a face} $\alpha$ of $\Delta$ if $e_v=0$ unless $v\in \alpha$. The \textbf{Stanley-Reisner ideal} $I_\Delta$ is the ideal of $R_\Delta$ generated by all squarefree monomials that are {\em not} supported on faces of $\Delta$. (It is of course sufficient to use the minimal such squarefree monomials.) The \textbf{Stanley-Reisner ring} (or \textbf{face ring}) $A[\Delta]$ of $\Delta$ is the quotient ring $R_\Delta / I_\Delta$. It has a natural $A$-basis consisting of (residue classes of) monomials that {\em are} supported on faces.
\end{definition}

\begin{example}
The Stanley-Reisner ideal of the complex $K$ depicted in figure \ref{fig:examplecomplex} is generated by $x_2x_4,x_2x_5, x_3x_4, x_3x_5$, and $x_1x_4x_5$. Thus the Stanley-Reisner ring of $K$ over the integers is
\[
\Z[K] = \Z[x_1,\dots,x_5]/(x_2x_4,x_2x_5, x_3x_4, x_3x_5, x_1x_4x_5).
\] 
\end{example}

\begin{remark}
The Stanley-Reisner ring was defined independently in the 1970's by Richard Stanley and Melvin Hochster. Hochster's student Gerald Reisner made a critical contribution, which we will see in the next section, by showing that the algebraic properties of the ring $A[\Delta]$ are closely tied to the topology of the geometric realization $|\Delta|$. The first major application was Stanley's proof of the Upper Bound Conjecture for simplicial spheres (\cite{stanley14}). See \cite{franciscoetal} for a short, self-contained introduction to the theory, and \cite{cca} for a fuller account. In \cite{brunsherzog}, Chapter 5, one can find a complete proof of the Upper Bound Conjecture.
\end{remark}

One also defines the Stanley-Reisner ring of a partially ordered set ({\em poset}). Indeed, a fundamental tool in the topological study of posets (see \cite{wachs}, whose notation we are following) is to associate to a poset a certain simplicial complex called its {\em order complex}, and then one can take the Stanley-Reisner ring of this complex.

\begin{definition}\label{def:ordercomplex}
Given a finite poset $(P,\leq)$, the \textbf{order complex} $\Delta(P)$ of $P$ is the abstract simplicial complex with vertex set the elements of $P$, and with faces the chains of $P$ (i.e. the totally ordered subsets).
\end{definition}

\begin{definition}[Stanley-Reisner ring of a poset]\label{def:SRringofposet}
The \textbf{Stanley-Reisner ring} $A[P]$ of a poset $P$ (over a coefficient ring $A$) is the Stanley-Reisner ring of its order complex, i.e. $A[P]$ is defined as $A[\Delta(P)]$.
\end{definition}

\begin{remark}\label{rmk:compactSRposet}
Tracing through the three definitions just given, one gets a compact description of $A[P]$: the indeterminates are indexed by the elements of $P$, and a product of indeterminates is zero whenever the corresponding elements of $P$ are incomparable. 
\end{remark}

\begin{example}
In figure \ref{fig:faceposet} is shown the face poset (to be defined below) of a triangle. This is order-isomorphic to the poset $P=B_3\setminus\{\emptyset\}$ of nonempty subsets of $[3]$, ordered by inclusion. The Stanley-Reisner ring $\Z[P]$ of this poset has indeterminates
\[
y_1,y_2,y_3,y_{12},y_{13},y_{23},y_{123},
\]
corresponding to the elements of the poset. The Stanley-Reisner ideal is generated by the products
\[
y_1y_2,y_1y_3,y_2y_3,y_1y_{23}, y_2y_{13}, y_3y_{12}, y_{12}y_{13},y_{12}y_{23},y_{13}y_{23}
\]
coming from each pair of incomparable sets in $B_3\setminus\{\emptyset\}$, and the Stanley-Reisner ring is the quotient by this ideal.
\end{example}

\begin{definition}[Cohen-Macaulayness of complexes and posets]\label{def:CMnessofcomplexesandposets}
A finite poset $P$ is said to be \textbf{Cohen-Macaulay}, over $\Z$ or a field $k$, if $\Z[P]$, respectively $k[P]$, is a Cohen-Macaulay ring. Likewise, a finite simplicial complex $\Delta$ is said to be \textbf{Cohen-Macaulay} over $\Z$ or $k$ if $\Z[\Delta]$, respectively $k[\Delta]$, is Cohen-Macaulay.
\end{definition}

Definition \ref{def:ordercomplex} gives a way of creating a simplicial complex out of a poset. There is an even more natural operation in the reverse direction:

\begin{definition}
To a simplicial complex $\Delta$, or more generally a CW complex, is associated a poset $P(\Delta)$ called its \textbf{face poset}: the elements are the cells of the complex, and the order relation is given by inclusion (or inclusion of closures in the case of the CW complex). See figure \ref{fig:faceposet}. A \textbf{flag} of $\Delta$ is a collection of faces, respectively cells, that form a chain in $P(\Delta)$. A \textbf{full flag} is a maximal chain.
\end{definition}

\begin{remark}\label{rmk:emptyface}
The definition of the face poset of a complex raises a point of tension between the abstract definition of a simplicial complex (\ref{def:simplicial}) and the view of its geometric realization as a CW complex (\ref{rmk:totalspaceasCWcomplex}). From the abstract definition \ref{def:simplicial}, it is clear that if $\Delta$ is nonempty then $\emptyset \in \Delta$ (the \textbf{empty face}). But if one then regards the total space as a regular CW complex via definition \ref{def:totalspace} and remark \ref{rmk:totalspaceasCWcomplex}, the CW complex structure no longer contains evidence of this empty face. Thus the face poset $P(\Delta)$ built from the abstract simplicial complex will have a minimal element corresponding to the empty face that is missing from the face poset built from the regular CW complex structure. See figure \ref{fig:faceposet}. 

In some respects the theory works best if one leans toward the simplicial point of view by always adding a minimal ``empty face" to regular CW complexes when taking their face posets. However, the relationship between face posets and order complexes, explicated momentarily in remark \ref{rmk:barycentricsubdivision}, is cleaner in the CW complex point of view. Therefore, when we need to, we will explicitly add the empty face to the face posets of regular CW complexes, as in definition \ref{def:booleancomplex}, or remove it from those of simplicial complexes, as in remark \ref{rmk:barycentricsubdivision}.
\end{remark}

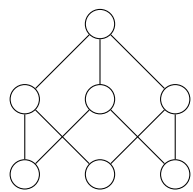
\begin{figure}
\begin{center}
\begin{tikzpicture}

\tikzstyle{every node} = [draw, shape=circle]

\node (v1) at (-1,1) {};
\node (v2) at (0,1) {};
\node (v3) at (1,1) {};
\node (e12) at (-1,2) {};
\node (e13) at (0,2) {};
\node (e23) at (1,2) {};
\node (tri) at (0,3) {};

\foreach \from/\to in {v1/e12, v1/e13, v2/e12, v2/e23, v3/e13, v3/e23, e13/tri, e23/tri, e12/tri} 
    \draw (\from) -- (\to);
    
\end{tikzpicture}
\end{center}
\caption{The face poset of a triangle, regarded as regular CW complex.}\label{fig:faceposet}
\end{figure}

\begin{remark}\label{rmk:barycentricsubdivision}
The operations of taking the order complex and the face poset are not inverse to each other. Indeed, the face poset of the order complex of a poset has many more elements than the poset, and likewise the order complex of the face poset of a regular CW complex has many more cells. However, if $\Delta$ is a regular CW complex, we do have a homeomorphism of the total space of $\Delta$ with that of the order complex of its face poset $\Delta(P(\Delta))$. This is because the operation $\Delta(P(\cdot))$ is nothing but barycentric subdivision. The points of the face poset $P(\Delta)$ correspond with the cells of $\Delta$. They then become the vertices of the order complex $\Delta(P(\Delta))$, and we can interpret them as the barycenters of the cells. This interpretation makes the simplices of $\Delta(P(\Delta))$, being chains in $P(\Delta)$, into the simplices spanned by the barycenters of cells in $\Delta$ that form flags. See figure \ref{fig:barycentric}.

If $\Delta$ is a simplicial complex, we need to remove the empty face from the face poset before taking the order complex in order to get this construction to work right, since the empty face ought not to have a barycenter. (If we retain the empty face in the face poset, the total space of $\Delta(P(\Delta))$ is the cone over the total space of $\Delta$.)
\end{remark}

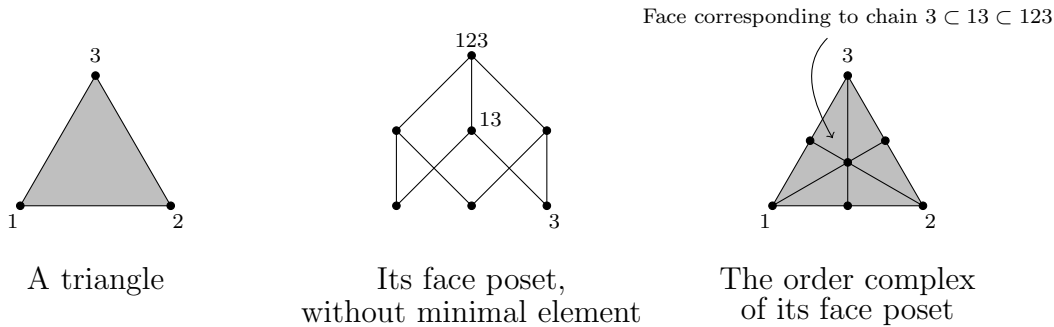
\begin{figure}
\begin{center}
\begin{tikzpicture}

\draw [fill = lightgray] (-1,0) -- (-3,0) -- (-2,1.732) -- (-1,0);
\draw [fill] (-1,0) circle [radius=0.05];
\draw [fill] (-3,0) circle [radius=0.05];
\draw [fill] (-2,1.732) circle [radius=0.05];
\node at (-2,-1) {A triangle};
\node at (-3.1,-0.2) {\scriptsize 1};
\node at (-0.9, -0.2) {\scriptsize 2};
\node at (-2, 2) {\scriptsize 3};

\draw [fill] (2,0) circle [radius = 0.05];
\draw [fill] (3,0) circle [radius = 0.05];
\draw [fill] (4,0) circle [radius = 0.05];
\draw [fill] (2,1) circle [radius = 0.05];
\draw [fill] (3,1) circle [radius = 0.05];
\draw [fill] (4,1) circle [radius = 0.05];
\draw [fill] (3,2) circle [radius = 0.05];
\draw (2,0) -- (2,1) -- (3,2) -- (4,1) -- (4,0) -- (3,1) -- (2,0);
\draw (2,1) -- (3,0) -- (4,1);
\draw (3,1) -- (3,2);
\node at (4.1,-0.2) {\scriptsize 3};
\node at (3.25, 1.15) {\scriptsize 13};
\node at (3, 2.2) {\scriptsize 123};
\node at (3,-1) {Its face poset,}; 
\node at (3,-1.4) {without minimal element};

\draw [fill = lightgray] (9,0) -- (7,0) -- (8,1.732) -- (9,0);
\draw (7,0) -- (8.5,0.866);
\draw (9,0) -- (7.5,0.866);
\draw (8,0) -- (8,1.732);
\draw [fill] (9,0) circle [radius=0.05];
\draw [fill] (8,0) circle [radius=0.05];
\draw [fill] (7,0) circle [radius=0.05];
\draw [fill] (7.5,0.866) circle [radius=0.05];
\draw [fill] (8.5,0.866) circle [radius=0.05];
\draw [fill] (8,1.732) circle [radius=0.05];
\draw [fill] (8,0.577) circle [radius=0.05];
\node at (8,-1) {The order complex};
\node at (8,-1.4) {of its face poset};
\node (labelr) at (8, 2.5) {\scriptsize Face corresponding to chain $3\subset 13 \subset 123$};
\draw [->] (labelr) to [out=225,in=120] (7.8, 0.9);
\node at (6.9,-0.2) {\scriptsize 1};
\node at (9.1, -0.2) {\scriptsize 2};
\node at (8, 2) {\scriptsize 3};

\end{tikzpicture}
\end{center}
\caption{The barycentric subdivision.}\label{fig:barycentric}
\end{figure}

Suppose we want to construct a Stanley-Reisner ring for a general regular CW complex analogous to the Stanley-Reisner ring of a simplicial complex. Remark \ref{rmk:barycentricsubdivision} shows that it would not be entirely unreasonable to take the Stanley-Reisner ring of its face poset: this at least preserves the underlying topology. However, this method does not strictly generalize definition \ref{def:SRringofsimplicial}, i.e. it does not produce an isomorphic ring in the simplicial case. However, there is a class of ``almost simplicial" complexes for which we have a definition that does strictly generalize \ref{def:SRringofsimplicial}.

\begin{definition}
A (finite) \textbf{boolean algebra} is a poset that is order-isomorphic to the set of subsets of some finite set, under inclusion.
\end{definition}

Note that the face poset of a simplex (as simplicial complex) is a boolean algebra. Regarded as a regular CW complex it is missing the empty face, as in remark \ref{rmk:emptyface}. See figure \ref{fig:faceposet}.

\begin{definition}\label{def:booleancomplex}
A \textbf{boolean complex} is a finite regular CW complex whose face poset has the property that if one appends a minimal element, all lower intervals are boolean algebras.\footnote{The name {\em boolean complex} was coined by Garsia and Stanton in \cite{garsiastanton}. They did not clearly distinguish the complex from its face poset (with minimal element), and worked primarily with the latter. The same concept is referred to by Stanley in terms of its face poset, which he calls a {\em simplicial poset} -- see \cite{cca} and \cite{stanley}. The literature has not decisively sided with one or the other of these terms. One could also get at the same concept by calling it a regular $\Delta$-complex, using Allan Hatcher's (\cite{hatcher}) coinage {\em $\Delta$-complex} for a CW complex in which the cells are simplices and the attaching maps are simplicial, but Hatcher's term does not seem to have caught on. We go with {\em boolean complex} over {\em simplicial poset} to allow us to refer to the complex itself (not only the poset) without confusion with a simplicial complex.} As with simplicial complexes, we refer to the cells as \textbf{faces}, the maximal cells as \textbf{facets}, and the $0$-cells as \textbf{vertices}, and the complex is \textbf{pure} if all facets have the same dimension.
\end{definition}

\begin{remark}
In a simplicial complex, the intersection of two faces $\alpha$ and $\beta$ is a single face of each. (In the abstract point of view it is $\alpha\cap\beta$. In the CW complex point of view it is the intersection of their closures.) Also, if there is any face containing both $\alpha$ and $\beta$, there is only one minimal such face. (It is $\alpha \cup \beta$ in the abstract point of view.) One may think of a boolean complex as ``like a simplicial complex but without these constraints." Figure \ref{fig:booleancomplex} gives an example.
\end{remark}

\begin{figure}
\begin{center}
\begin{tikzpicture}

\node at (-1.3,0) {A};
\draw [fill=black] (-1,0) circle [radius = 0.1];
\node at (1.3,0) {B};
\draw [fill=black] (1,0) circle [radius = 0.1];

\draw (-1,0) arc [radius = 1, start angle = 180, end angle = 360];
\node at (0,-1.2) {C};
\draw (1,0) arc [radius = 1, start angle = 0, end angle = 180];
\node at (0,1.2) {D};

\node at (5,-1.5) {$\emptyset$};
\node at (3.9,0) {A};
\node at (6.1,0) {B};
\node at (3.9,1) {C};
\node at (6.1,1) {D};

\tikzstyle{every node} = [draw, shape=circle]

\node (empt) at (5,-1) {};
\node (A) at (4.3,0) {};
\node (B) at (5.7,0) {};
\node (C) at (4.3,1) {};
\node (D) at (5.7,1) {};

\foreach \from/\to in {empt/A, empt/B, A/C, A/D, B/C, B/D}
    \draw (\from) -- (\to);

\end{tikzpicture}
\end{center}
\caption{Left: a boolean complex with total space homeomorphic to a circle. Right: its face poset, with minimal element appended.}\label{fig:booleancomplex}
\end{figure}
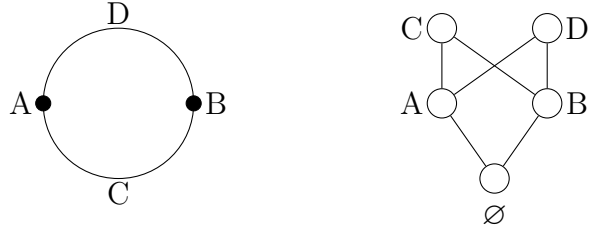

\begin{definition}
Let $P$ be a poset. $P$ is \textbf{ranked} if for every $\alpha \in P$, every chain maximal among those descending from $\alpha$ (i.e. of the form $\alpha_0\leq \alpha_1\leq\dots\leq \alpha_r = \alpha$) has the same length. The common length $r$ is said to be the \textbf{rank of $\alpha$}, written $\rk_P\alpha$, or $\rk \alpha$ if $P$ is clear from context.
\end{definition}

\begin{remark}
Note that the face posets of finite simplicial and boolean complexes are automatically ranked, because boolean algebras are ranked. In the version of the face poset with the minimal element, the rank of a face is the number of vertices it spans.
\end{remark}

\begin{definition}
If two elements $\alpha,\beta$ in a poset $P$ have a unique maximal common lower bound $\gamma$, it is called their \textbf{meet}, and written $\alpha\wedge \beta$. Dually, if they have a unique minimal common upper bound, it is their \textbf{join}, and is written $\alpha\vee\beta$.
\end{definition}

For boolean complexes, the Stanley-Reisner ring can be defined in a way that generalizes definition \ref{def:SRringofsimplicial}. The definition was first given by Stanley in \cite{stanley}.

\begin{definition}[Stanley-Reisner ring of a boolean complex]\label{def:SRringofboolean}
Let $\Delta$ be a boolean complex, and let $\widehat{P}$ be its face poset, with minimal element $\emptyset$ appended, and let $A$ be a coefficient ring. Let 
\[
R_\Delta = A[\{y_\alpha\}_{\alpha\in \widehat{P}}]
\]
be a polynomial ring over $A$ with indeterminates indexed by the elements of $\widehat{P}$, and let $I_\Delta$ be the ideal generated by:
\begin{enumerate}
\item $y_\emptyset - 1$, 
\item $y_\alpha y_\beta$ if $\alpha,\beta\in \widehat{P}$ do not have a common upper bound, and
\item $y_\alpha y_\beta - y_{\alpha\wedge \beta}\sum_{\gamma} y_\gamma$ if $\alpha,\beta$ do have a common upper bound, where the sum is taken over the set of minimal common upper bounds for $\alpha,\beta$.
\end{enumerate}
Then the \textbf{Stanley-Reisner ring $A[\Delta]$ of $\Delta$} is $R_\Delta/I_\Delta$.
\end{definition}

\begin{remark}
The expression $\alpha\wedge \beta$ in the third line of the definition of $I$ is well-defined because if $\alpha,\beta$ have any common upper bound $\gamma$, then the interval $[\emptyset,\gamma]$ contains everything below $\alpha,\beta$, and it is furthermore a boolean algebra because $\Delta$ is a boolean complex. Boolean algebras are lattices, thus $\alpha,\beta$ have a unique greatest lower bound in $[\emptyset,\gamma]$ and thus in $\widehat{P}$. This is $\alpha\wedge \beta$. 

The reader may have wondered why we added the $\emptyset$ element to $\widehat{P}$ in this definition only to absorb its corresponding indeterminate $y_\emptyset$ into $A$ with the relation $y_\emptyset - 1$; the reason is the need to make sure $\alpha\wedge\beta$ is always defined.
\end{remark}

\begin{remark}
If $\Delta$ is a true simplicial complex on a vertex set $V$, then the isomorphism between definitions \ref{def:SRringofboolean} and \ref{def:SRringofsimplicial} is as follows. Fix a coefficient ring $A$. Let $S$ be the Stanley-Reisner ring defined in \ref{def:SRringofsimplicial}, with indeterminates labeled $x_v, \, v\in V$. Let $S'$ be the Stanley-Reisner ring defined in \ref{def:SRringofboolean}, with indeterminates labeled $y_\alpha, \, \alpha\in \widehat{P}$ where $\widehat{P}$ is the face poset of $\Delta$ including the minimal element $\emptyset$, as in that definition. Identifying the elements of $\widehat{P}$ with these labels, the isomorphism $S\rightarrow S'$ is given by $x_v \mapsto y_{\{v\}},\, v\in V$. The inverse is given by $y_{\alpha}\mapsto \prod_{v\in\alpha} x_v$.
\end{remark}

%add a figure perhaps? use a running example of a simplicial complex to illustrate the isomorphism?

\subsection{The Reisner-Munkres theorem}\label{sec:reisnermunkres}

The utility of Stanley-Reisner rings to our inquiry is the close relationship between the algebra of $A[\Delta]$ (respectively $A[P]$), and the topology of $|\Delta|$ (respectively $|\Delta(P)|$). The following beautiful theorem is due to Gerald Reisner and James Munkres. It requires a definition:

\begin{definition}\label{def:link}
Let $\Delta$ be a simplicial complex and $\alpha\in\Delta$ a face. The \textbf{link} $\lk_\Delta(\alpha)$ of $\alpha$ in $\Delta$ is the subcomplex consisting of all those $\beta\in \Delta$ disjoint from $\alpha$ and such that $\alpha\cup \beta\in \Delta$.
\end{definition}

Note that $\Delta = \lk_\Delta(\emptyset)$.

\begin{example}
Figure \ref{fig:link} shows a simplicial complex $\Delta$ consisting of the six tetrahedra $1234$, $1245$, $\dots$, $1283$  joined along a common edge $\alpha = 12$, shown in red. Its link $\lk_\Delta (\alpha)$ is in blue.
\end{example}

\begin{figure}
\begin{center}
\begin{tikzpicture}

\path [fill = lightgray] (2,0) -- (1.866,0.866) -- (0,2) -- (-2,0) -- (-1.866,-0.866) -- (0,-2) -- (2,0);

\draw [blue, ultra thick] (2,0) -- (1.866,0.866) -- (-0.134,0.866) -- (-2,0);
\draw (0,-2) -- (-0.134,0.866) -- (0,2);
\draw (0,-2) -- (-2,0) -- (0,2);
\draw (0,-2) -- (1.866,0.866) -- (0,2);
\draw [red, ultra thick] (0,-2) -- (0,2);
\draw (0,-2) -- (2,0) -- (0,2);
\draw (0,-2) -- (-1.866,-0.866) -- (0,2);
\draw (0,-2) -- (0.134,-0.866) -- (0,2);
\draw [blue, ultra thick] (-2,0) -- (-1.866,-0.866) -- (0.134,-0.866) -- (2,0);

\node at (0,-2.2) {1};
\node at (0,2.2) {2};
\node at (2.2,0) {3};
\node at (2,1) {4};
\node at (-0.25,1.1) {5};
\node at (-2.2,0) {6};
\node at (-2.05,-1) {7};
\node at (0.25,-1.1) {8};

\end{tikzpicture}
\end{center}
\caption{The link of a simplex.}\label{fig:link}
\end{figure}
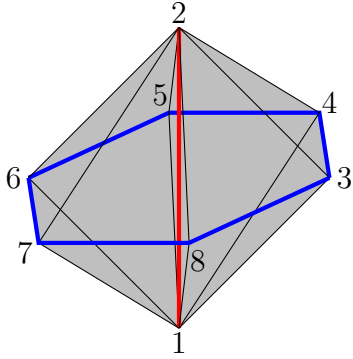

\begin{remark}
One can think of the link of a point topologically as the boundary of a ball around the point. In particular, if $\Delta$ is a PL triangulation of a PL manifold, then the link of a point is PL $(n-1)$-sphere. More generally, in a triangulated PL manifold, all links of nonempty faces are spheres: for any $\alpha\in \Delta$, $\lk_\Delta(\alpha)$ is PL-homeomorphic to a PL sphere of dimension $\dim \Delta - \dim \alpha - 1$.
\end{remark}

\begin{thm}[Reisner-Munkres theorem]\label{thm:reisnermunkres}
Let $k$ be a field and let $\Delta$ be a simplicial complex, of dimension $n$. Then the following are equivalent:
\begin{enumerate}
\item $k[\Delta]$ is a Cohen-Macaulay ring.
\item $\tilde H_i(|\lk_\Delta(\alpha)|; k) = 0$ for all $\alpha\in \Delta$ and all $i<\dim\lk_\Delta(\alpha)$.
\item $\tilde H_i(|\Delta|; k) = H_i(|\Delta|, |\Delta|\setminus p; k) = 0$ for all $p\in|\Delta|, i < n$.
\end{enumerate}
Here $\tilde H_* (-;k)$ is reduced homology with coefficients in $k$, and $H_*(-,-;k)$ is relative homology.\footnote{For definitions regarding homology, see \cite{hatcher}, Chapter 2.}
\end{thm}

For condition $2$ one may take either singular or simplicial homology. Condition $3$ is referring to singular homology.

\begin{proof}
(1)$\Leftrightarrow$(2) is due to Reisner (\cite{reisner}). (2)$\Leftrightarrow$(3) is due to Munkres (\cite{munkres}).
\end{proof}

\begin{remark}\label{rmk:CMistopological}
This theorem implies that Cohen-Macaulayness of $k[\Delta]$ is a purely topological condition on $|\Delta|$; in other words, two simplicial complexes with the same total space (up to homeomorphism) will be simultaneously Cohen-Macaulay over a given field $k$. In particular, in view of remark \ref{rmk:barycentricsubdivision}, a simplicial complex and its face poset minus the minimal element are simultaneously Cohen-Macaulay. (In fact, so is the face poset including the minimal element, since including a minimal element translates topologically to taking the cone, which does not affect the conditions in theorem \ref{thm:reisnermunkres}, as the reader can check.)
\end{remark}

\begin{remark}\label{rmk:kvsZ}
One can replace $k$ with $\Z$ in the statement of theorem \ref{thm:reisnermunkres}. Indeed, per \ref{prop:equivofZandallfields}, $\Z[\Delta]$'s Cohen-Macaulayness is equivalent to that of $k[\Delta]$'s for all fields $k$. Meanwhile, conditions 2 and 3 for $\Z$ are equivalent to conditions 2 and 3 for all $k$ by the universal coefficient theorem.
\end{remark}

Since there is a topological characterization of Cohen-Macaulayness, one can define Cohen-Macaulayness over a field $k$ for any CW complex (more generally any topological space such that there is a well-defined notion of dimension) using condition 3 of theorem \ref{thm:reisnermunkres} as the definition:

\begin{definition}\label{def:CMcomplex}
A CW complex $X$ of dimension $n$ is \textbf{Cohen-Macaulay} over a ring $A$ if it satisfies $\tilde H_i(X;A) = H_i(X,X\setminus p; A) = 0$ for every $i<n$ and every $p\in X$.
\end{definition}

Theorem \ref{thm:reisnermunkres} can then be summarized as stating that a simplicial complex is Cohen-Macaulay (i.e. its Stanley-Reisner ring is Cohen-Macaulay) if and only if its geometric realization is. We have a similar statement for posets if one takes the geometric realization of a poset to be the total space of its order complex. 

One can ask the same question about boolean complexes: if $\Delta$ is a boolean complex, is the Cohen-Macaulayness of the $k[\Delta]$ defined in \ref{def:SRringofboolean} equivalent to that of $|\Delta|$ as defined in \ref{def:CMcomplex}? This is also true. Half of this statement was proven by Stanley in the paper \cite{stanley} that introduced definition \ref{def:SRringofboolean}, and the other by his student Art Duval in \cite{duval}.

\begin{thm}[\cite{duval}, corollary 6.1]\label{thm:duval}
If $P$ is the face poset of a boolean complex $\Delta$ {\em without} a minimal element appended, then the depth (in the sense of remark \ref{rmk:depthofgraded}) of the ring $k[P]$ defined in \ref{def:SRringofposet} equals that of the ring $k[\Delta]$ defined in \ref{def:SRringofboolean}.
\end{thm}

\begin{cor}\label{cor:duval}
If $\Delta$ is a boolean complex and $k$ is a field, then $|\Delta|$ is Cohen-Macaulay over $k$ in the sense of definition \ref{def:CMcomplex} if and only if its Stanley-Reisner ring $k[\Delta]$ is Cohen-Macaulay.
\end{cor}

\begin{proof}
$\Delta$'s Cohen-Macaulayness over $k$, in the topological sense, is equivalent to that of its barycentric subdivision, since they are homeomorphic. The latter is the geometric realization of the order complex $\Delta(P)$, so its Cohen-Macaulayness is equivalent to the Cohen-Macaulayness of $k[\Delta(P)] = k[P]$, by theorem \ref{thm:reisnermunkres}, and this is equivalent to the Cohen-Macaulayness of $k[\Delta]$ by Duval's theorem \ref{thm:duval}.
\end{proof}

Thus when studying Stanley-Reisner rings (of simplicial complexes, posets, or boolean complexes), it is possible to think of Cohen-Macaulayness as a property either of a ring or of a topological space.

\begin{remark}
Again, one may replace $k$ by $\Z$ in the statements of \ref{thm:duval} and \ref{cor:duval}, for the same reasons as in \ref{rmk:kvsZ}.
\end{remark}

It is worth pausing to take stock of the geometric flavor of the topological Cohen-Macaulay condition. A Cohen-Macaulay space is much like a homology manifold with an additional global acyclicity constraint. In particular, spheres and balls are Cohen-Macaulay over every field (equivalently, $\Z$). The real projective plane $R\Proj^2$ is Cohen-Macaulay over fields of characteristic different from $2$, but not over $\F_2$, or over $\Z$, due to having nontrivial $H_1$. The torus $\mathbb{T}^2$ has $H_1(\mathbb{T}^2;\Z) = \Z^2$, so it is not Cohen-Macaulay over any field.

\subsection{The Garsia map}\label{sec:garsiamap}

The goal of this section is to prove results that will allow the Cohen-Macaulayness of a ring of polynomial permutation invariants to be deduced from the Cohen-Macaulayness of a corresponding Stanley-Reisner ring. Almost all of our work here is based on the pair of papers \cite{garsia}, \cite{garsiastanton} by Adriano Garsia, the second in collaboration with Dennis Stanton, published in 1980 and 1984 respectively. Garsia discovered that there is a natural $S_n$-equivariant linear isomorphism between the polynomial ring in $n$ variables and a certain Stanley-Reisner ring, which translates bases for the latter into bases for the former. Garsia and Stanton then applied this map to invariant rings of subgroups $G\subset S_n$. (They also showed how to obtain certain bases from the geometry of cell complexes; more on this in section \ref{sec:shellings}.)

While they stated all their results over $\Q$, a number of their arguments are independent of characteristic. This seems to have been first noted in print in 2003 by Victor Reiner in an appendix to the paper \cite{hersh2} of Patricia Hersh. We give our own, self-contained, account.

\subsubsection{Preliminary ideas}\label{sec:prelim}

\begin{notation}\label{not:definitionofR}
In this section, $A$ is an arbitrary commutative unital ring; sometimes it is specified to be an integral domain. In subsequent sections, it will always be $\Z$ or $\F_p$. Henceforth, $R$ is always the polynomial ring $A[x_1,\dots,x_n]$. 
\end{notation}

%sophie says I need to make this permeate the thesis more. She keeps forgetting R is a polynomial ring. I can put it in the notation page somehow? She suggested I change R to something more clearly a polynomial ring e.g. $A[\mathbf{x}]$, but I'm not sure I can stand to do this. Put something about notation in the introduction, in addition to the list of symbols?

The circle of ideas presented here begins with a definition, which we see as the key to understanding Garsia's accomplishment in \cite{garsia}, although he does not make it entirely explicit.

\begin{definition}\label{def:stackup}
We say two monomials $x_1^{k_1}\dots x_n^{k_n}$ and $x_1^{\ell_1}\dots x_n^{\ell_n}$ of $R$ \textbf{stack up} if there is an index $i$ that maximizes both $k_i$ and $\ell_i$, another index $i'$ that maximizes both $k_{i'}$ and $\ell_{i'}$ among the remaining indices, and so on.
\end{definition}

\begin{example}
$x_1^2x_2^3x_3$ and $x_1x_2^2$ stack up. $x_1^2x_2^3x_3$ and $x_1x_2$ stack up. $x_1^2x_2^3x_3$ and $x_1^2x_2$ do not stack up. (See figure \ref{fig:stackingup}.)
\end{example}

\begin{figure}
\begin{tikzpicture}
\node at (1.5,0) {};

\node at (3.4,-0.5) {These stack up.};

\draw (2,0) rectangle (2.9,0.4); \node at (2.5,0.2) {$x_1$};
\draw (2,0.5) rectangle (2.9,0.9); \node at (2.5,0.7) {$x_1$};
\draw (3,0) rectangle (3.9,0.4); \node at (3.5,0.2) {$x_2$};
\draw (3,0.5) rectangle (3.9,0.9); \node at (3.5,0.7) {$x_2$};
\draw (3,1) rectangle (3.9,1.4); \node at (3.5,1.2) {$x_2$};
\draw (4,0) rectangle (4.9,0.4); \node at (4.5,0.2) {$x_3$};

\draw (2,2) rectangle (2.9,2.4); \node at (2.5,2.2) {$x_1$};
\draw (3,2) rectangle (3.9,2.4); \node at (3.5,2.2) {$x_2$};
\draw (3,2.5) rectangle (3.9,2.9); \node at (3.5,2.7) {$x_2$};

\node at (8.9,-0.5) {These stack up.};

\draw (7.5,0) rectangle (8.4,0.4); \node at (8,0.2) {$x_1$};
\draw (7.5,0.5) rectangle (8.4,0.9); \node at (8,0.7) {$x_1$};
\draw (8.5,0) rectangle (9.4,0.4); \node at (9,0.2) {$x_2$};
\draw (8.5,0.5) rectangle (9.4,0.9); \node at (9,0.7) {$x_2$};
\draw (8.5,1) rectangle (9.4,1.4); \node at (9,1.2) {$x_2$};
\draw (9.5,0) rectangle (10.4,0.4); \node at (10,0.2) {$x_3$};

\draw (7.5,2) rectangle (8.4,2.4); \node at (8,2.2) {$x_1$};
\draw (8.5,2) rectangle (9.4,2.4); \node at (9,2.2) {$x_2$};

\node at (14.4,-0.5) {These do not stack up.};

\draw (13,0) rectangle (13.9,0.4); \node at (13.5,0.2) {$x_1$};
\draw (13,0.5) rectangle (13.9,0.9); \node at (13.5,0.7) {$x_1$};
\draw (14,0) rectangle (14.9,0.4); \node at (14.5,0.2) {$x_2$};
\draw (14,0.5) rectangle (14.9,0.9); \node at (14.5,0.7) {$x_2$};
\draw (14,1) rectangle (14.9,1.4); \node at (14.5,1.2) {$x_2$};
\draw (15,0) rectangle (15.9,0.4); \node at (15.5,0.2) {$x_3$};

\draw (13,2) rectangle (13.9,2.4); \node at (13.5,2.2) {$x_1$};
\draw (13,2.5) rectangle (13.9,2.9); \node at (13.5,2.7) {$x_1$};
\draw (14,2) rectangle (14.9,2.4); \node at (14.5,2.2) {$x_2$};

\end{tikzpicture}
\caption{Stacking up. $x_1^2x_2^3x_3$ and $x_1x_2^2$ stack up; $x_1^2x_2^3x_3$ and $x_1x_2$ stack up; but $x_1^2x_2^3x_3$ and $x_1^2x_2$ do not stack up.}\label{fig:stackingup}
\end{figure}
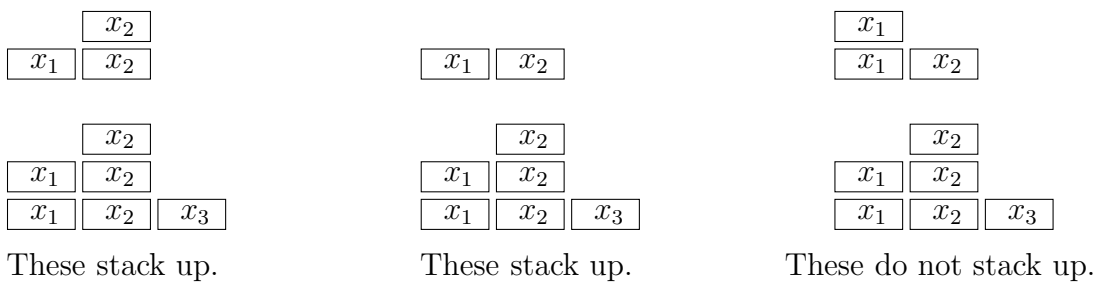

\begin{notation}\label{not:lambdaandP}
The tuple of exponents in a monomial of $R$ specifies a partition of the monomial's degree -- for example for the monomial $x_1^3x_2^4x_3$, this partition is $8 = 4 + 3 + 1$. Let the partition associated to a monomial $m\in R$ be denoted $\Lambda(m)$ and referred to as the \textbf{shape} of $m$. Since $R$ has $n$ indeterminates, $\Lambda(m)$ never has more than $n$ parts. Thus let $\lambda = (\lambda_1,\dots,\lambda_n)$ with $\lambda_1\geq \lambda_2\geq\dots\geq\lambda_n$ (with each $\lambda_i \geq 0$) be such a partition, and let $|\lambda| = \sum\lambda_i$. Let $\scrP$ be the set of all such partitions.
\end{notation}

One can add partitions by
\[
\mu + \lambda = (\mu_1+\lambda_1,\dots,\mu_n+\lambda_n).
\]
This makes $\scrP$ into a commutative monoid, with identity $\emptyset = (0,\dots,0)$. The map 
\[
\lambda \mapsto |\lambda|
\]
 is a monoid homomorphism $\scrP\rightarrow \N$.

The motivation for the definition of stacking up comes from the way multiplication of monomials interacts with this monoid structure of $\scrP$. Indeed, the definition is designed so that pairs of monomials that stack up are exactly those for whom multiplication corresponds to addition of the partitions describing their shapes. More precisely:

\begin{lemma}\label{lem:stackingup}
If $m,m'$ are monomials of $R$, then 
$$
\Lambda(m m') = \Lambda(m)+\Lambda(m')
$$
 if and only if $m,m'$ stack up.
\end{lemma}

\begin{proof}
If $m$ and $m'$ stack up, then the same $x_i$ occurs a maximum number of times in both, and thus also in the product $mm'$. Thus $\Lambda(mm')_1 = \Lambda(m)_1 + \Lambda(m')_1$. Similarly, among the remaining $x_i$'s, the same one occurs a maximum number of times in both, so $\Lambda(mm')_2 = \Lambda(m)_2 + \Lambda(m')_2$; similar reasoning shows that $\Lambda(mm')_i = \Lambda(m)_i + \Lambda(m')_i$ for all $i$. There is equality in this case.

Conversely, if $m$ and $m'$ do not stack up, then there is a $j\in \gls{[n]}$ such that while a set of $j-1$ greatest exponents in $m$ and in $m'$ can be chosen to have the same indices, the $j$th greatest exponent in each monomial occurs on different indices. Thus while the greatest $j-1$ exponents in the product $mm'$ can be taken to be sums of the greatest $j-1$ exponents in each of $m$ and $m'$, no exponent in the product $mm'$ other than these greatest $j-1$ can be as great as the sum of the two $j$th greatest exponents in $m$ and $m'$ (since they occur on different indeterminates). I.e. we must have $\Lambda(mm')_j < \Lambda(m)_j + \Lambda(m')_j$ even though $\Lambda(mm')_i = \Lambda(m)_i+\Lambda(m')_i$ for $i<j$.
\end{proof}

It may seem that this lemma does not do much for monomials that do not stack up. However, we can still say something important. In a sense to be made precise momentarily, when $m$ and $m'$ fail to stack up, $\Lambda(mm')$ is strictly lower than $\Lambda(m)+\Lambda(m')$.

\begin{definition}
The \textbf{degree lexicographic order} on partitions (\textit{deglex} for short) is the total order given by $\lambda > \mu$ if $|\lambda|>|\mu|$, or $|\lambda|=|\mu|$ but $\lambda_1 > \mu_1$, or $|\lambda|=|\mu|$ and $\lambda_1 = \mu_1$ but $\lambda_2 > \mu_2$, or etc.
\end{definition}

In all of what follows, all inequalities between partitions are with reference to this order.

\begin{remark}
The monoid structure of $\scrP$ is compatible with the degree lexicographic order in the sense that $\emptyset \leq \lambda$, and $\lambda > \mu \Rightarrow \lambda + \nu > \mu + \nu$, for all $\lambda,\mu,\nu\in \scrP$.
\end{remark}

\begin{remark}
The partial order $\leq$ not only totally orders $\scrP$ but well-orders it, i.e. $(\scrP,\leq)$ satisfies the descending chain condition. In fact, $(\mathscr{P},\leq)$ is order-isomorphic to $\N$, since there are only finitely many $\lambda\in \scrP$ with $|\lambda|=n$ for a fixed $n\in \N$, and thus only finitely many $\lambda\in\scrP$ less than a fixed $\mu\in\scrP$. In particular, we can do induction on $\lambda\in\scrP$.
\end{remark}

\begin{lemma}\label{lem:stackingandlex}
If monomials $m,m'$ of $R$ fail to stack up, then $\Lambda(mm') < \Lambda(m)+\Lambda(m')$ with respect to degree lexicographic order.
\end{lemma}

\begin{proof}
Examining the proof of lemma \ref{lem:stackingup}, we see that when $m,m'$ do not stack up, there is a $j\in \gls{[n]}$ such that $\Lambda(mm')_i = \Lambda(m)_i + \Lambda(m')_i$ for $i < j$ and $\Lambda(mm')_j < \Lambda(m)_j + \Lambda(m')_j$. This implies $\Lambda(mm') < \Lambda(m)+\Lambda(m')$ degree-lexicographically, in view of the fact that
\[
|\Lambda(mm')| = \deg mm' = \deg m + \deg m' = |\Lambda(m)| + |\Lambda(m')| = |\Lambda(m) + \Lambda(m')|
\]
with the last equality because $|\cdot|$ is a monoid homomorphism.
\end{proof}

\begin{remark}\label{rmk:partitiondecompSninvar}
In fact, in the situation of this lemma, $\Lambda(mm') < \Lambda(m)+\Lambda(m')$ not only with respect to degree lexicographic order but even with respect to {\em dominance order}, a partial order on partitions that is refined by the degree lexicographic total order. (In the dominance order, $\lambda \geq\mu$ means for any $k$, the biggest $k$ parts of $\lambda$ add up to at least as much as the biggest $k$ parts of $\mu$.) This is a somewhat stronger statement. Garsia and Stanton (in \cite{garsia},\cite{garsiastanton}) and Victor Reiner (in \cite{hersh2}) formulate the ideas on which this section is based in terms of the dominance order rather than deglex order. We prefer the deglex order because it is a total order on partitions, allowing it to be used to filter the ring $R$, which we do next.
\end{remark}

\begin{notation}
Given a partition $\lambda\in \scrP$, let $R_\lambda$ be the $A$-submodule of $R$ generated by the monomials of shape $\lambda$ (i.e. those $m\in R$ with $\Lambda(m)=\lambda$). We have
\[
R = \bigoplus_{\lambda\in \scrP} R_{\lambda}.
\]
Let $R_{\leq \lambda} = \bigoplus_{\nu \leq \lambda} R_\nu$, and let $R_{<\lambda} = \bigoplus_{\nu< \lambda} R_\nu$.
\end{notation}

\begin{remark}\label{rmk:A[S_n]-module}
The direct sum decomposition into $R_\lambda$'s is $S_n$-invariant; in fact, if $m\in R$ is a monomial and $\lambda = \Lambda(m)$, then $R_\lambda$ is precisely the $A$-span of the $S_n$-orbit of $m$, i.e. the cyclic $A[S_n]$-module generated by $m$, where $A[S_n]$ is the group algebra over $A$. Note that the $S_n$-invariance implies that if $G$ is a permutation group, 
\[
R^G = \bigoplus_{\lambda\in\scrP} R_\lambda^G
\]
since $G$ acts separately on each $R_\lambda$.
\end{remark}

\begin{remark}\label{rmk:dispersion}
One may think of the degree lexicographic order on the components $R_\lambda$ as refining the grading by measuring, within any given degree, the extent of the dispersion of the exponents of the monomials. Monomials with lexicographically higher shapes but the same degree have more spread-out exponents. We explicate this idea at length in \cite{FTSP}, where we also relate it to other measures of dispersion, specifically variance and higher moments (see \cite{FTSP}, Theorem 4). In that paper, we call the partial order on monomials induced by the lexicographic order on shapes the {\em symmetric lexicographic order}, to acknowledge its $S_n$-invariance, and to distinguish it from the lexicographic order on monomials themselves.
\end{remark}

The decomposition $R=\bigoplus R_\lambda$ is not a grading, because we do not have $R_\lambda R_\mu \subset R_{\lambda + \mu}$. However, there is a substitute:

\begin{prop}\label{prop:filtered}
We have
\[
R_\lambda R_\mu \subset R_{\leq \lambda +\mu}.
\]
\end{prop}

\begin{proof}
By lemmas \ref{lem:stackingup} and \ref{lem:stackingandlex}, the product of a monomial in $R_\lambda$ with a monomial in $R_\mu$ is a monomial lying in $R_{\leq \lambda + \mu}$. The result follows by $A$-linearity.
\end{proof}

\begin{remark}
Thus, the $R_{\leq \lambda}$'s give an ascending $A$-algebra filtration of $R$. Note that for any permutation group $G$, taking $G$-invariants one finds that the $R_{\leq\lambda}^G$'s give an ascending $A$-algebra filtration of $R^G$ in the same way.
\end{remark}

Now we introduce the Stanley-Reisner ring from which we can transfer bases to $R$. Recall remark \ref{rmk:compactSRposet} which gives a compact description of the Stanley-Reisner ring of a poset.

\begin{notation}\label{not:definitionofS}
Let $\Bn$ be the boolean algebra, i.e. the power set of $\gls{[n]}=\{1,\dots,n\}$, regarded as a poset under inclusion. Let $S$ be the Stanley-Reisner ring of the poset $\Bn\setminus\{\emptyset\}$ over $A$. Denote the indeterminate of $S$ corresponding to a set $U\subset \gls{[n]}$ by $y_U$. 
\end{notation}

\begin{remark}\label{rmk:SisaPLball}
As noted in the previous section, $\Bn\setminus\{\emptyset\}$ is the face poset of an $(n-1)$-dimensional simplex regarded as a regular CW complex (i.e. without its empty face). The order complex of $\Bn\setminus\{\emptyset\}$ is thus the barycentric subdivision of an $(n-1)$-simplex, so its total space is a PL ball. In particular, $S$ is Cohen-Macaulay.
\end{remark}

\begin{definition}\label{def:garsiamap}
The \textbf{Garsia map} $\garsia:S \rightarrow R$ is the $A$-linear isomorphism obtained by first mapping $y_U \mapsto \prod_{i\in U}x_i$, extending multiplicatively to the monomials of $S$, and then extending $A$-linearly to all of $S$. (See figure \ref{fig:garsiamap}.)
\end{definition}

\begin{remark}
Garsia himself (\cite{garsia},\cite{garsiastanton}), and Reiner following him (\cite{hersh2}), call this map the \textit{transfer map}. However, some authors in invariant theory (e.g. \cite{smith95}, \cite{neuselsmith}) use this phrase to refer to the map $R\rightarrow R^G$ given by $x\mapsto \sum_{g\in G} gx$ (though others call this the {\em trace}), and there are parallel usages in topology and group theory, so we are taking the opportunity to rename it to honor Garsia's discovery.
\end{remark}

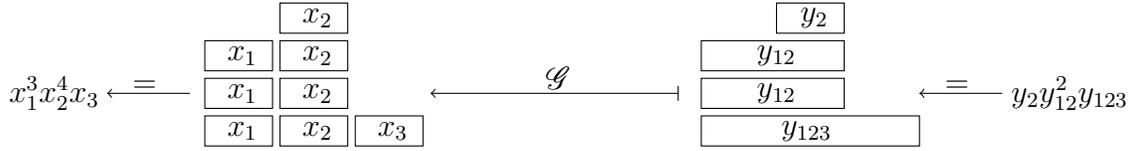
\begin{figure}
\begin{tikzpicture}
\node at (1,0.7) {$x_1^3x_2^4x_3$};
\draw [<-] (1.7,0.7) -- (2.8,0.7); \node at (2.2,0.8) {$=$};
\draw (3,0) rectangle (3.9,0.4); \node at (3.5,0.2) {$x_1$};
\draw (3,0.5) rectangle (3.9,0.9); \node at (3.5,0.7) {$x_1$};
\draw (3,1) rectangle (3.9,1.4); \node at (3.5,1.2) {$x_1$};
\draw (4,0) rectangle (4.9,0.4); \node at (4.5,0.2) {$x_2$};
\draw (4,0.5) rectangle (4.9,0.9); \node at (4.5,0.7) {$x_2$};
\draw (4,1) rectangle (4.9,1.4); \node at (4.5,1.2) {$x_2$};
\draw (4,1.5) rectangle (4.9,1.9); \node at (4.5,1.7) {$x_2$};
\draw (5,0) rectangle (5.9,0.4); \node at (5.5, 0.2) {$x_3$};
\draw [<-|] (6,0.7) -- (9.3,0.7); \node at (7.7,0.9) {$\garsia$};
\draw (9.6,0) rectangle (12.5,0.4); \node at (11,0.2) {$y_{123}$};
\draw (9.6,0.5) rectangle (11.5,0.9); \node at (10.6,0.7) {$y_{12}$};
\draw (9.6,1.0) rectangle (11.5,1.4); \node at (10.6,1.2) {$y_{12}$};
\draw (10.6,1.5) rectangle (11.5,1.9); \node at (11.1,1.7) {$y_2$};
\draw [<-] (12.5,0.7) -- (13.6,0.7); \node at (13,0.8) {$=$};
\node at (14.5,0.7) {$y_2y_{12}^2y_{123}$};
\end{tikzpicture}
\caption{The Garsia map}\label{fig:garsiamap}
\end{figure}

The Garsia map is not a ring homomorphism. However, it is well-behaved in a number of important respects. To begin with, it translates monomials whose product is nonzero in $S$ into monomials that stack up in $R$.

\begin{lemma}\label{lem:maximalchainandstacking}
If $m = \prod_i y_{U_i}$ is a monomial of $S$, then a maximal chain
\[
\{i_1\}\subset \{i_1,i_2\}\subset \dots \subset \gls{[n]}
\]
of $\Bn$ contains every $U_i$ if and only if in $\garsia(m) = \prod_j x_j^{k_j}$, $k_{i_1}$ is a maximal exponent, $k_{i_2}$ is a maximal exponent other than $k_{i_1}$, etc.
\end{lemma}

\begin{example}
This lemma is probably best apprehended by staring at figure \ref{fig:garsiamap} with the statement in mind, although we give a proof below. The maximal chain $\{2\}\subset \{1,2\}\subset\{1,2,3\}$ supports the monomial $y_2y_{12}^2y_{123}$ of the figure, and this is equivalent to the fact that $x_2$ has maximum exponent in $\garsia(y_2y_{12}^2y_{123}) = x_1^3x_2^4x_3$, and $x_1$ has maximum exponent among indeterminates other than $x_2$.
\end{example}

\begin{proof}[Proof of lemma \ref{lem:maximalchainandstacking}]
Since the $U_i$'s are supported on a chain (by the definition of the Stanley-Reisner ring $S$), without loss of generality we can reorder them so that $U_1\supset U_2\supset\dots \supset U_k$ (repetitions allowed). By definition of $\garsia$, we have
\[
\garsia(m) = \prod_i\prod_{j\in U_i} x_j
\]
The indeterminates with an exponent of $k$ in this product are precisely the $x_j$'s for which $j$ is in all $k$ of the $U_i$'s, or equivalently in $U_k$. More generally, for $\ell\in [k]$, the indeterminates with an exponent of at least $\ell$ are precisely those in $U_1,\dots,U_\ell$, or equivalently in $U_\ell$. Thus, in general, the indices giving the maximum $|U_\ell|$ exponents are exactly those in $U_\ell$. It follows that for a maximal chain 
\[
\{i_1\}\subset \{i_1,i_2\}\subset \dots \subset \gls{[n]},
\]
the property of supporting all the $U_i$'s of $m$ is equivalent to the property that $x_{i_1}$ has maximum exponent in $\garsia(m)$, $x_{i_2}$ has maximum exponent among the remaining indeterminates, etc.
\end{proof}

\begin{cor}\label{cor:prodnonzerostackup}
Let $m,m'\in S$ be monomials. Then $mm' \neq 0$ if and only if $\garsia(m),\garsia(m')\in R$ stack up.
\end{cor}

\begin{proof}
By the definition of $S$, $mm'\neq 0$ if and only if all the indeterminates in $m,m'$ are supported on the same maximal chain of $\Bn$. So we need to show that $\garsia(m),\garsia(m')$ stacking up is equivalent to the existence of a maximal chain in $\Bn$ such that every $U\in \Bn$ for which either $y_U\mid m$ or $y_U\mid m'$ belongs to this chain. This is immediate from the preceding lemma and the definition of stacking up.
\end{proof}

\begin{remark}
The proof of this corollary reveals why we chose the name ``stack up" for definition \ref{def:stackup}. Monomials in $R$ stack up if and only if their counterparts in $S$ contain only indeterminates supported on the same maximal chain of $\Bn$. Per the diagrams we have been using in figures \ref{fig:stackingup} and \ref{fig:garsiamap}, indeterminates supported on the same chain are depicted as a stack of blocks in which each block rests entirely on the one below it. If two blocks correspond to incomparable $U$'s in $\Bn$, then neither can rest entirely on the other -- they do not ``stack up."
\end{remark}

\begin{definition}
We will say that monomials $m,m'\in S$ \textbf{stack up} if their $\garsia$-images in $R$ stack up.
\end{definition}

Thus in $S$, monomials either stack up or multiply to zero.

A second good behavior of the Garsia map critical to our purposes is that it is $S_n$-equivariant for the natural actions of $S_n$ on the two rings. Furthermore, if $R$ and $S$ are given standard $\N$-gradings by $\deg y_U = |U|$ and $\deg x_i = 1$, then the $S_n$-actions are both graded actions, and $\garsia$ is a graded map. (In either case, when we speak of the \textit{degree} of an element of either $R$ or $S$, it is to this grading that we refer.) The ring $S$ even has an $\N^n$-grading which interacts well with the filtration of $R$ given above (in proposition \ref{prop:filtered}):

\begin{definition}\label{def:finegrading}
The \textbf{fine grading} of the Stanley-Reisner ring $S$ is the $\N^n$-grading $d$ given by
\[
y_U \xmapsto{d} e_{|U|}
\]
where $e_i$ is the $i$th member of the standard basis for $\N^n$.
\end{definition}

This indeed gives a grading of $S$ because it gives a grading of the polynomial ring $A[\{y_U\}_{U\in \Bn\setminus\emptyset}]$, of which $S$ is the quotient by a monomial (and therefore homogeneous) ideal.%add more to appendix checking this?

\begin{example}
The monomial $y_2y_{12}^2y_{123}$ of figure \ref{fig:garsiamap} has $d(y_2y_{12}^2y_{123}) = e_1 + 2e_2 + e_3$.
\end{example}

\begin{definition}
If an element of $S$ is homogeneous with respect to the fine grading, it is \textbf{finely homogeneous}. We refer to its image in $\N^n$ under the fine grading as its \textbf{fine grade}.
\end{definition}

Note that $S_n$ acts gradedly on $S$ even with respect to the fine grading, and in fact (similarly to remark \ref{rmk:A[S_n]-module}), the homogeneous components of $S$ with respect to the fine grading are exactly the cyclic $A[S_n]$-submodules generated by monomials.

\begin{notation}\label{not:lambdaconj}%overworking n here. Use a different letter.
The fine grading of $S$ reverts to the standard (degree) grading by mapping $e_i\mapsto i$. In this way, the fine grade of a monomial (more generally, a finely homogeneous element) gives a partition of its degree. In the example of figure \ref{fig:garsiamap}, this partition is 
\[
8 = 1 + 2\cdot 2 + 3.
\]
As a tuple this is $(3,2,2,1)$. If $m$ is a monomial of $S$, let this partition be called $\overline{\Lambda}(m)$. Note that for any indeterminate $y_U$ of $S$, $U\subset \gls{[n]}$, thus $|U|\leq n$. Therefore all the parts of this partition have size $\leq n$. Call the set of all such partitions $\scrPconj$.
\end{notation}

\begin{lemma}\label{lem:conjugatepartition}
Let $m\in S$ be a monomial. Then $\overline{\Lambda}(m)$ is the conjugate partition of $\Lambda(\garsia(m))$.
\end{lemma}

As with lemma \ref{lem:maximalchainandstacking}, one is probably most readily convinced of this lemma by examining figure \ref{fig:garsiamap} with its statement in mind, but for good form here is a proof:

\begin{proof}
Monomials 
\[
m=\prod_1^k y_{U_i}
\]
in $S$ (repetitions of $U_i$ allowed) are supported on chains of $\Bn$, i.e. by suitably ordering the $U_i$ we have 
\[
U_1\supset U_2\supset\dots\supset U_k.
\]
Then the partition $\overline{\Lambda}(m)$ is 
\[
\deg m = \sum |U_i|,
\]
i.e. $(|U_1|,|U_2|,\dots,|U_k|)$ as a tuple. Note that as the shape of the monomial $\garsia(m)$ is $S_n$-invariant, we may reorder the indeterminates of $R$ (with corresponding action on the $U_i$'s) without loss of generality. Therefore, using the fact that the $U_i$'s form a chain, reorder the $x_j$'s of $R$ such that each $U_i$ has the form $\{1,\dots,\ell_i\}$ for some $\ell_i\leq n$. (To wit: send the indices in $U_k$ to $1,\dots,|U_k|$; any indices in $U_{k-1}\setminus U_k$ to $|U_k|+1,\dots,|U_{k-1}|$; etc.)

Note that the achievement of this reordering is that in the monomial 
\[
\garsia(m)= \prod_i\prod_{j\in U_i} x_j,
\]
the exponents are nonincreasing. Thus $x_1,\dots,x_{|U_k|}$ all have an exponent of $k$, while if there is anything in $U_{k-1}\setminus U_k$ then $x_{|U_k|+1},\cdots,x_{|U_{k-1}|}$ all have exponent $k-1$, etc. Thus the partition $\Lambda(\garsia(m))$ giving the shape of this monomial is
\[
\deg \garsia(m) = |U_k|\cdot k + |U_{k-1}\setminus U_k| \cdot (k-1) + \dots + |U_1\setminus \bigcup_2^k U_i| \cdot 1
\]
If the $U_i$'s index the rows of a Ferrers diagram, then the parts of this partition measure the columns of this diagram, while the parts of $\overline{\Lambda}(m)$ measure the rows. Thus they are conjugate partitions.
\end{proof}

\begin{notation}
If $\lambda\in\scrP$ is a partition, then let $\overline{\lambda} \in \scrPconj$ be the conjugate partition. Given $\overline{\lambda}\in \scrPconj$, let $S_{\overline\lambda}$ be the $A$-submodule of $S$ generated by monomials $m$ with $\overline{\Lambda}(m) = \overline{\lambda}$. Note this is precisely a homogeneous component of $S$ for the fine grading.
\end{notation}

\begin{cor}\label{cor:conjugatepartition}
The Garsia map sends the homogeneous components $S_{\overline{\lambda}}, \overline{\lambda}\in \scrPconj$ to the direct summands $R_\lambda, \lambda\in \scrP$.\qed
\end{cor}

The theory we have developed so far allows us to limit in a precise way how far the Garsia map is from a ring homomorphism:

\begin{prop}\label{prop:approximatehomomorphism}
Let $f,g\in S$ be finely homogeneous, with $fg = h$ in $S$. Let 
\[
\lambda = \Lambda(\garsia(h)) = \overline{\overline{\Lambda}(h)}
\]
be the shape of $h$'s image in $R$ under the Garsia map. Then
\[
\garsia(f)\garsia(g) - \garsia(h) \in R_{<\lambda}.
\]
\end{prop}

\begin{proof}
By corollary \ref{cor:prodnonzerostackup}, the terms in the product $fg$ that contribute to $h$ are precisely those that come from monomials in $f$ and $g$ that stack up. The corresponding terms in the product $\garsia(f)\garsia(g)$ are precisely those that are in $R_\lambda$, by lemma \ref{lem:stackingup}, and the rest are in $R_{<\lambda}$, by lemma \ref{lem:stackingandlex}. Thus $\garsia(h)$ cancels the $R_\lambda$-part of $\garsia(f)\garsia(g)$, and what is left lies in $R_{<\lambda}$.
\end{proof}

\begin{remark}\label{rmk:approxhomformultiple}
This proposition extends easily by induction to any finite number of factors.
\end{remark}

We think of this proposition as telling us that $\garsia$ is a ``first-order approximation of a homomorphism," i.e. that if one approximates a product in $R$ with the corresponding product in $S$, one gets the ``biggest" terms (i.e. those maximal with respect to deglex order on shapes) correct.

\subsubsection{The main theorem about the Garsia map}\label{sec:garsiabasis}

The objective of all the theory developed in the last subsection is the following theorem, which is a slight generalization of \cite{garsiastanton}, Theorem 9.1:

\begin{thm}[Garsia-Stanton]\label{thm:garsiabasis}
Suppose $A$ is an integral domain, and $S^G$ is free as an $S^{S_n}$-module. Let 
\[
B = \{b_1,\dots,b_r\}
\]
be a finely homogeneous basis for $S^G$ as $S^{S_n}$-module. Then its image 
\[
\garsia(B)=\{\garsia(b_1),\dots,\garsia(b_r)\}
\]
under the Garsia map is a homogeneous basis for $R^G$ as an $R^{S_n}$-module.
\end{thm}

\begin{proof}
We need to show that $\garsia(b_1),\dots,\garsia(b_r)$ are linearly independent over $R^{S_n}$ and span $R^G$ over it.

Recall that $R^G = \bigoplus_{\lambda\in\scrP} R_\lambda^G$ (see remark \ref{rmk:A[S_n]-module}). We get spanningness from an induction on $\lambda$: we will assume that $R_{<\lambda}^G$ lies in the $R^{S_n}$-span of $\garsia(B)$, and show that $R_\lambda^G$ does too. The base case to be verified is that $R_{<\emptyset}^G$ lies in the $R^{S_n}$-span of $\garsia(B)$, but $R_{<\emptyset}^G$ is a void direct sum, i.e. it is the zero $A$-submodule of $R^G$, so of course it does.

Thus, let $f\in R_\lambda^G$, and let
\[
\garsia^{-1}(f) = \sum_1^r s_ib_i, \; s_i\in S^{S_n}
\]
be a representation of $\garsia^{-1}(f)\in S^G$, guaranteed to exist by the fact that $B$ spans $S^G$ as $S^{S_n}$-module. We claim that $f-\sum_1^r \garsia(s_i)\garsia(b_i)$ is contained in $R_{<\lambda}^G$. This is a consequence of proposition \ref{prop:approximatehomomorphism}, as follows:

By the linearity of $\garsia$, we have
\[
f = \sum_1^r \garsia(s_ib_i).
\]
Since the $b_i$'s are finely homogeneous and $\garsia^{-1}(f)$ is too (since $f\in R_\lambda$ and thus $\garsia^{-1}(f)\in S_{\overline{\lambda}}$), and since, in the graded ring $S$, an equality is separately an equality in each finely graded component, we can without loss of generality assume that the $s_i$'s are finely homogeneous and $\garsia(s_ib_i)\in R_\lambda$ for each $i$. (If not, drop any terms of any $s_i$'s whose product with the corresponding $b_i$ has fine grading different from $\overline{\lambda}$, since all such terms must ultimately cancel.) Then proposition \ref{prop:approximatehomomorphism} applies to each product $s_ib_i$, and we have $\garsia(s_ib_i) - \garsia(s_i)\garsia(b_i)\in R_{<\lambda}$. Since $R_{<\lambda}$ is closed under addition, we therefore have
\[
f - \sum_1^r \garsia(s_i)\garsia(b_i) = \sum_1^r\left(\garsia(s_ib_i) - \garsia(s_i)\garsia(b_i)\right) \in R_{<\lambda}.
\]
Note that everything in sight is $G$-invariant, so $f-\sum_1^r\garsia(s_i)\garsia(b_i)$ is actually in $R_{<\lambda}^G$. Thus, by the induction assumption, we have
\[
f - \sum_1^r\garsia(s_i)\garsia(b_i) = \sum_1^r t_i\garsia(b_i),\; t_i\in R^{S_n}.
\]
Since each $\garsia(s_i)$ is also in $R^{S_n}$ due to $\garsia$'s $S_n$-equivariance, it follows that the resulting expression
\[
f = \sum_1^r \left(\garsia(s_i) + t_i\right)\garsia(b_i)
\]
is an expression of $f$ as a linear combination of $\garsia(b_i)$'s with coefficients in $R^{S_n}$; this shows $R_\lambda^G$ lies in the $R^{S_n}$-span of $\garsia(B)$, completing the proof of spanningness.

For linear independence, suppose for a contradiction that we have a nontrivial $R^{S_n}$-linear relation
\[
0 = \sum_1^r t_i\garsia(b_i).
\]
There is a greatest $\lambda\in\scrP$ such that some nonzero monomial of some $t_i\garsia(b_i)$ occurs in $R_\lambda$. For each $i$ such that $t_i\garsia(b_i)$ does have a term of shape $\lambda$, it must be that $t_i$ contains at least one term of shape $\lambda - \Lambda(\garsia(b_i))$, and no terms of higher shape. This is because $t_i$ is in $R^{S_n}$, and therefore, for each of its terms, it also contains every term obtained by permuting the $x_i$'s. Thus one of its terms of maximum shape will stack up with any given term of $\garsia(b_i)$, and therefore $\lambda$ will be the sum of the shape of $\garsia(b_i)$ (i.e. $\Lambda(\garsia(b_i))$) and the shape of terms of maximum shape in $t_i$. (We are using the fact that the coefficient ring $A$ is an integral domain, so that any given product of terms will be nonzero.)

For each $i$, let $t_i'$ be $t_i$'s projection to $R_{\lambda-\Lambda(\garsia(b_i))}$, i.e. the result of discarding all terms of shape lower than $\lambda-\Lambda(\garsia(b_i))$. Since the $R_\lambda$'s are $S_n$-invariant, $t_i' \in R_{\lambda-\Lambda(\garsia(b_i))}^{S_n}$ for each $i$. (Some $t_i'$'s may be zero, but the ones coming from $t_i\garsia(b_i)$'s that have terms in $R_\lambda$ are nonzero by construction.) Now it may no longer be true in $R^G$ that $0 = \sum_1^r t_i'\garsia(b_i)$. However, by construction, since for each $i$ we have only dropped terms of shape lower than $\lambda - \Lambda(\garsia(b_i))$, the terms we lose from the product $t_i'\garsia(b_i)$ are all of shape lower than $\lambda$. It follows that the projection to $R_\lambda$ of $\sum_1^r t_i'\garsia(b_i)$ contains the same terms as the projection to $R_\lambda$ of $\sum_1^r t_i\garsia(b_i)$. We deduce that this projection is zero.

It follows that
\begin{equation}\label{eq:lincomb}
0 = \sum_1^r\garsia^{-1}(t_i') b_i
\end{equation}
in $S$. This is because the terms of each $t_i'\garsia(b_i)$ that are in $R_\lambda$ are exactly those coming from pairs of terms from $t_i'$ and $\garsia(b_i)$ that stack up, and by corollary \ref{cor:prodnonzerostackup}, these are exactly those whose corresponding pair of terms in $\garsia^{-1}(t_i')$ and $b_i$ have nonzero product. Thus for each $i$, the projection of $t_i'\garsia(b_i)$ to $R_\lambda$ is equal to $\garsia(\garsia^{-1}(t_i')b_i)$.

Thus \eqref{eq:lincomb} is a nontrivial $S^{S_n}$-linear combination of the $b_i$'s, contrary to assumption. This completes the proof of linear independence.
\end{proof}

\begin{remark}\label{rmk:spanningandLIareseparate}
The proof of spanningness given here, although formulated in the language of stacking up that we have developed, is essentially that found in \cite{garsiastanton}. In \cite{garsiastanton}, Garsia and Stanton use a counting argument based on Hilbert series to obtain that linear independence follows from spanningness. 

We have given the above proof in order to show that spanningness and linear independence are independent of each other. Thus we can conclude that $\garsia(B)$ spans $R^G$ over $R^{S_n}$ if $B$ spans $S^G$ over $S^{S_n}$, regardless of linear independence, and we can also deduce linear independence of $\garsia(B)$ over $R^{S_n}$ from that of $B$ over $S^{S_n}$, regardless of spanningness.

The spanningness proof makes no use of the assumption that $A$ is a domain, so we can lift spanningness from $S^G$ to $R^G$ over any coefficient ring $A$.
\end{remark}

\subsubsection{The FTSP, and h.s.o.p.'s for $R$ and $S$}\label{sec:FTSPhsops}

We promised a proof of theorem \ref{thm:ftsp} (the FTSP) in terms of the tools we have been developing. We have gotten far enough to give this proof. As a bonus, we will obtain explicit h.s.o.p.'s for the invariant rings of $R$ and $S$ under the action of any permutation group. (In fact, we have already done this for $R$, in the proof of \ref{thm:equivofquestions}.)

\begin{notation}\label{not:rankrowsums}
Let $A$ be any coefficient ring and let $S$ be as defined in \ref{not:definitionofS}. For $i=1,\dots,n$, let 
\[
\theta_i = \sum_{|U|=i} y_U\in S.
\]
These are called the \textbf{rank-row sums} of $S$.
\end{notation}

\begin{remark}
Note that the $\theta_i$'s are $S_n$-invariant, and that the elementary symmetric polynomials $\sigma_i$ are the images $\garsia(\theta_i)$ of the $\theta_i$'s under the Garsia map.
\end{remark}

\begin{prop}[FTSP for $S$]\label{prop:FTSPforS}
The rank-row sums $\theta_1,\dots,\theta_n$ are algebraically independent over $A$, and generate $S^{S_n}$ as an $A$-algebra.
\end{prop}

\begin{proof}
Algebraic independence is an immediate consequence of the fine grading of $S$ defined in \ref{def:finegrading}. Note that the fine grade of each $\theta_i$ is the $i$th basis vector $e_i$ of $\N^n$. Thus the map from monomials $\prod\theta_i^{a_i}$ to their fine grades is
\[
\prod \theta_i^{a_i} \mapsto \sum a_ie_i
\]
This map is injective (actually bijective) from tuples $(a_1,\dots,a_n)$ to $\N^n$. Since the graded pieces of a graded $A$-algebra are $A$-linearly independent, this shows that all distinct monomials in the $\theta_i$'s are $A$-linearly independent, which is the statement of algebraic independence.

$S$ has as an $A$-basis the monomials in the $y_U$ (for $U\in \Bn\setminus\{\emptyset\}$) that are supported on chains of $\Bn\setminus\{\emptyset\}$. Thus $S^{S_n}$ has as an $A$-basis the $S_n$-orbit monomials of these monomials. To prove that the $\theta_i$ generate $S^{S_n}$ as an $A$-algebra it is necessary and sufficient to show that they generate any individual one of these orbit monomials. Thus, let $U_1\subsetneq\dots\subsetneq U_r$ be a chain in $\Bn\setminus\{\emptyset\}$, and let $S_ny_{U_1}^{a_1}\dots y_{U_r}^{a_r}$ be the orbit monomial of a monomial $y_{U_1}^{a_1}\dots y_{U_r}^{a_r}$ supported on that chain. Then we claim that
\[
S_ny_{U_1}^{a_1}\dots y_{U_r}^{a_r} = \theta_{|U_1|}^{a_1}\dots\theta_{|U_r|}^{a_r}.
\]
We see this as follows: $S_n$ acts transitively on chains with any given rank set $\{|U_1|,\dots,|U_r|\}$, so the left side is the sum of all the monomials $y_{U'_1}^{a_1}\dots y_{U'_r}^{a_r}$ for every choice of chain $U'_1\subset\dots\subset U'_r$ that satisfies $|U'_i| = |U_i|,\; i = 1,\dots,r$. Meanwhile, each $\theta_{|U_i|}$ is the sum of $y_{U'_i}$ for every $U'_i$ with $|U'_i| = |U_i|$, and the resulting terms in the right side product that are nonzero are precisely those that stack up, i.e. are supported on chains. So the left and right sides coincide term-for-term.
\end{proof}

\begin{figure}
\begin{center}
\begin{tikzpicture}

\draw (0,0) rectangle (1.9,0.4);
\node at (1,0.2) {$y_{12}$};
\draw (0,0.5) rectangle (0.9,0.9);
\node at (0.5,0.7) {$y_1$};

\node at (1,-0.5) {A term of $\theta_1\theta_2$ that stacks up.};

\draw (7,0) rectangle (8.9,0.4);
\node  at (8,0.2) {$y_{12}$};
\draw (9,0.5) rectangle (9.9,0.9);
\node at (9.5,0.7) {$y_3$};

\node at (8,-0.5) {A term of $\theta_1\theta_2$ that does not stack up.};

\end{tikzpicture}
\end{center}
\caption{Illustration of the proof of \ref{prop:FTSPforS}. Cross-products in $\prod \theta_i$ that do not stack up are zero.}\label{fig:illusproof}
\end{figure}

\begin{example}
Let $n=3$. Consider the monomial $m=y_1y_{12}\in S$. Its orbit monomial is
\[
S_n m = y_1y_{12} + y_2y_{12} + y_1y_{13} + y_3y_{13} + y_2y_{23} + y_3y_{23}.
\]
The proof of \ref{prop:FTSPforS} argues that this is equal to
\[
\theta_1\theta_2 = (y_1+y_2+y_3)(y_{12} + y_{13} + y_{23}).
\]
This is so because the three cross-terms in this product that do not appear in $S_nm$ are precisely the three that do not stack up because they are not supported on chains. See figure \ref{fig:illusproof}.
\end{example}

Now the actual FTSP falls out as a corollary. We restate it for reference:

\begin{thm}[FTSP]\label{thm:ftsp2}
$\sigma_1,\dots,\sigma_n$ are algebraically independent over $A$, and generate $R^{S_n}$ as an $A$-algebra.
\end{thm}

\begin{proof}
Algebraic independence follows much as it did for the $\theta_i$ in \ref{prop:FTSPforS}. By proposition \ref{prop:approximatehomomorphism} and remark \ref{rmk:approxhomformultiple}, a monomial in the $\sigma_i$'s coincides with the $\garsia$-image of the corresponding monomial in the $\theta_i$'s in its degree-lexicographically greatest component. Since, as noted in the proof of \ref{prop:FTSPforS}, distinct monomials in the $\theta_i$ have distinct fine grades, they are in distinct $S_{\overline{\lambda}}$'s, and thus their $\garsia$-images are in distinct $R_\lambda$'s. Thus distinct monomials in the $\sigma_i$'s have distinct degree-lexicographically greatest components, and complete cancellation is impossible.

That $\sigma_1,\dots,\sigma_n$ generate $R^{S_n}$ follows from the corresponding part of \ref{prop:FTSPforS} by induction on $\lambda\in\scrP$, in parallel with the spanningness part of theorem \ref{thm:garsiabasis}. An element $f$ of $R_\lambda^{S_n}$ has a corresponding element $\garsia^{-1}(f) \in S_{\overline{\lambda}}^{S_n}$ which is equal to a polynomial $F(\theta_1,\dots,\theta_n)\in \Z[\theta_1,\dots,\theta_n]$ by \ref{prop:FTSPforS}. Then 
\[
f - F(\sigma_1,\dots,\sigma_n)
\]
lies in $R_{<\lambda}$ by proposition \ref{prop:approximatehomomorphism}, remark \ref{rmk:approxhomformultiple}, and the linearity of $\garsia$. Since $R_{<\lambda}$ lies in $\Z[\sigma_1,\dots,\sigma_n]$ by the induction hypothesis, we can conclude that $f\in\Z[\sigma_1,\dots,\sigma_n]$ as well.
\end{proof}

\begin{remark}
As stated previously, this proof is nothing but a dressed-up form of Gauss' proof in \cite{gauss}. The language of the Stanley-Reisner ring $S$ and the Garsia map focuses attention on the way that the proof represents an $f\in R^{S_n}$ one lexicographic layer at a time. A superficial difference is that Gauss orders the monomials themselves lexicographically, whereas we order their shapes; but the Gauss proof actually operates on the terms of $f$ in the order of their shapes (see \cite{FTSP}), so really there is no difference.

Per remark \ref{rmk:dispersion}, the order on shapes $\lambda\in\scrP$ can be seen as a measure of exponent spread. Thus this proof, and equivalently the Gauss proof, represent an $f\in R^{S_n}$ by aiming at the layer of terms with most spread-out exponents first. %This point of view is discussed in detail in \cite{FTSP}.
\end{remark}

We promised to extract h.s.o.p.'s for any permutation invariants of $R$ and $S$ from the above. We saw in \ref{thm:equivofquestions} that $\sigma_1,\dots,\sigma_n$ are an h.s.o.p. for $R$. It will not surprise the reader to learn that $\theta_1,\dots,\theta_n$ is an h.s.o.p. for $S$ by the same argument:

\begin{prop}\label{prop:hsops}
Take $A = \Z$. Then for any permutation group $G\subset S_n$, $\sigma_1,\dots,\sigma_n$, respectively $\theta_1,\dots,\theta_n$, is an h.s.o.p. (in the sense of definition \ref{def:hsopoverZ}) for $R^G$, respectively $S^G$.
\end{prop}

\begin{proof}
The rings $S^{S_n}=\Z[\theta_1,\dots,\theta_n]$ and $R^{S_n} = \Z[\sigma_1,\dots,\sigma_n]$ are polynomial rings by \ref{thm:ftsp2} and \ref{prop:FTSPforS}, thus they are both Krull dimension $n+1$. The inclusions $R^{S_n}\subset R^G$ and $S^{S_n}\subset S^G$ are both finite, by \ref{prop:integraloverSn}. Thus $R^G, S^G$ also have Krull dimension $n+1$. Since the two sequences of elements are self-evidently homogeneous of positive degree,  this verifies all the conditions of definition \ref{def:hsopoverZ} in each case.
\end{proof}

%\begin{remark}
%This last result, for both rings, can be extracted from a general theorem about {\em algebras with straightening law}, about which we will have more to say in section \ref{sec:notCM}. An algebra with straightening law (ASL) is a generalization of the Stanley-Reisner ring of a poset. It also has generators indexed by a poset, with monomials supported on chains, and the products of incomparable elements are constrained in a combinatorial way. The rank-row sums in an ASL always form (the analogue of) a homogeneous system of parameters (\cite{eisenbud2}, Proposition 3.7).
%\end{remark}

\begin{remark}\label{rmk:SGistorsionfree}
The ring $S$ has many zerodivisors. Nonetheless, if $A$ is an integral domain, then so is $S^{S_n}$, since it is a polynomial algebra over $A$. In this circumstance, $S^G$, although it may not be a domain, is nonetheless torsion-free as a module over this integral domain. We see this as follows: suppose $s\in S^{S_n}$ and $f\in S^G$, and $sf = 0$, with $f\neq 0$. We may suppose $s,f$ are finely homogeneous. Then $s$ is actually an $S_n$-orbit monomial (possibly times a scalar $\gamma\in A$), and therefore, by the argument in \ref{prop:FTSPforS}, is a single term in the $\theta_i$'s. I.e.
\[
s = \gamma\prod_i \theta_i^{a_i}
\]
with $\gamma \in A$. By inspection, this product is $\gamma$ times the sum of every monomial of the form $\prod_i y_{U_{ij}}^{a_i}$, where $|U_{ij}| = i$ and for each fixed $j$, the $U_{ij}$ form a chain.

Each of $f$'s terms is supported on some chain in $B_n\setminus\{\emptyset\}$. Because $f$ is finely homogeneous, the supports of two distinct terms $m,m'$ of $f$ cannot be contained in any of the same maximal chains, since they contain sets of the same size but are not the same.

Any individual term $m$ of $f$ has nonzero product with $s$ because $sm$ will be the sum of $\gamma m\prod_i y_{U_{ij}}^{a_i}$ for every $j$ such that the chain of $U_{ij}$'s is contained in a maximal chain also supporting $m$. And cancellation between $sm, sm'$ for two different terms $m,m'$ of $f$ is not possible because any terms of $sm$, resp. $sm'$, will be supported on maximal chains containing the support of $m$, resp. $m'$, and we noted in the last paragraph that these sets of maximal chains are disjoint.
\end{remark}

\subsection{Invariant subrings of $S$ and quotient complexes}\label{sec:reinerthm}

% also consider adding a bunch of content to do with product of arbitrary el. by the h.s.o.p. el's gives you everything of a certain rank set lying over it, so we can refer to this later when we use it.

Since theorem \ref{thm:garsiabasis} relates a Stanley-Reisner ring to an invariant ring, one begins to sense how topological results may be used to speak to invariant theoretic goals. However, there is a missing link, as things stand, between the results of section \ref{sec:reisnermunkres} and this theorem. Section \ref{sec:reisnermunkres} stated theorems about Stanley-Reisner rings themselves, while theorem \ref{thm:garsiabasis} relates a polynomial invariant ring to an invariant ring $S^G$ inside a Stanley-Reisner ring $S$. This link is supplied by an elegant theorem of Victor Reiner's, which shows that $S^G$ is itself the Stanley-Reisner ring of a boolean complex. Before giving Reiner's theorem, we recall the relevant notions and prove some useful facts.

\begin{definition}[Quotient of a poset by a group action]
Let $P$ be a finite poset with an action of $G$. The \textbf{quotient poset} $P/G$ is the poset whose elements are $G$-orbits $p^G$ of elements $p$ of $P$, and whose order relation is given by $p^G \leq q^G$ if there exists $p^\star\in p^G$ and $q^\star\in q^G$ with $p^\star\leq q^\star$.
\end{definition}

See figure \ref{fig:quotientposet} for an example.

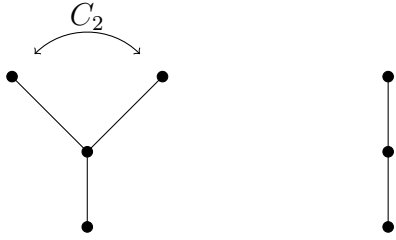
\begin{figure}
\begin{center}
\begin{tikzpicture}

\draw (0,-1) -- (0,0) -- (-1,1);
\draw (0,0) -- (1,1);
\draw [fill=black] (0,-1) circle [radius=0.07];
\draw [fill=black] (0,0) circle [radius=0.07];
\draw [fill=black] (-1,1) circle [radius=0.07];
\draw [fill=black] (1,1) circle [radius=0.07];

\draw [<->] (-0.7,1.3) to [out=45,in=135] (0.7,1.3);
\node at (0,1.8) {$C_2$};

\draw (4,-1) -- (4,0) -- (4,1);
\draw [fill=black] (4,-1) circle [radius=0.07];
\draw [fill=black] (4,0) circle [radius=0.07];
\draw [fill=black] (4,1) circle [radius=0.07];

\end{tikzpicture}
\end{center}
\caption{A poset with an action of $C_2$, and its quotient poset.}\label{fig:quotientposet}
\end{figure}

\begin{remark}
There is a canonical poset homomorphism 
\begin{align*}
P&\rightarrow P/G\\
p&\mapsto p^G. 
\end{align*}
Note that order automorphisms send chains to chains of the same length. Thus if $P$ is ranked, the action of $G$ is automatically rank-preserving, so the quotient poset $P/G$ is ranked.
\end{remark}

\begin{lemma}\label{lem:CWquotientposet}
If $\Delta$ is a finite CW complex with face poset $P$, and $G$ is a finite group acting on $\Delta$ by cellular automorphisms, then there is a natural induced action of $G$ on $P$, and the quotient CW complex complex $\Delta / G$ has face poset $P/G$.
\end{lemma}

\begin{proof}
The action of $G$ on $P$ arises because a cellular action preserves containment between closures of cells. The face poset of $\Delta/G$ is $P/G$ because a cell of $\Delta/G$ is a $G$-orbit $\alpha^G$ of cells of $\Delta$, and a cell $\alpha^G$ of $\Delta/G$ is contained in another $\beta^G$'s closure if and only if some representative $\alpha$ of $\alpha^G$ in $\Delta$ is contained in the closure of some representative $\beta$ of $\beta^G$.
\end{proof}

This lemma holds even if $\Delta$ is not regular.

\begin{remark}
The class of boolean complexes, even of regular complexes, is not closed under the forming of quotients by group actions. For example, the complex depicted in figure \ref{fig:booleancomplex} carries a $G = \Z/2\Z$-action corresponding to rotating the circle through a half-turn, thus transposing A with B and C with D. The quotient poset (including the minimal element) is just the total order $\emptyset \leq A^G \leq C^G$ since $A^G = B^G$ and $C^G = D^G$. This is not the face poset of a boolean complex since the lower interval from $C^G$ is not a boolean algebra. Geometrically, the rotation identifies the two $0$-cells and the two $1$-cells; the resulting complex is not even regular, since the boundary of the $1$-cell is no longer a pair of points. However, we can put conditions in place, satisfied by the posets of concern to us, that guarantee that the quotient remains the face poset of a boolean complex:
\end{remark}

\begin{definition}\label{def:balanced}
A simplicial or boolean complex of dimension $d$ is \textbf{balanced} if its vertices are partitioned into $d+1$ classes such that each facet contains exactly one vertex in each class. The classes are called \textbf{colors},  \textbf{labels}, or \textbf{ranks}. The face poset of a balanced simplicial or boolean complex is also said to be balanced.
\end{definition}

\begin{example}
Figure \ref{fig:balanced} shows a balanced $2$-complex consisting of two triangles glued at two vertices. The balancing is given by 3-coloring the vertices so that each triangle has one vertex of each color.
\end{example}

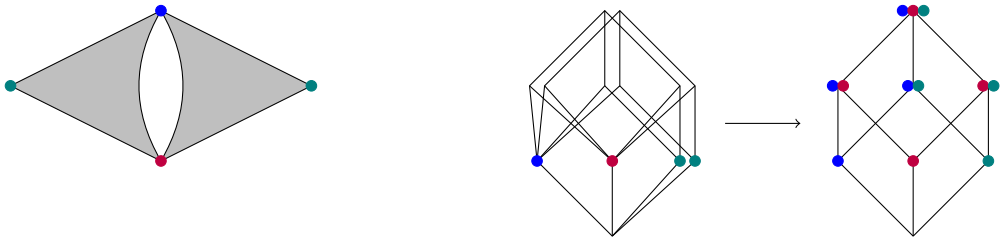
\begin{figure}
\begin{center}
\begin{tikzpicture}

\draw [fill=lightgray] (0,-1) to [out=120,in=240] (0,1) -- (-2,0) -- (0,-1);
\draw [fill=lightgray] (0,-1) to [out=60,in=300] (0,1) -- (2,0) -- (0,-1);
\draw [fill=purple, purple] (0,-1) circle [radius=0.07];
\draw [fill=blue,blue] (0,1) circle [radius=0.07];
\draw [fill=teal,teal] (-2,0) circle [radius=0.07];
\draw [fill=teal,teal] (2,0) circle [radius=0.07];

\draw (6,-2) -- (5,-1) -- (4.9,0) -- (5.9,1) -- (6.9,0) -- (6.9,-1) -- (6,-2);
\draw (6,-2) -- (6,-1) -- (4.9,0);
\draw (6,-1) -- (6.9,0);
\draw (5,-1) -- (5.9,0) -- (5.9,1);
\draw (5.9,0) -- (6.9,-1);

\draw (6,-2) -- (7.1,-1) -- (7.1,0) -- (6.1,1) -- (5.1,0) -- (5,-1);
\draw (5.1,0) -- (6,-1) -- (7.1,0);
\draw (5,-1) -- (6.1,0) -- (7.1,-1);
\draw (6.1,0) -- (6.1,1);

\draw [fill=blue,blue] (5,-1) circle [radius = 0.07];
\draw [fill=purple,purple] (6,-1) circle [radius = 0.07];
\draw [fill=teal,teal] (6.9,-1) circle [radius = 0.07];
\draw [fill=teal,teal] (7.1,-1) circle [radius = 0.07];

\draw (10,-2) -- (9,-1) -- (9,0) -- (10,1) -- (11,0) -- (11,-1) -- (10,-2);
\draw (10,-2) -- (10,-1) -- (9,0);
\draw (10,1) -- (10,0) -- (11,-1);
\draw (9,-1) -- (10,0);
\draw (10,-1) -- (11,0);

\draw [fill=blue,blue] (9,-1) circle [radius = 0.07];
\draw [fill=purple,purple] (10,-1) circle [radius = 0.07];
\draw [fill=teal,teal] (11,-1) circle [radius = 0.07];

\draw [fill=blue,blue] (8.93,0) circle [radius=0.07];
\draw [fill=purple,purple] (9.07,0) circle [radius=0.07];

\draw [fill=blue,blue] (9.93,0) circle [radius=0.07];
\draw [fill=teal,teal] (10.07,0) circle [radius=0.07];

\draw [fill=purple,purple] (10.93,0) circle [radius=0.07];
\draw [fill=teal,teal] (11.07,0) circle [radius=0.07];

\draw [fill=blue,blue] (9.86,1) circle [radius=0.07];
\draw [fill=purple,purple] (10,1) circle [radius=0.07];
\draw [fill=teal,teal] (10.14,1) circle [radius=0.07];

\draw [->] (7.5,-0.5) -- (8.5,-0.5);

\end{tikzpicture}
\end{center}
\caption{Left: a balanced boolean complex. Right: its face poset (with minimal element appended), showing the map $r$ to the boolean algebra on the labels.}\label{fig:balanced}
\end{figure}

\begin{remark}\label{rmk:labelsubsetbij}
The word {\em balanced} in this connection was introduced by Richard Stanley in \cite{stanley3} (in the simplicial context). His definition was slightly more general.  We are using the word to mean what Stanley there called {\em completely balanced}. This has become standard.

Note that a balanced simplicial or boolean complex $\Delta$ is pure, since every facet must be dimension $d$ as it spans $d+1$ vertices. A labeling that realizes $\Delta$ as balanced is equivalent to a $(d+1)$-vertex-coloring of the $1$-skeleton of $\Delta$, since the $1$-skeleton of any facet is a complete graph on its vertices.

Also note that in a balanced boolean complex, the subfaces of any fixed face $\alpha$ are in bijection with subsets of the labels of $\alpha$. There is a corresponding statement about the face poset.
\end{remark}

\begin{notation}\label{not:labelmap}
As a consequence of remark \ref{rmk:labelsubsetbij}, there is a natural poset homomorphism
\[
r :\widehat P\rightarrow B_{d+1}
\]
of the face poset with minimal element $\widehat P$ of a boolean complex $\Delta$ onto the boolean algebra $B_{d+1}$ on the set of labels, given by mapping a face to its label set. The face poset in figure \ref{fig:balanced} is drawn to emphasize the map to $B_3$.
\end{notation}

\begin{lemma}
Let $\Delta$ be a balanced boolean complex. Let $G$ be a finite group acting on $\Delta$ by cellular automorphisms that preserve the labels of vertices. Then the cell-complex quotient $\Delta / G$ is a balanced boolean complex as well.
\end{lemma}

\begin{proof}
Let $\widehat{P}$ be the face poset of $\Delta$ with minimal element appended. Then the face poset of $\Delta / G$ is $\widehat{P}/G$ by lemma \ref{lem:CWquotientposet}. Clearly it also has a unique minimal element $\emptyset$. Let $[\emptyset,\alpha^G]$ be an arbitrary lower interval in $\widehat{P}/G$, and pick a representative $\alpha\in \widehat{P}$ of the orbit $\alpha^G$. The lower interval $[\emptyset,\alpha]$ in $\widehat{P}$ is boolean since $\Delta$ is a boolean complex. The restriction of the canonical map 
\[
\pi:\widehat{P}\rightarrow\widehat{P}/G
\]
 to $[\emptyset,\alpha]$ is surjective onto $[\emptyset,\alpha^G]$, because if $\beta^G\leq \alpha^G$ in $\widehat{P}/G$ (i.e. $\beta^G$ is an arbitrary member of $[\emptyset,\alpha^G]$), then some $\beta$ representing $\beta^G$ lies below some $\alpha^\star$ representing $\alpha^G$; but since $\alpha,\alpha^\star$ lie in the same orbit, there is a $g\in G$ with $g(\alpha^\star) = \alpha$, and then $\beta\leq \alpha^\star$ implies $g(\beta) \leq \alpha$, since $g$ is an order automorphism; thus $g(\beta)\in [\emptyset,\alpha]$. Since $g(\beta)$ is in the orbit of $\beta$, i.e. 
 \[
 \pi(g(\beta)) = \beta^G,
 \]
we have found an element of $[\emptyset, \alpha]$ mapping to arbitrary $\beta^G\in [\emptyset,\alpha^G]$.

We have as yet made no use of the balancedness of $\Delta$ or the label-preservingness of $G$, so all of the above is true of any boolean complex with a group action. But these features force the restriction of $\pi$ to $[\emptyset,\alpha]$ to be injective. This is because the elements of the lower interval $[\emptyset,\alpha]$ correspond bijectively with subsets of labels of $\alpha$ (per remark \ref{rmk:labelsubsetbij}); meanwhile, because the $G$-action is label-preserving, elements of the same $G$-orbit have the same labels. So no $G$-orbit can meet $[\emptyset,\alpha]$ more than once, i.e. the restriction of $\pi$ to $[\emptyset,\alpha]$ is injective.

Since we already know 
\[
\pi|_{[\emptyset,\alpha]} :[\emptyset,\alpha]\rightarrow[\emptyset,\alpha^G]
\]
is surjective, we have now shown it is an order isomorphism to $[\emptyset,\alpha^G]$, which must therefore be boolean. Since the latter is an arbitrary lower interval of $\widehat{P}/G$, we conclude that $\Delta/G$ is a boolean complex. $G$'s label-preservingness means that the quotient map $\Delta\rightarrow\Delta/G$ is well-defined on labels, giving a labeling of the vertices of $\Delta/G$. This labeling realizes $\Delta/G$ as balanced, since every facet in $\Delta/G$ is the image of a facet in $\Delta$.
\end{proof}

\begin{definition}
We will say $G$ has a \textbf{balanced action} on $\Delta$ to mean that $\Delta,G$ meet the hypotheses of this lemma.
\end{definition}

\begin{remark}
In this proof we saw that if $G$ has a balanced action on $\Delta$ then lower intervals in the face poset map bijectively to the quotient. The geometric content of this statement is that the canonical map $\Delta \rightarrow \Delta/G$ is a homeomorphism on individual faces.
\end{remark}

Stanley-Reisner rings of balanced boolean complexes (recall definition \ref{def:SRringofboolean}) have a number of good properties generalizing those of the ring $S$ defined in \ref{not:definitionofS}.

\begin{notation} 
Let $\Delta$ be a balanced boolean complex with label set $[n]$ and face poset $\widehat P$, with minimal element appended.
\end{notation}

\begin{prop}\label{prop:balancedfinegrading}
The ring $A[\Delta]$ has an $\N^{n}$-grading given by
\[
y_\alpha \mapsto \sum_{i\in r(\alpha)} e_i,
\]
for each $\alpha \in \widehat P$, where $r(\alpha)$ is the set of labels of the face $\alpha$, as in \ref{not:labelmap}.
\end{prop}

Generalizing our language for $S$, we will refer to this as the \textbf{fine grading} on $A[\Delta]$.

\begin{proof}
The suggested map certainly gives an $\N^n$-grading of the polynomial ring 
\[
R_\Delta = A[\{y_\alpha\}_{\alpha\in\widehat P}]
\]
from the definition of $A[\Delta]$ (recall \ref{def:SRringofboolean}). We need to check that the ideal $I_\Delta$ of that definition is homogeneous with respect to this grading. We check the generators:
\begin{enumerate}
\item $y_\emptyset - 1$ is homogeneous, because $r(\emptyset)$ is $\emptyset$, so the fine grade
\[
\sum_{i\in r(\emptyset)} e_i
\]
of $y_\emptyset$ is a void sum, thus $=0\in\N^n$, which is also the fine grade of $1$.
\item $y_\alpha y_\beta$ (for $\alpha,\beta\in\widehat P$ lacking a common upper bound) is homogeneous because it is a monomial.
\item Consider
\[
y_\alpha y_\beta - y_{\alpha\wedge \beta }\sum y_\gamma
\]
where the sum is over minimal upper bounds for $\alpha,\beta$. The fine grade of $y_\alpha y_\beta$ is
\begin{equation}\label{eq:ab}
\sum_{i\in r(\alpha)} e_i + \sum_{i\in r(\beta)} e_i.
\end{equation}
Since $r$ is a poset homomorphism to the boolean algebra on the label set, we have
\[
r(\alpha\wedge \beta) = r(\alpha)\cap r(\beta),
\]
and similarly for any minimal upper bound $\gamma$ we have
\[
r(\gamma) = r(\alpha)\cup r(\beta).
\]
Thus the fine grade of each term in $y_{\alpha\wedge \beta }\sum y_\gamma$ is
\begin{equation}\label{eq:gsum}
\sum_{i\in r(\alpha)\cap r(\beta)} e_i + \sum_{i\in r(\alpha)\cup r(\beta)} e_i.
\end{equation}
The sums \eqref{eq:ab} and \eqref{eq:gsum} are equal term for term. Thus 
\[
y_\alpha y_\beta - y_{\alpha\wedge \beta }\sum y_\gamma
\]
is homogeneous with respect to the $\N^n$ grading on $R_\Delta$.
\end{enumerate}
Thus $I_\Delta$ is homogeneous, so the grading descends to $R_\Delta / I_\Delta = A[\Delta]$.
\end{proof}

\begin{prop}\label{prop:balancedhsop}
If $A=\Z$ or a field, the elements
\[
\psi_i = \sum_{r(\alpha) = \{i\}} y_\alpha,\; i=1,\dots,n
\]
form a finely homogeneous system of parameters for $A[\Delta]$.
\end{prop}

This h.s.o.p. specializes to the $\theta_i$'s of section \ref{sec:FTSPhsops} when $\Delta$ is the order complex of $B_n\setminus\{\emptyset\}$, so that $A[\Delta] = S$.

\begin{proof}
This is proven in \cite{cca}, Proposition III.4.3, in the case that $\Delta$ is a balanced simplicial complex, but the argument extends to the present case. It has two ingredients:
\begin{enumerate}
\item The Krull dimension of $A[\Delta]$ is $n$ if $A$ is a field or $n+1$ if $A=\Z$.\label{1dimisright}
\item The ring $A[\Delta]$ is module-finite over the subring $A[\psi_1,\dots,\psi_n]$.\label{2modulefinite}
\end{enumerate}
We have \ref{1dimisright} from a general theorem (\cite{eisenbud2}, Corollary 3.6(3)) about {\em algebras with straightening law}, of which $A[\Delta]$ is an example and about which we will have more to say in section \ref{sec:notCM}. We have \ref{2modulefinite} because $A[\Delta]$ is generated over $A$ by a finite set of elements $\{y_\alpha\}|_{\alpha\in\widehat P}$ each of which is integral over $A[\psi_1,\dots,\psi_n]$. This in turn is because any $y_\alpha$ satisfies the equation
\[
y_\alpha^2 - \left(\prod_{i\in r(\alpha)} \psi_i\right)y_\alpha =0.
\]
It is a worthwhile exercise to verify this equation from the definition of $A[\Delta]$. One can also deduce it from the following pair of lemmas.
\end{proof}

\begin{lemma}\label{lem:distinctsamelabels}
No two distinct elements of $\widehat P$ with the same label set have a common upper bound.
\end{lemma}

\begin{proof}
This follows from remark \ref{rmk:labelsubsetbij}. If $\alpha,\beta\in\widehat P$ have a common upper bound $\gamma$, then they are both in the interval $[\emptyset,\gamma]$, whose elements are in bijection with subsets of the label set of $\gamma$. Thus if $r(\alpha) = r(\beta)$, we must have $\alpha = \beta$.
\end{proof}

\begin{lemma}\label{lem:arithinbooleancomplexring}
If $\alpha\in \widehat P$ and $J$ is a set of labels disjoint from $\alpha$'s, i.e.
\[
J \subset [n]\setminus r(\alpha),
\]
then
\[
\left(\prod_{j\in J} \psi_j\right) y_\alpha = \sum_{\substack{\beta\geq \alpha\\r(\beta) = r(\alpha)\cup J}} y_\beta.
\]
\end{lemma}

In words, this is the statement that given $\alpha\in\widehat P$, one can obtain the sum of all elements of $\widehat P$ lying above $\alpha$ and with a fixed label set containing $\alpha$'s, by multiplying $y_\alpha$ by the $\psi_j$'s corresponding to the extra labels needed.

\begin{proof}
This is by induction on the cardinality of $J$. The base case is $\#J=1$. We need to show that if $j\notin r(\alpha)$, then
\[
\psi_j y_\alpha = \sum_{\substack{\beta\geq \alpha \\ r(\beta) = r(\alpha)\cup \{j\}}} y_\beta.
\]
By definition of $\psi_j$ and $A[\Delta]$, we have
\begin{align*}
\psi_j y_\alpha &= \sum_{r(\nu) = \{j\}} y_\gamma y_\alpha \\
 &= \sum_{r(\nu) =\{j\}} y_{\nu \wedge \alpha} \sum_{\gamma} y_\gamma
\end{align*}
where the last sum is taken over least common upper bounds $\gamma$ for $\nu$ and $\alpha$. First, $\nu\wedge\alpha = \emptyset$ because $\nu$ and $\alpha$ do not have any labels in common. Thus $y_{\nu\wedge\alpha} = y_\emptyset = 1$, so the above is
\begin{equation}\label{eq:sumofgammas}
\sum_{r(\nu) =\{j\}} \sum_{\gamma} y_\gamma.
\end{equation}
Second, any minimal common upper bound of $\nu$ and $\alpha$ must have label set 
\[
r(\alpha)\cup r(\nu) = r(\alpha) \cup \{j\}
\]
since $r$ is a poset homomorphism. Furthermore any upper bound of $\alpha$ with this label set appears somewhere in the sum \eqref{eq:sumofgammas}, since any such upper bound is the common upper bound of $\alpha$ and {\em some} $\nu$ with $r(\nu)=j$. Finally, each only appears once, since no two $\nu$'s have any common upper bounds, by lemma \ref{lem:distinctsamelabels}. Thus $\psi_jy_\alpha$ is exactly
\[
\sum_{\substack{\beta\geq \alpha\\ r(\beta) = r(\alpha)\cup\{j\}}} y_\beta,
\]
as was to be shown.

Now suppose $\#J > 1$, and let 
\[
J = J' \cup \{j^\star\}
\]
for some $j^\star\notin J'$. By the induction assumption,
\[
\left(\prod_{j\in J'} \psi_j\right) y_\alpha = \sum_{\substack{\beta' \geq \alpha \\ r(\beta') = r(\alpha)\cup J'}} y_{\beta'}.
\]
Thus
\begin{equation}\label{eq:bigsum}
\left(\prod_{j\in J}\psi_j\right) y_\alpha = \sum_{\substack{\beta' \geq \alpha \\ r(\beta') = r(\alpha)\cup J'}} \psi_{j^\star}y_{\beta'}.
\end{equation}
But by the base case, each term $\psi_{j^\star}y_{\beta'}$ is just the sum of $y_\beta$ for every $\beta\geq \beta'$ with label set
\[
r(\beta) = r(\beta')\cup\{j^\star\} = r(\alpha)\cup J.
\]
Any $\beta\geq \alpha$ with this label set is $\geq$ {\em some} $\beta'$, and none of them is $\geq$ two different $\beta'$'s (again by lemma \ref{lem:distinctsamelabels}), thus the sum \eqref{eq:bigsum} is precisely
\[
\sum_{\substack{\beta\geq \alpha \\ r(\beta) = r(\alpha)\cup J}} y_\beta,
\]
as was to be shown.
\end{proof}

\begin{remark}
An important special case of this theorem is that 
\[
\prod_J \psi_j = \sum_{r(\beta) = J} y_\beta.
\]
This explains the formula
\[
y_\alpha^2 - \left(\prod_{j\in r(\alpha)} \psi_j\right) y_\alpha = 0,
\]
as follows. The parenthetical product is the sum of $y_\beta$ for every $\beta$ of label set equal to $\alpha$. However, by lemma \ref{lem:distinctsamelabels}, the only one of these that has a common upper bound with $\alpha$ is $\alpha$ itself. Thus by the definition of $A[\Delta]$, all the products in the second term vanish except for $y_\alpha^2$.
\end{remark}

We now state Reiner's elegant theorem:

\begin{thm}[\cite{reiner92}, theorem 2.3.1]\label{thm:reinerSRquotient}
Suppose $G$ has a balanced action on $\Delta$. For the induced action of $G$ on the Stanley-Reisner ring $A[\Delta]$, we have
\[
A[\Delta/G]\cong A[\Delta]^G
\]
The isomorphism is given on the generators $y_{\alpha^G},\;\alpha^G\in \widehat{P}/G$ by
\[
\pushQED{\qed}
y_{\alpha^G} \mapsto \sum_{\alpha\in\alpha^G} y_\alpha.\qedhere
\popQED
\]
\end{thm}

\begin{remark}
Reiner formulated this theorem with the coefficient ring being a field, but his proof works word-for-word over any ring. We are interested in this special case:
\end{remark}

\begin{cor}\label{cor:deltaisbalanced}
Let $\Delta$ be the order complex of the poset $\Bn\setminus\{\emptyset\}$, so that $A[\Delta]$ is the ring $S$ defined in \ref{not:definitionofS}. Let $G\subset S_n$ be a permutation group, acting on $\Delta$ through its action on $\gls{[n]}$. Then $S^G \cong A[\Delta/G]$.
\end{cor}

Thus, the rings $S^G$ considered in section \ref{sec:garsiamap} are actually Stanley-Reisner rings of balanced boolean complexes.

\begin{proof}
This is immediate from theorem \ref{thm:reinerSRquotient} once we verify that the action of $G$ on $\Delta$ is balanced. But the vertices of $\Delta$ are precisely the elements of the ranked poset $\Bn\setminus\{\emptyset\}$, any maximal chain of $\Bn\setminus\{\emptyset\}$ hits every rank, and $G$'s action on $\Bn\setminus\{\emptyset\}$ necessarily preserves ranks; thus we can take the ranks as the labels of the vertices of $\Delta$, and the induced action on $\Delta$ will thus be balanced.
\end{proof}

\section{Quotients of spheres and balls}\label{sec:quotientsofspheres}

We have the pieces in place, which we will assemble in the next section, to see that the Cohen-Macaulayness of the integer polynomial invariant ring $R^G$ of interest to us is related to the topology of a boolean complex $\Delta/G$, where $\Delta$ is a triangulation of a ball. In order to apply this, we need to know something about the topology of $\Delta/G$. Specifically, for which groups $G$ is it Cohen-Macaulay?

We are able to answer this question completely, thanks in large part to remarkable recent work of Christian Lange in orbifold theory, completing a program initiated by Marina Mikhailova in the 70's and 80's (\cite{maerchik}, \cite{mikhailova78}, \cite{mikhailova82}, \cite{mikhailova85}).

\begin{definition}\label{def:reflectionrotation}
A linear transformation of $\R^n$ of finite order that fixes a hyperplane pointwise is called a \textbf{reflection}. If it fixes a subspace of codimension two pointwise, it is a \textbf{rotation}.
\end{definition}

\begin{remark}
This use of the word {\em reflection} with this sense is very standard; however, {\em rotation} is often used to refer to any element of $SO_n(\R)$. Our usage follows \cite{lange}, \cite{lange2}, \cite{langemikhailova}.

The familiar names reflect familiar geometric pictures:

For any finite-order element $g\in GL_n(\R)$ there exists an invariant inner product, so $g$ may be regarded as an element of $O_n(\R)$ after a change of basis. It restricts, on the orthogonal complement of its fixed point subspace, to an orthogonal transformation free of nontrivial fixed points. Furthermore it is determined by the data of its fixed point set, the invariant inner product, and this fixed-point-free restriction. 

For a reflection, the fixed-point subspace is codimension $1$, so its complement is a line. The only orthogonal transformation of a line that is free of nontrivial fixed points is multiplication by $-1$. Thus it fixes a hyperplane and reverses a line.

By the same token, a rotation is determined by an orthogonal transformation of a plane that lacks nontrivial fixed points; this is necessarily an element of $SO_2(\R)$, i.e. a rotation in the ordinary sense. Thus a rotation fixes a codimension 2 space and rotates its orthogonal plane.
\end{remark}

Lange, building on work of Mikha\^{i}lova, has proven the following theorem:

\begin{thm}[\cite{lange2}, main result]\label{thm:lange}
Let $G\subset O_n(\R)$ be a finite group of orthogonal transformations. Endow $\R^n$ with its standard PL manifold structure. Then the quotient PL space $\R^n/G$ is a PL manifold (with or without boundary) if and only if $G$ is generated by rotations and reflections. If it is, then $\R^n/G$ is PL homeomorphic to either $\R^{n-1}\times \R^{\geq 0}$ or $\R^n$, depending on whether or not $G$ contains a reflection.
\end{thm}

\begin{remark}\label{ref:linkinPLmanifold}
A basic fact about PL manifolds is that in any PL triangulation, the link of every face is PL-homeomorphic to a sphere, or, if the face is contained in the boundary, a ball.
\end{remark}

\begin{definition}
A finite group $G\subset GL_n(\R)$ generated by reflections, respectively rotations, respectively rotations and reflections, is a \textbf{reflection group}, respectively \textbf{rotation group}, respectively \textbf{rotation-reflection group}.
\end{definition}

Reflection groups are very important and well-studied objects. They are exactly the finite {\em Coxeter groups}. See \cite{bjornerbrenti}, \cite{davis}, and \cite{humphreys2} for three very different angles on their general theory. Their complex analogues are the {\em pseudoreflection groups}, the subject of the Chevalley-Shephard Todd theorem. See \cite{lehrertaylor} for an introduction. Pseudoreflection groups become rotation groups by regarding the underlying $\C$-vector space as an $\R$-vector space of twice the dimension and forgetting the complex structure. Lange and Mikha\^{i}lova's work may be seen as a first move toward a theory of rotation-reflection groups that extends these well-developed theories.

\begin{remark}
Most of the work of theorem \ref{thm:lange} lies in the ``if" direction. The proof relies on a complete classification of rotation-reflection groups, which Lange published jointly with Mikha\^{i}lova (\cite{langemikhailova}). This classification is used to verify that the quotient of a rotation-reflection group is always $\R^n$ or $\R^{n-1}\times\R^{\geq 0}$ via a delicate case-based induction on the group order. The argument uses a diverse and ad-hoc set of tools ranging from the Chevalley-Shephard-Todd theorem, to explicit construction of fundamental domains for the action, to the Poincar\'{e} conjecture. (This theorem is the principal aim of \cite{lange2}. For a more comprehensive and self-contained account, see Lange's thesis \cite{lange}, which collects everything in one place.)
\end{remark}

Using theorem \ref{thm:lange}, we are able to determine exactly when $\Delta / G$ is Cohen-Macaulay:

\begin{notation}
Let $S_n$ act on an $(n-1)$-simplex linearly by permutations of the vertices. Let $\Delta$ be any triangulation of the simplex such that the action by $S_n$ is simplicial. Let $G\subset S_n$ be any permutation group of degree $n$.
\end{notation}

\begin{thm}\label{thm:quotientCMforpermgroups}
With this notation, $\Delta / G$ is Cohen-Macaulay in the sense of definition \ref{def:CMcomplex}, over $\Z$, or equivalently every field $k$, if and only if $G$ is generated by transpositions, double transpositions, and 3-cycles.
\end{thm}

Before giving the proof, we recall some fundamentals about the geometry and combinatorics of $S_n$ acting on an $(n-1)$-simplex in the way indicated. We may take as a model the \textbf{standard simplex}, that is, the convex hull of the points $e_1,\dots,e_n\in \R^n$, i.e. the subset of $\R^n$ satisfying 
\[
x_i\geq 0,\;\forall i
\]
and
\[
\sum x_i = 1.
\]
Let $S_n$ act on it by permuting the axes of $\R^n$. Then a transposition $(ij)\in S_n$ fixes the hyperplane $x_i = x_j$; thus it is the reflection in this hyperplane with respect to standard dot product, since $S_n\subset O_n(\R)$. Because $S_n$ is generated by its transpositions, it is a reflection group; in fact it is a central example of a reflection group. 

The hyperplanes $x_i=x_j$ induce a triangulation $\Delta$ of the simplex. The boundary complex of this triangulation is the {\em Coxeter complex} of $S_n$ (see \cite{brown}, Chapter 1; \cite{reiner92}, Chapter 2; \cite{humphreys2}, \S 1.15). In fact, $\Delta$ is nothing but the barycentric subdivision of the original simplex. See figure \ref{fig:coxetercomplex}.

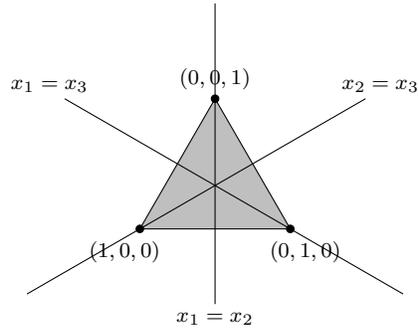
\begin{figure}
\begin{center}
\begin{tikzpicture}

\draw [fill=lightgray] (0,0) -- (2,0) -- (1,1.732) -- (0,0);
\node at (-0.2,-0.3) {\scriptsize $(1,0,0)$};
\node at (2.2,-0.3) {\scriptsize $(0,1,0)$};
\node at (1,2) {\scriptsize $(0,0,1)$};

\draw [fill = black] (0,0) circle [radius = 0.05];
\draw [fill = black] (2,0) circle [radius = 0.05];
\draw [fill = black] (1,1.732) circle [radius = 0.05];

\draw (-1.5,-0.866) -- (3,1.732);
\draw (-1,1.732) -- (3.5,-0.866);
\draw (1,-1) -- (1,3);

\node at (1,-1.2) {\scriptsize $x_1=x_2$};
\node at (3.2, 1.9) {\scriptsize $x_2=x_3$};
\node at (-1.2, 1.9) {\scriptsize $x_1=x_3$};

\end{tikzpicture}
\end{center}
\caption{The Coxeter complex of $S_3$: the standard simplex in the plane $x_1+x_2+x_3=1$, and its intersections with the planes $x_i=x_j$.}\label{fig:coxetercomplex}
\end{figure}

Since the face poset (without minimal element) of a simplex is $\Bn\setminus\{\emptyset\}$, this means that $\Delta$ is precisely the order complex of $\Bn\setminus\{\emptyset\}$, in other words, it is exactly the simplicial complex whose Stanley-Reisner ring is $S$.

\begin{notation}
The hyperplanes themselves form the \textbf{braid arrangement} $\mathcal{A}_{n-1}$ (see \cite{bjorner}, where it is called $\mathcal{A}_{n,2}$, and \cite{wachs}). The hyperplanes intersect along subspaces; the set of all of these subspaces, together with the whole space $\R^n$, ordered by inclusion, is the \textbf{intersection lattice} $L(\mathcal{A}_{n-1})$.
\end{notation}

All points of $\R^n$ that are fixed nontrivially by some element of $S_n$ are contained in the union of the hyperplanes. 

\begin{lemma}
The lattice $L(\mathcal{A}_{n-1})$ is order-isomorphic to the \textbf{partition lattice} $\Pi_n$ of partitions of the set $\gls{[n]}$, ordered by refinement.
\end{lemma}

\begin{proof}
The isomorphism is as follows. The partition
\[
\gls{[n]} = \overset{r}{\underset{1}{\gls{coprod}}} \lambda_i
\]
into disjoint sets $\lambda_i$ (call it $\pi$) corresponds to the subspace cut out by the conditions that for each $i$, all $x_j$ with $j\in \lambda_i$ are equal. 
\end{proof}

Observe that a partition with $m$ blocks corresponds to a subspace of dimension $m$. In particular, if a subspace has codimension $\leq 2$, the corresponding partition has $\geq n-2$ blocks.

A permutation $\sigma\in S_n$ determines a partition $\pi$ of $\gls{[n]}$ into orbits; its fixed-point set is the corresponding subspace in $L(\mathcal{A}_{n-1})$. Thus the reflections are exactly the transpositions, since a partition with $n-1$ blocks must have a single block of size two and the rest singletons, and the rotations are the double-transpositions and three-cycles, as the corresponding partitions are the only possibilities with $n-2$ blocks.%\com{figures?}

The stabilizer of any point $x\in\R^n$ in any subgroup $G\subset S_n$ can be described in terms of $L(\mathcal{A}_{n-1})$. One finds the minimal element of $L(\mathcal{A}_{n-1})$ containing $x$. (It describes all the collisions between $x$'s coordinates.) This defines a partition $\pi\in \Pi_n$ according to the above scheme. Then the stabilizer of $x$ in $G$ consists of all elements of $G$ whose associated partition of $\gls{[n]}$ into orbits refines $\pi$. Note that this is also the stabilizer of $\pi$ for the natural action of $G$ on $\Pi_n$.

%\com{example!}

All this established, we need only one more technical lemma on permutation groups before we are ready to begin proving theorem \ref{thm:quotientCMforpermgroups}.

\begin{notation}
If $G$ acts on a set $X$ and $x\in X$, the stabilizer of $x$ in $G$ is written $G_x$.
\end{notation}

\begin{lemma}\label{lem:primestabilizer}
Let $G\subset S_n$, and let $N\triangleleft G$ be a normal subgroup. Let $\pi$ be maximal in $\Pi_n$ among partitions associated (via the map $\sigma\mapsto \pi$ just described) with elements of $G\setminus N$. Then the image in $G/N$ of the stabilizer of $\pi$ in $G$ is cyclic of prime order $p$, and any element of $G$ whose orbits are given by $\pi$ has order $p^k$ and its image in $G/N$ generates this stabilizer.
\end{lemma}

\begin{proof}
Let $g$ be an element of $G\setminus N$ whose orbits are given by $\pi$, and let $h$ be any other nontrivial element of $G_\pi$. Pick any element $a\in \gls{[n]}$ acted on nontrivially by $h$. Since $h$ preserves $\pi$ and $g$ acts transitively on each block of $\pi$, there is a $k\in\Z$ such that $g^k(a) = h(a)$. Then $h^{-1}g^k(a) = a$, so that $h^{-1}g^k$ both preserves $\pi$ and has a fixed point $a$ that $g$ does not have. Thus its orbits properly refine $\pi$, and maximality of $\pi$ among partitions associated to elements of $G\setminus N$ implies that $h^{-1}g^k\in N$. Thus $hN = g^kN$. This shows that $g$ generates the image of $G_\pi$ in $G/N$; thus this image is cyclic. Meanwhile, for any prime $p$ dividing the order of $g$, $g^p$'s orbits also properly refine $g$'s, so $g^p$ is in $N$ too; thus $g$'s image in $G/N$ has order dividing $p$. Since $g\notin N$ by construction, there must be one such $p$ but there can only be one. We conclude $g$ has $p$-power order in $G$ and its image in $G/N$ has order $p$.
\end{proof}

\begin{proof}[Proof of theorem \ref{thm:quotientCMforpermgroups}]
Because the Cohen-Macaulay condition is a topological property, it is insensitive to the CW structure on $|\Delta|/G$. Thus we may choose a triangulation of the total space $|\Delta|$ to suit our needs, and we may embed this triangulation into $\R^n$ in any way that respects the action of $G$.

$\Rightarrow$ This is an immediate consequence of Lange's result. If $G$ is generated by transpositions, double transpositions, and three-cycles, then it is a rotation-reflection group; in this case its restriction to the hyperplane 
\[
H = \left\{\sum x_i = 0\right\},
\]
which is fixed under the action of $S_n$, is a rotation-reflection group as well. Take $\Delta$ to be the orthogonal projection into $H$ of the barycentric subdivision of the standard simplex. Extend $\Delta$ to a triangulation of $H$ in any way, subject to the constraint that $G$ acts simplicially. Then $\Delta / G$ is the cone over the link of the origin in the induced triangulation of $H/G$. Theorem \ref{thm:lange} implies $H/G$ is a PL manifold, thus the link of the origin is a sphere or a ball, and the cone over it is a ball. It is thus Cohen-Macaulay over any field.

$\Leftarrow$ Now we choose as $\Delta$ the second barycentric subdivision of the standard simplex in $\R^n$, orthogonally projected into $H$. As we have seen, the first barycentric subdivision yields the order complex of $\Bn\setminus\{\emptyset\}$, for which the $G$-action is balanced, thus the quotient by $G$ is a balanced boolean complex. Thus the second barycentric subdivision yields a $G$-quotient $\Delta/G$ that is the barycentric subdivsion of a balanced boolean complex. This is the order complex of the face poset of this complex, so it is even a simplicial complex. Similarly, $\Delta / G_1$ is a simplicial complex for any subgroup $G_1\subset G$. (Thus we will be able to apply definition \ref{def:link} and theorem \ref{thm:reisnermunkres} without any adjustments for boolean complexes.)

Note also that all of the fixed-point sets in $\Delta$ of elements of $G$ (in fact of $S_n$) are subcomplexes of $\Delta$, since the reflecting hyperplanes of $S_n$ induce the first barycentric subdivision and $\Delta$ is a refinement of this triangulation. Furthermore, any face of $\Delta$ fixed setwise by an element of $G$ (in fact of $S_n$) is even fixed pointwise.

Let $\Grr$ be the normal subgroup of $G$ generated by transpositions, double transpositions, and three-cycles. By assumption, $\Grr\subsetneq G$. Let $\sigma \in G\setminus \Grr$ be an element whose fixed-point set in $H$ is maximal among such elements; equivalently, such that the partition $\pi$ of $\gls{[n]}$ into orbits of $\sigma$ is maximal. Since all rotations and reflections are contained in $\Grr$, $\sigma$'s fixed-point set must have codimension $\geq 3$. Since $\dim \Delta = n-1$, this means its dimension is $\leq n-4$. Let $F\in \Delta$ be a face of full dimension contained in the subcomplex of $\Delta$ fixed by $\sigma$, and let $\overline{F}$ be its image in the quotient $\Delta / \Grr$.

By our construction, $\Delta / \Grr$ is a simplicial complex, of which $\overline{F}$ is a face. By the $\Rightarrow$ direction, $\Delta/\Grr$ is a PL triangulation of a ball. In particular, the link $\lk_{\Delta/\Grr}(\overline{F})$ of $\overline{F}$ in $\Delta / \Grr$ is topologically a sphere or a ball, of dimension 
$$
\dim \Delta/\Grr - \dim \overline{F} - 1 = \dim \Delta - \dim F - 1 = n - 2 - \dim F,
$$ 
with the first equality because the quotient map $\Delta \rightarrow \Delta / \Grr$ is dimension-preserving, and the second because $\dim \Delta = n-1$. Since $\dim F \leq n-4$, we have that 
\[
\dim \lk_{\Delta/\Grr}(\overline{F}) \geq (n-2) - (n-4) = 2.
\]

Now $\Delta / \Grr$ has an action of $G/\Grr$, and 
$$
\Delta / G \cong (\Delta / \Grr) / (G/\Grr).
$$ 
Let $F'$ be the image of $\overline{F}$ under the quotient map $\Delta / \Grr \rightarrow \Delta / G$, or equivalently the image of $F$ under $\Delta \rightarrow \Delta / G$. Then $\lk_{\Delta / G} (F')$ is the image of $\lk_{\Delta/\Grr}(\overline{F})$ under the action on the latter by the stabilizer $(G/\Grr)_{\overline{F}}$ of $\overline{F}$ in $G/\Grr$. By an elementary result in group theory (see \ref{lem:quotientstabilizer} in the algebraic lemmas appendix), this stabilizer is the image of the stabilizer $G_F$ under $G\mapsto G/\Grr$.

Since $F$ is full dimension in the fixed-point set of $\sigma$, its stabilizer is exactly the stabilizer of this fixed-point set, which in turn is the stabilizer in $G$ of the partition $\pi$ describing $\sigma$'s orbits. Thus lemma \ref{lem:primestabilizer} applies, and 
\[
(G/\Grr)_{\overline{F}} = \langle \overline{\sigma}\rangle \cong \Z/p\Z
\]
for some prime $p$, where $\overline{\sigma}$ is the image of $\sigma$ in $G/\Grr$. (As an addendum, we learn $\sigma$ has prime-power order.)

Furthermore, this stabilizer acts freely on $\lk_{\Delta/\Grr}(\overline{F})$. For if any nontrivial element of $(G/\Grr)_{\overline F}$ fixed a nontrivial face of $\lk_{\Delta/\Grr}(\overline{F})$, by the same basic result of group theory (\ref{lem:quotientstabilizer}), there would be a nontrivial element in $G\setminus \Grr$ stabilizing $F$ and also fixing a face $\alpha$ of $\lk_{\Delta}(F)$, and then $\alpha$ and $F$ would generate a nontrivially-fixed face of $\Delta$ strictly containing $F$, contradicting $F$'s maximality.

It follows by the Brouwer fixed-point theorem that $\lk_{\Delta/\Grr}(\overline{F})$ is a sphere, not a ball.\footnote{In fact, it follows that it is odd-dimensional unless $p=2$, as a corollary to the well-known fact that $\Z/2\Z$ is the only group that can act freely on an even-dimensional sphere, which is a consequence of the Lefschetz fixed point theorem.} Thus $\lk_{\Delta/G}(F')$ is the quotient of a sphere of dimension at least 2, which is simply connected, by a free action of $\Z/p\Z$. It follows that 
\[
H_1(\lk_{\Delta/G}(F');\Z) = \Z/p\Z.
\]
In particular, $p$-torsion shows up in the homology of this link in $\Delta/G$ below top dimension. Thus $\Delta/G$ is not Cohen-Macaulay over $\Z$ or $\F_p$.
\end{proof}

%There is, however, a much more elementary way to see this: a permutation of odd prime-power order has a fixed-point set of codimension $\sum_{\text{cycles}}\left(\text{cycle length}-1\right)$, which is even; thus the dimension of the link of any top-dimensional face in this fixed-point set is odd.

\begin{remark}
We expect that the $\Leftarrow$ direction can also be extracted from Lange's work on when $\R^n/G$ is a homology manifold -- see \cite{lange}, Chapter 4.
\end{remark}

\section{When the invariant ring is Cohen-Macaulay}\label{sec:mainresult}

We now assemble the pieces we have built into a proof of the chapter's main objective, theorem \ref{thm:mainresult}, which we restate for convenience:

\begin{thm}
Let $R$ be the integer polynomial ring $\Z[x_1,\dots,x_n]$. Let $G\subset S_n$ be a permutation group of degree $n$. Then if $G$ is generated by its transpositions, double transpositions, and three-cycles, the invariant ring $R^G$ is Cohen-Macaulay.
\end{thm}

\begin{proof}
Let $\Delta$ be the order complex of $\Bn\setminus\{\emptyset\}$. Under the hypothesis on $G$, $\Delta/G$ is a Cohen-Macaulay complex, by theorem \ref{thm:quotientCMforpermgroups}. This means that $\Z[\Delta / G]$ is a Cohen-Macaulay ring, by \ref{thm:duval}. This ring is isomorphic to $\Z[\Delta]^G$, by \ref{thm:reinerSRquotient}. Cohen-Macaulayness of this ring implies it has a free basis as a $\Z[\theta_1,\dots,\theta_n]$-module, by \ref{thm:HironakaoverZ}. This basis $B$ can be taken to be finely homogeneous (see proposition \ref{prop:gradedandfreeisgradedfree} in the algebraic lemmas appendix). But then $\garsia(B)$ is a basis for $R^G$ as a $\Z[\sigma_1,\dots,\sigma_n]=R^{S_n}$-module, by \ref{thm:garsiabasis}. Therefore $R^G$ is Cohen-Macaulay, again by \ref{thm:HironakaoverZ}. This completes the argument.
\end{proof}

\begin{remark}
The class of groups $G$ satisfying the hypothesis of this theorem is fairly restricted. The transitive such groups were classified in 1979 by W. Cary Huffman (\cite{huffman}, Theorem 2.1). Huffman's classification theorem is given in appendix \ref{appendix:huffman}. 
\end{remark}

\section{Shellings and bases}\label{sec:shellings}

%Need to revise this now we know what we've really done.

We have seen a condition on $G$ under which $R^G$ is guaranteed to be a free $R^{S_n}$-module. The present section is concerned with explicitly finding bases for this module, when this occurs. More generally, since $R^G\otimes \Q$ is always free as an $R^{S_n}\otimes \Q$-module, we seek constructions that provide bases for the former over the latter.

This latter goal was the original motivation of Garsia and Stanton in \cite{garsiastanton}. Their idea was to seek information with which to construct a basis in the combinatorial structure of the balanced boolean complex $\Delta/G$. This idea leads to a dramatic clarification regarding why the module-freeness of the invariant ring is related to the topology of $\Delta/G$.

Garsia and Stanton gave a construction that works some of the time, specifically when the complex is {\em shellable}. We describe the method in subsection \ref{subsec:shellings}.

The ideal construction would be one that provides a $R^{S_n}\otimes\Q$-basis for $R^G\otimes \Q$ in every case, but that is also a $R^{S_n}$-basis for $R^G$ in the case that the latter is free. We give a construction that conjecturally meets this criterion in subsection \ref{subsec:cellbases}. To describe it, we introduce the notion of a {\em cell basis} of a balanced boolean complex.

We also describe, given a basis so constructed, a method to write an arbitrary invariant on this basis.  Its starting point is an algorithm, due to Manfred G\"{o}bel, that represents an arbitrary element of $R^G$ as an $R^{S_n}$-linear combination of certain ``special" orbit monomials, which works for any group $G$, over any coefficient ring. G\"{o}bel's special orbit monomials do not constitute a basis because they are not $R^{S_n}$-linearly independent. However, they are in bijective correspondence with the cells of $\Delta / G$, and the bases we construct are subsets of them. In the next subsection we explicate G\"{o}bel's method and its connection to $\Delta/G$.

We remark that the goals of this section (finding a basis and writing an arbitrary polynomial on this basis) can also be approached using Gr\"{o}bner basis techniques that are not specific to permutation groups -- see \cite{sturmfels}, section 2.5, and \cite{derksenkemper}, section 3.5. However, the methods we present have the advantage of illuminating the connection to the cell complex $\Delta / G$.

\subsection{G\"{o}bel's algorithm and the cell complex $\Delta/G$}\label{subsec:gobel}

G\"{o}bel's method is well-suited to a description in terms of the Garsia map, so recall the notations of section \ref{sec:prelim}: $R$ (\ref{not:definitionofR}), $S$ (\ref{not:definitionofS}), $\Lambda(m)$ and $\scrP$ (\ref{not:lambdaandP}), $\overline{\Lambda}(m)$ and $\scrPconj$ (\ref{not:lambdaconj}), $\garsia$(\ref{def:garsiamap}), and $\theta_i$ (\ref{not:rankrowsums}).

\begin{definition}[G\"{o}bel]
A monomial $m\in R$ is called \textbf{special} if its shape $\Lambda(m) = (\lambda_1,\dots,\lambda_n)$ satisfies $\lambda_n=0$, and $\lambda_i-\lambda_{i+1}\leq 1$ for $i=1,\dots,n-1$. The orbit monomial of a special monomial is a \textbf{special orbit monomial}.
\end{definition}

\begin{remark}
Note that the ``empty monomial," $1$, satisfies this definition.
\end{remark}

\begin{remark}
In \cite{gobel}, which introduced this notion, G\"{o}bel also included $\sigma_n$ as a special monomial, in order be able to make the assertion that the special orbit monomials always generate $R^G$ as an algebra. Since our interest is in the fact that they generate $R^G$ as an $R^{S_n}$ module (note that $R^{S_n}$ already contains $\sigma_n$), we can avoid this exception.
\end{remark}

\begin{lemma}\label{lem:specialfromsquarefree}
The special monomials in $R$ are precisely the $\garsia$-images of squarefree monomials in $S$ that do not contain $y_{\gls{[n]}}$.
\end{lemma}

\begin{definition}
We will also refer to such monomials in $S$ as \textbf{special}.
\end{definition}

\begin{proof}[Proof of \ref{lem:specialfromsquarefree}]
If $\garsia^{-1}(m) = \prod y_{U_j}$ with $U_1\supset \dots\supset U_k$ as in the proof of \ref{lem:conjugatepartition}, then $\lambda_i - \lambda_{i+1}$ measures the number of $U_j$'s for which $|U_j| = i$, as can be seen from figure \ref{fig:garsiamap} (where $\lambda = (4,3,1)$). Thus $\lambda_n = 0$ is equivalent to saying $y_{\gls{[n]}}$ is not in $\garsia^{-1}(m)$, and $\lambda_i-\lambda_{i+1}\leq 1, \; i\in[n-1]$ is equivalent to saying that $\garsia^{-1}(m)$ is squarefree.
\end{proof}

This lemma was used implicitly by Victor Reiner in \cite{reiner95}, which built on ideas in \cite{garsiastanton} to generalize G\"{o}bel's results to other Weyl groups.

\begin{thm}[G\"{o}bel]\label{thm:gobel}
For any permutation group $G\subset S_n$, over any ground ring $A$, the invariant ring $R^G$ is generated as an $R^{S_n}$-module by special orbit monomials.
\end{thm}

This was proven via an explicit algorithm (\cite{gobel}, Algorithm 3.12), inspired by Gauss' proof of the FTSP, to represent an arbitrary element of $R^G$ as an $R^{S_n}$-linear combination of special orbit monomials. The proof can be reformulated elegantly in terms of the Garsia map:

\begin{proof}
By \ref{thm:garsiabasis} and \ref{rmk:spanningandLIareseparate}, it suffices to show that the $\garsia$-corresponding orbit monomials in $S^G$ span it over $S^{S_n}$. It will be enough to represent the orbit monomial $Gm$ of any individual monomial $m\in S$. So let 
\[
m = \prod y_{U_i}^{e_i}
\]
where $U_1\supset\dots\supset U_k$ are distinct and each $e_i\geq 1$, and let 
\[
Gm = \sum_{g\in G/G_m} g(m)
\]
be the corresponding orbit monomial (where the sum is over a set of coset representatives for $m$'s stabilizer $G_m$). Let
\[
m_\star = \prod_{i:U_i\neq \gls{[n]}} y_{U_i}
\]
and let $Gm_\star$ be the corresponding orbit monomial. Note that $m_\star$ is special, and since it contains $y_{U_i}$ for each $U_i$ except possibly $\gls{[n]}$, any permutation in $S_n$ that stabilizes it also stabilizes $m$ and vice versa. Let $f_1 = e_1$ if it so happens that $U_1=\gls{[n]}$, and let $f_i = e_i-1$ otherwise. We claim that
\begin{equation}\label{eq:buildingm}
m = \left(\prod \theta_{|U_i|}^{f_i}\right) m_\star.
\end{equation}
Just as in the proof of proposition \ref{prop:FTSPforS}, the parenthetical expression is equal to the sum of all monomials
\[
\prod y_{U_i'}^{f_i}
\]
for which the $U_i'$'s form a chain and $|U_i'| = |U_i|$. But only the one of these for which $U_i' = U_i$, for each $i$, is supported on the same maximal chain as $m_\star$. Thus the right side of \eqref{eq:buildingm} only has one nonzero term, and it is
\[
\left(\prod y_{U_i}^{f_i}\right) m_\star = \prod y_{U_i}^{e_i} = m
\]
by definition of the $f_i$. This establishes \eqref{eq:buildingm}.

Furthermore, $m$ and $m_\star$ have the same stabilizer $G_m$ in $G$ (since they even have the same stabilizer in $S_n$). Therefore we can sum \eqref{eq:buildingm} over a set of coset representatives for $G_m$ to find that
\[
Gm = \left(\prod \theta_{|U_i|}^{f_i}\right)Gm_\star.
\]
Thus every orbit monomial lies in the $S^{S_n}$-span of the special orbit monomials.
\end{proof}

\begin{remark}
This proof sheds light on why the special orbit monomials are sufficient to generate $R^G$ over $R^{S_n}$. The chain $U_1\supset\dots\supset U_k$ on which a monomial $m\in S$ is supported contains all the information needed to determine $m$'s ``symmetry type," i.e. its stabilizer in $S_n$ -- more than enough, in fact, since $y_{\gls{[n]}}$'s presence or absence does not affect this stabilizer. The associated special monomial $m_\star$ in the proof is a pared down version of $m$ that retains all and only the information needed to determine this stabilizer. The theorem can be thought of as the statement that the special orbit monomials contain enough information about the possible symmetry types of monomials in order to reconstruct $R^G$.

As with other results in this thesis that appeal to theorem \ref{thm:garsiabasis} and therefore come down to an induction on monomial shapes $\lambda$, this proof can be arranged into an algorithm. We have implemented this algorithm in Magma. The reader can find the implementation in section \ref{sec:gobelmethod} of the appendices.
\end{remark}

\begin{remark}
G\"{o}bel's result was published in 1995. It was theoretically important at the time because it established a bound on the maximum degrees required to generate $R^G$ as an algebra that is independent of the coefficient ring $A$, and is usually much smaller than Noether's bound.\footnote{Noether's bound states that in the nonmodular case, an invariant ring is generated in degree at most $|G|$; see chapter introduction.} The ring $R^{S_n}$ is generated by $\sigma_1,\dots,\sigma_n$, and $R^G$ is generated over it (even as a module, so certainly as an algebra) by special orbit monomials, the maximum degree of which is $n(n-1)/2$. Thus $R^G$ is always generated in the degrees up to $\max (n,n(n-1)/2)$. This is known as {\em G\"{o}bel's bound}.

It was already known in 1995 that this bound holds over a characteristic-zero field. This fact is sometimes attributed to Garsia and Stanton, e.g. in \cite{neuselbounds}. This was not a goal of Garsia and Stanton's, though it does follow from their work. There is also a beautiful proof dated 1991, due to Barbara Schmid, based on Hilbert series (\cite{schmid}, section 9). It works uniformly in the nonmodular case, though it was formulated in characteristic zero. %(in view of \ref{thm:hochstereagon} and \ref{cor:CMfreemodule})

However, all of these authors were anticipated by Leopold Kronecker, by over a century. Kronecker showed (\cite{kronecker}, \S 12) that with $A = \Q$, $R^G$ is generated over $R^{S_n}$ by the orbit sums of monomials of the form $\prod x_i^{e_i}$ with each $e_i<i$, which also have maximum degree $n(n-1)/2$. This implies G\"{o}bel's bound in the characteristic zero case. He also proved that a subset of these ``Kronecker-special" orbit monomials forms a free basis for $R^G$ as an $R^{S_n}$-module, which implies the Hochster-Eagon theorem (\ref{thm:hochstereagon}) in the permutation group case, in view of the Hironaka criterion (\ref{cor:CMfreemodule}).\footnote{It would be anachronistic to say that Kronecker proved the Hochster-Eagon theorem in the permutation group case since the notion of Cohen-Macaulayness did not exist yet.}  Kronecker even claimed that, ``obviously,"\footnote{{\em Offenbar.}} an arbitrary orbit monomial can be written as a $\Z$-linear combination of the Kronecker-special ones, which implies G\"{o}bel's bound in general. However, he did not give a proof.\footnote{We thank Harold Edwards for alerting us to Kronecker's contribution.}
\end{remark}

Theorem \ref{thm:gobel} reduces the work a proposed generating set for $R^G$ or $S^G$ has to do to prove itself, to the finite problem of representing all the special orbit monomials. Our promised clarification of the relationship of $\Delta/G$ to the module structure of $R^G$ and $S^G$ starts from the fact that the special orbit monomials are in bijection with the cells in the boundary of $\Delta/G$.

Recall that $\Delta$ is the order complex of the poset $B_n\setminus\{\emptyset\}$, and that $\Delta/G$ is the boolean complex obtained by taking the quotient of this complex by the simplicial action of $G$ on $\Delta$. Then $\partial \Delta$, the simplicial complex corresponding to the boundary of the topological space $|\Delta|$, can be identified with the order complex of $B_n\setminus\{\emptyset,\gls{[n]}\}$, as follows: $B_n\setminus\{\emptyset\}$ is the face poset of a simplex (interpreted as a regular CW complex); $\gls{[n]}$ represents the top-dimensional cell. Thus its order complex is the barycentric subdivision of this simplex, and now $\gls{[n]}$ corresponds to the barycenter of the simplex, and the boundary is exactly the link of this point. By excluding $\gls{[n]}$ from $B_n\setminus\{\emptyset\}$ before taking the order complex, we are left with just the boundary. See figure \ref{fig:boundarycomplex}. The boundary operator $\partial$ commutes with the action of $G$, so $(\partial \Delta)/G = \partial(\Delta/G)$, and we can write $\partial\Delta/G$ unambiguously.

\begin{figure}
\begin{center}
\begin{tikzpicture}

\node (v1) at (-2,1) {$1$};
\node (v2) at (-1,1) {$2$};
\node (v3) at (0,1) {$3$};
\node (e12) at (-2,2) {$12$};
\node (e13) at (-1,2) {$13$};
\node (e23) at (0,2) {$23$};
\node (tri) at (-1,3) {[3]};
\node at (-1,0) {$B_3\setminus\{\emptyset\}$};

\foreach \from/\to in {v1/e12, v1/e13, v2/e12, v2/e23, v3/e13, v3/e23, e13/tri, e23/tri, e12/tri} 
    \draw (\from) -- (\to);

\draw [fill = lightgray] (2,1) -- (4,1) -- (3,2.732) -- (2,1);
\draw (2,1) -- (3.5,1.866);
\draw (3,1) -- (3,2.732);
\draw (4,1) -- (2.5,1.866);
\draw [fill] (2,1) circle [radius=0.03];
\draw [fill] (4,1) circle [radius=0.03];
\draw [fill] (3,1) circle [radius=0.03];
\draw [fill] (2.5,1.866) circle [radius=0.03];
\draw [fill] (3,2.732) circle [radius=0.03];
\draw [fill] (3.5,1.866) circle [radius=0.03];
\draw [fill] (3,1.577) circle [radius=0.03];
\node (V1) at (1.8,0.9) {$1$};
\node (V2) at (4.2,0.9) {$2$};
\node (V3) at (3,3) {$3$};
\node (E12) at (3,0.75) {$12$};
\node (E13) at (2.2,2.016) {$13$};
\node (E23) at (3.8,2.016) {$23$};
\node (Tri) at (3,1.577) {$[3]$};
\node at (3,0) {$\Delta(B_3\setminus\{\emptyset\})$};

\node (v1) at (6,1) {$1$};
\node (v2) at (7,1) {$2$};
\node (v3) at (8,1) {$3$};
\node (e12) at (6,2) {$12$};
\node (e13) at (7,2) {$13$};
\node (e23) at (8,2) {$23$};
\node at (7,0) {$B_3\setminus\{\emptyset,[3]\}$};

\foreach \from/\to in {v1/e12, v1/e13, v2/e12, v2/e23, v3/e13, v3/e23} 
    \draw (\from) -- (\to);
    
\draw (10,1) -- (12,1) -- (11,2.732) -- (10,1);
\draw [fill] (10,1) circle [radius=0.03];
\draw [fill] (12,1) circle [radius=0.03];
\draw [fill] (11,1) circle [radius=0.03];
\draw [fill] (10.5,1.866) circle [radius=0.03];
\draw [fill] (11,2.732) circle [radius=0.03];
\draw [fill] (11.5,1.866) circle [radius=0.03];
\node (V1) at (9.8,0.9) {$1$};
\node (V2) at (12.2,0.9) {$2$};
\node (V3) at (11,3) {$3$};
\node (E12) at (11,0.75) {$12$};
\node (E13) at (10.2,2.016) {$13$};
\node (E23) at (11.8,2.016) {$23$};
\node at (11,0) {$\Delta(B_3\setminus\{\emptyset,[3]\})$};

\end{tikzpicture}
\end{center}
\caption{The order complex of $B_3\setminus\{\emptyset,[3]\}$ is the boundary of the order complex of $B_3\setminus\{\emptyset\}$.}\label{fig:boundarycomplex}
\end{figure}
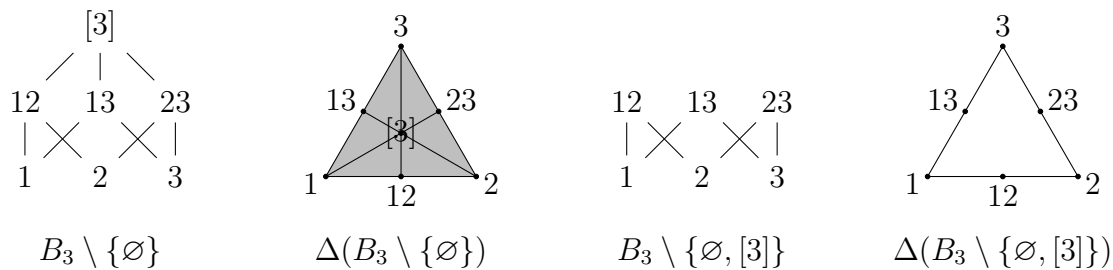

\begin{prop}\label{prop:orbitmonsandfaces}
The special orbit monomials of $R^G$ and $S^G$ other than $1$ are in bijection with the faces of $\partial \Delta / G$.
\end{prop}

\begin{proof}
A face in $\partial \Delta / G$ is the $G$-orbit of a face in $\partial \Delta$, and a special orbit monomial is the sum over a $G$-orbit of a special monomial in $R$ or $S$, so it will suffice to give a $G$-equivariant bijection between cells of $\partial \Delta$ and special monomials.

The bijection is this. A face of $\partial \Delta$ is a face of the order complex of $B_n\setminus\{\emptyset,\gls{[n]}\}$, which is a chain $U_1\supset\dots\supset U_k$ in $B_n$ not containing either $\gls{[n]}$ or $\emptyset$. This face corresponds with the monomial $y_{U_1}\dots y_{U_k}$ of $S$, and the monomial $\garsia(y_{U_1}\dots y_{U_k})$ of $R$. See figure \ref{fig:cellsandspecialmonomials}.
\end{proof}

\begin{figure}
\begin{center}
\begin{tikzpicture}

\draw (-2,0) -- (2,0) -- (0,3.464) -- (-2,0);
\draw [fill] (-2,0) circle [radius = 0.05];
\draw [fill] (2,0) circle [radius = 0.05];
\draw [fill] (0,3.464) circle [radius = 0.05];
\draw [fill] (0,0) circle [radius = 0.05];
\draw [fill] (-1,1.732) circle [radius = 0.05];
\draw [fill] (1,1.732) circle [radius = 0.05];
\node at (-2.2,-0.3) {$y_1$};
\node at (0,-0.3) {$y_{12}$};
\node at (2.2,-0.3) {$y_2$};
\node at (-1.5,1.932) {$y_{13}$};
\node at (1.5,1.932) {$y_{23}$};
\node at (0,3.764) {$y_3$};
\node at (-1.1,-0.3) {$y_1y_{12}$};
\node at (1.1,-0.3) {$y_2y_{12}$};
\node at (-2,1) {$y_1y_{13}$};
\node at (2,1) {$y_2y_{23}$};
\node at (-1,2.9) {$y_3y_{13}$};
\node at (1,2.9) {$y_3y_{23}$};

\end{tikzpicture}
\end{center}
\caption{The bijection between faces of $\partial \Delta$ and special monomials of $S$, for $n=3$.}\label{fig:cellsandspecialmonomials}
\end{figure}
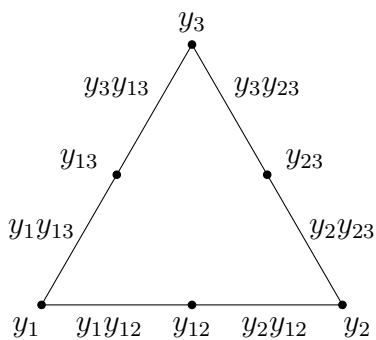

\begin{remark}\label{rmk:posetbij}
This is a point in the theory where we pay a minor price for our decision not to automatically regard the empty face as part of a boolean complex (cf. remark \ref{rmk:emptyface}). The boundary complex $\partial \Delta = \Delta(B_n\setminus\{\emptyset,\gls{[n]}\})$, as a simplicial complex, has an empty face. In the bijection just described, the empty face naturally corresponds with the special orbit monomial $1$, which is the empty product. In order to keep this feature of the correspondence when passing to the quotient $\partial \Delta / G$, which is no longer a simplicial but instead a boolean complex, we need to add the minimal element $\emptyset$ to its face poset and use the elements of the face poset for the bijection. See figure \ref{fig:posetbij}. This is the point of view we take going forward.
\end{remark}

\begin{figure}
\begin{center}
\begin{tikzpicture}

\node (empty) at (0,-1) {$\emptyset$};
\node (v1) at (-5,0) {$1$};
\node (v2) at (-1,0) {$2$};
\node (v3) at (3,0) {$3$};
\node (v12) at (-3,0) {$12$};
\node (v23) at (1,0) {$23$};
\node (v13) at (5,0) {$13$};
\node (e112) at (-5,1) {$1\subset 12$};
\node (e212) at (-3,1) {$2\subset 12$};
\node (e223) at (-1,1) {$2\subset 23$};
\node (e323) at (1,1) {$3 \subset 23$};
\node (e313) at (3,1) {$3\subset 13$};
\node (e113) at (5,1) {$1\subset 13$};

\node at (0,-2) {Face poset of $\partial \Delta$, labeled by chains in $B_3\setminus\{\emptyset,[3]\}$.};
\node at (0,-3) {};

\foreach \from/\to in {empty/v1, empty/v12, empty/v2, empty/v23, empty/v3, empty/v13, v1/e112, v12/e112, v12/e212, v2/e212, v2/e223, v23/e223, v23/e323, v3/e323, v3/e313, v13/e313, v13/e113, v1/e113}
    \draw (\from) -- (\to);

\end{tikzpicture}
\begin{tikzpicture}

\node (empty) at (0,-1) {$1$};
\node (v1) at (-5,0) {$x_1$};
\node (v2) at (-1,0) {$x_2$};
\node (v3) at (3,0) {$x_3$};
\node (v12) at (-3,0) {$x_1x_2$};
\node (v23) at (1,0) {$x_2x_3$};
\node (v13) at (5,0) {$x_1x_3$};
\node (e112) at (-5,1) {$x_1^2x_2$};
\node (e212) at (-3,1) {$x_1x_2^2$};
\node (e223) at (-1,1) {$x_2^2x_3$};
\node (e323) at (1,1) {$x_2x_3^2$};
\node (e313) at (3,1) {$x_1x_3^2$};
\node (e113) at (5,1) {$x_1^2x_3$};

\node at (0,-2) {Special monomials of $R$, including $1$.};

\foreach \from/\to in {empty/v1, empty/v12, empty/v2, empty/v23, empty/v3, empty/v13, v1/e112, v12/e112, v12/e212, v2/e212, v2/e223, v23/e223, v23/e323, v3/e323, v3/e313, v13/e313, v13/e113, v1/e113}
    \draw (\from) -- (\to);
    
\end{tikzpicture}
\end{center}
\caption{The face poset of $\partial \Delta$, for $n=3$, including the minimal element $\emptyset$, showing the bijection with special monomials of $R$.}\label{fig:posetbij}
\end{figure}

\begin{example}
For an example where $G$ is nontrivial, let $n=4$ and $G = D_4 = \langle (1234),(13)\rangle \subset S_4$. In this case, $\Delta$ is the barycentric subdivision of a $3$-simplex, thus $\partial \Delta$ is homeomorphic to a $2$-sphere. The quotient complex $\partial \Delta / G$ is homeomorphic to a disk. It has three facets. See figure \ref{fig:faceposetD4}.
\end{example}

\begin{figure}
\begin{center}
\begin{tikzpicture}

\node (A) at (0,9) {$1^32^23$};
\node (B) at (1.5,9) {$1^32^24$};
\node (C) at (-1.5,9) {$1^323^2$};
\node (D) at (-4,6) {$1^22^23$};
\node (E) at (-5.5,6) {$1^223^2$};
\node (F) at (5.5,6) {$1^224$};
\node (G) at (4,6) {$1^223$};
\node (H) at (0,6) {$1^22$};
\node (I) at (-1.5,6) {$1^23$};
\node (J) at (-4,3) {$12$};
\node (K) at (-5.5,3) {$13$};
\node (L) at (4,3) {$1$};
\node (M) at (0,3) {$123$};
\node (N) at (0,0) {$\emptyset$};

\foreach \from/\to in {A/H, A/D, A/G, B/D, B/H, B/F, C/G, C/E, C/I, D/M, D/J, E/M, E/K, F/M, F/L, G/M, G/L, H/J, H/L, I/L, I/K, J/N, K/N, L/N, M/N}
    \draw (\from) -- (\to);

\end{tikzpicture}

\begin{tikzpicture}

\node (L) at (0,3) {};
\node (M) at (0,-3) {};
\node (J) at (0,0) {};
\node (K) at (-5,0) {};

\draw [fill=lightgray] (M) to [out=180,in=270] (K) to [out=90,in=180] (L) to [out=0,in=90] (5,0) to [out=270,in=0] (M);

\node at (0,3.25) {1};
\draw [fill] (L) circle [radius=0.05];
\node at (0,-3.25) {123};
\draw [fill] (M) circle [radius=0.05];
\node at (0.25,0) {12};
\draw [fill] (J) circle [radius=0.05];
\node at (-5.25,0) {13};
\draw [fill] (K) circle [radius=0.05];

%\draw (M) to [out=180,in=270] (K);
\node at (-3.9,-2.2) {$1^223^2$};

%\draw (K) to [out=90,in=180] (L);
\node at (-3.9,2.1) {$1^23$};

\draw (M) to [out=180,in=270] (-2.5,0) to [out=90,in=180] (L);
\node at (-2.5,0.7) {$1^223$};
\node at (-3.7,-0.2) {$1^323^2$};

\draw (M) -- (L);
\node at (0,-1.5) {$1^22^23$};
\node at (0, 1.5) {$1^22$};
\node at (-1.3,-0.2) {$1^32^23$};

%\draw (M) to [out=0,in=270] (4,0) to [out=90,in=0] (L);
\node at (5,0.3) {$1^224$};
\node at (2.5,-0.2) {$1^32^24$};

\end{tikzpicture}

\end{center}
\caption{The face poset (above) and geometric realization (below) of $\partial \Delta/G$ in the case $G=D_4 = \langle (1234),(13)\rangle$. Each face is labeled by a term in the special orbit monomial corresponding to that face. The monomial $x_1^3x_2x_3^2$ is abbreviated $1^323^2$ etc., so we write the monomial $1$ as $\emptyset$.}\label{fig:faceposetD4}
\end{figure}
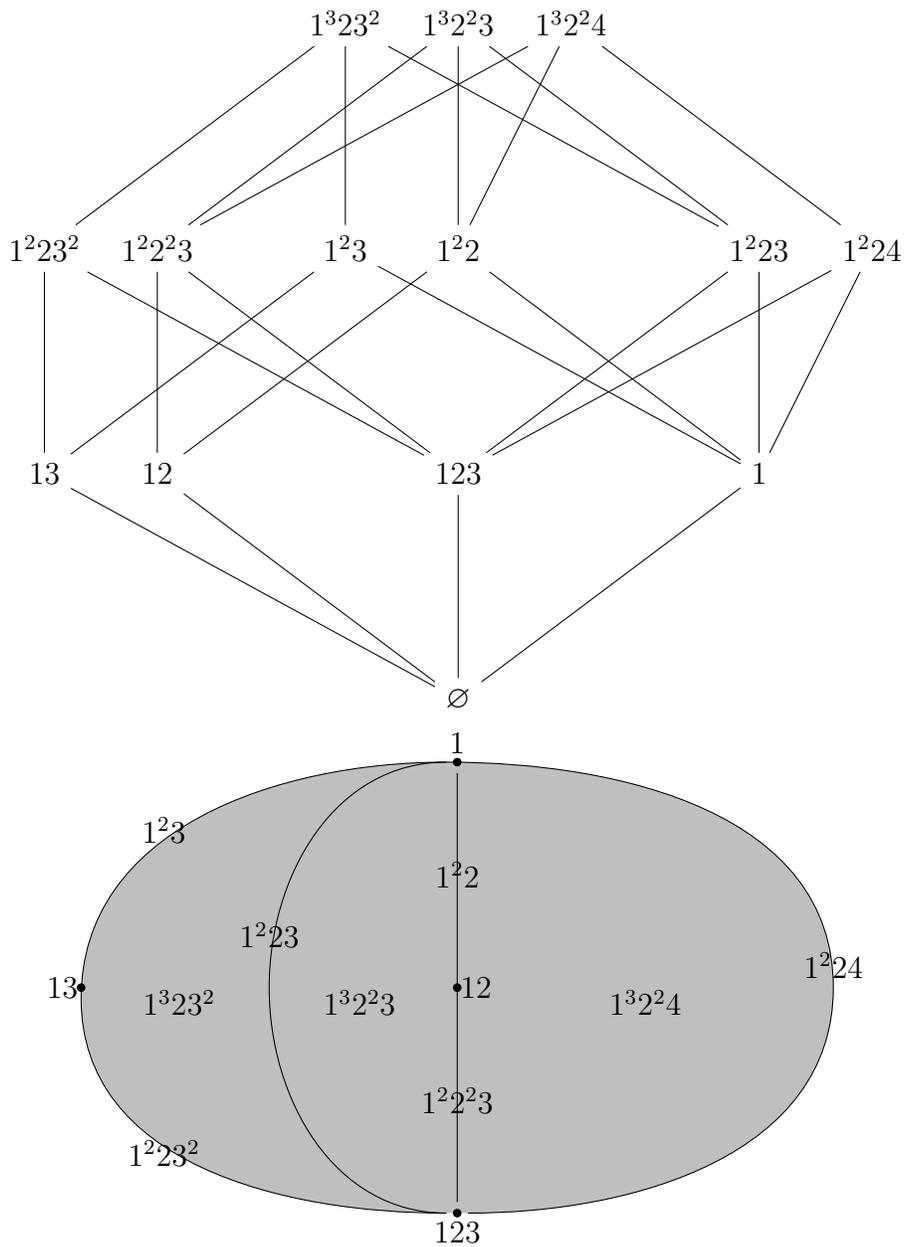

The special orbit monomials are thus a reflection of the combinatorial structure of $\partial \Delta / G$. %This is a moral explanation for the fact that the topology of $\partial \Delta / G$ has such an impact on the Cohen-Macaulayness of $S^G$ and $R^G$.

\subsection{Shellings}\label{subsec:shellings}

Shellability is a concept originating in polyhedral geometry. One of Garsia and Stanton's remarkable accomplishments in \cite{garsiastanton} was to show that a shelling of $\partial \Delta / G$ automatically gives rise to an $S^{S_n}$-basis of $S^G$, whose image under the Garsia map is therefore an $R^{S_n}$-basis of $R^G$. 

\begin{definition}
Let $K$ be a pure boolean complex of dimension $d$. If the facets of $K$ admit an ordering $F_1,\dots,F_r$ such that for each $j>1$, the intersection of $F_j$ with the union $\bigcup_{i<j} F_i$ of the earlier facets is a pure subcomplex of $F_j$ of dimension $d-1$, then $K$ is said to be \textbf{shellable}, and the ordering is a \textbf{shelling}.
\end{definition}

\begin{notation}
In what follows, we will use the same symbols $F_j$, etc. whether $K$ is being viewed as a boolean complex or as its face poset. Thus $\alpha\subset F_j$ and $\alpha\leq F_j$ mean the same thing.
\end{notation}

\begin{lemma}\label{lem:minimalfaces}
Let $\widehat{P}$ be the face poset of a pure boolean complex $\Delta$ with minimal element appended. Then an order $F_1,\dots,F_k$ is a shelling if and only if for each $j$, among the faces of $F_j$ not contained in $\bigcup_{i<j} F_i$ there is a unique minimal face $\alpha_j$.
\end{lemma}

\begin{proof}
Since $\Delta$ is a boolean complex, each facet $F_j$ is combinatorially a simplex. Thus the faces it contains are in bijection with subsets of the vertices it contains. We may therefore speak of the face of $F_j$ spanned by a specific set of vertices.

If in each $F_j$ there is a unique minimal face $\alpha_j$ not contained in $\bigcup_{i<j}F_i$, then $F_j\cap \bigcup_{i<j} F_i$ consists of all of the faces of $F_j$ missing at least one vertex of $\alpha_j$. Any such face is contained in a face of $F_j$ that is missing exactly one vertex of $\alpha_j$. These faces are all codimension $1$ in $F_j$, i.e. they are dimenison $d-1$.

In the other direction, if the intersection of $F_j$ with $\bigcup_{i<j} F_i$ is pure of dimension $d-1$, then its facets $f_1,\dots,f_r$ are faces of $F_j$ omitting exactly one vertex each. Let $\alpha_j$ be the subcomplex of $F_j$ spanned by these missing vertices. Any subcomplex of $F_j$ not containing $\alpha_j$ is thus contained in one of $f_1,\dots,f_r$. But meanwhile, $\alpha_j$ cannot be contained in any of $f_1,\dots,f_r$ since they are all missing at least one of its vertices. Therefore $\alpha_j$ is the minimal face of $F_j$ not already contained in $\bigcup_{i<j} F_i$.
\end{proof}

\begin{example}
Consider $n=4$ and $G=D_4$ as in figure \ref{fig:faceposetD4}. A shelling cannot begin with the two facets labeled $1^323^2$ and $1^32^24$ (on the left and right in the figure) since they intersect in codimension two. However, any of the other four orders of the facets is a shelling. Figure \ref{fig:shellingD4} depicts the shelling 
\[
1^323^2,\;1^32^23,\;1^32^24.
\]
It illustrates how each new facet intersects the previous facets in codimension one, and also how in the face poset, there is a unique minimal face among those added at each stage.
\end{example}

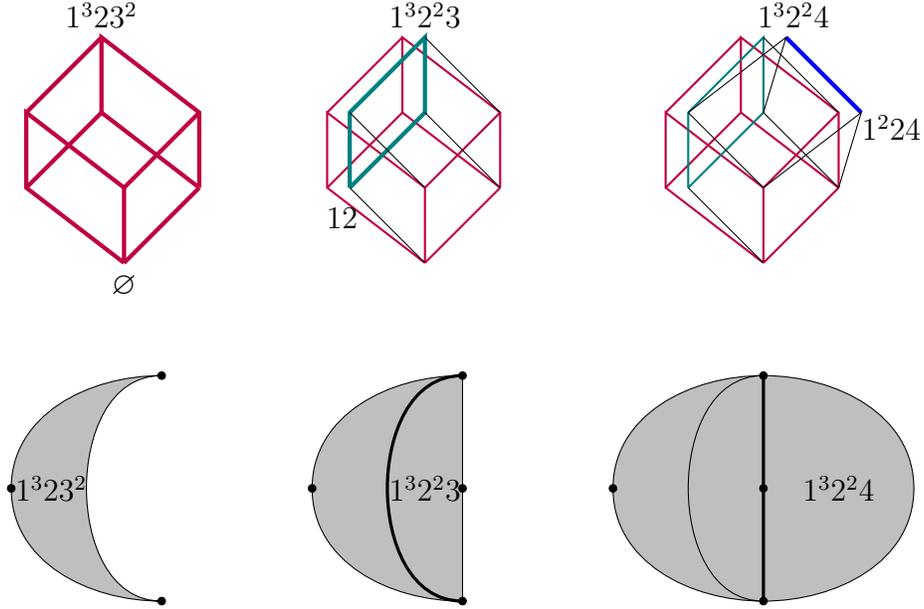
\begin{figure}
\begin{center}
\begin{tikzpicture}

\draw [fill=lightgray] (0,-1.5) to [out=180,in=270] (-2,0) to [out=90,in=180] (0,1.5) to [out=180,in=90] (-1,0) to [out=270,in=180] (0,-1.5);
\draw [fill] (-2,0) circle [radius=0.05];
\draw [fill] (0,-1.5) circle [radius=0.05];
\draw [fill] (0,1.5) circle [radius=0.05];
\node at (-1.47,0) {$1^323^2$};

\draw [purple, ultra thick] (-0.5,3) -- (-1.8,4) -- (-1.8,5) -- (-0.5,4) -- (-0.5,3);
\draw [purple, ultra thick] (0.5,4) -- (-0.8,5) -- (-0.8,6) -- (0.5,5) -- (0.5,4);
\draw [purple, ultra thick] (-0.5,3) -- (0.5,4);
\draw [purple, ultra thick] (-1.8,4) -- (-0.8,5);
\draw [purple, ultra thick] (-1.8,5) -- (-0.8,6);
\draw [purple, ultra thick] (-0.5,4) -- (0.5,5);
\node at (-0.8,6.3) {$1^323^2$};
\node at (-0.5, 2.7) {$\emptyset$};

\draw [fill=lightgray] (4,-1.5) to [out=180,in=270] (2,0) to [out=90,in=180] (4,1.5) -- (4,-1.5);
\draw [fill] (2,0) circle [radius=0.05];
\draw [fill] (4,-1.5) circle [radius=0.05];
\draw [fill] (4,1.5) circle [radius=0.05];
\draw [fill] (4,0) circle [radius=0.05];
\draw [very thick] (4,1.5) to [out=180,in=90] (3,0) to [out=270,in=180] (4,-1.5);
\node at (3.5,0) {$1^32^23$};

\draw [purple, thick] (3.5,3) -- (2.2,4) -- (2.2,5) -- (3.5,4) -- (3.5,3);
\draw [purple, thick] (4.5,4) -- (3.2,5) -- (3.2,6) -- (4.5,5) -- (4.5,4);
\draw [purple, thick] (3.5,3) -- (4.5,4);
\draw [purple, thick] (2.2,4) -- (3.2,5);
\draw [purple, thick] (2.2,5) -- (3.2,6);
\draw [purple, thick] (3.5,4) -- (4.5,5);
\draw [teal, ultra thick] (2.5,4) -- (2.5,5) -- (3.5,6) -- (3.5,5) -- (2.5,4);
\draw (2.5,4) -- (3.5,3);
\draw (2.5,5) -- (3.5,4);
\draw (3.5,6) -- (4.5,5);
\draw (3.5,5) -- (4.5,4);
\node at (3.5, 6.3) {$1^32^23$};
\node at (2.4,3.6) {$12$};

\draw [fill=lightgray] (8,-1.5) to [out=180,in=270] (6,0) to [out=90,in=180] (8,1.5) to [out=0,in=90] (10,0) to [out=270,in=0] (8,-1.5);
\draw [fill] (6,0) circle [radius=0.05];
\draw [fill] (8,-1.5) circle [radius=0.05];
\draw [fill] (8,1.5) circle [radius=0.05];
\draw [fill] (8,0) circle [radius=0.05];
\draw (8,1.5) to [out=180,in=90] (7,0) to [out=270,in=180] (8,-1.5);
\draw [very thick] (8,1.5) -- (8,-1.5);
\node at (9,0) {$1^32^24$};

\draw [purple, thick] (8,3) -- (6.7,4) -- (6.7,5) -- (8,4) -- (8,3);
\draw [purple, thick] (9,4) -- (7.7,5) -- (7.7,6) -- (9,5) -- (9,4);
\draw [purple, thick] (8,3) -- (9,4);
\draw [purple, thick] (6.7,4) -- (7.7,5);
\draw [purple, thick] (6.7,5) -- (7.7,6);
\draw [purple, thick] (8,4) -- (9,5);
\draw [teal, thick] (7,4) -- (7,5) -- (8,6) -- (8,5) -- (7,4);
\draw (7,4) -- (8,3);
\draw (7,5) -- (8,4);
\draw (8,6) -- (9,5);
\draw (8,5) -- (9,4);
\draw [blue, ultra thick] (9.3,5) -- (8.3,6);
\draw (8,4) -- (9.3,5) -- (9,4);
\draw (7,5) -- (8.3,6) -- (8,5);
\node at (8.4,6.3) {$1^32^24$};
\node at (9.7,4.8) {$1^224$};

\end{tikzpicture}
\end{center}
\caption{A shelling of $\partial \Delta / D_4$. Above: among the faces added by each new facet, there is a unique minimal one. Below: the intersection of each new facet with the union of the previous ones is codimension one. The notation is the same as in figure \ref{fig:faceposetD4}.}\label{fig:shellingD4}
\end{figure}

Because $\Delta = \Delta(B_n\setminus\{\emptyset\})$ is always homeomorphic to a simplex, $\partial \Delta$ is always homeomorphic to a sphere. Thus each codimension 1 face is incident to exactly two facets. This property is almost inherited by $\partial\Delta / G$: each codimension 1 face is incident to at most two facets, but perhaps just one, if the two facets incident to any of its preimages in $\partial \Delta$ get identified by the action of $G$. This makes $\partial \Delta / G$ what is called a {\em pseudomanifold} -- see \cite{reiner92}, Proposition 2.4.2.

In this circumstance, shellability of $\partial \Delta / G$ implies it is homeomorphic to a sphere or a ball (\cite{bjorner2}, Proposition 4.3; see also \cite{danarajklee}, Proposition 1.2). Thus shellability automatically implies Cohen-Macaulayness over any field. However, the situation is even better:

\begin{notation}
Let $F_1,\dots,F_r$ be a shelling of $\partial \Delta/G$, and let $\alpha_1,\dots,\alpha_r$ be the corresponding minimal faces guaranteed by lemma \ref{lem:minimalfaces}. Let $b_1,\dots,b_r$ be the orbit monomials in $S^G$ corresponding to $\alpha_1,\dots,\alpha_r$ according to the correspondence given in \ref{rmk:posetbij}. We fix the coefficient ring $A$ as $\Z$ or $\F_p$.
\end{notation}

\begin{thm}[Garsia-Stanton]\label{thm:shellingbasis}
If $\partial \Delta / G$ admits a shelling $F_1,\dots,F_r$, then the orbit monomials $b_1,\dots,b_r$ form a module basis for $S^G$ over $S^{S_n}$. Consequently, their images $\garsia(b_1),\dots,\garsia(b_r)$ form a module basis for $R^G$ over $R^{S_n}$.
\end{thm}

This is a slight generalization of \cite{garsiastanton}, Theorem 6.2.

Before giving the proof, we introduce some machinery.

\begin{definition}
A \textbf{partitioning}, or \textbf{E-R decomposition}, of a boolean complex is a decomposition of its face poset into disjoint intervals.
\end{definition}

\begin{lemma}\label{lem:shellingpartitioning}
Given a shelling $F_1,\dots,F_r$ of a boolean complex with face poset $\widehat{P}$ (including the minimal element), the intervals $[\alpha_j,F_j]$ form a partitioning, where each $\alpha_j$ is the minimal face of $F_j$ not contained in $\bigcup_{i<j}F_i$, per lemma \ref{lem:minimalfaces}.
\end{lemma}

\begin{proof}
Since each $\alpha_j$ is minimal in $F_j\setminus \bigcup_{i<j} F_i$ and $F_j$ is obviously maximal in it, we have
\[
[\alpha_j,F_j] = F_j\setminus\bigcup_{i<j} F_i
\]
for each $j$. (The union on the right is empty when $j=1$.) Thus the intervals $[\alpha_j,F_j]$ are precisely the partition of $\widehat{P}$ into the elements that are added by each new $F_j$.
\end{proof}

\begin{definition}
Let $F_1,\dots,F_r$ be the facets of a boolean complex with face poset $\widehat{P}$ (including the minimal element). Let $\alpha_1,\dots,\alpha_r$ be any $r$ elements of $P$. Then the \textbf{incidence matrix} of the sequence of $\alpha_i$'s is the $r\times r$ matrix with $i,j$th entry equal to $1$ if $\alpha_i\leq F_j$ and $0$ otherwise.
\end{definition}

\begin{lemma}\label{lem:unitriangular}
If $F_1,\dots,F_r$ is a shelling, and $\alpha_1,\dots\alpha_r$ are the corresponding minimal elements guaranteed by lemma \ref{lem:minimalfaces}, then the incidence matrix of $\alpha_1,\dots,\alpha_r$ is unitriangular, i.e. upper triangular with $1$'s on the main diagonal.
\end{lemma}

\begin{proof}
The diagonal consists of $1$'s because $\alpha_j\in F_j$ for each $j$. The matrix is upper triangular because $\alpha_j \notin \bigcup_{i<j} F_i$.
\end{proof}

Since special monomials of $S$ are square-free, their fine grades are in bijection with subsets of $[n-1]$. It follows that the same is true for special orbit monomials of $S^G$. Given a special orbit monomial $Gm\in S^G$, we refer to the set of ranks $i$ for which its fine grade contains $e_i$ as its \textbf{rank set}. If $m=\prod y_{U_i}^{e_i}$, this is just the set of cardinalities $|U_i|$. Recall from the proof of \ref{cor:deltaisbalanced} that these ranks are the labels that realize $\Delta$ and $\Delta / G$ as balanced boolean complexes.

\begin{lemma}\label{lem:facetproduct}
Let $b\in S^G$ be a special orbit monomial, with rank set $I$, corresponding to a face $\alpha$ in $\partial\Delta/G$. Let $s$ be a squarefree product of $\theta_j$'s for some set of ranks $J\subset [n-1]\setminus I$. Then $sb$ is the sum of all the special orbit monomials that have rank set $I\cup J$ and correspond to faces of $\partial\Delta/G$ that contain the face $\alpha$. In particular, if
\[
s = \prod_{j\in [n-1]\setminus I} \theta_j
\]
then $sb$ is the sum of the special orbit monomials corresponding to the facets $F_\ell$ of $\partial\Delta/G$ that contain $\alpha$.
\end{lemma}

See figure \ref{fig:facetproduct}.

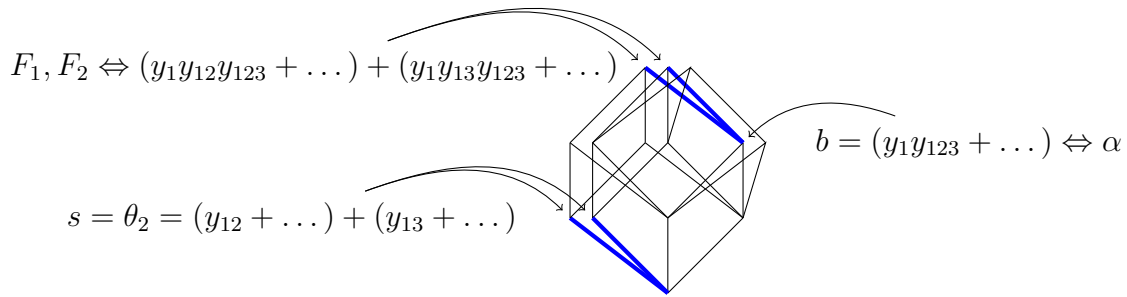
\begin{figure}
\begin{center}
\begin{tikzpicture}

\draw [blue, ultra thick] (8,3) -- (6.7,4);
\draw (6.7,4) -- (6.7,5);
\draw (6.7,5) -- (8,4);
\draw (8,3) -- (8,4);
\draw (9,4) -- (7.7,5);
\draw (7.7,5) -- (7.7,6);
\draw [blue, ultra thick] (7.7,6) -- (9,5);
\draw (9,4) -- (9,5);
\draw (8,3) -- (9,4);
\draw (6.7,4) -- (7.7,5);
\draw (6.7,5) -- (7.7,6);
\draw (8,4) -- (9,5);
\draw (7,4) -- (7,5) -- (8,6) -- (8,5) -- (7,4);
\draw [blue, ultra thick] (7,4) -- (8,3);
\draw (7,5) -- (8,4);
\draw [blue, ultra thick] (8,6) -- (9,5);
\draw (8,5) -- (9,4);
\draw (9.3,5) -- (8.3,6);
\draw (8,4) -- (9.3,5) -- (9,4);
\draw (7,5) -- (8.3,6) -- (8,5);
\node (term) at (12,5) {$b = (y_1y_{123}+\dots)\Leftrightarrow\alpha$};
\draw [->] (term) to [out=160,in=45] (9.075,5.075);
\node (theta) at (3,4) {$s = \theta_2 =( y_{12} + \dots) + (y_{13} + \dots)$};
\draw [->] (theta) to [out=20,in=135] (6.6,4.1);
\draw [->] (theta) to [out=20,in=135] (6.9,4.1);
\node (facets) at (3.3,6) {$F_1,F_2\Leftrightarrow(y_1y_{12}y_{123} + \dots) + (y_1y_{13}y_{123} + \dots)$};
\draw [->] (facets) to [out=20,in=135] (7.6,6.1);
\draw [->] (facets) to [out=20,in=135] (7.9,6.1);

\end{tikzpicture}
\end{center}
\caption{Illustration of lemma \ref{lem:facetproduct} in the case $n=4,\; G = D_4$. The orbit monomial of $y_1y_{123}$ is $b$. Its fine grade is $e_1+e_3$; the missing rank is $2$, so $s= \theta_2$. Note how the product $sb$ corresponds to the sum of the orbit monomials for each facet $F_1,F_2$ containing $\alpha$.}\label{fig:facetproduct}
\end{figure}

\begin{proof}
By \ref{thm:reinerSRquotient} and \ref{cor:deltaisbalanced}, $S^G\cong A[\Delta / G]$ is the Stanley-Reisner ring of the balanced boolean complex $\Delta / G$. The rank sets are the label sets for the balancing. Then this is precisely lemma \ref{lem:arithinbooleancomplexring}, applied to the present case. (The restriction to $\partial \Delta / G$ comes from excluding $n$ as a rank.)
\end{proof}

\begin{lemma}\label{lem:nonsingularLI}
If a sequence of faces $\alpha_1,\dots,\alpha_r$ in $\partial \Delta / G$ has an incidence matrix $M$ that is nonsingular over $A$, then the corresponding orbit monomials $b_1,\dots,b_r$ in $S^G$ are linearly independent over $S^{S_n}$.
\end{lemma}

\begin{proof}
Suppose for contradiction that there is a linear relation
\[
0=\sum s_ib_i
\]
with the $s_i\in S^{S_n}$. Without loss of generality we can assume that the relation is finely homogeneous and thus the $s_i$'s are finely homogeneous, since the $b_i$'s are, by construction. It follows by remark \ref{rmk:SGistorsionfree} that each $s_i$ is a single term in the $\theta_j$'s. Since $S^G$ is torsion free as an $S^{S_n}$-module, also by remark \ref{rmk:SGistorsionfree}, we can also assume that the $s_i$'s do not have a common factor. Since $b_i$'s are special orbit monomials, the fine grade of a $b_i$ has the form $\sum a_je_j$ with each $0\leq a_j\leq 1$ for $j\leq n-1$, and $a_n=0$. Then the common fine grade $\sum a_j'e_j$ of $s_ib_i$ cannot have any $a_j'\geq 2$ without forcing the $s_i$'s to have the corresponding $\theta_j$ as a common factor, and likewise $a_n'$ must be $0$. In other words, each $s_ib_i$ is also a sum of special orbit monomials.

Let $t$ be the product 
\[
\prod_{j\leq n-1} \theta_j^{1-a_j'},
\]
i.e. the product of $\theta_j$'s with rank set complementary in $[n-1]$ to that of the $s_ib_i$'s. Then certainly
\[
0=\sum ts_ib_i.
\]
By lemma \ref{lem:facetproduct}, each $ts_ib_i$ is the sum of the orbit monomials corresponding to the facets $F_\ell$ of $\partial \Delta / G$ that contain $\alpha_i$, all with the same coefficient, say $c_i\in A$. If $m_i$ is the $i$th row of $M$, this means
\[
0=\sum c_i m_i.
\]
But this linear relation over $A$ between the rows of $M$ contradicts the assumption about $M$.
\end{proof}

\begin{proof}[Proof of Theorem \ref{thm:shellingbasis}]%attend. would be great to roll this and a number of other proofs into some more basic theory of SR rings for balanced boolean.
We only need to prove the assertion about $S^G$ because the assertion about $R^G$ then follows by theorem \ref{thm:garsiabasis}.

Since the incidence matrix of the shelling is unitriangular by lemma \ref{lem:unitriangular}, it is nonsingular for any choice of $A$, so the $b_i$ are linearly independent over $S^{S_n}$ by lemma \ref{lem:nonsingularLI}. It remains to show spanningness.

By theorem \ref{thm:gobel}, it is enough to show that all special orbit monomials lie in the span of the $b_i$. We give a procedure to represent an arbitrary special orbit monomial as an $S^{S_n}$-linear combination of $b_i$'s.

Let $f$ be an arbitrary special orbit monomial, corresponding to a face $\tau$ in $\partial\Delta / G$. Because the shelling induces a partitioning of $\partial \Delta / G$ by \ref{lem:shellingpartitioning}, there is a unique $j$ with $\tau \in [\alpha_j,F_j]$. Then there is a unique squarefree product $s$ of $\theta_\ell$'s such that the rank set of $sb_j$ coincides with that of $f$. By lemma \ref{lem:facetproduct}, $sb_j$ is the sum of all the orbit monomials with that rank set that correspond to faces containing $\alpha_j$. One of these is $f$. 

We claim that the remaining special orbit monomials in $f - sb_j$ all correspond with faces that lie in various $[\alpha_k,F_k]$'s for $k>j$. Thus if we repeat the procedure on each of these, and then on any new orbit monomials that appear as a result, etc., we will finish in a finite number of steps. The claim holds because every such special orbit monomial corresponds to a face lying over $\alpha_j$ that is different from $\tau$. No face lying over $\alpha_j$ is in any $F_i$ for $i<j$, since $[\alpha_j,F_j]$ is disjoint from $\bigcup_{i<j}F_i$; and the only face lying over $\alpha_j$ and inside $F_j$ and with the right rank set is $\tau$, because $\partial\Delta / G$ is balanced, by \ref{lem:distinctsamelabels}.
\end{proof}

\begin{example}
In our running example $n=4, G = D_4$, the minimal elements $\alpha_1,\alpha_2,\alpha_3$ given by the shelling are the faces labeled $\emptyset, 12, 1^224$ in \ref{fig:shellingD4}, corresponding to the orbit monomials $1, Gy_{12}, Gy_1y_{124}$ in $S^G$ and $1, Gx_1x_2, Gx_1^2x_2x_4$ in $R^G$. Theorem \ref{thm:shellingbasis} implies these form bases, respectively, over $S^{S_n}$ and $R^{S_n}$. We given an example calculation exposing the mechanism of proof of the theorem.

Suppose we want to represent the orbit monomial 
\[
f = G y_{13}y_{123} = y_{13}y_{123} + y_{24}y_{234} + \dots 
\]
as an $S^{S_n}$ linear combination of the basis $1, Gy_{12}, Gy_1y_{124}$. This orbit monomial corresponds to the node in the face poset $\widehat{P}$ labeled $1^223^2$ in figure \ref{fig:faceposetD4}. Consulting figure \ref{fig:shellingD4}, the interval in the partitioning induced by the shelling that contains this face is $[\emptyset, 1^323^2]$. The basis element $b_j$ corresponding to this interval is $1$, which has empty rank set (i.e. fine grade zero), while our target orbit monomial has rank set $\{2,3\}$, i.e. its fine grade is $e_2+e_3$. Thus we take $s = \theta_2\theta_3$. We have
\[
\theta_2\theta_3 \cdot 1 = \text{every monomial with fine grade }e_2+e_3 = Gy_{13}y_{123} + Gy_{12}y_{123}
\]
Thus $f - sb_j$ is $-G_{12}y_{123}$. This is a single orbit monomial, corresponding to the node labeled $1^22^23$ in figure \ref{fig:faceposetD4}. Again consulting figure \ref{fig:shellingD4}, this node lies in the interval $[12,1^32^23]$, so the corresponding basis element is $Gy_{12}$. This has rank set $2$ i.e. its fine grade is $e_2$, so to hit the target of $e_2+e_3$ we need to multiply by $s = \theta_3$. We have
\[
-\theta_3 \cdot Gy_{12} = -Gy_{12}y_{123}
\]
so the target term has been expressed. Combining the two steps we obtain
\[
Gy_{13}y_{123} = \theta_2\theta_3\cdot 1 -\theta_3\cdot Gy_{12}
\]
which is the desired representation.
\end{example}

\subsection{Cell bases}\label{subsec:cellbases}

As we have seen, a shelling of $\partial \Delta / G$ implies Cohen-Macaulayness of $R^G$ and $S^G$ for $A=\Z$ and therefore any $\F_p$ as well. Nonetheless, Cohen-Macaulayness is always present if $A=\Q$. Garsia and Stanton sought combinatorial methods to find bases of $R^G$ and $S^G$ in this latter setting, which meant that shelling, though an elegant solution when available, was inadequate to their broader purpose. They achieved the following generalization in the case $A=\Q$.

\begin{thm}[\cite{garsiastanton}, Theorem 6.1]\label{thm:partitioningbasis}
If $\partial \Delta /G$ admits a partitioning $\bigcup [\alpha_j,F_j]$ such that $\alpha_1,\dots,\alpha_r$ has a nonsingular incidence matrix, then the special orbit monomials corresponding with $\alpha_1,\dots,\alpha_r$ form a basis for $S^G$, respectively $R^G$, over $S^{S_n}$, respectively $R^{S_n}$. \qed
\end{thm}

The argument works with $A=\F_p$ as well, as noted in \cite{hersh2}. More generally it works when the determinant of the incidence matrix is a unit of $A$. In view of this theorem, much of the subsequent work that has applied Garsia and Stanton's ideas to invariant theory has sought partitionings when shellings were not available (e.g. see \cite{reiner92}, \cite{hersh}, \cite{hersh2}). However, a basis of orbit monomials (over $\Z$, $\F_p$, or $\Q$) need not come from a partitioning.

More broadly, it was long thought that Cohen-Macaulayness might imply the existence of a partitioning. (Stanley, in \cite[p.~85]{cca}, called this ``a central combinatorial conjecture on Cohen-Macaulay complexes.") This has recently turned out to be false: even Cohen-Macaulayness over $\Z$ does not guarantee a partitioning -- see \cite{duvaletal}, which gives an explicit family of counterexamples. The question is still open if we add the assumption that the complex is balanced.\footnote{In the other direction, it has been known since the beginning that partitionability does not imply Cohen-Macaulayness, even in the balanced case. For example, the abstract simplicial complex $\{\emptyset, \{1\},\{2\},\{3\},\{4\},\{1,2\},\{3,4\}\}$, whose geometric realization is two disjoint line segments, is balanced (with color classes $\{1,3\}$ and $\{2,4\}$), but is not Cohen-Macaulay over any field, because it is $1$-dimensional but has nontrivial $\tilde H_0$.}

Also, the incidence matrices of various partitionings of a given complex need not be simultaneously nonsingular in a given characteristic. For example, all connected graphs (i.e. one-dimensional simplicial complexes) are shellable, so they have partitionings with incidence matrix with determinant $1$. On the other hand, the graph with 5 vertices and edges $12,13,23,34,45$ (a triangle with an antenna) has a partitioning 
\[
[1,12],[2,23],[3,13],[34,34], [\emptyset,45]
\]
with incidence matrix $2$. (Example due to Victor Reiner, personal communication.) The same idea with a square in place of a triangle leads to a partitioning with incidence matrix $0$. This latter complex is even balanced.

All this suggests that the relationship between Cohen-Macaulayness and partitionability might not be as close as previously thought.

Thus, we propose to refocus attention on the sets of cells in $\partial \Delta / G$ that correspond to orbit monomial bases, as objects of combinatorial study in themselves. 

\begin{notation}
The natural setting for the definitions we wish to make are a pure, balanced boolean complex $K$ of dimension $d$, and its Stanley-Reisner ring $A[K]$ over $A=k$ or $\Z$. (Recall definition \ref{def:SRringofboolean}.) As we have seen, in this context, $A[K]$ has a natural choice of h.s.o.p., namely the sums
\[
\psi_i = \sum_{v\text{ has label }i} v
\]
across the vertices with each label (\ref{prop:balancedhsop}), and $A[K]$ has a natural fine $\N^{d+1}$-grading (\ref{prop:balancedfinegrading}).
\end{notation}

\begin{remark}
The $\psi_i$'s just defined specialize to our $\theta_i$'s when we specialize $K$ to $\Delta / G$, and the fine grading specializes to the fine grading on $S^G$.
\end{remark}

\begin{definition}
A set of faces $\alpha_1,\dots,\alpha_r$ in the complex is a \textbf{cell basis} of $K$ over $A$ if the corresponding elements $y_{\alpha_1},\dots,y_{\alpha_r}$ are a basis for $A[K]$ over the subring $A[\psi_1,\dots,\psi_{d+1}]$.
\end{definition}

\begin{definition}
Suppose $K$ has facets $F_1,\dots,F_r$. For any cell $\alpha\in K$, the vector whose $j$th entry is $1$ if $\alpha \leq F_j$ and zero otherwise will be called the \textbf{facet vector} of $\alpha$.
\end{definition}

Thus the incidence matrix of a set of cells $\alpha_1,\dots,\alpha_r$ has their facet vectors as rows.

\begin{definition}
If the incidence matrix of $\alpha_1,\dots,\alpha_r$ is a unit of the coefficient ring $A$, then the facet vector of any face $\alpha$ of $K$ has a unique representation as an $A$-linear combination of the facet vectors of the $\alpha_i$'s. When this happens, we will say that $\alpha$'s facet vector is \textbf{supported on} those $\alpha_i$'s whose facet vectors occur with nonzero coefficient in this representation.
\end{definition}

\begin{prop}
A set of faces $\alpha_1,\dots,\alpha_r$ forms a cell basis for $K$ if and only if the following two conditions are met:
\begin{enumerate}
\item The determinant of the incidence matrix of $\alpha_1,\dots,\alpha_r$ is a unit of $A$.\label{cond:unitofA}
\item For any face $\alpha\in K$, with label set $J$, its facet vector is supported only on $\alpha_i$'s whose label sets are subsets of $J$.\label{cond:labelsupport}
\end{enumerate}
\end{prop}

\begin{proof}[Proof sketch]
If condition \ref{cond:unitofA} is met, then the corresponding elements $y_{\alpha_1},\dots, y_{\alpha_r}$ are linearly independent by the same argument as in \ref{lem:nonsingularLI}, which works without significant change in this more general setting. Likewise, if condition \ref{cond:labelsupport} is met, then they span $A[K]$, by the argument of \ref{thm:shellingbasis}. For this, one needs a generalization of G\"{o}bel's theorem (\ref{thm:gobel}). Fortunately, it is available: $A[K]$ is spanned as an $A[\psi_1,\dots,\psi_{d+1}]$-module by the elements $y_\alpha$ for $\alpha\in K$. This is because, as an {\em algebra with straightening law} (more on these in the next section), $A[K]$ has an $A$-basis consisting of monomials $\prod y_\alpha^{e_\alpha}$ supported on chains; but we can write any such monomial as a product of $\psi_i$'s times a single $y_{\alpha^\star}$, by taking $\alpha^\star$ to be maximal in the chain supporting the monomial. Then we can replace each other $y_\alpha$ with the product of $\psi_j$ over the label set of $\alpha$.

In the other direction, if $y_{\alpha_1},\dots,y_{\alpha_r}$ form a cell basis, then one can prove the incidence matrix of $\alpha_1,\dots,\alpha_r$ is nonsingular by representing each $y_{F_i}$ for facets $F_i$ as an $A[\psi_1,\dots,\psi_{d+1}]$-linear combination of $y_{\alpha_j}$'s, and interpreting the results in terms of the $\alpha_j$s' facet vectors. To prove condition \ref{cond:labelsupport}, one writes an arbitrary $y_\alpha$ with label set $J$ as
\[
y_\alpha = \sum s_i y_{\alpha_i}
\] 
with the $s_i\in A[\psi_1,\dots,\psi_{d+1}]$. Taking stock of the fine grading, it is clear that the label sets of the $\alpha_i$'s that appear on the right with nonzero $s_i$'s must be subsets of $\alpha$'s label set. Then one multiplies through by
\[
\prod_{j\in [d+1]\setminus J} \psi_j
\]
to obtain an equation expressing $\alpha$'s facet vector in terms of the facet vectors of the same $\alpha_i$'s.
\end{proof}

Garsia and Stanton give a different if-and-only-if criterion (of which \ref{thm:partitioningbasis} is a consequence) -- see \cite{garsiastanton}, Theorem 5.1 -- but it is equivalent after specializing to their setting. The proof involves Hilbert series.

When a cell basis exists, it can often (conjecture: always) be found inductively, beginning with the empty face, by choosing any face minimal among those whose facet vectors do not lie in the $A$-span of the facet vectors of the faces already selected, and repeating until a basis is achieved.

\begin{example}
In figure \ref{fig:greedycellbasis} we use this procedure to find a cell basis for our running example of $\partial\Delta / G$ for $G=D_4\subset S_4$. Each node is marked by its facet vector. We pick the minimal node first (mauve), and then cross out all nodes whose facet vectors lie in the $\Q$-span of its facet vector. Then we pick a minimal node among those remaining (teal), and cross out all nodes newly in the span. Finally, we pick a minimal node among those remaining (blue), and all nodes now lie in the span. Because there are three facets and it only took three facet vectors to span everything, the incidence matrix is nonsingular over $\Q$. (By inspection, it is even nonsingular over $\Z$.) Because at each stage, each node crossed out had rank set containing those of the nodes already selected, condition \ref{cond:labelsupport} was also met. Thus the end result is a cell basis. Comparing to figure \ref{fig:faceposetD4}, we obtain a basis for $R^G$ consisting of $\emptyset$, $13$, and $1^223$, i.e. of the orbit monomials $1$, $Gx_1x_3$, and $Gx_1^2x_2x_3$.
\end{example}

\begin{figure}
\begin{center}
\begin{tikzpicture}

\node (A) at (0,9) {$010$};
\node (B) at (1.5,9) {$001$};
\node (C) at (-1.5,9) {$100$};
\node (D) at (-4,6) {$011$};
\node (E) at (-5.5,6) {$100$};
\node (F) at (5.5,6) {$001$};
\node (G) at (4,6) {$110$};
\node (H) at (0,6) {$011$};
\node (I) at (-1.5,6) {$100$};
\node (J) at (-4,3) {$011$};
\node (K) at (-5.5,3) {$100$};
\node (L) at (4,3) {$111$};
\node (M) at (0,3) {$111$};
\node (N) at (0,0) {$111$};

\foreach \from/\to in {A/H, A/D, A/G, B/D, B/H, B/F, C/G, C/E, C/I, D/M, D/J, E/M, E/K, F/M, F/L, G/M, G/L, H/J, H/L, I/L, I/K, J/N, K/N, L/N, M/N}
    \draw (\from) -- (\to);

\draw [purple, thick] (-0.5,-0.2) rectangle (0.5,0.2);
\draw [purple, thick] (-0.5,2.8) -- (0.5,3.2);
\draw [purple, thick] (3.5,2.8) -- (4.5,3.2);

\draw [teal, thick] (-6,2.8) rectangle (-5,3.2);
\draw [teal, thick] (-4.5,2.8) -- (-3.5,3.2);
\draw [teal, thick] (-6,5.8) -- (-5,6.2);
\draw [teal, thick] (-4.5,5.8) -- (-3.5,6.2);
\draw [teal, thick] (-2,5.8) -- (-1,6.2);
\draw [teal, thick] (-0.5,5.8) -- (0.5,6.2);
\draw [teal, thick] (-2,8.8) -- (-1,9.2);

\draw [blue, thick] (3.5,5.8) rectangle (4.5,6.2);
\draw [blue, thick] (5,5.8) -- (6,6.2);
\draw [blue, thick] (-0.5,8.8) -- (0.5,9.2);
\draw [blue, thick] (1,8.8) -- (2,9.2);

\end{tikzpicture}
\end{center}
\caption{Finding a cell basis greedily. The nodes of $\Delta / G$ are marked by their facet vectors.}\label{fig:greedycellbasis}
\end{figure}
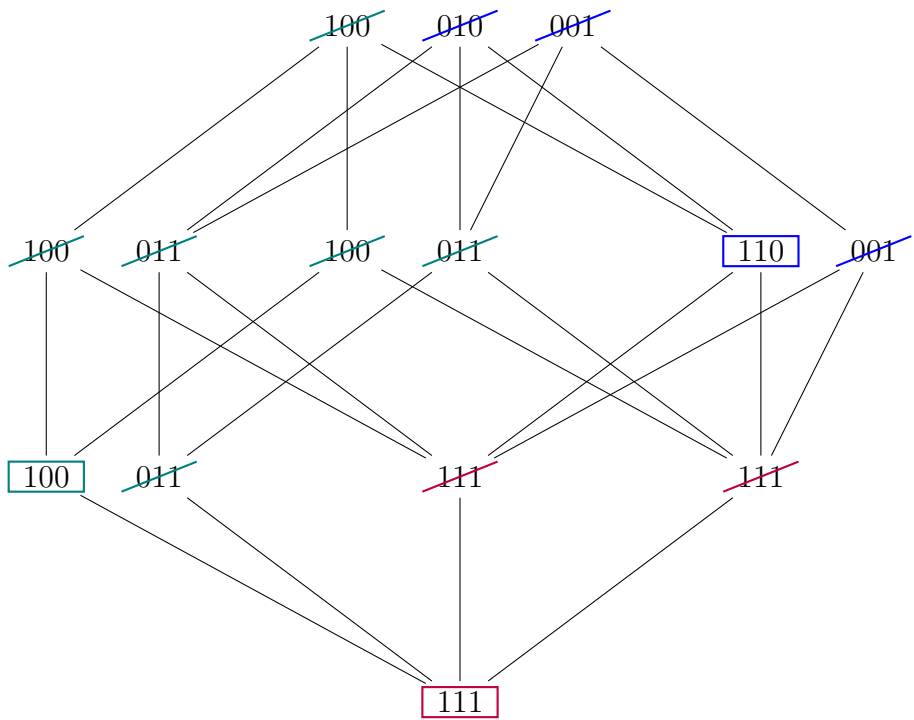

\section{When the invariant ring is not Cohen-Macau\-lay, and other questions}\label{sec:notCM}

In this section we collect many questions and conjectures that this chapter raises, beginning with the question of when the converse of our main result holds.

\subsection{Does Cohen-Macaulayness of $R^G$ imply that of $S^G$?}

In theorem \ref{thm:quotientCMforpermgroups}, we proved a topological theorem that determines when $S^G$ is a Cohen-Macaulay ring, based on whether $G=\Grr$. In section \ref{sec:mainresult}, we combined this with the fact that $S^G$'s Cohen-Macaulayness implies $R^G$'s to deduce that $G=\Grr$ implies $R^G$ is Cohen-Macaulay. The missing piece of this story is the question of whether $R^G$ can be Cohen-Macaulay but not $S^G$. We conjecture that it cannot.

\begin{conj}\label{conj:mainconj}
With $A=\Z$, $R^G$ is not Cohen-Macaulay unless $S^G$ is too (and therefore $G=\Grr$).
\end{conj}

This question fits into a broader context. By proposition \ref{prop:approximatehomomorphism}, $S$ is a sort of ``coarse approximation" of $R$. This is parallel to the theory of {\em algebras with straightening laws}, also known as {\em ordinal Hodge algebras} (\cite{hodge}, \cite{eisenbud2}, \cite[III.6]{cca}, \cite[Ch. 7]{brunsherzog}). An algebra with straightening law (ASL) is a ring $B$ with a set of generators (over some specified coefficient ring $A$, usually a field) identified with the elements of a finite poset $P$, such that (1) the monomials supported on chains of $P$ form an $A$-basis, and (2) when products of incomparable elements $x,y\in P$ are expressed on this basis, each term has a factor $z\in P$ lying below both $x$ and $y$. Every ASL $B$ is associated to a {\em discrete ASL} $\overline{B}$, which is nothing but the Stanley-Reisner ring of the poset. In fact, $R$, $S$, and $S^G \cong A[\Delta / G]$ are all ASLs, and $S$ is the discrete ASL associated to $R$.\footnote{The ASL structures of $R$, $S$, and $S^G \cong A[\Delta / G]$ require working with the {\em order duals} of the posets we have so far considered, i.e. the same underlying sets but with all order relations reversed. For example, $R$ and $S$ become ASLs on the order dual of $B_n\setminus\{\emptyset\}$ by associating $U\in B_n\setminus\{\emptyset\}$ with $y_U\in S$ and $\garsia(y_U)\in R$. Then in $R$, we have 
\[
\garsia(y_{1})\garsia(y_2) = x_1x_2 = \garsia(y_{12})
\]
and in order to meet the definition of an ASL we need to insist that $12$ is below instead of above both $1$ and $2$ in $B_n\setminus\{\emptyset\}$.}

The theory of ASLs is designed so that good properties like Cohen-Macaulayness pass from $\overline{B}$ to $B$, in parallel to how, in the theory we have developed here, Cohen-Macaulayness passes from $S^G$ to $R^G$. The reverse question of how $B$ constrains $\overline{B}$ turns out to be a hard open problem (Naoki Terai, personal communication). At present, although there is no known proof that $B$'s Cohen-Macaulayness implies $\overline{B}$'s, there does not appear to be any known counterexample. The best available current result seems to be Mitsuhiro Miyazaki's theorem that if $B$ is Cohen-Macaulay and $\overline{B}$ is Buchsbaum, then $\overline{B}$ is Cohen-Macaulay (\cite{miyazaki}).\footnote{Terai claimed in \cite{terai} that the depth of $\overline{B}$ cannot differ from the depth of $B$ by more than $1$. However, Miyazaki (\cite{miyazaki}) found a gap in Terai's proof, though not a counterexample.}

On the strength of this parallel, it is at least plausible that $R^G$'s Cohen-Macaulay\-ness implies $S^G$'s, although it might be very hard to prove.

The situation for our conjecture is slightly better than this, though. The relationship between $R^G$ and $S^G$ is more rigid than that between a general ASL $B$ and its discrete counterpart $\overline{B}$, since there are many ASLs built on the same poset. In another equally rigid situation, the desired implication is present. Namely, the Stanley-Reisner ring of a boolean complex $K$ (definition \ref{def:SRringofboolean}) is an ASL built on the order dual of the face poset, without the minimal element, and the corresponding discrete ASL is the Stanley-Reisner ring of this poset. As the Stanley-Reisner ring of a poset is isomorphic to that of its order dual (since reversing the ordering does not affect whether a pair of elements is incomparable), Duval's theorem (\ref{thm:duval}) tells us that this particular ASL and its discrete counterpart are simultaneously Cohen-Macaulay. Perhaps the situation with $R^G$ and $S^G$ is analogous.

Furthermore, the conjecture is true in some special cases:

{\em It is true for groups of low degree.} By a computer calculation we have included in appendix \ref{sec:smalldegree}, for permutation groups $G$ of degree $\leq 6$, $R^G$ is Cohen-Macaulay exactly when $G=\Grr$ and thus when $S^G$ is Cohen-Macaulay.

{\em It is true for p-groups.} By a result of Kemper (\cite{kemper99}, Corollary 3.7), which applies in the setting of linear $p$-groups over a field of characteristic $p$, Cohen-Macaulayness implies that the group is generated by elements fixing subspaces of codimension $\leq 2$. Cohen-Macaulayness of $R^G$ over $\Z$ implies Cohen-Macaulayness over $\F_p$, whereupon Kemper's result implies $G=\Grr$ (and then $S^G$ is Cohen-Macaulay as well).

\subsection{Questions about cell bases and partitionings}

Above, we proposed {\em cell bases} for a balanced boolean complex as an object of combinatorial study. There are many basic questions to answer.

\begin{conj}
The procedure described in the last paragraph of section \ref{subsec:cellbases} always leads to a cell basis, when one exists.
\end{conj}

This would be a combinatorial analogue to the fact that one can select a basis for a vector space inductively simply by picking new vectors not in the span of those already picked.

%\begin{conj}
%If $K$ is Cohen-Macaulay over $A=k$ or $\Z$, then a cell basis over $A$ exists.
%\end{conj}

%attend
%We expect this to be true. By \ref{}, it is equivalent to the statement that if there is a basis for $A[K]$ at all, then the basis can be chosen from among the elements of $K$. Garsia and Stanton stated the analogous claim for their setting, in passing, as a well-known fact (\cite[p.~113]{garsiastanton}), although we were not able to find a reference or an easy proof.

\begin{conj}
Any two cell bases over $\Q$ have incidence matrices with the same determinant.
\end{conj}

We have very preliminary computational evidence suggesting this.

If cell bases are replaced with the minimal elements from the intervals of a partitioning, this statement fails for boolean complexes -- even for simplicial complexes, and even in the balanced case. 

However, in the balanced case, it is at least true that if a partitioning's incidence matrix is nonsingular over $\Q$, then its minimal elements form a cell basis over $\Q$. Thus this conjecture would imply that if a balanced boolean complex is Cohen-Macaulay, the incidence matrices of any two partitionings have determinants that are equal unless one of them is zero.

\subsection{Questions about shellings and sphere quotients}

Lange's theorem, \ref{thm:lange}, 
is an extremely elegant if-and-only-if statement. Nonetheless, as mentioned in section \ref{sec:quotientsofspheres}, the ``if" direction was proven via a delicate and ad-hoc induction, based on a full classification of rotation-reflection groups and a case analysis of the corresponding quotients of $\R^n$. %The ``only if" direction is much easier.

In this way, its history mirrors that of the Chevalley-Shepard-Todd theorem. The ``if" direction (a pseudoreflection group has a polynomial invariant ring) was first proven by Shepard and Todd via a complete classification of pseudoreflection groups and a case-by-case determination of their invariant rings. The ``only if" direction followed much more easily via a combinatorial argument based on Hilbert series (see \cite{stanley4}, \S 4).

Chevalley then came up with a uniform, classification-free proof of the ``if" direction. Lange's result seems ripe for a similar treatment. 

\begin{question}\label{q:classificationfreeLange}
Is there a classification-free proof of Lange's theorem?
\end{question}

Lange himself has proposed a program for such a proof in \cite{lange2}, \S 7. We propose another program below.

Prior to Lange's theorem, in the known cases in which the complex $\Delta / G$ was Cohen-Macaulay over $\Z$, this was proven via an explicit shelling of $\partial \Delta / G$. For Young subgroups in $S_n$ and their counterparts inside other finite Coxeter groups, a family of shellings is due to Garsia and Stanton (\cite{garsiastanton}, sections 7, 8). For $A_n$, and for $S_n$ diagonally embedded in $S_n\times S_n\subset S_{2n}$, there is a shelling due to Reiner (\cite{reiner92}, chapter 4) -- this also works for other finite Coxeter groups. For the wreath product $S_2\wr S_n\subset S_{2n}$, a shelling is due to Patricia Hersh (\cite{hersh}).

Although it is possible for a complex to be Cohen-Macaulay over $\Z$ without being shellable, practitioners in the fields of combinatorial commutative algebra and poset topology have observed that in practice, one usually can prove Cohen-Macaulayness by proving shellability (\cite[p.~68]{wachs}, and Josephine Yu, personal communication). Thus, it is natural to ask:

\begin{question}\label{q:shellingalways}
If $G=\Grr$ for a permutation group, is $\partial\Delta / G$ shellable? Is there a uniform proof?
\end{question}

We have some computational evidence (omitted) that the quotient of the Coxeter complex associated to a finite Coxeter group by a subgroup generated by elements of length $\leq 2$ can always be given a shelling based on length order which generalizes the shelling given for Young subgroups by Garsia and Stanton. We do not have a proof at this time.

If the answer to both of these questions turns out to be {\em yes}, this suggests another approach to question \ref{q:classificationfreeLange}. Given an arbitrary rotation-reflection group $G\in O_n(\R)$, one can take a point $p\in S^{n-1}$ that is not fixed by any nontrivial element of $G$, and consider the Voronoi diagram in $S^{n-1}$ of its $G$-orbit. This gives a tiling of $S^{n-1}$ by congruent convex polytopes on which $G$ acts freely and transitively. In suitable circumstances, the boundaries of the Voronoi cells give a regular CW complex structure that generalizes the Coxeter complex in the case that $G$ is a finite Coxeter group. One can pass to the barycentric subdivision in order to guarantee that the result is a balanced simplicial complex on which $G$ acts balancedly, so that the quotient by $G$ is a balanced boolean complex. Suppose the CW complex structure on $S^{n-1}$ so constructed is called $\Gamma$.

\begin{question}
Is there a uniform proof of shellability of $\Gamma / G$?
\end{question}

If the answer is {\em yes}, then this goes most of the way toward a uniform proof of Lange's theorem, since the fact that the codimension-one faces of $\Gamma$ are incident to two facets will imply $\Gamma / G$ is a pseudomanifold, and shellability of a pseudomanifold (possibly with boundary) implies it is a sphere or a ball.

An alternative, potentially more flexible, approach to proving Lange's theorem using $\Gamma / G$, suggested by Robert Young (personal communication), would be to apply the discrete Morse theory of Robin Forman (\cite{forman1}, \cite{forman2}). This is another theory designed for deducing topological results about a cell complex from its combinatorics. The relationship between shelling and discrete Morse functions is discussed in \cite{babsonhersh}. Although discrete Morse theory is primarily designed to make statements about homotopy type rather than homeomorphism class, it could potentially be employed in Lange's proposed program, which calls for proving that $S^{n-1}/G$ has the homology of a sphere or point (see \cite{lange2}, \S 7).

\section{Appendix: Magma calculations}

\subsection{Testing Cohen-Macaulayness in small degree}\label{sec:smalldegree}

The following routines describe the Magma calculation we used to verify conjecture \ref{conj:mainconj} for permutation groups of degree $<6$.

The following three functions input a degree $n$ and output, in a format Magma understands, the cycle structures of transpositions, double transpositions, and three-cycles of degree $n$:

\begin{verbatim}
Transp := function(n)
    if n gt 2 then
        return [ <2,1>, <1, n-2> ];
    elif n eq 2 then
        return [ <2,1> ];
    else
        return [];
    end if;
end function;

DoubleTransp := function(n)
    if n gt 4 then
        return [ <2,2>, <1, n-4> ];
    elif n eq 4 then
        return [ <2,2> ];
    else
        return [];
    end if;
end function;

ThreeCycle := function(n)
    if n gt 3 then
        return [ <3,1>, <1, n-3> ];
    elif n eq 3 then
        return [ <3,1> ];
    else
        return [];
    end if;
end function;
\end{verbatim}

The function \texttt{SubgroupGenByCycleStructures} inputs a permutation group $G$ and a list of cycle structures $L$, and outputs the subgroup of $G$ generated by the elements having cycle structures in $L$.

\begin{verbatim}
SubgroupGenByCycleStructures := function(G,L)
    S := [];
    for g in G do
        if CycleStructure(g) in L then
            Append(~S,g);
        end if;
    end for;
    return sub< G | S >;
end function;
\end{verbatim}

The function \texttt{TestCMness} inputs a permutation group $G$, identifies the primes $p$ dividing $[G:\Grr]$, and tests Cohen-Macaulayness of $R^G\otimes \F_p$ for each of these primes until it finds a failure of Cohen-Macaulayness. It uses the fact (\cite{derksenkemper}, Theorem 3.7.1) that the size of a set of minimal module generators (``secondary invariants") for $R^G\otimes\F_p$ as a module over the subring generated by an h.s.o.p. (``primary invariants") is given by the product of the h.s.o.p.'s degrees divided by the group order if and only if $R^G\otimes\F_p$ is Cohen-Macaulay. It outputs whether $R^G$ is Cohen-Macaulay and also the index of $\Grr$ in $G$.

\begin{verbatim}
TestCMness := function(G)
    n := Degree(G);
    L := [Transp(n), DoubleTransp(n), ThreeCycle(n)];
    H := SubgroupGenByCycleStructures(G,L);
    Ind := Index(G,H);
    P := PrimeFactors(Ind);
    IsCM := true;
    for p in P do
        F := GaloisField(p);
        R := InvariantRing(G,F);
        Pri := PrimaryInvariants(R);
        Predicted := 1;
        for i in [1..#Pri] do
            Predicted *:= TotalDegree(Pri[i]);
        end for;
        Predicted := Predicted / Order(G);
        if Predicted ne #SecondaryInvariants(R) then
            IsCM := false;
            break p;
        end if;
    end for;
    return IsCM, Ind;
end function;
\end{verbatim}

This function outputs \texttt{true} if the Cohen-Macaulayness of $R^G$ coincides with the statement $G=\Grr$ and \texttt{false} otherwise.

\begin{verbatim}
DoTheyMatch := function(G)
    CM, Ind := TestCMness(G);
    Generated := (Ind eq 1);
    Match := (Generated eq CM);
    return Match;
end function;
\end{verbatim}

The function \texttt{TestTheGroup} inputs a group and calls \texttt{DoTheyMatch} for each of its conjugacy classes of subgroups. Applied to $S_n$, it allows us to see if Cohen-Macaulayness of $R^G$ always coincides with $G=\Grr$ for permutation groups of degree $n$. If there are subgroups for which $R^G$ is Cohen-Macaulay but $G\neq \Grr$, it outputs these subgroups.

\begin{verbatim}
function TestTheGroup(G)
    Cla := SubgroupLattice(G);
    NaughtyGroups := [];
    m := #Cla;
    for i in [1..m] do
        H := Cla[i];
        print i;
        if DoTheyMatch(H) then
            print H, "is good.";
        else
            print H, "is bad.";
            Append(~NaughtyGroups,H);
        end if;
    end for;
    if #NaughtyGroups eq 0 then
        print "Everything matched!";
    else
        print "Uh-oh! There were", #NaughtyGroups, "naughty subgroups!";
    end if;
    return NaughtyGroups;
end function;
\end{verbatim}

We have run this code for $n\leq 6$. For $n>6$ it is infeasible on our equipment. Here is the tail end of the output for $n=6$:

\begin{verbatim}
54
Permutation group H acting on a set of cardinality 6
Order = 120 = 2^3 * 3 * 5
    (1, 6)
    (1, 4, 5)(3, 6)
is good.
55
Permutation group H acting on a set of cardinality 6
Order = 360 = 2^3 * 3^2 * 5
    (1, 6)(4, 5)
    (1, 2)(3, 6, 4, 5)
is good.
56
Symmetric group G acting on a set of cardinality 6
Order = 720 = 2^4 * 3^2 * 5
    (1, 2, 3, 4, 5, 6)
    (1, 2)
is good.
Everything matched!
[]
\end{verbatim}

\subsection{Constructing $\partial\Delta / G$}\label{sec:cellconstruction}

The following routines build the boolean complex $\partial\Delta / G$ as a combinatorial object. They rely on a beautiful, purely group-theoretic description of the face poset of $\partial \Delta / G$ in terms of double cosets in $S_n$:

\begin{definition}
Let $s_i = (i,i+1)\in S_n$ for $i=1,\dots,n-1$. The $s_i$ form a system $\mathcal{S}$ of Coxeter generators for $S_n$. For each finite subset $J\subset\mathcal{S}$, let $\langle J\rangle$ be the subgroup of $S_n$ generated by the elements in $J$. For each $J$ and each $\pi\in S_n$, form the double coset $G\pi\langle J\rangle$ in $S_n$. Let $\Sigma(n,G)$ be the poset consisting of all distinct ordered pairs $(G\pi\langle J\rangle, J)$, ordered by reverse inclusion of both factors, i.e.
\[
(G\pi\langle J\rangle, J) \leq (G\pi'\langle J'\rangle, J') \Leftrightarrow G\pi\langle J\rangle \supset G\pi'\langle J'\rangle \text{ and } J\supset J'.
\]
\end{definition}

\begin{lemma}
The face poset of $\partial\Delta /G$, with minimal face appended, is order-isomorphic to $\Sigma(n,G)$.
\end{lemma}

\begin{remark}
This description of $\partial \Delta / G$ was made use of in \cite{garsiastanton}, in the more general setting of finite Coxeter groups. It was made explicit in \cite{reiner92}, but the requirement that $J\supset J'$ was missing from the description of the order relation on the $(G\pi\langle J\rangle, J)$'s. This was corrected in \cite{babsonreiner}. The lemma is a straightforward consequence of the special case when $G$ is the trivial subgroup, which is explained in detail in \cite{brown}, section I.5H.
\end{remark}

This established, here is the code:

The function \texttt{LexLower} inputs two permutations of the same degree and compares them lexicographically. If the first is greater, it outputs 0; if less, 1; if equal, 2.

\begin{verbatim}
LexLower := function(perm1,perm2)
    assert Degree(Parent(perm1)) eq Degree(Parent(perm2));
    t := 1;
    while t le Degree(Parent(perm1)) do
        if t^perm1 lt t^perm2 then
            return 1;
        elif t^perm1 gt t^perm2 then
            return 0;
        else
            t := t+1;
        end if;
    end while;
    if t gt Degree(Parent(perm1)) then
        return 2;
    end if;
end function;
\end{verbatim}

The function \texttt{LowestCosetMap} inputs a permutation group $G$, and outputs the right transversal consisting of the lexicographically lowest representative in each right coset, as a sequence, and a map from $S_n$ to this transversal mapping an arbitrary element to the lexicographically lowest representative of its coset.

\begin{verbatim}
LowestCosetMap := function(G);
    n := Degree(G);
    S := Generic(G);
    Transv, Phi := Transversal(S,G);
    LowestTransv := [];
    for i in [1..#Transv] do
        iLowest := Transv[i];
        for g in G do
            if LexLower(g*Transv[i],iLowest) eq 1 then
                iLowest := g*Transv[i];
            end if;
        end for;
        Append(~LowestTransv,iLowest);
    end for;
    Translato := map< Set(Transv) -> Set(LowestTransv) | 
                  [Transv[i] -> LowestTransv[i]: i in [1..#Transv]]>;
    return LowestTransv, Phi*Translato;
end function;
\end{verbatim}

The function \texttt{WhichCosets}, given a group $S$, two subgroups $G$ and $Y$, a right transversal of $G$ in $S$, and an element $b$ of $S$, returns the set of elements of the transversal that represent the right cosets of $G$ that are contained in the double coset $GbY$.

\begin{verbatim}
WhichCosets := function(S,G,Y,Transv,b)
    Reps := {};
    B := Base(S);
    BI := DoubleCosetCanonical(S,G,b,Y: B:=B);
    for g in Transv do
        if DoubleCosetCanonical(S,G,g,Y: B:=B) eq BI then
            Include(~Reps, g);
        end if;
    end for;
    return Reps;
end function;
\end{verbatim}

The function \texttt{CellConstruction}, given a permutation group $G$ of degree $n$, constructs the quotient of the Coxeter complex of $S_n$ by $G$. The output is expressed as a sequence of ordered pairs \texttt{< j , Reps >} each representing a double coset $GbY$, where \texttt{j} is a set of integers indexing the Coxeter generators that generate $Y$, and \texttt{Reps} lists the lexicographically minimal representatives of the cosets of $G$ that are contained in $GbY$. Below, we also refer to the output of this function as a CellConstruction.

\begin{verbatim}
CellConstruction := function(G)
    n := Degree(G);
    S := Generic(G); 
    Transv, Phi := LowestCosetMap(G);
    gens := [S!(i,i+1) : i in {1..n-1}];
    f := map< {1..n-1} -> S | [i -> gens[i]: i in {1..n-1}] >;
    P := [Subsets({1..n-1},k): k in [0..n-1]];
    CellCollection := [];
    ParabolicSubgroups := [];
    for i in [1..n] do
        Append(~CellCollection,AssociativeArray(P[i]));
        Append(~ParabolicSubgroups,AssociativeArray(P[i]));
        for j in P[i] do
            ParabolicSubgroups[i][j] := sub< S | f(j) >;
            CellCollection[i][j] := DoubleCosetRepresentatives(S,G,
                                     ParabolicSubgroups[i][j]);
        end for;
    end for;
    OutputArray := [];
    for i in [1..n] do
        for j in P[i] do
            for k in CellCollection[i][j] do
                Reps := WhichCosets(S,G,ParabolicSubgroups[i][j],
                                     Transv,k);
                Append(~OutputArray,<j,Reps>);
            end for;
        end for;
    end for;
    return OutputArray;
end function;
\end{verbatim}

\subsection{G\"{o}bel's algorithm}\label{sec:gobelmethod}

The routines in this subsection use those in the previous one to implement G\"{o}bel's algorithm. An example calculation is given at the end.

The function \texttt{ReorderList} inputs a sequence of length $n$ and a permutation of degree $n$, and outputs the result of applying the permutation to the sequence.

\begin{verbatim}
ReorderList := function(TargetSeq, Perm)
    n := Degree(Parent(Perm));
    assert n eq #TargetSeq;
    Answer := TargetSeq;
    for i in [1..n] do
        Answer[i^Perm] := TargetSeq[i];
    end for;
    return Answer;
end function;
\end{verbatim}

The function \texttt{DetectOrder} inputs a sequence of integers, sorts the sequence into ascending order, and outputs the reordered sequence and a permutation that returns it to the original order.

\begin{verbatim}
DetectOrder := function(IntegerSeq)
    WorkingSeq := IntegerSeq;
    n := #IntegerSeq;
    FinalOrder := [];
    TruncatedRanks := [];
    while #WorkingSeq gt 0 do
        MinGuy, MinIndex := Min(WorkingSeq);
        Remove(~WorkingSeq,MinIndex);
        Append(~TruncatedRanks,MinIndex);
        Append(~FinalOrder,MinGuy);
    end while;
    for i in [0..n-2] do
        for j in [1..n-i-1] do
            if TruncatedRanks[n-i] ge TruncatedRanks[n-i-j] then
                TruncatedRanks[n-i] +:= 1;
            end if;
        end for;
    end for;
    g := SymmetricGroup(n)!TruncatedRanks;
    assert FinalOrder eq ReorderList(IntegerSeq, g^-1);
    return g, FinalOrder;
end function;
\end{verbatim}

The function \texttt{FindSymMultipliers} inputs a monomial, assesses whether it is special, and if it is, outputs the fine grading. If it is not, it outputs the associated special monomial (as a sequence of its exponents), and the monomial in the elementary symmetric polynomials that would have to be multiplied to reach this one from its associated special monomial (also as a sequence of exponents).

\begin{verbatim}
FindSymMultipliers := function(InMonom)
    OrigExps := Exponents(InMonom);
    g, Exps := DetectOrder(OrigExps);
    n := #Exps;
    Sort(~Exps);
    Reverse(~Exps);
    IsSpecial := true;
    OutDegs := {};
    OutExps := [0 : i in [1..n]];
    dif := [0: i in [1..n]];
    for i in [1..n-1] do
        dif[i] := Exps[i] - Exps[i+1];
    end for;
    dif[n] := Exps[n];
    for i in [1..n-1] do
        if dif[i] gt 1 then
            IsSpecial := false;
            break i;
        elif dif[i] eq 1 then
            Include(~OutDegs, i);
        else
            assert dif[i] eq 0;
        end if;
    end for;
    if dif[n] gt 0 then
        IsSpecial := false;
    end if;
    if IsSpecial then
        return IsSpecial, OutDegs, _;
    else
        for i in [1..n-1] do
            OutExps[i] := Max(0,dif[i]-1);
        end for; 
        OutExps[n] := dif[n];
        AmountTaken := [0: i in [1..n]];
        for i in [1..n] do
            for j in [i..n] do
                AmountTaken[i] +:= OutExps[j];
            end for;
        end for;
        AssociatedSpecial := [ Exps[i] - AmountTaken[i] : 
                                     i in [1..n] ];
        Reverse(~AssociatedSpecial);
        AssociatedSpecial := ReorderList(AssociatedSpecial, g);
        return IsSpecial, OutExps, AssociatedSpecial;
    end if;
end function;
\end{verbatim}

The function \texttt{CellToOrbitMonomial} turns an individual face in $\partial \Delta / G$ into the corresponding orbit monomial. It inputs the face, specified as an entry in a CellConstruction (see subsection \ref{sec:cellconstruction}), the group $G$, and the ambient polynomial ring $R$, and returns the orbit monomial as an element of $R$.

\begin{verbatim}
CellToOrbitMonomial := function(CellPair,G,R)
    assert Degree(G) eq Rank(R);
    assert CellPair[2] subset Set(Generic(G));
    n := Degree(G);
    assert CellPair[1] subset {1..n-1};
    y := [ R | 1: i in [1..n]];
    for i in [1..n] do
        for j in [1..i] do
            y[i] *:= R.j;
        end for;
    end for;
    StartMonomial := R!1;
    for i in ({1..n-1} diff CellPair[1]) do
        StartMonomial *:= y[i];
    end for;
    CellMonomial := StartMonomial^(Rep(CellPair[2])^(-1));
    for g in CellPair[2] do
        assert CellMonomial^G eq (StartMonomial^(g^(-1)))^G;
    end for;
    OrbitMonomial := R!0;
    for x in CellMonomial^G do
        OrbitMonomial +:= x;
    end for;
    return OrbitMonomial;
end function;
\end{verbatim}

The function \texttt{CellComplexToOrbitMons} takes in $\partial \Delta / G$ (given as a CellConstruction) and returns the sequence of orbit monomials corresponding to each of the cells.

\begin{verbatim}
CellComplexToOrbitMons := function(CellConstr, G, R)
    MonomialArray := [CellToOrbitMonomial(CellConstr[i], G, R): 
                                   i in [1..#CellConstr] ];
    return MonomialArray;
end function;
\end{verbatim}

The function \texttt{SymDegLex} compares two monomials with respect to the degree lexicographic order on their shapes. It outputs the comparison in the form of a string: \texttt{"gt"}, \texttt{"lt"}, or \texttt{"eq"}.

\begin{verbatim}
SymDegLex := function(Monom1,Monom2)
    Exps1 := Exponents(Monom1); Exps2 := Exponents(Monom2);
    Sort(~Exps1); Sort(~Exps2);
    Reverse(~Exps1); Reverse(~Exps2);
    n := #Exps1;
    assert n eq #Exps2;
    Output := "eq";
    if Degree(Monom1) gt Degree(Monom2) then
        Output := "gt";
    elif Degree(Monom1) lt Degree(Monom2) then
        Output := "lt";
    else
        for i in [1..n] do
            if Exps1[i] gt Exps2[i] then
                Output := "gt";
                break i;
            elif Exps1[i] lt Exps2[i] then
                Output := "lt";
                break i;
            end if;
        end for;
    end if;
    return Output;
end function;
\end{verbatim}

The function \texttt{IdentifySymLexHighest} inputs a sequence of monomials and outputs an element in the sequence with the degree-lexicographically highest shape.

\begin{verbatim}
IdentifySymLexHighest := function(Monoms)
    Highest := Monoms[1];
    for i in [2..#Monoms] do
        if SymDegLex(Monoms[i],Highest) eq "gt" then
            Highest := Monoms[i];
        end if;
    end for;
    return Highest;
end function;
\end{verbatim}

The function \texttt{IdentifySymLexLeading} identifies a monomial of degree lexicographically leading shape in a polynomial.

\begin{verbatim}
IdentifySymLexLeading := function(InputPoly)
    Monoms := Monomials(InputPoly);
    return IdentifySymLexHighest(Monoms);
end function;
\end{verbatim}

The function \texttt{GobelMethod} implements a version of G\"{o}bel's algorithm to calculate a representation on special orbit monomials for a given invariant polynomial. It inputs the group $G$, the target polynomial \texttt{InvarPoly} to be represented, the ambient polynomial ring $R$, a second polynomial ring \texttt{SymR} which will house the coefficients of the representation, and a CellConstruction. The format of the output is a function $F$ from the set of special orbit monomials to \texttt{SymR}, that, for each special orbit monomial, gives its coefficient (as a polynomial in the elementary symmetric polynomials) in the representation.

\begin{verbatim}
GobelMethod := function(InvarPoly, G, R, SymR, CellComplex)
    assert IsInvariant(InvarPoly, G);
    n := Degree(G);
    assert Rank(R) eq n;
    assert Rank(SymR) eq n;
    s := [ElementarySymmetricPolynomial(R,i): i in [1..n]];
    SpecialOrbitMons := CellComplexToOrbitMons(CellComplex, G, R);
    SymCoefficients := [SymR!0 : i in [1..#SpecialOrbitMons]];
    f := hom< SymR -> R | [ElementarySymmetricPolynomial(R,i): 
                                   i in [1..n]] >;
    Remainder := InvarPoly;
    while Remainder ne 0 do
        m := IdentifySymLexLeading(Remainder);
        C := MonomialCoefficient(Remainder, m);
        Orbitm := R!0;
        for x in m^G do
            Orbitm +:= x;
        end for;
        IsSpecial, Multipliers, AssociatedSpecial := 
                                    FindSymMultipliers(m);
        if IsSpecial eq false then
            AssociatedSpecial := Monomial(R,AssociatedSpecial);
            OrbitSpecial := R!0;
            for x in AssociatedSpecial^G do
                OrbitSpecial +:= x;
            end for;
            Multipliers := Monomial(SymR,Multipliers);
            SymCoefficients[Index(SpecialOrbitMons, OrbitSpecial)] 
                                              +:= SymR!C*Multipliers;
            Remainder -:= C*f(Multipliers)*OrbitSpecial;
        else
            SymCoefficients[Index(SpecialOrbitMons, Orbitm)] 
                                              +:= SymR!C;
            Remainder -:= C*Orbitm;
        end if;
    end while;
    F := map< Set(SpecialOrbitMons) -> SymR | [<SpecialOrbitMons[i],
               SymCoefficients[i]>: i in [1..#SpecialOrbitMons] ] >;
    IsItRight := R!0;
    for m in SpecialOrbitMons do
        IsItRight +:= f(F(m))*m;
    end for;
    assert IsItRight eq InvarPoly;
    return F;
end function;
\end{verbatim}

Here is an example calculation. We choose $G = A_3 \subset S_3$.  First we construct the ambient polynomial ring $R$, the coefficient ring \texttt{SymR}, the group $G$, the cell complex $\partial \Delta / G$ (which is homeomorphic to the real projective plane in this case) as a CellConstruction, and the list of special orbit monomials:

\begin{verbatim}
R<[x]> := PolynomialRing(IntegerRing(), 3);
SymR<[s]> := PolynomialRing(IntegerRing(), 3);
G := PermutationGroup< 3 | (1,2,3)>;
CellComplex := CellConstruction(G);
OrbitMons := CellComplexToOrbitMons(CellComplex, G, R);
\end{verbatim}

Then we construct a miscellaneous invariant polynomial to be represented. We have chosen the orbit monomial of $x_1x_3^4$, which is $x_1^4x_2 + x_2^4x_3 + x_1x_3^4$.

\begin{verbatim}
m := x[1]*x[3]^4;
Orbitm := R!0;
for b in m^G do
    Orbitm +:= b;
end for;
\end{verbatim}

Then we use \texttt{GobelMethod} to calculate and then display a representation as an $R^{S_n}$-linear combination of special orbit monomials.

\begin{verbatim}
F := GobelMethod(Orbitm, G, R, SymR, CellComplex);

print Orbitm;
print "= sum of";
for SpecialMon in OrbitMons do
    if F(SpecialMon) ne 0 then
        print F(SpecialMon), "* (", SpecialMon, "),";
    end if;
end for;
\end{verbatim}

The output looks like this:

\begin{verbatim}
x[1]^4*x[2] + x[1]*x[3]^4 + x[2]^4*x[3]
= sum of
s[1]^2 - 2*s[2]
* ( x[1]^2*x[2] + x[1]*x[3]^2 + x[2]^2*x[3]
),
-s[2]
* ( x[1]^2*x[3] + x[1]*x[2]^2 + x[2]*x[3]^2
),
-2*s[3]
* ( x[1]*x[2] + x[1]*x[3] + x[2]*x[3]
),
s[1]*s[3]
* ( x[1] + x[2] + x[3]
),
\end{verbatim}

This expresses the fact that
\[
Gx_1x_3^4 = (\sigma_1^2 - 2\sigma_2)(Gx_1^2x_2) -\sigma_2(Gx_1^2x_3) - 2\sigma_3(Gx_1x_2) + \sigma_1\sigma_3(Gx_1)
\]
where we are abbreviating by $Gm$ the orbit monomial of a monomial $m$. %This example illustrates the redundancy of the complete set of special orbit monomials as a generating set: indeed, $Gx_1$ is already in $R^{S_n}$, so it would never appear in an actual basis.

%\section{Appendix: Shelling quotients of Coxeter complexes}

\section{Appendix: Algebraic lemmas}\label{appendix:algebraiclemmas2}

%Here we collect some elementary algebraic results we used in the chapter that are not part of our main story.

\begin{lemma}\label{lem:quotientstabilizer}
Let $G$ be a group acting on a set $X$, let $N$ be a normal subgroup, and let $\overline{X} = X/N$ be the set of $N$-orbits. Let $\overline{x} \in\overline{X}$ be any element, and let $x$ be any of its preimages in $X$. Then the preimage in $G$ of $(G/N)_{\overline{x}}$ is $G_xN$, so the canonical map $G\rightarrow G/N$ maps the stabilizer $G_x$ onto $(G/N)_{\overline{x}}$.
\end{lemma}

\begin{proof}
If $g\in G$'s image in $G/N$ stabilizes $\overline{x}$, it means it sends $x$ to another preimage $x'$ of $\overline{x}\in \overline{X}$. Then there exists $n\in N$ with $n(x) = x'$ as well, so that $n^{-1}g$ stabilizes $x$ in $X$, i.e. it is $\in G_x$, so that $g\in NG_x = G_xN$. Conversely, if $g\in G_xN$, then $g=g'n$ with $g'\in G_x$ and thus its image in $G/N$ stabilizes $\overline{x}$.
\end{proof}

\begin{definition}
Let $\mathscr{M}$ be a commutative monoid, written additively with identity element $0$. We will say $\mathscr{M}$ is \textbf{positive} if it is cancellative and has the additional property that for any $a,b\in \mathscr{M}$, $a+b=0$ implies that $a=b=0$.
\end{definition}

\begin{example}
$(\N^n,+)$ is a positive monoid.
\end{example}

\begin{lemma}
A positive monoid $\mathscr{M}$ has a canonical partial order given by $a\leq b$ if there exists $c\in\mathscr{M}$ with $a+c=b$.
\end{lemma}

\begin{proof}
Transitivity follows from associativity of $\mathscr{M}$ and reflexivity from the existence of $0\in\mathscr{M}$. Antisymmetry is a consequence of the positivity: $a\leq b$ and $b\leq a$ implies the existence of $c,c'$ with $a+c = b,\; b+c'=a$, so that 
\[
(a+c) + c' = a + (c+c') = a,
\]
whereupon $c+c'=0$ by cancellation, and then $c=c'=0$ by positivity, implying that $a=b$.
\end{proof}

\begin{definition}
We will say that a positive monoid $\mathscr{M}$ is \textbf{archimedean} if for any fixed $a\in\mathscr{M}$, there is an $\ell\in \N$ such that no $b\leq a$ can be expressed as a sum of $\ell$ nonzero elements of $\mathscr{M}$.
\end{definition}

\begin{example}
$(\N^n,+)$ is archimedean. $(\N\setminus\{0\},\times)$ is archimedean. $(\N,\times)$ is not archimedean, since for $a=0$, no $\ell$ can prevent $b^\ell$ from dividing $a$. (In fact, it is not even positive, since it is not cancellative.)
\end{example}

\begin{definition}
For a positive monoid $\mathscr{M}$, the ring $R$ is \textbf{$\mathscr{M}$-graded} if it has a direct sum decomposition
\[
R = \bigoplus_{a\in\mathscr{M}} R_a
\]
such that $R_aR_b\subset R_{a+b}$ for all $a,b\in\mathscr{M}$. Note that $\bigoplus_{a\neq 0} R_a$ is an ideal of $R$; call it $R_+$. Also, $R_0$ is a subring. An $R$-module $M$ with a direct sum decomposition
\[
M = \bigoplus_{a\in\mathscr{M}} M_a
\]
satisfying $R_aM_b\subset M_{a+b}$ for all $a,b\in\mathscr{M}$ will also be called \textbf{$\mathscr{M}$-graded}. The $R_a$ and $M_a$ are called the \textbf{homogeneous components} of $R$ and $M$, respectively,  and their elements are called \textbf{homogeneous elements}, as usual. The projections of an arbitrary element $x$ of $R$ or $M$ to their homogeneous components are called the {\em homogeneous components of $x$}. If $x\in R$ or $M$ is homogeneous, we will refer to the unique $a\in\mathscr{M}$ such that $x$'s image in $R_a$ or $M_a$ is nonzero, as usual, as $x$'s \textbf{degree}. A \textbf{homogeneous submodule} of $M$ is one that is the direct sum of its projections to each $M_a$, and a \textbf{homogeneous ideal} is a homogeneous submodule of the $R$-module $R$. Note $R_+$ is a homogeneous ideal.
\end{definition}

\begin{lemma}[Nakayama lemma for $\mathscr{M}$-graded modules]\label{lem:gradedNakayama}
If $\mathscr{M}$ is an archimedean, positive monoid, $R$ is an $\mathscr{M}$-graded ring, and $M$ is a $\mathscr{M}$-graded $R$-module satisfying
\[
R_+M = M,
\]
then $M=0$.
\end{lemma}

\begin{proof}
We will show that $M$ does not have any homogeneous nonzero elements. It will follow that it is zero since it is the direct sum of its homogeneous components. So suppose for a contradiction that $m\in M_a$ is nonzero. Because $\mathscr{M}$ is archimedean, there is an $\ell$ for which no sum $s$ of $\ell$ nonzero elements of $\mathscr{M}$ can satisfy $s+c = a$ for any $c\in\mathscr{M}$. But meanwhile,
\[
R_+^\ell M = R_+^{\ell-1}(R_+M) = R_+^{\ell - 1} M = \dots = R_+M = M.
\]
Thus there exists an expression for $m$ as a finite sum
\[
m = \sum_ir_{1i}\dots r_{\ell i} n_i
\]
for a finite set of $n_i$'s in $M$ and $r_{ij}$'s in $R_+$. By splitting each element into its homogeneous components and expanding out all the products, we may assume that all the $n_i$'s and $r_{ij}$'s are homogeneous.

But no $r_{i1}\dots r_{i\ell}n_i$ can lie in $M_a$, since each $r_{ij}$ is homogeneous and not in $R_0$, so its degree is a nonzero element of $\mathscr{M}$; if $r_{i1}\dots r_{i\ell}n_i$ lay in $M_a$ this would imply
\[
\deg r_{i1} + \dots + \deg r_{i\ell} + \deg n_i = a
\]
which contradicts the construction of $\ell$. We conclude $m$ cannot exist.
\end{proof}

\begin{definition}
For $a\in\mathscr{M}$, let $R(a)$ denote the $\mathscr{M}$-graded $R$-module given by $R(a)_b = R_c$ for $b\geq a$, with $c$ the element of $\mathscr{M}$ satisfying $a+b=c$ (it must be unique since $\mathscr{M}$ is cancellative); and $R(a)_b = 0$ for $b\ngeq a$.
\end{definition}

\begin{prop}\label{prop:gradedandfreeisgradedfree}
Let $\mathscr{M}$ be an archimedean positive monoid, and let $R$ be an $\mathscr{M}$-graded ring. If $R_0$ is a ring with the property that projective modules over it are free (for example a local ring or a p.i.d.), then any $\mathscr{M}$-graded $R$-module that is free as a module has a homogeneous basis.
\end{prop}

\begin{proof}
Suppose $M$ is a free $R$-module. By tensoring with $R_0 = R/R_+$, we obtain that $M/R_+M$ is a free $R_0$-module, and it inherits an $\mathscr{M}$-grading in the obvious way, since $R_+M$ is a homogeneous submodule.

The homogeneous components of $M/R_+M$ are direct summands of it and thus projective $R_0$-modules. It follows by the assumption about $R_0$ that they are free. Thus each one has a basis over $R_0$; combining these bases, one sees that $M/R_+M$ has a homogeneous basis.

Lift it to a set of homogeneous elements $B\subset M$. By construction, 
\[
RB + R_+M = M.
\]
Thus $M/RB$ satisfies
\[
R_+\left(\frac{M}{RB}\right) = \frac{R_+M}{RB} = \frac{M}{RB},
\]
so the $\mathscr{M}$-graded Nakayama lemma tells us $M/RB = 0$, i.e. $B$ generates $M$. Then we have a surjection from the free $\mathscr{M}$-graded module $R^B = \bigoplus_{x\in B} R(\deg x)$, obtained by sending the generator of the $R(\deg x)$ component to $x$:
\[
\varphi: R^B \rightarrow M.
\]
Because $B$'s image in $M/R_+M$ is a basis, it cannot have any relations mod $R_+$, i.e. it must be that $\ker\varphi\subset R_+ R^B$.

Since $M$ is a free (and thus projective) module by presumption, this surjection splits:
\[
R^B = \ker \varphi \oplus M'
\]
where $M'$ is mapped isomorphically to $M$ by $\varphi$. Note that $\ker \varphi$ is a graded submodule of $R^B$ because $\varphi$ is a degree-preserving map.

Now we have
\[
\ker \varphi \subset R_+R^B = R_+\left(\ker\varphi \oplus M'\right) = R_+\ker \varphi\oplus R_+ M'
\]
Now $\ker \varphi$ cannot meet $R_+M'\subset M'$ in a nonzero element, since $M'$ is mapped isomorphically to $M$ by $\varphi$. But it does contain $R_+\ker\varphi$ completely. Thus its containment in the direct sum $R_+\ker \varphi\oplus R_+ M'$ implies that it $=R_+\ker\varphi$.

Then, a second application of the $\mathscr{M}$-graded Nakayama lemma shows that $\ker\varphi = 0$, and therefore that $\varphi: R^B\rightarrow M$ is an isomorphism. Thus $B$ is a homogeneous basis for $M$.
\end{proof}

\begin{remark}
We first heard this argument from Manny Reyes in the context of an $\N$-graded algebra over a field. It was pointed out to us by Eric Wofsey that the argument also works for an $\N^n$-graded algebra over any ring such that projective modules are free, and the foregoing is our construction of an appropriate abstract setting for this result.
\end{remark}

%Lifting a basis lemma? (if I need it for my basis for Young subgroups?)

\section{Appendix: Huffman's theorem}\label{appendix:huffman}

The following classification theorem, due to W. Cary Huffman, gives a sense of the groups to which our main result, theorem \ref{thm:mainresult}, applies.

\begin{thm}[\cite{huffman}, Theorem 2.1]
Let $G$ be a transitive permutation group of degree $n$, satisfying $G=\Grr$. Then the following case analysis holds:
\begin{enumerate}
\item If $G$ contains a transposition and a three-cycle, $G=S_n$.
\item If $G$ does not contain a transposition but does contain a three-cycle, $G=A_n$.
\item If $G$ does not contain a three-cycle but does contain a transposition, then $n=2m$ and $G$ is the wreath product $S_2\wr S_m$.
\item If $G$ does not contain any three-cycles or transpositions, then we have one of the following:
\begin{enumerate}
\item $n=2m$ and $G$ is $A_n \cap S_2\wr S_m$.
\item $n=5$ and $G=D_{10}$.
\item $n=6$ and $G$ is the transitive embedding of $A_5$ in $S_6$.
\item $n=7$ and $G\cong PSL(2,7)$.
\item $n=8$ and $G\cong C_2^3 \rtimes PSL(2,7)\cong AGL(3,2)$.
\end{enumerate}
\end{enumerate}
\end{thm}

Note that the theorem does not encompass intransitive groups such as the Young subgroups, or $S_n$ diagonally embedded in $S_n\times S_n\subset S_{2n}$. However, an intransitive permutation group $G$ generated by its transpositions, double transpositions, and three-cycles has a surjective homomorphism, for each of its orbits $\Omega$, into $S_\Omega$, and the image is one of the groups on this list. Furthermore, transpositions and three-cycles can act in only one of the orbits $\Omega$ at a time, and double transpositions can act in at most two orbits. If they do this, they act as transpositions in each. Thus $G$'s image in any $S_\Omega$ splits off as a direct factor if it doesn't contain any transpositions, i.e. unless it is case $1$ or $3$ on the list above. So the theorem goes a long way toward completely describing permutation groups for which $G=\Grr$.

%\section{Appendix: direct proof of a group-theoretic corollary}

% \section{Appendix: Artin's and Galois' developments of Galois theory} 
%%%%% chap3 %%%%%%%%%%%%
% \newpage
% \input{chap3}
%%%% Input bibliography file %%%%%%%%%%%%%%%


\newpage
\begin{thebibliography}{99}\addcontentsline{toc}{section}{Bibliography}

% First I used this two lines, that import my bibliography from references.bib
%\bibliographystyle{ametsoc}
%\bibliography{../my_papers/references}

% When you get to your final version, comment out the above two lines and then
% copy and paste the content of your thesis.bbl
% it should like to what is below, plus a few lines at the begining.

\bibitem{artin} Artin, Michael, {\em Algebra}. Prentice-Hall, Upper Saddle River, New Jersey, 1991.

\bibitem{atiyahmacdonald} Atiyah, M. F. and I. G. MacDonald, {\em Introduction to Commutative Algebra}. Addison-Wesley Publishing Co., Reading, Massachusetts - Menlo Park, California - London - Don Mills, Ontario, 1969.

\bibitem{babsonhersh} Babson, Eric and Patricia Hersh, Discrete Morse functions from lexicographic orders. {\em Transactions of the American Mathematical Society} 357(2):509-534, 2004.

\bibitem{babsonreiner} Babson, Eric and Victor Reiner, Coxeter-like complexes. {\em Discrete Mathematics and Theoretical Computer Science (DMTCS)} 6(2):223--252, 2004.

\bibitem{barryward} Barry, Michael J. J. and Michael B. Ward, Simple groups contain minimal simple groups. {\em Publicacions Matem\`{a}tiques}, 41:411--415, 1997.

%\bibitem{batyrev1} Batyrev, Victor V., Non-archimedean integrals and stringy Euler numbers of log terminal pairs. {\em J. Euro. Math. Soc.}, 1(1):5--33, 1999.

%\bibitem{batyrev2} Batyrev, Victor V., Canonical abelianization of finite group actions. Preprint: arXiv:math/0009043, 2000.

\bibitem{bhargavasatriano} Bhargava, Manjul and Matthew Satriano, On a notion of ``Galois closure" for extensions of rings.  {\em J. Eur. Math. Soc. (JEMS)} 16(9):1881--1913, 2014.

\bibitem{biesel} Biesel, Owen, {\em Galois closures for rings}. PhD thesis, Princeton University, Princeton, NJ 2013.

\bibitem{bjornerbrenti} Bj\"{o}rner, Anders and Francesco Brenti, {\em Combinatorics of Coxeter Groups}. Springer, New York, New York, 2005.

\bibitem{bjorner2} Bj\"{o}rner, Anders, Posets, regular CW complexes, and Bruhat order. {\em European Journal of Combinatorics} 5:7--16, 1984.

\bibitem{bjorner} Bj\"{o}rner, Anders, Subspace arrangements. In A. Joseph et. al., eds., {\em First European Congress of Mathematics, Volume I: Invited Lectures (Part 1)}, 321--370, Birkh\"{a}user, Basel, Germany, 1994.

\bibitem{FTSP} Blum-Smith, Ben and Samuel Coskey, The fundamental theorem on symmetric polynomials: history's first whiff of Galois theory. {\em College Mathematics Journal}, 48(1):18-29, 2017.

%\bibitem{bogomolov} Bogomolov, Fedor, Stable cohomology of groups and algebraic varieties. {\em Russian Acad. Sci. Sb. Math.}, 76(1):1--21, 1993.

\bibitem{brown} Brown, Kenneth S., {\em Buildings}. Springer, New York, New York, 1989.

\bibitem{brunsherzog}
Bruns, Winfried and J\"{u}rgen Herzog, \textit{Cohen-Macaulay Rings}. Cambridge studies in advanced mathematics. Cambridge University Press, Cambridge, 1993.

\bibitem{campbelletal} Campbell, H. E. A., A. V. Geramita, I. P. Hughes, R. J. Shank, and D. L. Wehlau, Non-Cohen-Macaulay vector invariants and a Noether bound for a Gorenstein ring of invariants. {\em Canadian Mathematical Bulletin} 42(2):155-161, 1999.

\bibitem{campbellwehlau} Campbell, H. E. A. Eddy, and David L. Wehlau, {\em Modular Invariant Theory}. Springer-Verlag, Berlin, 2011.

\bibitem{cox} Cox, David A., John B. Little, and Henry K. Schenck, {\em Toric Varieties}. The American Mathematical Society, Providence, Rhode Island, 2011.

\bibitem{danarajklee} Danaraj, Gopal and Victor Klee, Shellings of spheres and polytopes. {\em Duke Mathematical Journal} 41:443-451, 1974.

\bibitem{davis} Davis, Michael W., {\em The Geometry and Topology of Coxeter Groups}. Princeton University Press, Princeton, New Jersey, 2008.

\bibitem{hodge} De Concini, Corrado, David Eisenbud, and Claudio Procesi, Hodge algebras. {\em Ast\'{e}risque} 91: 1-87, 1982.

\bibitem{derksenkemper} Derksen, Harm and Gregor Kemper, {\em Computational Invariant Theory}. Springer-Verlag, Berlin Heidelberg, 2002.

\bibitem{dickson} Dickson, L. E., {\em Linear Groups, with an Exposition of the Galois Field Theory}. Teubner, Leipzig, 1901.

\bibitem{dornhoff} Dornhoff, Larry, {\em Group Representation Theory, Part A: Ordinary Representation Theory}. Dekker, New York, New York, 1971.

\bibitem{duvaletal} Duval, Art M., Bennet Goeckner, Caroline J. Klivans, and Jeremy L. Martin, A non-partitionable Cohen-Macaulay complex. {\em Advances in Mathematics} 299:381--395, 2016.

\bibitem{duval} Duval, Art M., Free resolutions of simplicial posets. {\em Journal of Algebra} 188: 363--399, 1997.

\bibitem{edwards} Edwards, Harold, {\em Galois Theory}. Springer-Verlag, New York, 1984.

\bibitem{eisenbud} Eisenbud, David, {\em Commutative Algebra with a View Toward Algebraic Geometry}. Springer, New York, New York, 2004.

\bibitem{eisenbud2} Eisenbud, David, Introduction to algebras with straightening laws. {\em Ring theory and algebra III (Proceedings of the Third Oklahoma Conference)} 243--268, Lecture Notes in Pure and Applied Mathematics 55, Dekker, New York, New York,1980.

\bibitem{ellingsrud} Ellingsrud, Geir and Tor Sjelbred, Profondeur d'anneaux d'invariants en caract\'{e}ristique $p$. {\em Compositio Mathematica} 41(2): 233-244, 1980.

\bibitem{ferrario} Ferrario, Riccardo, {\em Galois closures for monogenic degree-4 extensions of rings}. Master's thesis, Universiteit Leiden / Universita degli Studi di Padova, Leiden, Netherlands, 2014.

\bibitem{fleischmann} Fleischmann, Peter, The Noether bound in invariant theory of finite groups. {Advances in Mathematics} 156:23-32, 2000.

\bibitem{fogarty} Fogarty, John, On Noether's bound for polynomial invariants of a finite group. {Electronic Research Announcements of the American Mathematical Society} 7:5-7, 2001.

\bibitem{forman1} Forman, Robin, Morse theory for cell complexes. {\em Advances in Mathematics} 134:90-145, 1998.

\bibitem{forman2} Forman, Robin, A user's guide to discrete Morse theory. {\em S\'{e}minaire Lotharingien de Combinatoire} 48: Art B48c, 35pp., 2002.

\bibitem{franciscoetal} Francisco, Christopher A, Jeffrey Mermin, and Jay Schweig, A survey of Stanley-Reisner theory. In {\em Connections Between Algebra, Combinatorics, and Geometry} 209--234, Springer, New York, 2014.

\bibitem{garsia} Garsia, Adriano M., Combinatorial methods in the theory of Cohen-Macaulay rings. {\em Advances in Mathematics} 38, 229--266, 1980.

\bibitem{garsiastanton} Garsia, A. M. and D. Stanton, Group actions on Stanley-Reisner rings and invariants of permutation groups. {\em Advances in Mathematics} 51, 107--201, 1984.

\bibitem{gauss} Carl Friedrich Gauss, Demonstratio nova altera theorematis omnem functionem algebraicam rationalem integram unius variabilis in factores reales primi vel secundi gradus resolvi posse, \textit{Comm. Recentiores} 3:107--142, 1816, reprinted in \textit{Werke}, vol. 3, 31--56.

\bibitem{gobel} G\"{o}bel, Manfred, Computing bases for rings of permutation-invariant polynomials. {\em J. Symbolic Computation} 19, 285--291, 1995.

%\bibitem{gordeevkemper} Gordeev, Nikolai and Gregor Kemper, On the branch locus of quotients by finite groups and the depth of the algebra of invariants. {\em Journal of Algebra} 268:22-38, 2003.

\bibitem{hatcher} Hatcher, Allen, {\em Algebraic Topology}. Cambridge University Press, Cambridge, United Kingdom, 2002.

\bibitem{hersh} Hersh, Patricia, Lexicographic shellability for balanced complexes. {\em Journal of Algebraic Combinatorics} 17, 225--254, 2003.

\bibitem{hersh2} Hersh, Patricia, A partitioning and related properties for the quotient complex $\Delta(B_{lm})/S_l\wr S_m$. With an appendix by Victor Reiner. {\em Journal of Pure and Applied Algebra} 178, 255--272, 2003.

\bibitem{hilbert1} Hilbert, David, \"{U}ber die Theorie der algebraischen Formen, {\em Mathematische Annalen} 36: 473--531, 1890.

\bibitem{hilbert2} Hilbert, David, \"{U}ber die vollen Invariantensysteme, {\em Mathematische Annalen} 42: 313--370, 1893.

\bibitem{hochstereagon} Hochster, Melvin, and John A. Eagon, Cohen-Macaulay rings, invariant theory, and the generic perfection of determinantal loci. {\em American Journal of Mathematics}, 93(4): 1020-1058, 1971.

\bibitem{huffman} Huffman, W. Cary, Imprimitive linear groups generated by elements containing an eigenspace of codimension two. {\em Journal of Algebra} 63: 499--513, 1980.

\bibitem{humphreys2} Humphreys, James E., {\em Reflection Groups and Coxeter Groups}. Cambridge University Press, Cambridge, United Kingdom, 1990.

%\bibitem{humphreys} Humphreys, James E., The representations of $SL(2,p)$. {\em The American Mathematical Monthly} 82:21--39, 1975.

\bibitem{jacobson} Jacobson, Nathan, {\em Basic Algebra I}. 2nd ed. W. H. Freeman and Co., San Francisco, 1985.

\bibitem{kaplansky} Kaplansky, Irving, Projective modules. {\em Annals of Mathematics, Second Series} 68:372-377, 1958.

\bibitem{kemper99} Kemper, Gregor, On the Cohen-Macaulay property of modular invariant rings. {\em Journal of Algebra} 215(1):330-351, 1999.

\bibitem{kemper01} Kemper, The depth of invariant rings and cohomology. {\em Journal of Algebra} 245(2): 463--531, 2001.

\bibitem{kemper12} Kemper, Gregor, The Cohen-Macaulay property and depth in invariant theory. In {\em Proceedings of the 33rd Symposium on Commutative Algebra in Japan} 53--63, 2012.

\bibitem{king} King, Oliver H., The subgroup structure of finite classical groups in terms of geometric configurations. {\em Surveys in Combinatorics} 5:29-56, 2005.

\bibitem{kronecker} Kronecker, Leopold, Grundz\"{u}ge einer arithmetischen theorie der algebraischen gr\"{o}ssen, \textit{Crelle, Journal f\"{u}r die reine und angewandte Mathematik} 92:1-122, 1881, reprinted in \textit{Werke}, vol. 2, 237--387.

\bibitem{lam} Lam, T. Y., {\em Serre's Problem on Projective Modules}. Springer-Verlag, Berlin -- Heidelberg, Germany, 2006.

\bibitem{lange} Lange, Christian, {\em Some results on orbifold quotients and related objects}. PhD thesis, University of K\"{o}ln, K\"{o}ln, Germany, 2016.

\bibitem{lange2} Lange, Christian, Characterization of finite groups generated by reflections and rotations. {\em Journal of Topology} 9(4):1109--1129, 2016.

\bibitem{langemikhailova} Lange, Christian, and Marina Mikha\^{i}lova, Classification of finite groups generated by reflections and rotations. {\em Transformation Groups} 21(4):1155--1201, 2016.

\bibitem{lehrertaylor} Lehrer, Gustav I. and Donald E. Taylor, {\em Unitary Reflection Groups}. Cambridge University Press, Cambridge, United Kingdom, 2009.

\bibitem{mastnaka} Mastnaka, Mitja and Heydar Radjavib, Structure of finite, minimal nonabelian groups and triangularization. {\em Linear Algebra and Its Applications} 430(7):1838--1848, 2009.

\bibitem{maerchik} Maerchik, M. A. (maiden name of Mikhailova), Finite groups generated by pseudoreflections in four-dimensional Euclidean space. {\em Trudy Kirgiz Gos. Univ. Ser. Mat. Nauk} 11:66--72, 1976. (Russian)

\bibitem{mikhailova78} Mikhailova, M. A. Finite imprimitive groups generated by pseudoreflections. {\em Studies in geometry and algebra, Kirgiz. Gos. Univ., Frunze} 82--93, 1978. (Russian)

\bibitem{mikhailova82} Mikhailova, M. A. Finite reducible groups generated by pseudoreflections, deposited at VINITI, man. no. 1248-82, 1982. (Russian)

\bibitem{mikhailova85} Mikhailova, M. A. On the quotient space modulo the action of a finite group generated by pseudoreflections. {\em Math. USSR-Izvestiya} 24(1):99--119, 1985.

\bibitem{millermoreno} Miller, G. A. and H. C. Moreno, Non-abelian groups in which every subgroup is abelian. {\em Transactions of the American Mathematical Society}, 4(4):398--404, 1903.

\bibitem{miyazaki} Miyazaki, Mitsuhiro, On the discrete counterparts to algebras with straightening laws. {\em Journal of Commutative Algebra} 2(1): 79-89, 2010.

\bibitem{munkres} Munkres, James R., Topological results in combinatorics. {\em Michigan Math. J.} 31, 113--128, 1984.

\bibitem{neuselsmith} Neusel, Mara and Larry Smith, {\em Invariant Theory of Finite Groups}. American Mathematical Society, Providence, Rhode Island, 2002.

\bibitem{neuselbounds} Neusel, Mara, Degree bounds -- an invitation to postmodern invariant theory. {\em Topology and Its Applications} 154:792--814, 2007.

\bibitem{piatetskishapiro} Piatetski-Shapiro, Ilya, {\em Complex Representations of $GL(2,K)$ for Finite Fields $K$}. Volume 16. American Mathematical Society, 1983.

\bibitem{reiner92} Reiner, Victor, {\em Quotients of Coxeter complexes and $P$-partitions}. American Mathematical Society, Providence, Rhode Island, 1992.

\bibitem{reiner95} Reiner, Victor, On G\"{o}bel's bound for invariants of permutation groups. {\em Archiv der Mathematik} 65(6):475--480, 1995.

\bibitem{reisner} Reisner, Gerald Allen, Cohen-Macaulay quotients of polynomial rings. {\em Advances in Mathematics} 21:30--49, 1976.

\bibitem{schmid} Schmid, Barbara, Finite groups and invariant theory. {\em Topics in Invariant Theory.} Lecture Notes in Mathematics, vol. 1478. Springer, Berlin, Heidelberg, 1991.

\bibitem{serre} Serre, Jean-Pierre, {\em Linear Representations of Finite Groups}. Springer-Verlag, New York, New York, 1977.

\bibitem{smith95} Smith, Larry, {\em Polynomial Invariants of Finite Groups}. A K Peters, Wellesley, MA, 1995.

\bibitem{smith96} Smith, Larry, Some rings of invariants that are Cohen-Macaulay. {\em Canadian Mathematical Bulletin} 39(2):238--240, 1996.

\bibitem{smith97} Smith, Larry, Polynomial invariants of finite groups: a survey of recent developments. {\em Bulletin of the American Mathematical Society} 34(3):211--250, 1997.

\bibitem{stanley4} Stanley, Richard P., Invariants of finite groups and their applications to combinatorics. {\em Bulletin of the American Mathematical Society} 1(3):475--511, 1979.

\bibitem{stanley3} Stanley, Richard P., Balanced Cohen-Macaulay complexes. {\em Transactions of the American Mathematical Society} 249(1):139--157, 1979.

\bibitem{stanley} Stanley, Richard P., $f$-vectors and $h$-vectors of simplicial posets. {\em Journal of Pure and Applied Algebra} 71:319--331, 1991.

\bibitem{cca} Stanley, Richard P., {\em Combinatorics and Commutative Algebra}. 2nd ed. Birkh\"{a}user, Boston, Massachusetts, 1996.

\bibitem{stanley14} Stanley, Richard P., How the upper bound conjecture was proved. {\em Annals of Combinatorics} 18(3):533--539, 2014.

\bibitem{sturmfels} Sturmfels, Bernd, {\em Algorithms in Invariant Theory}. 2nd ed. Springer-Verlag/Wien, M\"{o}rlenbach, Germany, 2008.

\bibitem{suzuki} Suzuki, Michio, On a class of doubly transitive groups. {\em Annals of Mathematics}, Second Series, 75(1):105--145, 1962.

\bibitem{terai} Terai, Naoki, Some remarks on algebras with straightening laws. {\em Journal of Pure and Applied Algebra} 95: 87-101, 1994.

\bibitem{thompson} Thompson, John G., Nonsolvable finite groups all of whose local subgroups are solvable. {\em Bulletin of the American Mathematical Society} 74:383--437, 1968.

\bibitem{tignol} Tignol, Jean-Pierre, {\em Galois' theory of algebraic equations}. World Scientific Publishing Co., Inc., River Edge, NJ, 2001.

%\bibitem{vistoli} Vistoli, Angelo, Intersection theory on algebraic stacks and on their moduli spaces. {\em Inventiones Mathematicae} 97: 613--670, 1989.

\bibitem{wachs} Wachs, Michelle, Poset topology: tools and applications. In Ezra Miller, Victor Reiner, and Bernd Sturmfels, eds., {\em Geometric Combinatorics: lectures from the Graduate Summer School held in Park City, UT, 2004}, IAS/Park City Mathematics Series, 13:497--615, American Mathematical Society, Providence, Rhode Island -- Institute for Advanced Study, Princeton, New Jersey, 2007.

\bibitem{waring} Waring, E., {\em Meditationes Algebraicae}. 3rd ed., 1782. English translation by Dennis Weeks, American Mathematical Society, Providence, Rhode Island, 1991.

\end{thebibliography}
\end{document}